\documentclass[a4paper,12pt,reqno,twoside,openright]{memoir}
\usepackage[utf8]{inputenc}
\usepackage{amsmath}
\usepackage{amssymb}
\usepackage{amsthm}
\usepackage[arrow, matrix, curve]{xy}
\usepackage{color}
\usepackage{enumerate}
\usepackage{a4wide}
\usepackage{rotating}
\usepackage{longtable}
\usepackage{imakeidx}
\makeindex[intoc]
\usepackage{hyperref}

\hypersetup{
	pdftitle={Conformally invariant differential operators on Heisenberg groups and minimal representations},
	pdfauthor={Jan Frahm},
	bookmarksnumbered=true,
	bookmarksopen=true,
	bookmarksopenlevel=1,
	colorlinks=true,            
	pdfstartview=Fit,           
	pdfpagemode=UseOutlines,
	pdfpagelayout=TwoPageRight
}

\makeatletter

\DeclareMathOperator{\GL}{GL}
\DeclareMathOperator{\gl}{\mathfrak{gl}}
\DeclareMathOperator{\SL}{SL}

\let\sl\relax
\DeclareMathOperator{\sl}{\mathfrak{sl}}
\DeclareMathOperator{\upO}{O}
\DeclareMathOperator{\SO}{SO}
\DeclareMathOperator{\so}{\mathfrak{so}}
\DeclareMathOperator{\Sp}{Sp}
\DeclareMathOperator{\Mp}{Mp}
\let\sp\relax
\DeclareMathOperator{\sp}{\mathfrak{sp}}
\DeclareMathOperator{\SU}{SU}
\DeclareMathOperator{\su}{\mathfrak{su}}

\newcommand{\fraka}{\mathfrak{a}}
\newcommand{\frake}{\mathfrak{e}}
\newcommand{\frakf}{\mathfrak{f}}
\newcommand{\frakg}{\mathfrak{g}}
\newcommand{\frakh}{\mathfrak{h}}
\newcommand{\frakk}{\mathfrak{k}}

\newcommand{\frakm}{\mathfrak{m}}
\newcommand{\frakn}{\mathfrak{n}}
\newcommand{\frakp}{\mathfrak{p}}
\newcommand{\frakq}{\mathfrak{q}}
\newcommand{\fraks}{\mathfrak{s}}
\newcommand{\fraku}{\mathfrak{u}}

\newcommand{\CC}{\mathbb{C}}

\newcommand{\EE}{\mathbb{E}}
\newcommand{\FF}{\mathbb{F}}
\newcommand{\HH}{\mathbb{H}}
\newcommand{\NN}{\mathbb{N}}
\newcommand{\OO}{\mathbb{O}}

\newcommand{\RR}{\mathbb{R}}
\newcommand{\ZZ}{\mathbb{Z}}

\newcommand{\calC}{\mathcal{C}}
\newcommand{\calD}{\mathcal{D}}

\newcommand{\calF}{\mathcal{F}}
\newcommand{\calH}{\mathcal{H}}
\newcommand{\calJ}{\mathcal{J}}
\newcommand{\calO}{\mathcal{O}}

\newcommand{\calS}{\mathcal{S}}

\newcommand{\calU}{\mathcal{U}}

\newcommand{\0}{\textbf{0}}
\renewcommand{\1}{\textbf{1}}

\DeclareMathOperator{\Ind}{Ind}
\DeclareMathOperator{\tr}{tr}
\DeclareMathOperator{\ad}{ad}
\DeclareMathOperator{\Ad}{Ad}
\DeclareMathOperator{\Cas}{Cas}
\DeclareMathOperator{\End}{End}
\DeclareMathOperator{\Hom}{Hom}
\DeclareMathOperator{\Sym}{Sym}
\DeclareMathOperator{\Herm}{Herm}

\DeclareMathOperator{\id}{id}
\DeclareMathOperator{\sgn}{sgn}
\DeclareMathOperator{\const}{const}
\DeclareMathOperator{\diag}{diag}

\DeclareMathOperator{\met}{met}
\DeclareMathOperator{\Bz}{Bz}
\DeclareMathOperator{\supp}{supp}
\renewcommand\Re{\operatorname{Re}}

\newcommand{\Omin}{\mathcal{O}_{\textup{min}}}
\newcommand{\HS}{{\textup{HS}}}
\renewcommand{\min}{{\textup{min}}}
\DeclareMathOperator{\otimeshat}{\widehat{\otimes}}

\theoremstyle{plain}
\newtheorem{theorem}{Theorem}[section]
\newtheorem{proposition}[theorem]{Proposition}
\newtheorem{lemma}[theorem]{Lemma}
\newtheorem{corollary}[theorem]{Corollary}

\newtheorem{thmalph}{Theorem}

\theoremstyle{definition}
\newtheorem{definition}[theorem]{Definition}

\newtheorem{remark}[theorem]{Remark}

\numberwithin{equation}{section}

\title{Conformally invariant differential operators on Heisenberg groups and minimal representations}

\author{Jan Frahm}
\date{\small April 27, 2022}

\begin{document}
	
	\maketitle
	
	\begin{abstract}
		For a simple real Lie group $G$ with Heisenberg parabolic subgroup $P$, we study the corresponding degenerate principal series representations. For a certain induction parameter the kernel of the conformally invariant system of second order differential operators constructed by Barchini, Kable and Zierau is a subrepresentation which turns out to be the minimal representation. To study this subrepresentation, we take the Heisenberg group Fourier transform in the non-compact picture and show that it yields a new realization of the minimal representation on a space of $L^2$-functions. The Lie algebra action is given by differential operators of order $\leq3$ and we find explicit formulas for the functions constituting the lowest $K$-type.
		
		These $L^2$-models were previously known for the groups $\SO(n,n)$, $E_{6(6)}$, $E_{7(7)}$ and $E_{8(8)}$ by Kazhdan and Savin, for the group $G_{2(2)}$ by Gelfand, and for the group $\widetilde{\SL}(3,\RR)$ by Torasso, using different methods. Our new approach provides a uniform and systematic treatment of these cases and also constructs new $L^2$-models for $E_{6(2)}$, $E_{7(-5)}$ and $E_{8(-24)}$ for which the minimal representation is a continuation of the quaternionic discrete series, and for the groups $\widetilde{\SO}(p,q)$ with either $p\geq q=3$ or $p,q\geq4$ and $p+q$ even. 
		
		As a byproduct of our construction, we find an explicit formula for the group action of a non-trivial Weyl group element that, together with the simple action of a parabolic subgroup, generates $G$.
	\end{abstract}

\vspace*{\fill}

\small\textit{Address:} Department of Mathematics, Aarhus University, Ny Munkegade 118, 8000 Aarhus, Denmark

\small\textit{E-Mail:} \texttt{frahm@math.au.dk}



\cleardoublepage

\setcounter{tocdepth}{1}

\tableofcontents

\cleardoublepage

\chapter*{Introduction}
\addcontentsline{toc}{chapter}{Introduction}

The classification of all irreducible unitary representations of a semisimple Lie group is one of the key problems in representation theory, and it is still unsolved for most groups. A guiding principle for the classification is the orbit philosophy which proposes a tight relation between the unitary dual of a semisimple group $G$ and the set of coadjoint orbits in the dual space $\frakg^*$ of the Lie algebra $\frakg$ of $G$. For elliptic and hyperbolic coadjoint orbits, cohomological and parabolic induction provide explicit constructions of the corresponding unitary representations, and the resulting representations make up a large part of the unitary dual. For nilpotent orbits, however, it is not clear in general how to apply the orbit philosophy to construct unitary representations. Among the finitely many nilpotent coadjoint orbits, there are one or two of minimal dimension, depending on whether the group is of Hermitian type or not. Irreducible unitary representations corresponding to a minimal nilpotent coadjoint orbit are called \emph{minimal representations}. They are often unique and in general a group can only have finitely many equivalence classes of minimal representations.

The most prominent example of a minimal representation is the metaplectic representation (also referred to as oscillator or Segal--Shale--Weil representation) of the metaplectic group $\Mp(n,\RR)$, a double cover of the symplectic group $\Sp(n,\RR)$ (see \cite{Wei64} for Weil's original work). We refer the reader to Folland's book \cite{Fol89} for a detailed account on the construction of this representation and some of its properties. Although the metaplectic representation plays an important role within the representation theory of the metaplectic group, it is mostly its relevance in other areas of mathematics and physics that have made this particular representation a truly fascinating object within the last few decades.

The key role of the metaplectic representation in the representation theory of real reductive groups is in the context of the theta correspondence, also referred to as Howe's dual pair correspondence. This correspondence was defined by Howe~\cite{How79} and relates irreducible representations of two different groups $G_1$ and $G_2$ that occur inside the metaplectic group $\Mp(n,\RR)$ as a so-called \emph{dual pair}, i.e. $G_1$ and $G_2$ are mutual centralizers of each other. The theta correspondence is a map that associates to a representation $\pi$ of $G_1$ occurring inside the metaplectic representation a representation $\theta(\pi)$ of $G_2$ such that $\pi\otimes\theta(\pi)$ occurs as a quotient of the metaplectic representation restricted to $G_1\times G_2$. It was shown by Howe~\cite{How89} that this establishes a bijection between certain irreducible representations of $G_1$ and $G_2$.

The metaplectic representation also has a version over non-Archimedean local fields, over global fields and even over finite fields. Corresponding theta correspondences were established by Waldspurger~\cite{Wal90}, M\'{\i}nguez~\cite{Min08}, Gan--Takeda~\cite{GT16} and Gan--Sun~\cite{GS17} for the case of local non-Archimedean fields, and by Rallis~\cite{Ral84} for global fields, and this is still a very active line of research. For instance, the theta correspondence has been applied to the construction of new representations, to classification problems, and to branching problems.

Apart from the theta correspondence, the metaplectic representation has also been used in various other contexts such as classical invariant theory~\cite{How89,HTW05}, theta series and the Maslov index~\cite{LV80}, the Siegel--Weil formula in automorphic forms~\cite{Gan00,KR88,Wei65}, harmonic analysis~\cite{Fol89,GS81} and quantum mechanics~\cite{Woi17}, just to mention a few.\\

It seems natural to try to extend the extremely rich theory of the metaplectic representation as a minimal representation of the metaplectic group $\Mp(n,\RR)$ to minimal representations of more general reductive groups. The first step in this program is the construction and classification of all minimal representations. For reductive groups over Archimedean local fields, many different constructions can be found in the literature. Let us list a few of them without claiming to be complete: Brylinski--Kostant~\cite{BK94}, Binegar--Zierau~\cite{BZ91}, Dvorsky--Sahi~\cite{DS99}, Gelfand~\cite{Gel80}, Gross--Wallach~\cite{GW96}, Hilgert--Kobayashi--M\"{o}llers~\cite{HKM14}, Kazhdan--Savin~\cite{KS90}, Kobayashi--{\O}rsted~\cite{KO03b}, Li~\cite{Li00}, M\"{o}llers--Schwarz~\cite{MS17}, Sabourin~\cite{Sab96}, Savin~\cite{Sav93}, Torasso~\cite{Tor83,Tor97}, Vogan~\cite{Vog81,Vog94}. It was believed that the thus obtained list of minimal representations is in fact exhaustive, and this was recently shown by Tamori~\cite{Tam19}, building on ideas of Gan--Savin~\cite{GS05}. For $p$-adic groups, minimal representations were constructed by Kazhdan--Savin~\cite{KS90}, Rumelhart~\cite{Rum97}, Savin~\cite{Sav93}, Torasso~\cite{Tor97} and Weissman~\cite{Wei03}. An overview together with a global picture can be found in the paper by Gan and Savin~\cite{GS05}.

Subsequently, minimal representations of general reductive groups over local fields have turned out to be useful from various different points of view. For instance, their geometric realizations are a particularly rich source to do classical harmonic analysis, a viewpoint that has been particularly advocated by Kobayashi~\cite{Kob14}. Special functions such as orthogonal polynomials or Bessel functions are often used to express explicit $K$-finite vectors in the representation spaces, see e.g. \cite{HKMM11}. Moreover, explicit geometric realizations are often connected to interesting analytic and geometric problems such as partial differential equations on manifolds, see~\cite{KO03a,KO03b,KO03c}. Many of these features also become apparent in this work.

Minimal representations can further be used to study certain Fourier coefficients of automorphic forms \cite{BS18,GGKPS22a,GGKPS22b,KPW02} and to construct more general theta series, see \cite{KPW02}. Finally, versions of the theta correspondence have been established in some cases involving exceptional groups, see e.g. the work of Ginzburg--Jiang~\cite{GJ01}, Ginzburg--Rallis--Soudry~\cite{GRS97}, Gross--Savin~\cite{GS98}, Huang--Pandzic--Savin~\cite{HPS96}, Li~\cite{Li97,Li00}, Loke--Savin~\cite{LS19}, Magaard--Savin~\cite{MS97} and Weissman~\cite{Wei06}, just to name a few.\\

What many of these applications have in common is that they rely on a rather explicit realization of the minimal representation. In particular in the Archimedean setting, where delicate analytic problems arise, it is often vital to be able to construct explicit vectors in the representation spaces and compute group actions on them. There is one particular type of realization that has turned out to be extremely useful for such purposes, a realization referred to as \emph{$L^2$-model} or \emph{Schr\"{o}dinger model}. In the Archimedean context, one could define an $L^2$-model of a unitary representation as a realization on a Hilbert space of $L^2$-functions such that the Lie algebra acts by differential operators. $L^2$-models of minimal representations have essentially been constructed in two different settings which we now briefly describe.

The first class of groups for which $L^2$-models of minimal representations have been constructed in a uniform way are simple Lie groups $G$ possessing a parabolic subgroup $P=MAN$ with abelian nilradical $N$, also referred to as \emph{Siegel parabolic subgroups}. For such groups $G$, minimal representations can often be found as proper subrepresentations of the corresponding degenerate principal series $\Ind_P^G(\chi)$ for a certain real-valued character $\chi$ of $P$. The subrepresentations arise as the kernel of a system of second order differential operators. These differential operators can most easily be described in the so-called \emph{non-compact picture} of the degenerate principal series. The non-compact picture is a realization on a space of functions on the opposite nilradical $\overline{N}$. Identifying $\overline{N}$ with its Lie algebra $\overline{\frakn}$, the differential operators become second order constant coefficient differential operators on $\overline{\frakn}$. For instance, for $G=\upO(p,q)$ we have $\overline{\frakn}\simeq\RR^{p+q-2}$ and the system consists of a single differential operator whose symbol is a quadratic form of signature $(p-1,q-1)$ (see \cite{KO03a}), and for $G=\Mp(n,\RR)$ the system of differential operators on $\overline{\frakn}\simeq\Sym(n,\RR)$, the real symmetric $n\times n$ matrices, has as symbols the $2\times2$ minors.

In order to obtain an $L^2$-model of this subrepresentation, the Euclidean Fourier transform $\calS'(\overline{\frakn})\to\calS'(\frakn),\,u\mapsto\widehat{u},$ is employed. It turns the system of constant coefficient differential operators $P(\partial)$ into a system of multiplication operators $P(x)$ and hence the differential equations $P(\partial)u=0$ turn into $P(x)\widehat{u}=0$ which implies a support condition $\supp\widehat{u}\subseteq\{P(x)=0\}$ on $\frakn$. For instance, for $G=\upO(p,q)$ the Fourier transform of a function in the subrepresentation is a distribution supported on the isotropic cone in $\frakn\simeq\RR^{p+q-2}$ associated with a quadratic form of signature $(p-1,q-1)$, and for $G=\Mp(n,\RR)$ they are supported on the zero set of all $2\times2$ minors in $\frakn\simeq\Sym(n,\RR)$ which is the subvariety of rank one symmetric $n\times n$ matrices. In general, the submanifold on which these Fourier transforms are supported is an orbit $\calO\subseteq\frakn$ of $\Ad(MA)$ of minimal possible dimension, and one can show that in many cases this yields a realization of the minimal representation on $L^2(\calO,d\mu)$ with respect to a certain $\Ad(MA)$-equivariant measure $d\mu$ on $\calO$. This was first observed by Vergne--Rossi~\cite{VR76} in the case of Hermitian groups and later generalized by Dvorsky--Sahi~\cite{DS99}, Kobayashi--{\O}rsted~\cite{KO03b} and M\"{o}llers--Schwarz~\cite{MS17} to cover all cases (see also Goncharov~\cite{Gon82} for the underlying Lie algebra representation and Hilgert--Kobayashi--M\"{o}llers~\cite{HKM14} for a uniform construction including the full action of the Lie algebra).

The second class of groups where $L^2$-models of minimal representations are known are simple real Lie groups $G$ having a Heisenberg parabolic subgroup $P=MAN$, i.e. the nilradical $N$ is a Heisenberg group. For some of these groups $G$, $L^2$-models of minimal representation can be found in the literature, but the constructions differ from case to case, and some groups are not treated at all although they do possess a minimal representation. The probably most famous representation among them is Torasso's representation of $G=\widetilde{\SL}(3,\RR)$, the double cover of $\SL(3,\RR)$, see \cite{Tor83}. Strictly speaking, this representation is not minimal, because there is no notion of minimality for type $A$ groups, but it is thought to correspond to the minimal nilpotent coadjoint orbit in a certain sense. Torasso's representation is realized on $L^2(\RR^\times\times\RR)$ and he provides explicit formulas for the action of both the group $G$ and its Lie algebra as well as for the lowest $K$-type. Similar formulas for the group and Lie algebra action can be found in the construction of Kazhdan and Savin~\cite{KS90} for the split simply-laced groups $\SO(n,n)$, $E_{6(6)}$, $E_{7(7)}$ and $E_{8(8)}$, although the method seems to be quite different. They construct in a natural way a representation of $P$ on $L^2(\RR^\times\times\Lambda)$ for some Lagrangian subspace $\Lambda\subseteq V$ of the symplectic vector space $V$ defining the Heisenberg group $N$, and prove that this representation extends in a unique way to a minimal representation of $G$. What remains mysterious in their construction is why the extension from $P$ to $G$ exists. Later, Savin~\cite{Sav93} applied the same construction to the group $G=G_2(\CC)$. We remark that in the case of $G=G_{2(2)}$, formulas for the Lie algebra action already appear in the work of Gelfand~\cite{Gel80}, but without reference to the unitary structure of this representation. Later, Sabourin~\cite{Sab96} obtained a similar model for the minimal representation of $\widetilde{\SO}(4,3)$ using a variant of the orbit method. Note that some of these constructions also work over non-Archimedean local fields.

While the above constructions for groups with Heisenberg parabolic subgroup are different in nature, the obtained realizations seem to be closely related. The main motivation for this work was to find a uniform construction of all these minimal representations in the spirit of the construction in the case of Siegel parabolic subgroups. The key ideas here were to exhibit the minimal representation inside a degenerate principal series $\Ind_P^G(\chi)$ as the kernel of a system of differential operators, and then take an appropriate Fourier transform. Roughly speaking, these ideas can indeed be applied in the Heisenberg parabolic case. In fact, several authors already noted that the minimal representation is a subrepresentation of a degenerate principal series of the form $\Ind_P^G(\chi)$, see \cite[Corollary 13.7 and Proposition 14.11]{GW96} for $\SO(4,4)$, $E_{6(2)}$, $E_{7(-5)}$, $E_{8(-24)}$ and $G_{2(2)}$ and \cite[Theorem 4.2.2]{Sab96} for $\widetilde{\SO}(4,3)$ (see also \cite[Theorem 4]{KS90} for the $p$-adic groups $D_4$, $E_6$, $E_7$ and $E_8$). In some of these works, it is even indicated that the minimal representation is the kernel of a system of second order differential operators, but this is not pursued further.

More recently, Barchini, Kable and Zierau~\cite{BKZ08} constructed in a systematic way \emph{conformally invariant systems of differential operators} on the opposite nilradical $\overline{N}$ of a Heisenberg parabolic subgroup $P$. As we explain in detail below, these systems are related to certain polynomial maps on the symplectic vector space defining the Heisenberg group $\overline{N}$ that we call \emph{symplectic invariants}. There are symplectic invariants of order one (the symplectic form), two (a moment map), three and four, and each of them gives rise to a system of differential operators of the same order whose joint kernel is a subrepresentation of a certain degenerate principal series $\Ind_P^G(\chi)$. In this work we focus on the system of order two and show that its joint kernel is in many cases a minimal representation. For this, we make extensive use of the structure theory for Heisenberg parabolic subalgebras developed by Slupinski--Stanton~\cite{SS,SS12,SS15}.

Let us remark that the joint kernel of the second order conformally invariant operators has been studied in a few examples, mostly algebraically (see e.g. the work of Kable~\cite{Kab11,Kab12a,Kab12b} and Kubo--{\O}rsted~\cite{KO19}). In particular, an analysis of the corresponding representations was missing so far, probably due to the fact that the differential operators do no longer have constant coefficients but are left-invariant operators on the Heisenberg group. This suggests to consider the Heisenberg group Fourier transform in order to understand the joint kernel of these differential operators, and in this work we attempt to carry out this analysis in a uniform way in the case where the group $G$ is non-Hermitian. For the Hermitian case we refer to a subsequent paper~\cite{FWZ22}.

The Heisenberg group Fourier transform is more difficult to deal with since the Fourier transform of a function is operator-valued, thus adding a non-commutative flavour to the theory. Moreover, it is not clear how to define the Fourier transform of a general tempered distribution on the Heisenberg group. It turns out that for the non-compact picture of the degenerate principal series $\Ind_P^G(\chi)$ one can make sense of the Heisenberg group Fourier transform in the distribution sense if the character is sufficiently positive. The first key observation is that the second order conformally invariant system of Barchini--Kable--Zierau can be expressed in terms of the metaplectic representation on the Fourier transformed side (see Theorem~\ref{thm:IntroThmA}). This allows us to solve the corresponding system of equations on the Fourier transformed side in the distribution sense (see Theorem~\ref{thm:IntroThmB}) and obtain a new realization of their joint kernel (see Theorem~\ref{thm:IntroThmC}). Already at this point, we recognize the same formulas for the Lie algebra representation as in the work of Gelfand~\cite{Gel80}, Torasso~\cite{Tor83} and Kazhdan--Savin~\cite{KS90}, thus obtaining both a new conceptual explanation for their constructions as well as a uniform treatment. However, it is not immediate at this point whether there are $K$-finite vectors among the solutions, so we find explicit formulas for the functions constituting the lowest $K$-type in this new realization (see Theorem~\ref{thm:IntroThmD}). The lowest $K$-type generates an irreducible $(\frakg,K)$-module which we can integrate to a minimal representation of (a finite cover of) $G$, realized on an $L^2$-space (see Theorem~\ref{thmintro:IntegrationMinRep}). Finally, in the spirit of the work of Kazhdan--Savin~\cite{KS90} and Kobayashi--Mano~\cite{KM11}, we obtain the group action of a Weyl-group element in the $L^2$-model which, together with a parabolic subgroup whose action is also explicit, generates $G$ (see Theorem~\ref{thm:IntroThmF}).

The groups for which our construction provides $L^2$-models, can be divided into several classes. For the split groups $E_{6(6)}$, $E_{7(7)}$ and $E_{8(8)}$ and $\SO(n,n)$, our construction yields the same model as obtained by Kazhdan--Savin~\cite{KS90} by different methods. In these cases, the minimal representation is spherical, and the expression we find for the spherical vector matches the one in \cite{KPW02} found using case-by-case computations. For $G_{2(2)}$, the model can be found in the work of Gelfand~\cite{Gel80} and Savin~\cite{Sav93}. Here the lowest $K$-type is three-dimensional. For $\SL(n,\RR)$ our construction actually yields two one-parameter families of representations which turn out to be unitary principal series representations induced from a different maximal parabolic subgroup. Additionally, for $n=3$ we also construct Torasso's representation which lives on the double cover $\widetilde{\SL}(3,\RR)$. The groups for which the obtained $L^2$-models seem to be new, are the quaternionic groups $E_{6(2)}$, $E_{7(-5)}$ and $E_{8(-24)}$, for which the minimal representation is an analytic continuation of the quaternionic discrete series, and the indefinite orthogonal groups $\SO(p,q)$ with $p\neq q$. For the latter groups, one has to assume either that $p+q$ is even or that $\min(p,q)=3$, in which case the representation actually lives on a double cover $\widetilde{\SO}(p,q)$. For the case $\widetilde{\SO}(4,3)$ our formulas do not quite match the ones by Sabourin~\cite{Sab96}, but seem to be closely related (see Section~\ref{sec:TheCaseSO43}).

\section*{Acknowledgments}

We thank Marcus Slupinski and Robert Stanton for sharing early versions of their manuscript \cite{SS} with us. We are particularly indebted to Robert Stanton for numerous discussions about the structure of Heisenberg graded Lie algebras and for his help with Lemma~\ref{lem:BezoutianSum}.

\chapter{Statement of the results}

In this first chapter, we give a more precise statement of the results than in the introduction.

\section{Conformally invariant systems}

Let $G$ be a connected non-compact simple real Lie group with finite center and denote by $\frakg$ its Lie algebra. We assume that $G$ has a parabolic subgroup $P=MAN$ whose nilradical $N$ is a Heisenberg group and write $\frakm$, $\fraka$ and $\frakn$ for the Lie algebras of $M$, $A$ and $N$. There exists a unique element $H\in\fraka$ such that $\ad(H)$ has eigenvalues $+1$ and $+2$ on $\frakn$ and $-1$ and $-2$ on the opposite nilradical $\overline{\frakn}$. We decompose
$$ \frakg=\frakg_{-2}\oplus\frakg_{-1}\oplus\frakg_0\oplus\frakg_1\oplus\frakg_2 $$
into eigenspaces for $\ad(H)$ where $\frakg_0=\frakm\oplus\fraka$. In all cases but $\frakg\simeq\sl(n,\RR)$ the parabolic subgroup $P$ is maximal and hence $\fraka=\RR H$. To simplify notation we therefore put
$$ \fraka=\RR H \qquad \mbox{and} \qquad \frakm=\{T\in\frakg_0:\ad(T)|_{\frakg_2}=0\} $$
also in the case $\frakg\simeq\sl(n,\RR)$. Further, we write $\rho=\frac{1}{2}\ad|_{\frakn}\in\fraka^*$ as usual.

For an irreducible smooth admissible representation $(\zeta,V_\zeta)$ of $M$ and $\nu\in\fraka_\CC^*$ we form the degenerate principal series (smooth normalized parabolic induction)
$$ \pi_{\zeta,\nu} = \Ind_P^G(\zeta\otimes e^\nu\otimes\1) $$
and realize it on a subspace $I(\zeta,\nu)\subseteq C^\infty(\overline{\frakn})\otimeshat V_\zeta$ of $V_\zeta$-valued functions on the opposite nilradical $\overline{\frakn}\simeq\overline{N}$ which is a Heisenberg Lie algebra (the \emph{non-compact picture}). The representation $\pi_{\zeta,\nu}$ is irreducible for generic $\nu$, but may contain irreducible subrepresentations for singular parameters.

Extending $H$ to an $\sl_2$-triple by $E\in\frakg_2$ and $F\in\frakg_{-2}$, we can endow $V=\frakg_{-1}$ with a symplectic form $\omega$ characterized by
$$ [x,y]=\omega(x,y)F \qquad \mbox{for }x,y\in V. $$
The identity component $M_0$ of $M$ acts symplectically on $(V,\omega)$ and the $5$-grading of $\frakg$ gives rise to three additional symplectic invariants:
\begin{align*}
	\mu:V\to\frakm, &\quad \mu(x)=\frac{1}{2!}\ad(x)^2E && \mbox{(the moment map)}\\
	\Psi:V\to V, &\quad \Psi(x)=\frac{1}{3!}\ad(x)^3E && \mbox{(the cubic map)}\\
	Q:V\to\RR, &\quad Q(x)=\frac{1}{4!}\ad(x)^4E && \mbox{(the quartic)}
\end{align*}
which are all $M$-equivariant polynomials. In \cite{BKZ08} Barchini, Kable and Zierau constructed for each of the invariants $\omega$, $\mu$, $\Psi$ and $Q$ a system of differential operators on $\overline{\frakn}$ which is \emph{conformally invariant}. These systems can be seen as quantizations of the symplectic invariants. The joint kernel of each system gives rise to a subrepresentation of a degenerate principal series representation. For instance, for the conformally invariant system $\Omega_\mu(T)$ ($T\in\frakm$) corresponding to the moment map, the conformal invariance implies that for every simple or one-dimensional abelian ideal $\frakm'\subseteq\frakm$ there exists a parameter $\nu=\nu(\frakm')\in\fraka^*$ such that the joint kernel
$$ I(\zeta,\nu)^{\Omega_\mu(\frakm')} = \{u\in I(\zeta,\nu):\Omega_\mu(T)u=0\mbox{ for all }T\in\frakm'\} $$
is a subrepresentation of $I(\zeta,\nu)$ whenever $\zeta$ is trivial on the connected component $M_0$ of $M$. Note that $I(\zeta,\nu)^{\Omega_\mu(\frakm')}$ could be trivial.

\section{The Heisenberg group Fourier transform}

The infinite-dimensional irreducible unitary representations $(\sigma_\lambda,\calH_\lambda)$ of $\overline{N}$ are parameterized by their central character $\lambda\in\RR^\times$ in the sense that $\sigma_\lambda(e^{tF})=e^{i\lambda t}\id$. They give rise to operator-valued maps
$$ \sigma_\lambda:L^1(\overline{N}) \to \End(\calH_\lambda), \quad \sigma_\lambda(u)=\int_{\overline{N}}u(\overline{n})\sigma_\lambda(\overline{n})\,d\overline{n}. $$
The Heisenberg group Fourier transform is the collection of all $\sigma_\lambda$ and it extends to a unitary isomorphism
$$ \calF:L^2(\overline{N})\to L^2(\RR^\times,\HS(\calH);|\lambda|^{\frac{\dim V}{2}}\,d\lambda)\simeq L^2(\RR^\times;|\lambda|^{\frac{\dim V}{2}}\,d\lambda)\otimeshat\HS(\calH), \quad \calF u(\lambda)=\sigma_\lambda(u), $$
where $\calH=\calH_\lambda$ is a Hilbert space which realizes all representations $\sigma_\lambda$ and $\HS(\calH)$ denotes the Hilbert space of all Hilbert--Schmidt operators on $\calH$. It is a non-trivial problem to extend $\calF$ to tempered distributions (see e.g. \cite{Dah19,Fab91} on this issue). For our purpose it is enough to show that $\calF$ extends for $\Re\nu>-\rho$ to an injective linear map (see Corollary~\ref{cor:FTinjectiveOnPS})
$$ \calF:I(\zeta,\nu)\to\calD'(\RR^\times)\otimeshat\Hom(\calH^\infty,\calH^{-\infty})\otimeshat V_\zeta, $$
where $\calH^\infty$ denotes the space of smooth vectors in $\calH=\calH_\lambda$ and $\calH^{-\infty}=(\calH^\infty)'$ its dual space, the space of distribution vectors (both spaces are independent of $\lambda$).

For $\Re\nu>-\rho$ we call the realization
$$ \widehat{\pi}_{\zeta,\nu}(g)=\calF\circ\pi_{\zeta,\nu}(g)\circ\calF^{-1} \qquad (g\in G) $$
on the subspace
$$ \widehat{I}(\zeta,\nu):=\calF(I(\zeta,\nu))\subseteq\calD'(\RR^\times)\otimeshat\Hom(\calH^\infty,\calH^{-\infty})\otimeshat V_\zeta $$
the \emph{Fourier transformed picture}. In order to understand the Fourier transformed picture of the subrepresentations $I(\zeta,\nu)^{\Omega_\mu(\frakm')}$, we study the Fourier transform of the conformally invariant system $\Omega_\mu$. For this denote by $d\omega_{\met,\lambda}$ the metaplectic representation of $\sp(V,\omega)$ which is uniquely defined by (see e.g. \cite[Chapter 4]{Fol89})
$$ d\sigma_\lambda([T,X]) = [d\omega_{\met,\lambda}(T),d\sigma_\lambda(X)] \qquad (T\in\sp(V,\omega),X\in\overline{\frakn}). $$
Note that the adjoint representation $\ad:\frakm\to\gl(V),\,T\mapsto\ad(T)|_V$ identifies $\frakm$ with a subalgebra of $\sp(V,\omega)$ so that we can restrict $d\omega_{\met,\lambda}$ to $\frakm$.

\begin{thmalph}[see Theorem~\ref{thm:FTofOmegaMu}]\label{thm:IntroThmA}
	For $\lambda\in\RR^\times$, $T\in\frakm$ and $u\in I(\zeta,\nu)$ we have, in the distribution sense,
	$$ \sigma_\lambda(\Omega_\mu(T)u) = 2i\lambda\, \sigma_\lambda(u)\circ d\omega_{\met,\lambda}(T). $$
\end{thmalph}

This implies that the Fourier transform of $u\in I(\zeta,\nu)^{\Omega_\mu(\frakm')}$ satisfies, again in the distribution sense,
$$ \calF u(\lambda)\circ d\omega_{\met,\lambda}(T) = 0 \qquad \mbox{for all }T\in\frakm'. $$

We also obtain formulas for the Fourier transform of the conformally invariant systems $\Omega_\omega$, $\Omega_\Psi$ and $\Omega_Q$ associated to $\omega$, $\Psi$ and $Q$ in Sections~\ref{sec:FTOmegaOmega} and \ref{sec:FTOmegaPsiQ}, but do not study the corresponding subrepresentations any further.

\section{The Fourier transformed picture of the minimal representation}

In this work we restrict our attention to the case of non-Hermitian $G$. More details about the differences between Hermitian and non-Hermitian $G$ can be found in Section~\ref{sec:HermVsNonHerm}. We hope to return to the Hermitian case in a future work.

Since $G$ is non-Hermitian, one can use the structure theory developed in \cite{SS,SS15} to obtain a bigrading on $\frakg$ (see Section~\ref{sec:Bigrading} for details). This results in a particular choice of a Lagrangian subspace $\Lambda\subseteq V$ with decomposition $\Lambda=\RR A\oplus\calJ$, where in most cases $\calJ$ is a semisimple Jordan algebra of degree $3$ with norm function $n(z)$ defined by $\Psi(z)=n(z)A$. (In fact, all semisimple Jordan algebras of rank three arise in this way, cf. Table~\ref{tab:Classification}). The two exceptions are $\frakg\simeq\frakg_{2(2)}$ where $\calJ\simeq\RR$ and $n(z)=z^3$ and $\frakg\simeq\sl(n,\RR)$ where $n(z)=0$ for all $z\in\calJ$. We write $(a,z)$ for $aA+z\in\RR A\oplus\calJ=\Lambda$.

We realize the representations $\sigma_\lambda$ on the common Hilbert space $\calH=L^2(\Lambda)$, the Schr\"{o}dinger model of $\sigma_\lambda$. In this realization we have $\calH^\infty=\calS(\Lambda)$, the space of Schwartz functions on $\Lambda$, and $\calH^{-\infty}=\calS'(\Lambda)$, the space of tempered distributions, so that the Schwartz Kernel Theorem implies
$$ \Hom(\calH^\infty,\calH^{-\infty}) \simeq \calS'(\Lambda\times\Lambda) \simeq \calS'(\Lambda)\otimeshat\calS'(\Lambda). $$

From here on we assume that the parameter $\nu=\nu(\frakm')\in\fraka^*$, for which the joint kernel of $\Omega_\mu(\frakm')$ is a subrepresentation of $I(\zeta,\nu)$, is the same for all factors of $\frakm$. This is in particular the case when $\frakm$ is simple, but also for $\frakg\simeq\sl(3,\RR)$ where $\frakm\simeq\gl(1,\RR)$ and for $\frakg\simeq\so(4,4)$ where $\frakm\simeq\sl(2,\RR)\oplus\sl(2,\RR)\oplus\sl(2,\RR)$. The reason for this assumption is that we need invariance under the full Lie algebra $\frakm$ in the following result:

\begin{thmalph}[see Theorem~\ref{thm:InvDistributionVector}]\label{thm:IntroThmB}
	For every $\lambda\in\RR^\times$ the space $L^2(\Lambda)^{-\infty,\frakm}$ of $\frakm$-invariant distribution vectors in $d\omega_{\met,\lambda}$ is two-dimensional and spanned by $\xi_{\lambda,\varepsilon}$ ($\varepsilon\in\ZZ/2\ZZ$), where
	$$ \xi_{\lambda,\varepsilon}(a,z) = \sgn(a)^\varepsilon|a|^{s_\min}e^{-i\lambda\frac{n(z)}{a}}, \qquad (a,z)\in\RR\times\calJ, $$
	with $s_\min=-\frac{1}{6}(\dim\Lambda+2)$.
\end{thmalph}

We remark that these distributions also occur in the classification \cite{EKP02} of certain generalized functions whose Euclidean Fourier transform is of the same type.

Let us for simplicity also assume that the representation $\zeta$ is a character of the component group of $M$. Then, Theorem~\ref{thm:IntroThmB} implies that the Fourier transform $\calF u\in\calD'(\RR^\times)\otimeshat\calS'(\Lambda)\otimeshat\calS'(\Lambda)$ of a function $u\in I(\zeta,\nu)^{\Omega_\mu(\frakm)}$ in the kernel of $\Omega_\mu(\frakm)$ can be written as
$$ \calF u(\lambda,x,y) = \xi_{-\lambda,\varepsilon}(x)\widetilde{u}(\lambda,y) $$
for some $\widetilde{u}(\lambda,\cdot)\in\calS'(\Lambda)$, where $\varepsilon\in\ZZ/2\ZZ$ is determined in Corollary~\ref{cor:PSEmbedding}. The map
$$ I(\zeta,\nu)^{\Omega_\mu(\frakm)} \to \calD'(\RR^\times)\otimeshat\calS'(\Lambda), \quad u\mapsto\widetilde{u} $$
is injective and provides a new realization $\rho_\min$ of the subrepresentation $I(\zeta,\nu)^{\Omega_\mu(\frakm)}$ on
$$ J_\min\subseteq\calD'(\RR^\times)\otimeshat\calS'(\Lambda). $$
Note that still $J_\min$ could be trivial. To show that there exists a representation $\zeta$ of $M$ such that $J_\min\neq\{0\}$, we compute the Lie algebra action in the new realization and find explicit $K$-finite vectors.

We remark that this construction only excludes the non-Hermitian Lie algebras $\frakg=\sl(n,\RR)$ ($n>3$) and $\frakg=\so(p,q)$ ($p,q\geq3$, $(p,q)\neq(4,4)$). In Section~\ref{sec:FTpictureMinRepSLn} we explain how for $\frakg=\sl(n,\RR)$ a generalization of the first order system $\Omega_\omega$ (a quantization of the symplectic form $\omega$) to the case of vector-valued principal series induced from characters $\zeta=\zeta_r$ of $M$ ($r\in\CC$) yields a one-parameter family $d\rho_{\min,r}$ of subrepresentations of $I(\zeta_r,\nu)$ on $\calD'(\RR^\times)\otimeshat\calS'(\Lambda)$. And for $\frakg=\so(p,q)$ we combine in Section~\ref{sec:FTpictureMinRepSOpq} generalizations of both $\Omega_\omega$ and $\Omega_\mu$ to the vector-valued degenerate principal series induced from representations of the $\SL(2)$-factor of $M$ to find the analogous representation $d\rho_\min$ of $\frakg$ on $\calD'(\RR^\times)\otimeshat\calS'(\Lambda)$.

\section{The Lie algebra action and lowest $K$-types}

Heuristic arguments involving the standard Knapp--Stein intertwining operators (see Remark~\ref{rem:MotivationL2}) show that $\widetilde{\pi}_\min$ should be unitary on $L^2(\RR^\times\times\Lambda,|\lambda|^{\dim\Lambda-2s_\min}d\lambda\,dy)$. To obtain a unitary representation on $L^2(\RR^\times\times\Lambda)$, we twist the representation with the isomorphism
$$ \Phi_\delta:\calD'(\RR^\times)\otimeshat\calS'(\Lambda)\to\calD'(\RR^\times)\otimeshat\calS'(\Lambda), \quad \Phi_\delta u(\lambda,x) = \sgn(\lambda)^\delta|\lambda|^{-s_\min}u(\lambda,\tfrac{x}{\lambda}) $$
which restricts to an isometry $L^2(\RR^\times\times\Lambda,|\lambda|^{\dim\Lambda-2s_\min}d\lambda\,dy)\to L^2(\RR^\times\times\Lambda)$ (see Corollary~\ref{cor:PSEmbedding} for the choice of $\delta\in\ZZ/2\ZZ$). Let
$$ I_\min := \Phi_\delta(J_\min), \qquad \pi_\min(g) := \Phi_\delta\circ\rho_\min(g)\circ\Phi_\delta^{-1}. $$

\begin{thmalph}\label{thm:IntroThmC}
	The Lie algebra action $d\pi_\min$ is by algebraic differential operators of degree $\leq3$ on $\RR^\times\times\Lambda$ and is explicitly computed in Proposition~\ref{prop:dpimin}. It extends naturally to $\calD'(\RR^\times)\otimeshat\calS'(\Lambda)$ and is infinitesimally unitary on $L^2(\RR^\times\times\Lambda)$.
\end{thmalph}

Comparing this action with formulas in the literature shows that our representation agrees with the one for $\frakg=\so(n,n),\frake_{6(6)},\frake_{7(7)},\frake_{8(8)}$ in \cite{KPW02,KS90}, for $\frakg=\frakg_{2(2)}$ in \cite{Gel80, Sav93}, and for $\frakg=\sl(3,\RR)$ in \cite{Tor83} (see Section~\ref{sec:LAactionLiterature} for details). There also seems to be a relation to the formulas for $\frakg=\so(4,3)$ in Sabourin \cite{Sab96}. Similar formulas also appear in \cite{GNPP08,GP05,GP06} but without addressing the question of unitarizability. In this sense, our computations give a new explanation of the formulas in the literature and provide an explicit (degenerate) principal series embedding of the representations as well as generalize them to a larger class of groups.

In order to show that the Lie algebra representation $d\pi_\min$ on $I_\min\subseteq\calD'(\RR^\times)\otimeshat\calS'(\Lambda)$ integrates to an irreducible unitary representation on $L^2(\RR^\times\times\Lambda)$ for some representation $\zeta$, we find the lowest $K$-type in the representation.

\begin{thmalph}\label{thm:IntroThmD}
	\begin{enumerate}[(1)]
		\item For $\frakg=\frake_{6(2)},\frake_{7(-5)},\frake_{8(-24)}$ there exists a $\frakk$-subrepresentation $W\subseteq\calD'(\RR^\times)\otimeshat\calS'(\Lambda)$, explicitly given in Theorem~\ref{thm:LKTQuat}, which is isomorphic to the representation $S^{2,4,8}(\CC^2)\boxtimes\CC$ of $\frakk\simeq\su(2)\oplus\frakk''$.
		\item For $\frakg=\frake_{6(6)},\frake_{7(7)},\frake_{8(8)}$ there exists a $\frakk$-subrepresentation $W\subseteq\calD'(\RR^\times)\otimeshat\calS'(\Lambda)$, explicitly given in Theorem~\ref{thm:LKTSplit}, which is isomorphic to the trivial representation.
		\item For $\frakg=\frakg_{2(2)}$ there exists a $\frakk$-subrepresentation $W\subseteq\calD'(\RR^\times)\otimeshat\calS'(\Lambda)$, explicitly given in Theorem~\ref{thm:LKTG2}, which is isomorphic to the representation $\CC\boxtimes S^2(\CC^2)$ of $\frakk\simeq\su(2)\oplus\su(2)$.
		\item For $\frakg=\sl(n,\RR)$ there exist for every $r\in\CC$ two $\frakk$-subrepresentations $W_{\varepsilon,r}\subseteq\calD'(\RR^\times)\otimeshat\calS'(\Lambda)$ ($\varepsilon\in\ZZ/2\ZZ$) of the representation $d\pi_{\min,r}$ of $\frakg$, explicitly given in Theorem~\ref{thm:LKTSLn}. The representation $W_{0,r}$ is isomorphic to the trivial representation of $\frakk$ and $W_{1,r}$ is isomorphic to the standard representation $\CC^n$ of $\frakk\simeq\so(n)$.
		\item For $\frakg=\sl(3,\RR)$ and $r=0$ there exists a third $\frakk$-subrepresentation $W_{\frac{1}{2}}\subseteq\calD'(\RR^\times)\otimeshat\calS'(\Lambda)$ of $d\pi_{\min,r}$, explicitly given in Theorem~\ref{thm:LKTSL3}, which is isomorphic to the representation $\CC^2$ of $\frakk\simeq\su(2)$.
		\item For $\frakg=\so(p,q)$, $p\geq q\geq4$ with $p+q$ even, there exists a $\frakk$-subrepresentation $W\subseteq\calD'(\RR^\times)\otimeshat\calS'(\Lambda)$, explicitly given in Theorem~\ref{thm:LKTSOpq}, which is isomorphic to the representation $\CC\boxtimes\calH^{\frac{p-q}{2}}(\RR^q)$ of $\frakk\simeq\so(p)\oplus\so(q)$.
		\item For $\frakg=\so(p,3)$, $p\geq3$, there exists a $\frakk$-subrepresentation $W\subseteq\calD'(\RR^\times)\otimeshat\calS'(\Lambda)$, explicitly given in Theorem~\ref{thm:LKTSOp3}, which is isomorphic to the representation $\CC\boxtimes S^{p-3}(\CC^2)$ of $\frakk\simeq\so(p)\oplus\su(2)$.
	\end{enumerate}
\end{thmalph}

The explicit form for the spherical vector for $\frakg=\so(n,n),\frake_{6(6)},\frake_{7(7)},\frake_{8(8)}$ has previously been found by Kazhdan, Pioline and Waldron~\cite{KPW02} and a similar formula for a $K$-finite vector in the case $\frakg=\frakg_{2(2)}$ can be found in \cite{GNPP08}. Further, for $\sl(3,\RR)$ the lowest $K$-types $W_{0,r}$, $W_{1,r}$ and $W_{\frac{1}{2}}$ were obtained by Torasso~\cite{Tor83}. We believe that the other formulas are new.

\section{The minimal representation}

A careful study of the action of $\frakg$ on the lowest $K$-type $W$ shows:

\begin{thmalph}\label{thmintro:IntegrationMinRep}
	The $K$-type $W$ generates an irreducible $(\frakg,K)$-module $\overline{W}=d\pi_\min(U(\frakg))W$ with lowest $K$-type $W$. This $(\frakg,K)$-module integrates to an irreducible unitary representation of the universal cover $\widetilde{G}$ of $G$ on $L^2(\RR^\times\times\Lambda)$ which is minimal in the sense that its annihilator is a completely prime ideal with associated variety equal to the minimal nilpotent coadjoint orbit. For $\frakg$ not of type $A$, the annihilator is the Joseph ideal.
\end{thmalph}

For the split groups $G=\SO(n,n),E_{6(6)},E_{7(7)},E_{8(8)}$ the same realization has been constructed by Kazhdan and Savin~\cite{KS90} and for $G=G_{2(2)}$ by Gelfand~\cite{Gel80} (see also Savin~\cite{Sav93}). For $G=\widetilde{\SL}(3,\RR)$ the representations were studied in detail by Torasso~\cite{Tor83}, and for $G=\widetilde{\SO}(4,3)$ the realization is similar to the one constructed by Sabourin~\cite{Sab96}. We believe that the realization is new for the quaternionic groups $G=E_{6(2)},E_{7(-5)},E_{8(-24)}$ in which case the representations can also be obtained as continuation of the quaternionic discrete series by algebraic methods (see Gross and Wallach~\cite{GW94,GW96}) and also for the groups $G=\widetilde{\SO}(p,q)$.

In view of the classification of minimal representations in \cite{Tam19}, the construction in Theorem~\ref{thmintro:IntegrationMinRep} together with the constructions in \cite{HKM14}, \cite{MS17} and \cite{Sav93} yield $L^2$-models for all minimal representations except for the ones of $F_{4(4)}$ and the complex groups $E_8(\CC)$ and $F_4(\CC)$. We believe that the case of $F_{4(4)}$ can be treated with a slight generalization of our methods to the vector-valued case (see Remark~\ref{rem:LKTF4}). It is further feasible that a construction similar to the one in \cite[Section 7]{Sav93} may construct an $L^2$-model for the complex groups $E_8(\CC)$ and $F_4(\CC)$ from the one for $E_{8(8)}$ and $F_{4(4)}$. This would give $L^2$-models for all minimal representations of simple Lie groups.

\section{The action of a non-trivial Weyl group element}

The group $G$ is generated by its maximal parabolic subgroup $\overline{P}$ and a non-trivial Weyl group element $w_1\in K$. While the action of $\overline{P}$ in $\pi_\min$ is relatively simple (see Lemma~\ref{lem:RepOfPbar}), it is non-trivial to find explicit formulas for other group elements. In \cite{KS90,Sav93} the authors were able to construct the above $L^2$-models for split groups by extending the action of $\overline{P}$ to $G$ in terms of $\pi_\min(w_1)$ (see also \cite{Rum97} for the case of $p$-adic groups). Here, one has to verify that the definition for the operator $\pi_\min(w_1)$ satisfies several relations, and the methods does not seem to generalize in a straightforward way.

Having constructed the representations $\pi_\min$ by different methods, we are able to find the action $\pi_\min(w_1)$ of the Weyl group element $w_1$ in all cases explicitly. The action depends on the eigenvalues of a certain Lie algebra element on the lowest $K$-type. Those can be integers of half-integers, and we refer to these two cases as the \emph{integer case} and the \emph{half-integer case} (see Section~\ref{sec:ActionWeylGroupElts} for details).

\begin{thmalph}[see Theorem~\ref{thm:ActionW1}]\label{thm:IntroThmF}
	The element $w_1$ acts in the $L^2$-model of the minimal representation by
	\begin{equation*}
		\pi_\min(w_1)f(\lambda,a,x) = e^{-i\frac{n(x)}{\lambda a}}f(\sqrt{2}a,-\tfrac{\lambda}{\sqrt{2}},x)\times\begin{cases}1&\mbox{in the integer case,}\\\varepsilon(a\lambda)&\mbox{in the half-integer case,}\end{cases}\label{eq:ActionW1}
	\end{equation*}
	where
	$$ \varepsilon(x) = \begin{cases}1&\mbox{for $x>0$,}\\i&\mbox{for $x<0$.}\end{cases} $$
\end{thmalph}

This gives a complete description of $\pi_\min$ on the generators $\overline{P}$ and $w_1$ and generalizes the formulas for $\pi_\min(w_1)$ in \cite{KS90,Sav93} . We remark that this viewpoint was also advocated in \cite{KM11} where the action of a non-trivial Weyl group element was obtained in a different $L^2$-model for the minimal representation of $\upO(p,q)$.

\section{Outlook}

Minimal representations have shown to be of importance in the theory of automorphic representations, for instance in the construction of exceptional theta series (see e.g. \cite{KPW02}), the study of Fourier coefficients of automorphic forms (see e.g. \cite{GGKPS22a,GGKPS22b}) or the study of local components of global automorphic representations (see e.g. \cite{BS18,KS15}). Some of these works use $L^2$-realizations of minimal representations. We hope that our new $L^2$-models might help to generalize some of these results.

Another possible application concerns branching laws for unitary representations, i.e. the restriction of representations to subgroups. $L^2$-models have proven to be useful in the decomposition of restricted representations since here classical spectral theory of differential operators can be applied (see e.g. \cite{Dvo07,KO03b,MO15}). We expect our new $L^2$-models to be useful for the decomposition of restrictions of minimal representations. In particular, we hope that these models allow for a more far-reaching and more explicit version of the theta correspondence.

It has further been observed that explicit realizations of small representations have fruitful connections to geometry, analysis and special functions (see e.g. \cite{FL12,HKMM11,HSS12,Kob14,KM11,KO03a,KO03c,KoMo11}). In our $L^2$-models the explicit $K$-finite vectors exhibited in Chapter~\ref{ch:LKT} are for instance expressed in terms of $K$-Bessel functions. We believe that there are many additional connections between our new realizations and other branches of mathematics. For instance, it would be interesting to relate the $L^2$-model for the minimal representation of $G=\SO(p,q)$ constructed in this paper to the one obtained in \cite{KM11,KO03c} by an explicit integral transformation. Similarly, one could try to find an explicit intertwining operator between our $L^2$-model and the realization constructed in \cite{AF12,BK94} on sections of the half-form bundle over the minimal nilpotent $K_\CC$-orbit.

A more direct further line of research is the investigation of the missing case $G=F_{4(4)}$ for which we expect a similar, possibly vector-valued, $L^2$-model. Also the case of Hermitian groups, which is missing in this work, is a possible further research question (see Section~\ref{sec:HermVsNonHerm} for some structural results in this situation and the recent preprint \cite{Zha21} for structural results about the corresponding spherical degenerate principal series).

\chapter[Structure theory]{Structure theory for Heisenberg parabolic subgroups}

In this preliminary section we study the structure of Heisenberg graded real Lie algebras. This includes the Langlands decomposition (see Section~\ref{sec:HeisenbergParabolicSubgroups}), a canonical $\sl(2)$-triple (see Section~\ref{sec:W0}), the minimal adjoint orbits (see Section~\ref{sec:MinimalAdjointOrbits}), the associated symplectic vector space together with its invariants, including the moment map (see Section~\ref{sec:SymplecticInvariants}), explicit formulas for the Killing form (see Section~\ref{sec:KillingForm}) as well as the structure of the Heisenberg nilradical (see Section~\ref{sec:HeisenbergNilradical}). For much of this we follow \cite{SS,SS15}, the statements in Sections~\ref{sec:BruhatDecomposition} about the Bruhat decomposition, in Section~\ref{sec:MaxCptSubgroups} about Cartan involutions and maximal compact subgroups, and in Section~\ref{sec:HermVsNonHerm} about the difference between Hermitian and non-Hermitian Lie algebras are new.

\section{Heisenberg parabolic subgroups}\label{sec:HeisenbergParabolicSubgroups}

Let $G$\index{G1@$G$} be a connected non-compact simple real Lie group with finite center. We assume that $G$ has a parabolic subgroup $P$\index{P1@$P$} whose nilradical is a Heisenberg group, i.e. two-step nilpotent with one-dimensional center. Then $P$ is maximal parabolic except in the case where $G$ is locally isomorphic to $\SL(n,\RR)$. Let $P=MAN$\index{M1@$M$}\index{A1@$A$}\index{N1@$N$} be a Langlands decomposition of $P$ and denote by $\frakg$\index{g3@$\frakg$}, $\frakp$\index{p3@$\frakp$}, $\frakm$\index{m3@$\frakm$}, $\fraka$\index{a3@$\fraka$} and $\frakn$\index{n3@$\frakn$} the corresponding Lie algebras of $G$, $P$, $M$, $A$ and $N$. See Table~\ref{tab:Classification} for a classification due to Cheng~\cite{Che87}.

There is a unique grading element $H\in\fraka$\index{H1@$H$} such that $\ad(H)$ has eigenvalues $1$ and $2$ on $\frakn$. Write
$$ \frakg = \frakg_{-2}+\frakg_{-1}+\frakg_0+\frakg_1+\frakg_2\index{g3i@$\frakg_i$} $$
for the decomposition of $\frakg$ into eigenspaces of $\ad(H)$, so that $\frakm\oplus\fraka=\frakg_0$ and $\frakn=\frakg_1+\frakg_2$. Denote by $\overline{\frakp}=\frakg_{-2}+\frakg_{-1}+\frakg_0$\index{p3@$\overline{\frakp}$} the opposite parabolic subalgebra with nilradical $\overline{\frakn}=\frakg_{-2}+\frakg_{-1}$\index{n3@$\overline{\frakn}$} and let $\overline{P}$\index{P1@$\overline{P}$} and $\overline{N}$\index{N1@$\overline{N}$} be the corresponding groups. In all cases but $\frakg\simeq\sl(n,\RR)$ we then have
$$ \fraka = \RR H, \qquad \frakm = \{T\in\frakg_0:\ad(T)|_{\frakg_2}=0\}. $$
To simplify notation, we use this as a definition for $\fraka$ and $\frakm$ in the case $\frakg=\sl(n,\RR)$, although this does not yield a Langlands decomposition of $P$.

Since $M$ commutes with $A=\exp(\RR H)$, $M$ preserves the $5$-grading of $\frakg$. In particular, $M$ acts on the subspaces $\frakg_{\pm2}$ by the adjoint representation. These subspaces are one-dimensional since they are the center of the Heisenberg algebra $\frakn$ resp. $\overline{\frakn}$, so there exists a character $\chi:M\to\{\pm1\}$\index{1xchi@$\chi$} such that
$$ \Ad(m)|_{\frakg_{\pm2}}=\chi(m)\cdot\id_{\frakg_{\pm2}}. $$

\section{The $\sl_2$-triple and the Weyl group element $w_0$}\label{sec:W0}

We choose $E\in\frakg_2$\index{E@$E$} and $F\in\frakg_{-2}$\index{F@$F$} such that $[E,F]=H$. Then $\{E,H,F\}$ forms an $\sl_2$-triple. Put
\begin{equation}
	w_0 := \exp\left(\frac{\pi}{2}(E-F)\right),\index{w20@$w_0$}\label{eq:DefW0}
\end{equation}
then $\Ad(w_0):\frakg\to\frakg$ is of order four. On the $\sl_2$-triple it is given by
$$ \Ad(w_0)E = -F, \qquad \Ad(w_0)H = -H, \qquad \Ad(w_0)F = -E, $$
and it acts trivially on $\frakm$. Further, $\Ad(w_0)$ restricts to isomorphisms $\frakg_1\to\frakg_{-1}$ and $\frakg_{-1}\to\frakg_1$ which compose to $-1$ times the identity. We put
\begin{align}
	\overline{x} &:= \ad(x)F = \Ad(w_0)x = -\Ad(w_0^{-1})x, && x\in\frakg_1,\label{eq:DefBar1}\\
	\overline{y} &:= \ad(y)E = -\Ad(w_0)y = \Ad(w_0^{-1})y, && y\in\frakg_{-1}.\index{1ABar@$\overline{(\cdot)}$}\label{eq:DefBar2}
\end{align}
Note that these maps are mutually inverse and
$$ \overline{\Ad(m)x} = \chi(m)\Ad(m)\overline{x} \qquad \mbox{for all }m\in M,x\in\frakg_{\pm1}. $$
This implies in particular that
\begin{equation}
	\overline{[T,x]} = [T,\overline{x}] \qquad \mbox{for all }T\in\frakm,x\in\frakg_{\pm1}.\label{eq:BarMInvariant}
\end{equation}
We further remark that $P$ and $\overline{P}$ are conjugate via $w_0$:
$$ \overline{P} = w_0Pw_0^{-1} = w_0^{-1}Pw_0. $$
The element $w_0$ defines a non-trivial coset in the Weyl group of $G$ with respect to a maximal abelian subalgebra of $\frakm\oplus\fraka$, and we therefore also refer to it as \emph{Weyl group element}.

\section{Minimal (co)adjoint orbits}\label{sec:MinimalAdjointOrbits}

The adjoint orbit
$$ \Omin = \Ad(G)E\subseteq\frakg\index{Omin@$\Omin$} $$
is a minimal nilpotent orbit. If $G$ is non-Hermitian then $\Omin=-\Omin$ is the unique minimal nilpotent adjoint orbit. If $G$ is Hermitian then $\Omin$ and $-\Omin$ are the two distinct minimal nilpotent adjoint orbits.

\begin{lemma}
The stabilizer of $E$ in $G$ is given by $M_1N$ where $M_1=\{m\in M:\Ad(m)E=E\}=\{m\in M:\chi(m)=1\}\subseteq M$\index{M11@$M_1$} is a subgroup of index at most $2$. In particular, $\dim_\RR\calO_\min=\dim\frakn+1$.
\end{lemma}

\begin{proof}
Let $g\in G$ such that $\Ad(g)E=E$. We claim that $g\in N_G(\frakp)=P$. In fact, for $x\in\frakg_1$ and $T\in\frakm$ we have
$$ [\Ad(g)x,E] = \Ad(g)[x,E] = 0 \qquad \mbox{and} \qquad  [\Ad(g)T,E] = \Ad(g)[T,E] = 0 $$
and hence $\Ad(g)x,\Ad(g)T\in Z_\frakg(E)=\frakm+\frakg_1+\frakg_2$. Further,
$$ [\Ad(g)H,E] = \Ad(g)[H,E] = 2\Ad(g)E = 2E $$
and hence $\Ad(g)H\in H+\frakm+\frakg_1+\frakg_2$. This shows that $\Ad(g)\frakp\subseteq\frakp$ and therefore $g\in N_G(\frakp)=P=MAN$. Write $g=man$ with $m\in M$, $a=\exp(rH)\in A$ and $n\in N$, then $\Ad(g)E=\chi(m)e^{2r}E$. This shows that $\Ad(man)E=E$ if and only if $a=e$, the identity, and $m\in M_1$.
\end{proof}

Important for the definition of minimal representations is the minimal nilpotent coadjoint orbit $\calO_{\min,\CC}^*$ in $\frakg_\CC^*$. This is the unique non-trivial nilpotent coadjoint orbit in $\frakg^*$ of minimal dimension. Identifying $\frakg_\CC^*\simeq\frakg_\CC$ using the Killing form, the orbit $\calO_{\min,\CC}^*$ is identified with the minimal nilpotent adjoint orbit $\calO_{\min,\CC}\subseteq\frakg_\CC$. Since $E$ is a highest root vector, it follows that $\calO_{\min,\CC}$ is the orbit through $E$ and therefore a complexification of $\calO_\min\subseteq\frakg$. This implies:

\begin{corollary}\label{cor:DimMinOrbit}
	$\dim_\CC\calO_{\min,\CC}^*=\dim_\RR\frakn+1$.
\end{corollary}

\section{The symplectic invariants}\label{sec:SymplecticInvariants}

Following \cite[Section 2]{SS15}, we now discuss the symplectic structure of $V:=\frakg_{-1}$\index{V@$V$}. The bilinear form $\omega:V\times V\to\RR$\index{1zomegaxy@$\omega(x,y)$} given by
$$ [x,y] = \omega(x,y)F, \qquad x,y\in V, $$
turns $V$ into a symplectic vector space. The group $M$ preserves the symplectic form up to a sign:
$$ \omega(mx,my) = \chi(m)\omega(x,y) \qquad (m\in M,x,y\in V), $$
where we abbreviate $mx=\Ad(m)x$. In particular, the Lie algebra $\frakm$ of $M$ can be viewed as a subalgebra of $\sp(V,\omega)$. We note that for $x,y\in\frakg_1$:
\begin{equation}
	[x,y] = -\omega(\overline{x},\overline{y})E.\label{eq:CommutatorG1}
\end{equation}

We now define three maps on $V$ related to the symplectic structure.

\subsection{The moment map}

Put
$$ \mu:V\to\frakg_0, \quad x\mapsto\frac{1}{2!}\ad(x)^2E.\index{1mux@$\mu(x)$} $$
Then $\mu$ actually maps into $\frakm$ and is $\Ad(M)$-equivariant, i.e.
$$ \mu(\Ad(m)x) = \chi(m)\Ad(m)\mu(x), \qquad m\in M. $$

\subsection{The cubic map}

Let
$$ \Psi:V\to V, \quad x\mapsto\frac{1}{3!}\ad(x)^3E.\index{1Psix@$\Psi(x)$} $$
Then $\Psi$ is $\Ad(M)$-equivariant, i.e.
$$ \Psi(\Ad(m)x) = \chi(m)\Ad(m)\Psi(x), \qquad m\in M. $$

\subsection{The quartic}

Define $Q:V\to\RR$ by
$$ Q(x)F = \frac{1}{4!}\ad(x)^4E.\index{Qx@$Q(x)$} $$
Then $Q$ transforms under the action $\Ad(M)$ by the character $\chi$, i.e.
$$ Q(\Ad(m)x) = \chi(m)Q(x), \qquad m\in M. $$

\subsection{Symplectic formulas}

We state some formulas for the symplectic invariants $\mu$, $\Psi$ and $Q$ proved in \cite{SS,SS15}. (Note that $\mu$, $\Psi$ and $Q_n$ in their notation are $-2\mu$, $6\Psi$ and $36Q$ in our notation.)

\begin{lemma}[{\cite[Corollary 4.2]{SS15}}]\label{lem:SymplecticFormulas}
For $x\in V$ the following identities hold:
\begin{enumerate}[(1)]
\item\label{lem:SymplecticFormulas1} $\mu(ax+b\Psi(x))=(a^2-b^2Q(x))\mu(x)$,
\item\label{lem:SymplecticFormulas2} $\Psi(ax+b\Psi(x))=(a^2-b^2Q(x))(bQ(x)x+a\Psi(x))$,
\item\label{lem:SymplecticFormulas3} $Q(ax+b\Psi(x))=(a^2-b^2Q(x))^2Q(x)$,
\item\label{lem:SymplecticFormulas4} $\mu(x)\Psi(x)=-3Q(x)x$.
\end{enumerate}
\end{lemma}

Let $B_\mu:V\times V\to\frakm$\index{Bmuxy@$B_\mu(x,y)$}, $B_\Psi:V\times V\times V\to V$\index{BPsixyz@$B_\Psi(x,y,z)$} and $B_Q:V\times V\times V\times V\to\RR$\index{BQxyzw@$B_Q(x,y,z,w)$} denote the symmetrizations of $\mu$, $\Psi$ and $Q$. Further, define $\tau:V\to\sp(V,\Omega)$ and its symmetrization $B_\tau:V\times V\to\sp(V,\omega)$ by
$$ \tau(x)y=\omega(x,y)x \qquad \mbox{and} \qquad B_\tau(x,y)z=\tfrac{1}{2}(\omega(x,z)y+\omega(y,z)x).\index{1tauxy@$\tau(x,y)$}\index{Bmuxy@$B_\tau(x,y)$} $$

\begin{lemma}[{\cite[Proposition 2.2]{SS15}}]\label{lem:SymmetrizationsOfSymplecticCovariants}
For $x,y,z,w\in V$ the following identities hold:
\begin{enumerate}[(1)]
\item\label{lem:SymmetrizationsOfSymplecticCovariants1} $B_\mu(x,y) = \tfrac{1}{4}([x,[y,E]]+[y,[x,E]])$,
\item\label{lem:SymmetrizationsOfSymplecticCovariants2} $B_\Psi(x,y,z) = -\tfrac{1}{3}B_\mu(x,y)z-\tfrac{1}{6}B_\tau(x,y)z$,
\item\label{lem:SymmetrizationsOfSymplecticCovariants3} $B_Q(x,y,z,w) = \tfrac{1}{4}\omega(x,B_\Psi(y,z,w))$.
\end{enumerate}
\end{lemma}

Note that this implies in particular that
$$ \omega(B_\mu(x,y)z,w) = \omega(B_\mu(z,w)x,y) \qquad \mbox{for all }x,y,z,w\in V. $$

\begin{lemma}[{\cite[Definition 2.1~(2) and Theorem 2.16]{SS15}}]\label{lem:RewriteBmu}
For $x,y,z\in V$:
$$ 4B_\mu(x,y)z-4B_\mu(x,z)y = \omega(x,y)z-\omega(x,z)y-2\omega(y,z)x. $$
\end{lemma}

Using Lemma~\ref{lem:RewriteBmu}, it is possible to reconstruct the Heisenberg graded Lie algebra $\frakg$ from the symplectic vector space $(V,\omega)$ the subalgebra $\frakm\subseteq\sp(V,\Omega)$ and the map $B_\mu:V\times V\to\frakm$ (see \cite[Theorem 2.16]{SS15} where the tuple $(\frakm,V,\omega,B_\mu)$ is called a \emph{special symplectic representation}).

\subsection{The Lie bracket}

We express the Lie bracket $[\frakg_1,\frakg_{-1}]$ in terms of the moment map and the symplectic form.

\begin{lemma}\label{lem:G1bracketG-1}
For $x\in\frakg_1$ and $y\in\frakg_{-1}$ the decomposition of $[x,y]\in\frakg_0$ in terms of the decomposition $\frakg_0=\frakm+\fraka$ is given by
$$ [x,y] = -2B_\mu(\overline{x},y) - \tfrac{1}{2}\omega(\overline{x},y)H. $$
\end{lemma}

\begin{proof}
We have
\begin{align*}
 [x,y] &= [\overline{\overline{x}},y] = [[\overline{x},E],y] = -[y,[\overline{x},E]]\\
 &= -\tfrac{1}{2}\left([\overline{x},[y,E]]+[y,[\overline{x},E]]\right) +\tfrac{1}{2}[\ad(\overline{x}),\ad(y)]E\\
 &= -2B_\mu(\overline{x},y)+\tfrac{1}{2}\ad([\overline{x},y])E = -2B_\mu(\overline{x},y)+\tfrac{1}{2}\omega(\overline{x},y)\ad(F)E\\
 &= -2B_\mu(\overline{x},y)-\tfrac{1}{2}\omega(\overline{x},y)H.\qedhere
\end{align*}
\end{proof}

\subsection{Simple factors of $\frakm$}

We note that $\frakm$ is reductive with at most one-dimensional center. The proof of the following result was communicated to us by R. Stanton:

\begin{lemma}\label{lem:BezoutianSum}
Let $(e_\alpha)$ be a basis of $V$ and $(\widehat{e}_\alpha)$ its dual basis with respect to the symplectic form $\omega$, i.e. $\omega(e_\alpha,\widehat{e}_\beta)=\delta_{\alpha\beta}$. Then the map
\begin{equation}
	\frakm\to\frakm, \quad T\mapsto\sum_\alpha B_\mu(Te_\alpha,\widehat{e}_\alpha)\label{eq:BezoutianSum}
\end{equation}
is a scalar multiple $\calC(\frakm')\cdot\id_{\frakm'}$\index{Cmprime@$\calC(\frakm')$} of the identity on the center and on each simple factor $\frakm'$ of $\frakm$. In particular, if $\frakm$ is simple the map is a scalar multiple of the identity.
\end{lemma}

\begin{proof}
The expression is obviously independent of the chosen basis, so that we may assume $\{e_\alpha\}$ to be a symplectic basis of the form $\{e_i,f_j\}$ with $\omega(e_i,f_j)=\delta_{ij}$ and $\omega(e_i,e_j)=\omega(f_i,f_j)=0$. Then $\{\widehat{e}_\alpha\}=\{f_i,-e_j\}$ and the sum becomes
$$ \sum_i\Big(B_\mu(Te_i,f_i)+B_\mu(Tf_i,-e_i)\Big) = \sum_i\Big(B_\mu(Te_i,f_i)-B_\mu(e_i,Tf_i)\Big). $$
For $x,y\in V$ we define
$$ \Bz(x,y):\frakm\to\frakm, \quad \Bz(x,y)T=B_\mu(Tx,y)-B_\mu(x,Ty). $$
$\Bz$ is called the \textit{Bezoutian} and the $\frakm$-equivariance of $B_\mu$ implies that that $\Bz:\Lambda^2V\to\End(\frakm)$ is $\frakm$-equivariant. In this notation, the map \eqref{eq:BezoutianSum} can be viewed as the image of $\sum_i e_i\wedge f_i\in\Lambda^2V$ under the Bezoutian $\Bz:\Lambda^2V\to\End(\frakm)$. The element $\sum_i e_i\wedge f_i$ corresponds to the symplectic form $\omega$ on $V$ which is $\frakm$-invariant, so $\sum_i e_i\wedge f_i$ is $\frakm$-invariant. It follows that its image under the Bezoutian is $\frakm$-invariant, hence the map \eqref{eq:BezoutianSum} is $\frakm$-intertwining. If $\frakg_\CC$ is not of type $A$, then $\frakm$ is semisimple and hence the $\frakm$-intertwining map \eqref{eq:BezoutianSum} is a scalar multiple of the identity on every simple factor of $\frakm$. The case $\frakg_\CC=\sl(n,\CC)$ can finally be treated by an explicit computation using Appendix~\ref{app:SLn}. (Note that the statement of the lemma only depends on the complexification of $\frakg$.)
\end{proof}

In \cite[\S8.10]{BKZ08}, the numbers $\calC(\frakm')$ are computed for all simple factors $\frakm'$ of $\frakm$ and we have included them in Table~\ref{tab:Classification} (see also Lemma~\ref{lem:CmForSimpleM} for the computation of $\calC(\frakm)$ in the case where $G$ is non-Hermitian and $\frakm$ is simple).

\begin{corollary}\label{cor:TraceTMuProportionalSympForm}
For every factor $\frakm'$ of $\frakm$ we have
$$ \tr(T\mu(x)) = \mathcal{C}(\frakm')\omega(Tx,x) \qquad \mbox{for all }x\in V,T\in\frakm'. $$
\end{corollary}

\begin{proof}
Let $(e_\alpha)$ be a basis of $V$ and $(\widehat{e}_\alpha)$ the dual basis with respect to the symplectic form, i.e. $\omega(e_\alpha,\widehat{e}_\beta)=\delta_{\alpha\beta}$. Then
\begin{align*}
 \tr(T\mu(x)) &= \sum_\alpha\omega(T\mu(x)e_\alpha,\widehat{e}_\alpha) = \sum_\alpha\omega(T\widehat{e}_\alpha,\mu(x)e_\alpha)\\
 &= \sum_\alpha\omega(T\widehat{e}_\alpha,-3B_\Psi(x,x,e_\alpha)-\frac{1}{2}\tau(x)e_\alpha)\\
 &= -12\sum_\alpha B_Q(x,x,e_\alpha,T\widehat{e}_\alpha)+\frac{1}{2}\omega(Tx,x)\\
 &= -3\sum_\alpha\omega(x,B_\Psi(e_\alpha,T\widehat{e}_\alpha,x))+\frac{1}{2}\omega(Tx,x)\\
 &= \sum_\alpha\omega(x,B_\mu(e_\alpha,T\widehat{e}_\alpha)x)+\frac{1}{2}\omega(x,B_\tau(e_\alpha,T\widehat{e}_\alpha)x)+\frac{1}{2}\omega(Tx,x)\\
 &= \mathcal{C}(\frakm')\omega(Tx,x).
\end{align*}
Here we have used the $\frakm$-invariance of $\omega$ in the first line, Lemma~\ref{lem:SymmetrizationsOfSymplecticCovariants}~\eqref{lem:SymmetrizationsOfSymplecticCovariants2} together with the symmetry of $B_\Psi$ in the second and fifth line, Lemma~\ref{lem:SymmetrizationsOfSymplecticCovariants}~\eqref{lem:SymmetrizationsOfSymplecticCovariants3} together with the symmetry of $B_Q$ in the third and fourth line, and Lemma~\ref{lem:BezoutianSum} as well as the definition of $B_\tau$ in the last line.
\end{proof}

\section{The Killing form}\label{sec:KillingForm}

We compute the Killing form $\kappa(X,Y)=\tr(\ad(X)\circ\ad(Y))$\index{1kappaXY@$\kappa(X,Y)$} on $\frakg$. For this let
\begin{equation}
	\kappa_0 := \dim\frakg_1+4.\index{1kappa0@$\kappa_0$}\label{eq:DefKappa0}
\end{equation}

\begin{lemma}\label{lem:KillingForm}
Let $X\in\frakg_i$ and $Y\in\frakg_j$, then $\kappa(X,Y)=0$ unless $i+j=0$. Further,
\begin{align*}
	\kappa(E,F) &= \kappa_0,\\
	\kappa(x,y) &= -\kappa_0\omega(\overline{x},y), && x\in\frakg_1,y\in\frakg_{-1},\\
	\kappa(S+aH,T+bH) &= \kappa_\frakm(S,T) + 2\tr(\ad(S)\circ\ad(T))|_{\frakg_1} + 2\kappa_0ab, && S,T\in\frakm,a,b\in\RR,
\end{align*}
where $\kappa_\frakm$ denotes the Killing form on $\frakm$.
\end{lemma}

\begin{proof}
It is clear that $\kappa(\frakg_i,\frakg_j)=\{0\}$ unless $i+j=0$. The formulas for $\kappa(E,F)$ and $\kappa(H,H)$ are proven in \cite[Proposition 2.2]{SS} and the formula for $\kappa(S,T)$ is clear since $\frakm$ acts trivially on $E$ and $F$. It therefore remains to show the formula for $\kappa(x,y)$, $x\in\frakg_1$ and $y\in\frakg_{-1}$. Using $\ad$-invariance of the Killing form we have
\begin{equation*}
 \kappa(x,y) = \kappa([\overline{x},E],y) = -\kappa(E,[\overline{x},y]) = -\kappa(E,\omega(\overline{x},y)F) = -\kappa_0\omega(\overline{x},y).\qedhere
\end{equation*}
\end{proof}

\section{The Heisenberg nilradical}\label{sec:HeisenbergNilradical}

The unipotent subgroup $\overline{N}$ is a Heisenberg group and hence diffeomorphic to its Lie algebra $\overline{\frakn}$. We identify $\overline{N}\simeq\frakg_{-1}\oplus\frakg_{-2}\simeq V\times\RR$ via
$$ V\times\RR\stackrel{\sim}{\to}\overline{N}, \quad (x,s)\mapsto\overline{n}_{(x,s)}:= \exp(x+sF). $$
The group multiplication in $\overline{N}$ is given by
\begin{equation}
	\overline{n}_{(x,s)}\cdot\overline{n}_{(y,t)} = \overline{n}_{(x+y,s+t+\frac{1}{2}\omega(x,y))}, \qquad x,y\in V,s,t\in\RR.\label{eq:MultiplicationHeisenbergGroup}
\end{equation}
Hence, the map $V\times\RR\to\overline{N},\,(x,s)\mapsto\overline{n}_{(x,s)}$ turns into a group isomorphism if we equip $V\times\RR$ with the product
$$ (x,s)\cdot(y,t) := (x+y,s+t+\tfrac{1}{2}\omega(x,y)). $$

\section{Bruhat decomposition}\label{sec:BruhatDecomposition}

The natural multiplication map
$$ \overline{N}\times M\times A\times N\to G $$
is a diffeomorphism onto an open dense subset of $G$, the open dense Bruhat cell. Hence, every $g\in\overline{N}MAN\subseteq G$ decomposes uniquely into
\begin{equation}
	g = \overline{n}(g)m(g)a(g)n.\index{n2g@$\overline{n}(g)$}\index{m2g@$m(g)$}\index{a2g@$a(g)$}\label{eq:DefinitionBruhatDecomp}
\end{equation}

We identify $\fraka_\CC^*\simeq\CC$ by $\nu\mapsto\nu(H)$. For $\lambda\in\fraka_\CC^*$ we write $a^\lambda=e^{\lambda t}$ where $a=e^{tH}\in A$ with $t\in\RR$.

\begin{lemma}\label{lem:AProjectionOnw0Nbar}
For $(x,s)\in V\times\RR$ we have $w_0^{-1}\overline{n}_{(x,s)}\in\overline{N}MAN$ if and only if $s^2-Q(x)\neq0$. In this case
$$ \chi(m(w_0^{-1}\overline{(x,s)})) = \sgn(s^2-Q(x)), \qquad a(w_0^{-1}\overline{n}_{(x,s)})^\lambda = |s^2-Q(x)|^{\lambda/2} $$
and
$$ \log \overline{n}(w_0^{-1}\overline{n}_{(x,s)}) = \frac{1}{s^2-Q(x)}(\Psi(x)-sx,-s). $$
Moreover, the quartic $Q$ is non-positive if and only if the character $\chi:M\to\{\pm1\}$ is trivial.
\end{lemma}

\begin{proof}
Assume $w_0^{-1}\overline{n}_{(x,s)}=\overline{n}_{(y,t)}m\exp(rH)n\in\overline{N}MAN$. We let both sides act on $E$ by the adjoint representation and then compare the results. Let us first compute $\Ad(w_0^{-1}\overline{n}_{(x,s)})E$. We have
\begin{align*}
 \ad(x+sF)E &= \overline{x}-sH,\\
 \ad(x+sF)^2E &= 2(\mu(x)-sx-s^2F),\\
 \ad(x+sF)^3E &= 6\Psi(x),\\
 \ad(x+sF)^4E &= 24Q(x)F,\\
 \ad(x+sF)^5E &= 0,
\end{align*}
and hence
\begin{equation}
 \Ad(\overline{n}_{(x,s)})E = e^{\ad(x+sF)}E = E+\overline{x}+\mu(x)-sH+\big(\Psi(x)-sx\big)+(Q(x)-s^2)F.\label{eq:AdNbarOnE}
\end{equation}
Applying $\Ad(w_0^{-1})=\Ad(w_0)^{-1}$ yields
\begin{equation}
 \Ad(w_0^{-1}\overline{n}_{(x,s)})E = (s^2-Q(x))E+\overline{\Psi(x)-sx}+\mu(x)+sH-x-F.\label{eq:Adw0NbarOnE}
\end{equation}
Now let us compute $\Ad(\overline{n}_{(y,t)}m\exp(rH)n)E$. Note that $N$ acts trivially on $E$ and $M$ acts on $E$ by the character $\chi:M\to\{\pm1\}$. Therefore, using \eqref{eq:AdNbarOnE}:
\begin{multline}
 \Ad(\overline{n}_{(y,t)}m\exp(rH)n)E = \chi(m)e^{2r}\Ad(\overline{n}_{(y,t)})E\\
 = \chi(m)e^{2r}\left(E+\overline{y}+\mu(y)-tH+\Psi(y)-ty+(Q(y)-t^2)F\right).\label{eq:BruhatAppliedToE}
\end{multline}
Comparing with \eqref{eq:Adw0NbarOnE} shows that
$$ s^2-Q(x) = \chi(m)e^{2r}, \quad \Psi(x)-sx = \chi(m)e^{2r}y, \quad \mu(x) = \chi(m)e^{2r}\mu(y)$$
$$ s = -\chi(m)e^{2r}t, \qquad x = -\chi(m)e^{2r}(\Psi(y)-ty), \qquad 1 = \chi(m)e^{2r}(t^2-Q(y)). $$
The first identity shows that if $\{x\in V:Q(x)>0\}\neq\emptyset$ then there exists $m\in M$ such that $\chi(m)=-1$, because otherwise the non-empty open set of all $w_0^{-1}\overline{n}_{(x,t)}man$ with $t^2-Q(x)<0$ and $m\in M$, $a\in A$, $n\in N$, would have trivial intersection with the open dense Bruhat cell $\overline{N}MAN$. Further, the first, second and fourth identities show that
$$ s^2-Q(x)=\chi(m)e^{2r}\neq0 \qquad \mbox{and} \qquad (y,t)=\frac{1}{s^2-Q(x)}(\Psi(x)-sx,-s). $$
Conversely, if $s^2-Q(x)\neq0$ then let
$$ r:=\frac{1}{2}\log|s^2-Q(x)| \qquad \mbox{and} \qquad (y,t)=\frac{1}{s^2-Q(x)}(\Psi(x)-sx,-s), $$
and choose $m\in M$ such that $\chi(m)=\sgn(s^2-Q(x))$. Using the above computation as well as Lemma~\ref{lem:SymplecticFormulas} one can show that
$$ \Ad(w_0^{-1}\overline{n}_{(x,s)})E = \Ad(\overline{n}_{(y,t)}me^{rH})E. $$
Since the stabilizer of $E$ in $G$ is equal to $M_1N$ there exist $m'\in M_1$ and $n\in N$ such that $w_0^{-1}\overline{n}_{(x,s)}=\overline{n}_{(y,t)}mm'e^{rH}n\in\overline{N}MAN$. This show the claim.
\end{proof}

\begin{lemma}\label{lem:BruhatDecompForNNbar}
	For $(x,t)\in V\times\RR$ and $s\in\RR$ sufficiently close to $0$ we have
	\begin{align*}
		\log\overline{n}(e^{-sE}\overline{n}_{(x,t)}) &= \left(\frac{x+s(\Psi(x)-tx)}{1-2st-s^2(Q(x)-t^2)},\frac{t+s(Q(x)-t^2)}{1-2st-s^2(Q(x)-t^2)}\right),\\
		\left.\frac{d}{ds}\right|_{s=0}m(e^{-sE}\overline{n}_{(x,t)}) &= -\mu(x),\\
		a(e^{-sE}\overline{n}_{(x,t)})^\lambda &= (1-2st-s^2(Q(x)-t^2))^{\frac{\lambda}{2}}.
	\end{align*}
\end{lemma}

\begin{proof}
	For $s$ sufficiently close to $0$ we have $e^{-sE}\overline{n}_{(x,t)}\in\overline{N}MAN$ and write
	\begin{equation}
		e^{-sE}\overline{n}_{(x,t)} = \exp(y+uF)m_se^{rH}\exp(\overline{z}+vE)\label{eq:BruhatDecompForNNbar}
	\end{equation}
	for $y,z\in V$, $r,u,v\in\RR$ and $m_s\in M$. We first act with both sides of \eqref{eq:BruhatDecompForNNbar} on $E$ by the adjoint action. By \eqref{eq:AdNbarOnE} we have
	$$ \Ad(\overline{n}_{(x,t)})E = E+\overline{x}+\mu(x)-tH+(\Psi(x)-tx)+(Q(x)-t^2)F. $$
	Now, $\Ad(e^{-sE})=e^{-s\ad(E)}$ and
	\begin{align*}
		\ad(E)\Ad(\overline{n}_{(x,t)})E &= 2tE+\overline{tx-\Psi(x)}+(Q(x)-t^2)H,\\
		\ad(E)^2\Ad(\overline{n}_{(x,t)})E &= -2(Q(x)-t^2)E,\\
		\ad(E)^3\Ad(\overline{n}_{(x,t)})E &= 0,
	\end{align*}
	hence
	\begin{multline*}
		\Ad(e^{-sE}\overline{n}_{(x,t)})E = (1-2st-s^2(Q(x)-t^2))E + \overline{x+s(\Psi(x)-tx)}+\mu(x)\\
		-(t+s(Q(x)-t^2))H+(\Psi(x)-tx)+(Q(x)-t^2)F.
	\end{multline*}
	On the other hand, by \eqref{eq:BruhatAppliedToE}:
	\begin{multline*}
		\Ad\big(\exp(y+uF)m_se^{rH}\exp(\overline{z}+vE)\big)E\\
		= \chi(m_s)e^{2r}\left(E+\overline{y}+\mu(y)-uH+\Psi(y)-uy+(Q(y)-u^2)F\right).
	\end{multline*}
	Comparing the two expressions shows the formulas for $\overline{n}(e^{-sE}\overline{n}_{(x,t)})$ and $a(e^{-sE}\overline{n}_{(x,t)})$.
	To find $\left.\frac{d}{ds}\right|_{s=0}m_s$, we let both sides of \eqref{eq:BruhatDecompForNNbar} act on an element $\overline{a}\in\frakg_1$. By similar computations we arrive at
	\begin{multline*}
		\Ad\big(e^{-sE}\overline{n}_{(x,t)}\big)\overline{a} = \big(-s\omega(x,a)+s^2\omega(tx+\Psi(x),a)\big)E+\overline{(a-s(ta+\mu(x)a+\omega(x,a)x))}\\
		+2B_\mu(x,a)+\big(-\tfrac{1}{2}\omega(x,a)+s\omega(tx+\Psi(x),a)\big)H\\
		-\big(ta+\mu(x)a+\omega(x,a)x\big)-\omega(tx+\Psi(x),a)F
	\end{multline*}
	and
	\begin{multline*}
		\Ad\big(\exp(y+uF)m_se^{rH}\exp(\overline{z}+vE)\big)\overline{a} \\= e^r\Big(\overline{m_sa}+2B_\mu(y,m_s)-\tfrac{1}{2}\omega(y,m_sa)H-\big(um_sa+\mu(y)m_sa+\omega(y,m_sa)y\big)-\omega(uy+\Psi(y),m_sa)F\Big)\\
		+\omega(a,z)e^{2r}\Big(E+\overline{y}+\mu(y)-uH+\big(\Psi(y)-uy\big)+(Q(y)-u^2)F\Big).
	\end{multline*}
	Comparing the coefficients of $E$ shows
	$$ z = \frac{sx-s^2(tx+\Psi(x))}{1-2st-s^2(Q(x)-t^2)}. $$
	Next, comparing the terms in $\frakg_1$ and using the previously obtained formula for $y$ yields
	$$ e^rm_sa = a-s(ta+\mu(x)a+\omega(x,a)x) + \omega(z,a)(x-s(tx-\Psi(x))). $$
	Note that $\left.\frac{d}{ds}\right|_{s=0}e^r=-t$ and $\left.\frac{d}{ds}\right|_{s=0}z=x$, hence differentiating the above identity gives
	$$ -ta+\left.\frac{d}{ds}\right|_{s=0}m_sa = -(ta+\mu(x)a+\omega(x,a)x) + \omega(x,a)x = -ta-\mu(x)a $$
	and the claim follows.
\end{proof}

\section{Maximal compact subgroups}\label{sec:MaxCptSubgroups}

Let $\theta$\index{1htheta@$\theta$} be a Cartan involution of $G$ and denote by $\theta$ also the corresponding involution on $\frakg$. We may conjugate $\theta$ (or alternatively the parabolic subgroup $P$) such that $\theta H=-H$, then it follows that $\theta\frakg_i=\frakg_{-i}$, $i\in\{-2,-1,0,1,2\}$ and $\theta\frakm=\frakm$. After possibly rescaling $E$ and $F$, we may further assume that $\theta E=-F$ and $\theta F=-E$. Define $J\in\End(V)$\index{J1@$J$} by
$$ Jv := \overline{\theta x} = -\theta\overline{x}, \qquad x\in V, $$
with $\overline{\theta x}=[\theta x,F]$ and $\overline{x}=[x,E]$ as in \eqref{eq:DefBar1} and \eqref{eq:DefBar2}. Then $J^2=-\id_V$ and the bilinear form on $V$ given by
$$ (x|y) = \tfrac{1}{4}\omega(Jx,y), \qquad x,y\in V,\index{1AIP@$(\cdot\vert\cdot)$} $$
is positive definite. Write $|x|^2=(x|x)=\tfrac{1}{4}\omega(Jx,x)$\index{1ANorm@$\vert\cdot\vert$} for the corresponding norm on $V$. Further note that
$$ \Ad(\theta(T))|_V = J\circ\Ad(T)|_V\circ J^{-1} \qquad \mbox{for all }T\in\frakm. $$

Since $\theta$ is an automorphism, it follows that
$$ \omega(Jx,Jy)=\omega(x,y) \qquad \mbox{for all }x,y\in V, $$
so that $J\in\Sp(V,\omega)$. For the symplectic invariants we further have for $x\in V$:
\begin{equation}
 \mu(Jx)=J\circ\mu(x)\circ J^{-1}, \qquad \Psi(Jx) = J\Psi(x), \qquad Q(Jx) = Q(x).\label{eq:JonSymplCov}
\end{equation}

The following result is a converse to the construction of $J$ from $\theta$:

\begin{lemma}\label{lem:CartanInvFromJ}
Let $J\in\End(V)$ and define $\theta:\frakg\to\frakg$ by
$$ \theta E=-F, \qquad \theta H=-H, \qquad \theta F=-E, $$
and for $x\in\frakg_1$, $y\in\frakg_{-1}$ and $T\in\frakm$ by
$$ \theta x=-J\overline{x}, \qquad \theta T=JTJ^{-1}, \qquad \theta y=\overline{Jy}. $$
Then $\theta$ is a Cartan involution of $\frakg$ if and only if the following conditions are satisfied:
\begin{enumerate}[(1)]
\item\label{lem:CartanInvFromJ1} $J^2=-\id_V$,
\item\label{lem:CartanInvFromJ2} $\omega(Jx,x)\geq 0$ for all $x\in V$,
\item\label{lem:CartanInvFromJ3} $\omega(Jx,Jy)=\omega(x,y)$ for all $x,y\in V$,
\item\label{lem:CartanInvFromJ4} $\mu(Jx)=J\mu(x)J^{-1}$ for all $x\in V$.
\end{enumerate}
\end{lemma}

\begin{proof}
Using \eqref{eq:DefBar1}, \eqref{eq:DefBar2} and Lemma~\ref{lem:G1bracketG-1}, one can show that conditions \eqref{lem:CartanInvFromJ3} and \eqref{lem:CartanInvFromJ4} imply that $\theta$ is indeed a Lie algebra automorphism. That $\kappa(\theta X,X)\geq0$ for all $X\in\frakg$ now follows from \eqref{lem:CartanInvFromJ2} and Lemma~\ref{lem:KillingForm}, and that $\theta^2=\id_\frakg$ is a consequence of \eqref{lem:CartanInvFromJ1}.
\end{proof}

Let $K=G^\theta$\index{K1@$K$} be the subgroup of $\theta$-fixed points in $G$. Then $K$ is maximal compact in $G$ and its Lie algebra $\frakk$\index{k3@$\frakk$} has the form
$$ \frakk = \RR(E-F)\oplus\{x+\theta(x):x\in\frakg_1\}\oplus\frakk_\frakm, $$
where $\frakk_\frakm=\frakk\cap\frakm\subseteq\frakm$\index{k3m@$\frakk_\frakm$} is maximal compact in $\frakm$. Write $K_M=K\cap M\subseteq M$ for the corresponding group. Using the decomposition $G=KMAN$ we define a map $H:G\to\fraka$ by
$$ g\in KMe^{H(g)}N.\index{H1g@$H(g)$} $$
We now compute $H$ on $\overline{N}$.

\begin{lemma}\label{lem:IwasawaAProjectionOnNbar}
For $(x,t)\in V\times\RR$ we have
\begin{equation*}
 e^{\lambda(H(\overline{n}_{(x,t)}))} = \Big(1+4|x|^2-\frac{1}{2}\omega(\mu(Jx)x,x)+2t^2+4|tx-\Psi(x)|^2+(t^2-Q(x))^2\Big)^{\lambda/4}.
\end{equation*}
\end{lemma}

We remark that for some special cases a similar formula was obtained in \cite[equation (5.5)]{GP09} and \cite[equation (5.4)]{GP10}.

\begin{remark}\label{rem:SphVectorPositive}
Since, by Lemma~\ref{lem:KillingForm},
\begin{align*}
 \kappa(\mu(x),T) &= \frac{1}{2}\kappa(\ad(x)^2E,T) = \frac{1}{2}\kappa(E,\ad(x)^2T)\\
 &= \frac{1}{2}\kappa(E,[[T,x],x]) = \frac{1}{2}\omega(Tx,x)\kappa(E,F) = \frac{\kappa_0}{2}\omega(Tx,x)
\end{align*}
we have
$$ \omega(\mu(Jx)x,x) = \frac{2}{\kappa_0}\kappa(\mu(x),\mu(Jx)) = \frac{2}{\kappa_0}\kappa(\mu(x),\theta\mu(x)) \leq 0. $$
\end{remark}

\begin{proof}
Let $(x,t)\in V\times\RR$ and write $\overline{n}_{(x,t)}=kman$, so that $a=e^{H(\overline{n}_{(x,t)})}$. Then
$$ \theta(\overline{n}_{(x,t)})^{-1}\overline{n}_{(x,t)} = \theta(n)^{-1}(\theta(m)^{-1}m)a^2n\in\overline{N}MAN. $$
Hence $a=a(\theta(\overline{n}_{(x,t)})^{-1}\overline{n}_{(x,t)})^{1/2}$ which we compute by letting $\theta(\overline{n}_{(x,t)})^{-1}\overline{n}_{(x,t)}$ act on $E$ via the adjoint action. First, by \eqref{eq:AdNbarOnE} we have
\begin{equation}
 \Ad(\overline{n}_{(x,t)})E = E+\overline{x}+\mu(x)-tH+\Psi(x)-tx+(Q(x)-t^2)F.\label{eq:ActionNbarOnE}
\end{equation}
Next, $\theta(\overline{n}_{(x,t)})^{-1}=\exp(-\overline{Jx}+tE)$. We let this act on each of the summands of \eqref{eq:ActionNbarOnE}. First, it acts trivially on $E$:
$$ \Ad(\exp(-\overline{Jx}+tE))E = E. $$
The action on the $\frakg_1$-part $\overline{x}$ is computed using \eqref{eq:CommutatorG1}:
$$ \Ad(\exp(-\overline{Jx}+tE))\overline{x} = \overline{x}+\ad(-\overline{Jx}+tE)\overline{x} = \overline{x}+\omega(Jx,x)E = \overline{x}+4|x|^2E. $$
Next, on the $\frakm$-part $\mu(x)$ we have, using Lemma~\ref{lem:SymmetrizationsOfSymplecticCovariants} and Lemma~\ref{lem:G1bracketG-1}:
\begin{align*}
 & \Ad(\exp(-\overline{Jx}+tE))\mu(x)\\
 ={}& \mu(x) + \ad(-\overline{Jx}+tE)\mu(x) + \tfrac{1}{2}\ad(-\overline{Jx}+tE)^2\mu(x)\\
 ={}& \mu(x) + [\mu(x),\overline{Jx}] + \tfrac{1}{2}\omega(Jx,[\mu(x),Jx])E\\
 ={}& \mu(x) + [\mu(x),\overline{Jx}] - \Big(\tfrac{3}{2}\omega(Jx,B_\Psi(x,x,Jx))+\tfrac{1}{4}\omega(Jx,\omega(x,Jx)x\Big)E\\
 ={}& \mu(x) + [\mu(x),\overline{Jx}] - (6B_Q(x,x,Jx,Jx)-4|x|^4)E.
\end{align*}
For the action on $H$ we find
$$ \Ad(\exp(-\overline{Jx}+tE))H = H + \ad(-\overline{Jx}+tE)H = H + \overline{Jx} -2tE. $$
Next, the action on the $\frakg_{-1}$-part $\Psi(x)-tx$ is computed using Lemma~\ref{lem:G1bracketG-1}:
\begin{align*}
 & \Ad(\exp(-\overline{Jx}+tE))(\Psi(x)-tx)\\
 ={}& (\Psi(x)-tx)+\ad(-\overline{Jx}+tE)(\Psi(x)-tx)+\tfrac{1}{2}\ad(-\overline{Jx}+tE)^2(\Psi(x)-tx)\\
 & +\tfrac{1}{6}\ad(-\overline{Jx}+tE)^3(\Psi(x)-tx)\\
 ={}& (\Psi(x)-tx) + \Big(2B_\mu(Jx,\Psi(x)-tx)+\tfrac{1}{2}\omega(Jx,\Psi(x)-tx)H-t\overline{(\Psi(x)-tx)}\Big)\\
 & +\tfrac{1}{2}\Big(2[B_\mu(Jx,\Psi(x)-tx),\overline{Jx}]+\tfrac{1}{2}\omega(Jx,\Psi(x)-tx)\overline{Jx}-2t\omega(Jx,\Psi(x)-tx)E\Big)\\
 & +\tfrac{1}{6}\Big(2\omega(Jx,[B_\mu(Jx,\Psi(x)-tx),Jx])E\Big)\\
 ={}& \Big(-t\omega(Jx,\Psi(x)-tx)+\tfrac{1}{3}\omega(Jx,[B_\mu(Jx,\Psi(x)-tx),Jx])\Big)E + \cdots.
\end{align*}
And finally, for the action on $F$ we have, once again using Lemma~\ref{lem:G1bracketG-1} as well as Lemma~\ref{lem:SymmetrizationsOfSymplecticCovariants}:
\begin{align*}
 \Ad(\exp(-\overline{Jx}+tE))F ={}& F + \ad(-\overline{Jx}+tE)F + \tfrac{1}{2}\ad(-\overline{Jx}+tE)^2F + \tfrac{1}{6}\ad(-\overline{Jx}+tE)^3F\\
 & + \tfrac{1}{24}\ad(-\overline{Jx}+tE)^4F\\
 ={}& F + \Big(-Jx+tH\Big) + \tfrac{1}{2}\Big(-2\mu(Jx)+2t\overline{Jx}-2t^2E\Big)\\
 & +\tfrac{1}{6}\Big(6\overline{\Psi(Jx)}\Big)+\tfrac{1}{24}\Big(24Q(Jx)E\Big).
\end{align*}
Altogether we obtain
\begin{multline*}
 \Ad(\theta(\overline{n}_{(x,t)})^{-1}\overline{n}_{(x,t)})E = \Big(1+4|x|^2-6B_Q(x,x,Jx,Jx)+4|x|^4+2t^2-t\omega(Jx,\Psi(x)-tx)\\
 +\tfrac{1}{3}\omega(Jx,[B_\mu(Jx,\Psi(x)-tx),Jx])+(Q(Jx)-t^2)(Q(x)-t^2)\Big)E + \cdots
\end{multline*}
This expression can be simplified. In fact, by Lemma~\ref{lem:SymmetrizationsOfSymplecticCovariants}:
\begin{align*}
 \omega(Jx,\Psi(x)) &= 4B_Q(Jx,x,x,x),\\
 \omega(Jx,[B_\mu(Jx,x),Jx]) &= -12B_Q(Jx,Jx,Jx,x),\\
 \omega(Jx,[B_\mu(Jx,\Psi(x)),Jx]) &= -12B_Q(Jx,Jx,Jx,\Psi(x)).
\end{align*}
Hence
\begin{multline*}
 1+4|x|^2-6B_Q(Jx,Jx,x,x)+4|x|^4+2t^2-4tB_Q(Jx,x,x,x)+4t^2|x|^2\\
 -4B_Q(Jx,Jx,Jx,\Psi(x))+4tB_Q(Jx,Jx,Jx,x)+(Q(Jx)-t^2)(Q(x)-t^2).
\end{multline*}
Finally, using \eqref{eq:JonSymplCov} we find that
\begin{align*}
 B_Q(Jx,x,x,x) &= (x|\Psi(x)), & B_Q(Jx,Jx,Jx,\Psi(x)) &= -|\Psi(x)|^2,\\
 B_Q(Jx,Jx,Jx,x) &= -(x|\Psi(x)), & B_Q(Jx,Jx,x,x) &= \frac{1}{12}\omega(\mu(Jx)x,x)+\frac{2}{3}|x|^4
\end{align*}
and $Q(Jx)=Q(x)$ and the claimed formula follows.
\end{proof}

\section{Hermitian vs. non-Hermitian}\label{sec:HermVsNonHerm}

We derive several equivalent properties characterizing the Hermitian Lie algebras among all Heisenberg graded Lie algebras.

\begin{theorem}\label{thm:CharacterizationHermitian}
The following are equivalent:
\begin{enumerate}[(1)]
\item\label{thm:CharacterizationHermitian1} The group $G$ is of Hermitian type,
\item\label{thm:CharacterizationHermitian2} There exists $J\in\frakm$ such that $\ad(J)^2|_{\frakg_{\pm1}}=-1$ and $(x,y)\mapsto\omega(\ad(J)x,y)$ is positive definite on $V$,
\item\label{thm:CharacterizationHermitian3} The minimal adjoint orbits $\Omin$ and $-\Omin$ are distinct,
\item\label{thm:CharacterizationHermitian4} The quartic $Q$ is non-positive,
\item\label{thm:CharacterizationHermitian5} The character $\chi$ of $M$ is trivial.
\end{enumerate}
\end{theorem}

\begin{proof}
We first show \eqref{thm:CharacterizationHermitian1}$\Leftrightarrow$\eqref{thm:CharacterizationHermitian2}. If $G$ is of Hermitian type then the center of $\frakk$ is non-trivial, i.e. there exists $0\neq X\in\frakk$ such that $[X,Y]=0$ for any $Y\in\frakk$. Write $X=s(E-F)+(x+\theta(x))+T$ with $s\in\RR$, $x\in\frakg_1$ and $T\in\frakk_\frakm$. Then, by \eqref{eq:DefBar1} and \eqref{eq:DefBar2}:
$$ 0 = [X,E-F] = -\overline{x}+\overline{\theta(x)} $$
where $\overline{x}\in\frakg_{-1}$ and $\overline{\theta(x)}\in\frakg_1$. Hence $x=0$. Further, for every $v\in\frakg_1$ we have
$$ 0 = [X,v+\theta(v)] = s(\overline{\theta(v)}-\overline{v}) + [T,v]+[T,\theta(v)], $$
and therefore $[T,v]+s\overline{\theta(v)}=0$ and $[T,\theta(v)]-s\overline{v}=0$. This implies
$$ \ad(T)v = -s\overline{\theta(v)} \qquad \mbox{and} \qquad \ad(T)w = s\overline{\theta(w)} $$
for $v\in\frakg_1$ and $w\in\frakg_{-1}$. Now, if $s=0$ then $\ad(T)=0$ on $\frakg_{\pm1}$, and trivially also on $\frakg_{\pm2}$, hence on $\frakg$ which is only possible if $T=0$, because $\frakg$ is assumed to be simple. So $s\neq0$ and therefore $\ad(T)^2|_{\frakg_{\pm1}}=-s^2$. Then $J:=s^{-1}T\in\frakm$ satisfies $\ad(J)^2|_{\frakg_{\pm1}}=-1$. Further, for $X,Y\in V=\frakg_{-1}$ we have, by Lemma~\ref{lem:KillingForm}
$$ 0\leq -B(X,\theta Y) = -B(\theta X,Y) = p\cdot\omega(\ad(J)X,Y) $$
and hence $\omega(\ad(J)X,Y)$ is positive definite.\\
Conversely, let $J\in\frakm$ with $\ad(J)^2|_{\frakg_{\pm1}}=-1$ and $\omega(\ad(J)X,Y)$ positive definite on $V$. Note that $\widetilde{J}:=\exp(\frac{\pi}{2}J)\in M$ satisfies $\Ad(\widetilde{J}^2)=-1$ and $\Ad(\widetilde{J})|_{\frakg_{\pm1}}=\ad(J)|_{\frakg_{\pm1}}$. Then one can define an involution $\theta$ on $\frakg$ by
$$ \theta(E):=-F, \qquad \theta(F):=-E, \qquad \theta(H):=-H, $$
and for $v\in\frakg_1$, $w\in\frakg_{-1}$ and $T\in\frakm$ by
$$ \theta(v) := -\overline{\ad(J)v}, \qquad \theta(w) := \overline{\ad(J)w}, \qquad \theta(T) = \Ad(\widetilde{J})T. $$
It is immediate that $\theta^2=1$ and hence $\theta$ is in fact an involution. We now show that $\theta$ is a Cartan involution. Then, by the same computations as above, the center of the corresponding maximal compact subalgebra $\frakk=\frakg^\theta$ is spanned by $X=E-F+J$ and hence $G$ is of Hermitian type. To show that $\theta$ is a Cartan involution we compute $B(X,\theta(X))$ for $X\in\frakg_i$, $i=-2,-1,0,1,2$.
\begin{enumerate}[(1)]
\item For $X=E\in\frakg_2$ we have $B(X,\theta(X))=-\kappa_0<0$.
\item For $X\in\frakg_1$ we have $B(X,\theta(X))=\kappa_0\cdot\omega(\overline{X},\overline{\ad(J)X})=-\kappa_0\cdot\omega(\ad(J)\overline{X},\overline{X})$ which is strictly negative for $X\neq0$.
\item For $X\in\frakm$ we have $\theta(X)=\widetilde{J}X\widetilde{J}^{-1}$. Since $\frakm\subseteq\sp(V,\omega)$ via $X\mapsto\ad(X)|_{\frakg_{-1}}$ the Killing form $B_\frakm$ has to be a scalar multiple of the Killing form of $\sp(V,\omega)$ on each simple factor of $\frakm$. It is well-known that the involution $\theta(X)=\widetilde{J}X\widetilde{J}^{-1}$ extends to a Cartan involution of $\sp(V,\omega)$ and hence $B(X,\theta(X))<0$ for all $X\neq0$.
\item For $X=H\in\fraka$ we have $B(X,\theta(X))=-2\kappa_0<0$.
\item For $X\in\frakg_{-1}$ we have $B(X,\theta(X))=-\kappa_0\cdot\omega(\ad(J)X,X)$ and hence strictly negative for $X\neq0$.
\item For $X=F\in\frakg_{-2}$ we have $B(X,\theta(X))=-\kappa_0<0$.
\end{enumerate}

Next, the equivalence \eqref{thm:CharacterizationHermitian1}$\Leftrightarrow$\eqref{thm:CharacterizationHermitian3} follows from \cite[Theorem 1.4]{Oku15}. (Note that $\frakg$ cannot have a complex structure, because in this case $\frakg_2$ as the highest root space would have a complex structure and have real dimension $\geq2$.)

Let us show \eqref{thm:CharacterizationHermitian3}$\Leftrightarrow$\eqref{thm:CharacterizationHermitian4}. If the orbits $\Omin$ and $-\Omin$ are distinct then $\Ad(M)\cdot E=\{E\}$. Hence, by Lemma~\ref{lem:AProjectionOnw0Nbar} the quartic $Q$ is non-positive. If conversely $Q\leq0$ then by Lemma~\ref{lem:AProjectionOnw0Nbar} we have $\Ad(m)E=E$ for all $m\in M$. We obtain $\Ad(MAN)E=\{e^{2r}E:r\in\RR\}$ and further
\begin{multline*}
	\Ad(\overline{N}MAN)E = \{e^{2r}\left(E+\overline{x}+\mu(x)-sH+(\Psi(x)-sx)+(Q(x)-s^2)F\right):\\
	r\in\RR,(x,s)\in V\times\RR\}.
\end{multline*}
In particular, the coefficient of $E$ of every element in $\Ad(\overline{N}MAN)E$ is positive. Since the set $\Ad(\overline{N}MAN)E$ is dense in $\Omin$, the element $-E$ cannot be in $\Omin$ and therefore, $\Omin$ and $-\Omin$ are distinct.

Finally, \eqref{thm:CharacterizationHermitian4}$\Leftrightarrow$\eqref{thm:CharacterizationHermitian5} follows from Lemma~\ref{lem:AProjectionOnw0Nbar}.
\end{proof}

\begin{corollary}
If $G$ is Hermitian, then the formula in Lemma~\ref{lem:IwasawaAProjectionOnNbar} simplifies to
$$ e^{\lambda(H(\overline{n}_{(x,t)}))} = (1+2|x|^2-Q(x)+t^2)^{\lambda/2}. $$
\end{corollary}

\begin{proof}
If $G$ is Hermitian, $J\in\frakm$ and hence, by the $\frakm$-equivariance of $B_\Psi$ and $B_Q$:
$$ J\Psi(x)=3B_\Psi(Jx,x,x) \qquad \mbox{and} \qquad B_Q(Jx,x,x,x)=0, $$
so that by Lemma~\ref{lem:SymmetrizationsOfSymplecticCovariants}:
\begin{align*}
 \omega(\mu(Jx)x,x) &= 12B_Q(Jx,Jx,x,x)-8|x|^4 = 3\omega(Jx,B_\Psi(Jx,x,x))-8|x|^4\\
 &= \omega(Jx,J\Psi(x)) - 8|x|^4 = 4Q(x) -8|x|^4,\\
 (x|\Psi(x)) &= \frac{1}{4}\omega(Jx,\Psi(x)) = B_Q(Jx,x,x,x) = 0.
\end{align*}
Further, polarizing Lemma~\ref{lem:SymplecticFormulas}~\eqref{lem:SymplecticFormulas2} gives $B_\Psi(\Psi(x),x,x)=\frac{1}{3}Q(x)x$ and hence
\begin{align*}
 |\Psi(x)|^2 &= \frac{1}{4}\omega(J\Psi(x),\Psi(x)) = B_Q(J\Psi(x),x,x,x) = -3B_Q(\Psi(x),x,x,Jx)\\
 &= -\frac{3}{4}\omega(Jx,B_\Psi(\Psi(x),x,x)) = -Q(x)|x|^2.
\end{align*}
Inserting this into the formula in Lemma~\ref{lem:IwasawaAProjectionOnNbar} and rearranging shows the claim.
\end{proof}

\begin{remark}
Note that if $\frakg=\su(n+1,1)$ then $V=\CC^n$ with $Q(x)=-|x|^4$, so that the above expression becomes
$$ e^{\lambda(H(\overline{n}_{(x,t)}))}=((1+|x|^2)^2+t^2). $$
This formula is well-known (see e.g. \cite[Theorem IX.3.8]{Hel78}).
\end{remark}

In the case where $G$ is Hermitian we further show that the pair $(\sp(V,\omega),\frakm)$ is of \emph{holomorphic type} in the sense of Kobayashi~\cite[Definition 1.4]{Kob08}. Recall that for reductive Hermitian Lie algebras $\frakh\subseteq\frakg$, the pair $(\frakg,\frakh)$ is said to be of holomorphic type if there exists a Cartan involution $\theta$ of $\frakg$ which leaves $\frakh$ invariant and an element $z\in\frakk_\frakh=\frakh^\theta$ such that $\ad(z)=0$ on $\frakk=\frakg^\theta$ and $\ad(z)^2=-1$ on $\frakp=\frakg^{-\theta}$. In this case the natural embedding $H/(K\cap H)\subseteq G/K$ of Hermitian symmetric spaces is holomorphic.

\begin{corollary}
If $G$ is Hermitian then $\frakm$ is also Hermitian and the pair $(\sp(V,\omega),\frakm)$ is holomorphic.
\end{corollary}

\begin{proof}
We can choose the element $z$ to be $z=\tfrac{1}{2}J\in\frakm$. A Cartan involution of $\sp(V,\omega)$ is given by $\theta(T)=JTJ^{-1}$, hence $z\in\sp(V,\omega)^\theta$. Further $\ad(z)=0$ on $\sp(V,\omega)^\theta=\{T\in\sp(V,\omega):TJ=JT\}$, and for $T\in\sp(V,\omega)^{-\theta}=\{T\in\sp(V,\omega):TJ=-JT\}$ we have $\ad(z)T=JT$ and $\ad(z)^2T=\tfrac{1}{2}[J,JT]=J^2T=-T$.
\end{proof}

\chapter[Principal series representations]{Principal series representations and intertwining operators}

In this chapter we define the degenerate principal series representations induced from the parabolic subgroup $P$ (see Section~\ref{sec:DegPrincipalSeries}), describe their realization in the non-compact picture on functions on the Heisenberg group (see Section~\ref{sec:NonCptPicture}) and briefly discuss standard intertwining operators between them (see Section~\ref{sec:IntertwiningOperators}). We believe that the formula in Proposition~\ref{prop:KnappSteinIntegralFormula} for the integral kernel of the intertwining operators is new. Moreover, in order to take the Fourier transform of the non-compact picture, we recall the Heisenberg group Fourier transform in Section~\ref{sec:HeisFT}. Finally, in Section~\ref{sec:SchroedingerModel} we use the Schr\"{o}dinger model of the irreducible unitary representations of the Heisenberg group to extend the Fourier transform to distributions and apply it to the non-compact model of the degenerate principal series to obtain a new model, the Fourier transformed picture.

\section{Degenerate principal series representations}\label{sec:DegPrincipalSeries}

For a smooth admissible representation $(\zeta,V_\zeta)$\index{1fzeta@$\zeta$}\index{Vzeta@$V_\zeta$} of $M$ and $\nu\in\fraka_\CC^*$ we let $(\widetilde{\pi}_{\zeta,\nu},\widetilde{I}(\zeta,\nu))$ be the induced representation $\Ind_P^G(\zeta\otimes e^\nu\otimes\1)$, acting by left-translation on
$$ \widetilde{I}(\zeta,\nu) = \{f\in C^\infty(G,V_\zeta):f(gman)=a^{-\nu-\rho}\zeta(m)^{-1}f(g)\mbox{ for all }man\in MAN\}. $$
Here $\rho\in\fraka^*$\index{1rho@$\rho$} denotes as usual the half sum of all positive roots.

\section{The non-compact picture}\label{sec:NonCptPicture}

Since $\overline{N}MAN\subseteq G$ is open dense, functions in $\widetilde{I}(\zeta,\nu)$ are uniquely determined by their restriction to $\overline{N}$. Therefore, we define for any $f\in\widetilde{I}(\zeta,\nu)$ a $V_\zeta$-valued function $f_{\overline{\frakn}}$ on $V\times\RR$ by
$$ f_{\overline{\frakn}}(x,s) := f(\overline{n}_{(x,s)}), \qquad (x,s)\in V\times\RR\index{fn@$f_{\overline{\frakn}}$} $$
and let
$$ I(\zeta,\nu) := \{f_{\overline{\frakn}}:f\in\widetilde{I}(\zeta,\nu)\}.\index{Izetanu@$I(\zeta,\nu)$} $$
The representation $\widetilde{\pi}_{\zeta,\nu}$ on $\widetilde{I}(\zeta,\nu)$ defines an equivalent representation $\pi_{\zeta,\nu}$\index{1pi1zetanu@$\pi_{\zeta,\nu}$} on $I(\zeta,\nu)$ by
$$ \pi_{\zeta,\nu}(g)f_{\overline{\frakn}} = (\widetilde{\pi}_{\zeta,\nu}(g)f)_{\overline{\frakn}}, \qquad g\in G,f\in\widetilde{I}(\zeta,\nu). $$
The realization $(I(\zeta,\nu),\pi_{\zeta,\nu})$ is called the \textit{non-compact picture} of the degenerate principal series. Note that
$$ \calS(V\times\RR)\otimeshat V_\zeta\subseteq I(\zeta,\nu)\subseteq C^\infty_{\mathrm{temp}}(V\times\RR)\otimeshat V_\zeta, $$
where $\calS(V\times\RR)$ denotes the space of Schwartz functions, i.e. smooth functions decreasing rapidly at infinity together with all their derivatives, and $C^\infty_{\mathrm{temp}}(V\times\RR)$ denotes the space of smooth functions which grow at most polynomially at infinity.

We compute the action of $MA\overline{N}$ and $w_0$ in this realization:

\begin{proposition}\label{prop:GroupActionNonCptPicture}
For $f\in I(\zeta,\nu)$ and $(x,s)\in V\times\RR$ we have
\begin{align*}
 \pi_{\zeta,\nu}(\overline{n}_{(y,t)})f(x,s) &= f(x-y,s-t+\tfrac{1}{2}\omega(x,y)), && \overline{n}_{(y,t)}\in\overline{N},\\
 \pi_{\zeta,\nu}(m)f(x,s) &= \zeta(m)f(m^{-1}x,\chi(m)^{-1}s), && m\in M,\\
 \pi_{\zeta,\nu}(e^{rH})f(x,s) &= e^{(\nu+\rho)r}f(e^rx,e^{2r}s), && e^{rH}\in A.
\end{align*}
Moreover, for $\zeta=\1$ the trivial representation of $M$ we have
$$ \pi_{\1,\nu}(w_0^{\pm1})f(x,s) = |s^2-Q(x)|^{-\frac{\nu+\rho}{2}}f\left(\pm\frac{\Psi(x)-sx}{s^2-Q(x)},-\frac{s}{s^2-Q(x)}\right). $$
\end{proposition}

\begin{proof}
The formulas for $M$, $A$ and $\overline{N}$ are obvious. For $\pi_{\1,\nu}(w_0)$ we use Lemma~\ref{lem:AProjectionOnw0Nbar}.
\end{proof}

We can use these formulas to find the differentiated action $d\pi_{\zeta,\nu}(X)=\left.\frac{d}{dt}\right|_{t=0}\pi_{\zeta,\nu}(\exp(tX))$\index{dpi1zetanu@$d\pi_{\zeta,\nu}$} of the Lie algebra $\frakg$ on $I(\zeta,\nu)$. To simplify the formulas, we let $\EE$ denote the weighted Euler operator on $V\times\RR$, i.e.
$$ \EE = \sum_\alpha x_\alpha\frac{\partial}{\partial x_\alpha} + 2s\frac{\partial}{\partial s},\index{E@$\EE$} $$
where $x=\sum_\alpha x_\alpha e_\alpha$ for any basis $(e_\alpha)$ of $V$.

\begin{corollary}\label{cor:LieAlgActionNonCptPicture}
The Lie algebra representation $d\pi_{\zeta,\nu}$ of $\frakg$ on a $V_\zeta$-valued function in $(x,s)\in V\times\RR$ is given by
\begin{align*}
 d\pi_{\zeta,\nu}(F) &= -\partial_s,\\
 d\pi_{\zeta,\nu}(v) &= -\partial_v+\tfrac{1}{2}\omega(x,v)\partial_s, && v\in\frakg_{-1}\\
 d\pi_{\zeta,\nu}(T) &= -\partial_{Tx}+d\zeta(T), && T\in\frakm,\\
 d\pi_{\zeta,\nu}(H) &= \EE+(\nu+\rho),\\
 d\pi_{\zeta,\nu}(w) &= \partial_{\mu(x)\overline{w}+\omega(x,\overline{w})x-s\overline{w}}+\tfrac{1}{2}\omega(sx+\Psi(x),\overline{w})\partial_s+\tfrac{\nu+\rho}{2}\omega(x,\overline{w})-2d\zeta(B_\mu(x,\overline{w})), && w\in\frakg_1,\\
 d\pi_{\zeta,\nu}(E) &= \partial_{sx+\Psi(x)}+(s^2+Q(x))\partial_s+(\nu+\rho)s+d\zeta(\mu(x)).
\end{align*}
\end{corollary}

\begin{proof}
The formulas for $d\pi_{\zeta,\nu}(F)$, $d\pi_{\zeta,\nu}(v)$, $d\pi_{\zeta,\nu}(T)$ and $d\pi_{\zeta,\nu}(H)$ follow by differentiating the corresponding group actions in Proposition~\ref{prop:GroupActionNonCptPicture}. For $d\pi_{\zeta,\nu}(E)$ we use Lemma~\ref{lem:BruhatDecompForNNbar} to find for $f\in\widetilde{I}(\zeta,\nu)$:
\begin{align*}
	d\pi_{\zeta,\nu}(E)f_{\overline{\frakn}}(x,s) ={}& \left.\frac{d}{dt}\right|_{t=0}f(\exp(-tE)\overline{n}_{(x,s)})\\
	={}& \left.\frac{d}{dt}\right|_{t=0}\Bigg[(1-2st-t^2(Q(x)-s^2))^{-\frac{\nu+\rho}{2}}\zeta(m(e^{-tE}\overline{n}_{(x,s)}))^{-1}\\
	& \hspace{3cm}\times f_{\overline{\frakn}}\left(\frac{x+t(\Psi(x)-sx)}{1-2st-t^2(Q(x)-s^2)},\frac{s+t(Q(x)-s^2)}{1-2st-t^2(Q(x)-s^2)}\right)\Bigg]\\
	={}& \left.\frac{d}{dt}\right|_{t=0}\left[(1-2st-t^2(Q(x)-s^2))^{-\frac{\nu+\rho}{2}}\right]f_{\overline{\frakn}}(x,s)\\
	& \hspace{.5cm}+\left.\frac{d}{dt}\right|_{t=0}\left[\zeta(m(e^{-tE}\overline{n}_{(x,s)}))^{-1}\right]f_{\overline{\frakn}}(x,s)\\
	& \hspace{.5cm}+\left.\frac{d}{dt}\right|_{t=0}f_{\overline{\frakn}}\left(\frac{x+t(\Psi(x)-sx)}{1-2st-t^2(Q(x)-s^2)},\frac{s+t(Q(x)-s^2)}{1-2st-t^2(Q(x)-s^2)}\right).
\end{align*}
The first of the three terms is seen to be $(\nu+\rho)sf_{\overline{\frakn}}(x,s)$. For the second term we use from Lemma~\ref{lem:BruhatDecompForNNbar} that $\frac{d}{dt}|_{t=0}m(e^{-tE}\overline{n}_{(x,s)})=-\mu(x)$, so that the second term becomes $d\zeta(\mu(x))f_{\overline{\frakn}}(x,s)$. The third term equals $s\partial_{\Psi(x)}f_{\overline{\frakn}}(x,s) + (s^2+Q(x))\partial_sf_{\overline{\frakn}}(s,x)$ by a careful application of the chain rule.\\
Finally, $d\pi_{\zeta,\nu}(w)$ can be obtained from $d\pi_{\zeta,\nu}(E)$ and $d\pi_{\zeta,\nu}(\overline{w})$ by $d\pi_{\zeta,\nu}(w)=[d\pi_{\zeta,\nu}(\overline{w}),d\pi_{\zeta,\nu}(E)]$.
\end{proof}

\section{Intertwining operators}\label{sec:IntertwiningOperators}

For $\Re\nu\gg0$ the standard Knapp--Stein intertwining operator $\widetilde{A}(\zeta,\nu):\widetilde{I}(\zeta,\nu)\to\widetilde{I}(w_0\zeta,-\nu)$ (with $[w_0\zeta](m)=\zeta(w_0^{-1}mw_0)$) is given by the convergent integral
$$ \widetilde{A}(\zeta,\nu)f(g) = \int_{\overline{N}} f(gw_0\overline{n}) \,d\overline{n}. $$
It is well-known that $\widetilde{A}(\zeta,\nu)$ extends meromorphically to all $\nu\in\CC$. We consider the Knapp--Stein operator $A(\zeta,\nu)$ in the non-compact picture:
$$ A(\zeta,\nu)f_{\overline{\frakn}} := (\widetilde{A}(\zeta,\nu)f)_{\overline{\frakn}} \qquad (f\in\widetilde{I}(\zeta,\nu)).\index{A1zetanu@$A(\zeta,\nu)$} $$

\begin{proposition}\label{prop:KnappSteinIntegralFormula}
Assume $\zeta=\1$ is the trivial representation, then the operator $A(\zeta,\nu)$ is the convolution operator
$$ A(\1,\nu)f(x,s) = \int_{V\times\RR} |t^2-Q(y)|^{\frac{\nu-\rho}{2}} f((x,s)\cdot(y,t)) \,d(y,t). $$
\end{proposition}

\begin{proof}
In \cite[Chapter VII, \S7]{Kna86} it is shown that
$$ \widetilde{A}(\zeta,\nu)f(g) = \int_{\overline{N}} a(w_0^{-1}\overline{n})^{\nu-\rho}\zeta(m(w_0^{-1}\overline{n}))f(g\overline{n}) \,d\overline{n} $$
with $m(w_0^{-1}\overline{n})$ and $a(w_0^{-1}\overline{n})$ as in \eqref{eq:DefinitionBruhatDecomp}. Then, Lemma~\ref{lem:AProjectionOnw0Nbar} immediately yields the claim.
\end{proof}

\begin{remark}
	Of course one can also write down a formula for $A(\zeta,\nu)$ for general $\zeta$, but we will not need this in what follows.
\end{remark}

\section{The Fourier transform on the Heisenberg group}\label{sec:HeisFT}

The infinite-dimensional irreducible unitary representations of the Heisenberg group $\overline{N}$ are parameterized by their central character $i\lambda\in i\RR^\times$. More precisely, for each $\lambda\in\RR^\times$ there exists a unique (up to equivalence) infinite-dimensional unitary representation $(\sigma_\lambda,\calH_\lambda)$\index{1sigmalambda@$\sigma_\lambda$}\index{H2lambda@$\calH_\lambda$} of $\overline{N}$ such that $d\sigma_\lambda(0,t)=i\lambda t$ for $(0,t)\in\overline{\frakn}\simeq V\times\RR$. There are two standard realizations of $\sigma_\lambda$, the Schrödinger model and the Fock model. The Schrödinger model is realized on the space $\calH_\lambda=L^2(\Lambda)$ for a Lagrangian subspace $\Lambda\subseteq V$, whereas the Fock model is realized on the Fock space $\calH_\lambda=\calF(V)$ consisting of holomorphic functions on $V$ (with respect to a certain complex structure) which are square-integrable with respect to a Gaussian measure.

Since $M$ acts on $\overline{N}$ by automorphisms, the map $\overline{n}\mapsto\sigma_\lambda(\Ad(m)\overline{n})$ defines an irreducible unitary representation of $\overline{N}$ with central character $i\chi(m)\lambda$, and hence there exists a projective unitary representation $\omega_{\met,\lambda}$\index{1zomegametlambda@$\omega_{\met,\lambda}$} of $M$ on the same representation space such that
\begin{equation}
 \sigma_\lambda(\Ad(m)\overline{n}) = \omega_{\met,\lambda}(m)\circ\sigma_{\chi(m)\lambda}(\overline{n})\circ\omega_{\met,\lambda}(m)^{-1} \qquad \mbox{for all }m\in M,\overline{n}\in\overline{N}.\label{eq:DefMetaplecticRep}
\end{equation}
Since $M_1=\{m\in M:\chi(m)=1\}$ acts symplectically on $V$, the representation $\omega_{\met,\lambda}|_{M_1}$ is simply the restriction of the metaplectic representation of $\operatorname{Sp}(V,\omega)$ to $M_1$, viewed as a projective representation.

For $f\in L^1(\overline{N})$ we form
\begin{equation}
	\sigma_\lambda(f) = \int_{\overline{N}} f(\overline{n})\sigma_\lambda(\overline{n})\,d\overline{n},\label{eq:DefSigmaOfF}
\end{equation}
where $d\overline{n}$ is a fixed Haar measure on $\overline{N}$. If we identify $X\in\overline{\frakn}$ with the corresponding left-invariant vector field on $\overline{N}$, then
\begin{equation}
	\sigma_\lambda(Xf) = -\sigma_\lambda(f)\circ d\sigma_\lambda(X).\label{eq:FTofVectorField}
\end{equation}
Further, if $f*g\in L^1(\overline{N})$ denotes the convolution of $f,g\in L^1(\overline{N})$ given by
$$ (f*g)(x) = \int_{\overline{N}}f(y)g(xy^{-1})\,dy,\index{1AStar@$*$} $$
then
$$ \sigma_\lambda(f*g) = \sigma_\lambda(g)\circ\sigma_\lambda(f). $$
We have the following Plancherel Formula (after appropriate normalization of the measures involved):
$$ \|f\|_{L^2(\overline{N})}^2 = \int_{\RR^\times}\|\sigma_\lambda(f)\|_\HS^2 |\lambda|^{\frac{\dim V}{2}}\,d\lambda. $$
Here $\|T\|_\HS^2=\tr(TT^*)$\index{1@$\|\cdot\|_\HS$} denotes the Hilbert--Schmidt norm of a Hilbert--Schmidt operator on a Hilbert space. Realizing all infinite-dimensional irreducible unitary representations of $\overline{N}$ on the same Hilbert space $\calH_\lambda=\calH$ and writing $\HS(\calH)$\index{H1SH@$\HS(\calH)$} for the Hilbert space of Hilbert--Schmidt operators on $\calH$, the Fourier transform can be viewed as an isometric isomorphism
\begin{equation}
	\calF: L^2(\overline{N})\to L^2(\RR^\times,\HS(\calH);|\lambda|^{\frac{\dim V}{2}}\,d\lambda), \quad \calF f(\lambda) = \sigma_\lambda(f).\label{eq:DefFT}
\end{equation}

\section[The Fourier transform of distributions]{The Schr\"{o}dinger model and the Fourier transform of distributions}\label{sec:SchroedingerModel}

For the Schr\"{o}dinger model of the infinite-dimensional irreducible unitary representations of $\overline{N}\simeq V\times\RR$ one has to choose a Lagrangian subspace $\Lambda\subseteq V$\index{1Lambda@$\Lambda$} and a Lagrangian complement $\Lambda^*\subseteq V$\index{1LambdaStar@$\Lambda^*$}. Then the Schr\"{o}dinger model is a realization of $\sigma_\lambda$ on $\calH=L^2(\Lambda)$ given by
\begin{equation}
 \sigma_\lambda(z,t)\varphi(x) = e^{i\lambda t}e^{i\lambda(\omega(z'',x)+\frac{1}{2}\omega(z',z''))}\varphi(x-z') \qquad (x\in\Lambda,\varphi\in L^2(\Lambda))\label{eq:DefSchroedingerModel}
\end{equation}
for $z=(z',z'')\in\Lambda\oplus\Lambda^*=V$ and $t\in\RR$. The corresponding differentiated representation of $\overline{\frakn}\simeq V\times\RR$ is given by
$$ d\sigma_\lambda(z,t) = -\partial_{z'}+i\lambda\omega(z'',x)+i\lambda t, \qquad z=(z',z'')\in\Lambda\oplus\Lambda^*=V, t\in\RR. $$

For $\varphi\in\calS(\overline{N})$ we have
\begin{align*}
[\sigma_\lambda(u)\varphi](y) &= \int_{\overline{N}} u(z,t)\sigma_\lambda(z,t)\varphi(y)\,dz\,dt\\
&= \int_{\Lambda}\int_{\Lambda^*}\int_\RR u(z',z'',t)e^{i\lambda t}e^{i\lambda(\omega(z'',y)+\frac{1}{2}\omega(z',z''))}\varphi(y-z')\,dt\,dz''\,dz'\\
&= \int_\Lambda\left(\int_{\Lambda^*}\int_\RR u(y-x,z'',t)e^{i\lambda t}e^{-\frac{i\lambda}{2}(\omega(x+y,z''))}\,dt\,dz''\right)\varphi(x)\,dx\\
&= \int_{\Lambda}\widehat{u}(\lambda,x,y)\varphi(x)\,dx,
\end{align*}
where
$$ \widehat{u}(\lambda,x,y) = \int_{\Lambda^*}\int_\RR u(y-x,z'',t)e^{i\lambda t}e^{-\frac{i\lambda}{2}\omega(x+y,z'')}\,dt\,dz'' = \calF_2\calF_3u(y-x,\tfrac{\lambda}{2}(x+y),-\lambda).\index{uhatlambdaxy@$\widehat{u}(\lambda,x,y)$} $$
Here $\calF_2$ denotes the symplectic Fourier transform with respect to $\omega$ in the second variable, and $\calF_3$ denotes the Euclidean Fourier transform with respect to the third variable.

\begin{proposition}\label{prop:KernelFT}
The linear map
$$ \calS'(\overline{N})\to\calD'(\RR^\times)\otimeshat\calS'(\Lambda\times\Lambda)\simeq\calD'(\RR^\times)\otimeshat\Hom(\calS(\Lambda),\calS'(\Lambda)), \quad u\mapsto\widehat{u}, $$
is defined and continuous. Its kernel is given by those distributions which are polynomial in $t$.
\end{proposition}

\begin{proof}
Clearly $\calF_2\calF_3$ is a topological isomorphism $\calS'(\overline{N})\to\calS'(\Lambda\times\Lambda\times\RR)$. Restricting the last coordinate to $\RR^\times$ defines a continuous linear map $\calS'(\Lambda\times\Lambda\times\RR)\simeq\calS'(\Lambda\times\Lambda)\otimeshat\calS'(\RR)\to\calS'(\Lambda\times\Lambda)\otimeshat\calD'(\RR^\times)$. Finally, the change of coordinates $(x,y,\lambda)\mapsto(y-x,\frac{\lambda}{2}(x+y),-\lambda)$ induces a continuous linear isomorphism on $\calS'(\Lambda\times\Lambda)\otimeshat\calD'(\RR^\times)$. Composing these three maps shows continuity of the map $u\mapsto\widehat{u}$. To determine its kernel we observe that the only non-bijective map in this three-fold composition is the restriction to $\RR^\times$. Its kernel is given by all distributions $v\in\calS'(\Lambda\times\Lambda\times\RR)$ with $\supp v\subseteq\Lambda\times\Lambda\times\{0\}$. Such distributions are necessarily of the form
$$ v(x,y,\lambda) = \sum_{k=0}^m v_k(x,y)\delta^{(k)}(\lambda) $$
for some distributions $v_k\in\calS'(\Lambda\times\Lambda)$. Taking the inverse Fourier transforms $\calF_2^{-1}\circ\calF_3^{-1}$ shows the claim.
\end{proof}

\begin{remark}
The map $u\mapsto\widehat{u}$ is essentially the group Fourier transform of $\overline{N}$, but only evaluated at the infinite-dimensional unitary representations $\sigma_\lambda$, $\lambda\in\RR^\times$. Therefore, it has a kernel which can be treated using the finite-dimensional unitary representations of $\overline{N}$.
\end{remark}

\begin{corollary}\label{cor:FTinjectiveOnPS}
Assume that $(\zeta,V_\zeta)$ is a smooth admissible Fr\'{e}chet representation of moderate growth. Then, for $\Re\nu>-\rho$ the Fourier transform
$$ \calF:I(\zeta,\nu)\subseteq\calS'(V\times\RR)\otimeshat V_\zeta \to \calD'(\RR^\times)\otimeshat\calS'(\Lambda\times\Lambda)\otimeshat V_\zeta\index{F@$\calF$} $$
is injective.
\end{corollary}

\begin{proof}
	We embed $\zeta$ into a principal series representation $\Ind_{M'A'N'}^M(\zeta'\otimes e^{\nu'}\otimes\1)$ of $M$ induced from a minimal parabolic subgroup $M'A'N'\subseteq M$, where $(\zeta',V_\zeta')$ is a finite-dimensional representation of the compact group $M'$ and $\nu'\in(\fraka'_\CC)^*$, with $\fraka'$ denoting the Lie algebra of $A'$. Using induction in stages, we may embed
	$$ I(\zeta,\nu) = \Ind_{MAN}^G(\zeta\otimes e^\nu\otimes\1) \subseteq \Ind_{M'(AA')(NN')}(\zeta'\otimes e^{\nu+\nu'}\otimes\1). $$
	By Proposition~\ref{prop:KernelFT}, the kernel of the Fourier transform consists of functions which, restricted to $\overline{N}\simeq V\times\RR$, are polynomial in $t\in\RR$. It therefore suffices to show that the restriction to $\overline{N}$ of a function $f\in\Ind_{M'(AA')(NN')}(\zeta'\otimes e^{\nu+\nu'}\otimes\1)$ cannot be a non-zero polynomial in $t$ if $\Re\nu>-\rho$. After possibly translating $f$ by an element of $\overline{N}$, it suffices to show that $f(\overline{n}_{(0,t)})$ cannot be a non-zero polynomial in $t\in\RR$.\\	
	Since $M'(AA')(NN')$ is a minimal parabolic subgroup of $G$,we have an Iwasawa decomposition $G=KAA'NN'$. Write $\overline{n}_{(0,t)}=kaa'nn'$ with $k\in K$, $a\in A$, $a'\in A'$, $n\in N$ and $n'\in N'$, then
	$$ f(\overline{n}_{(0,t)}) = a^{-\nu-\rho}(a')^{-\nu-\rho'}f(k). $$
	Since $\overline{n}_{(0,t)}=\exp(tF)$ is contained in a subgroup locally isomorphic to $\SL(2,\RR)$ (whose Lie algebra is spanned by $H$, $E$ and $F$) which contains $A=\exp(\RR H)$ but not $A'\subseteq M$, we can decompose $\overline{n}_{(0,t)}$ with respect to the Iwasawa decomposition of this subgroup and find that $a'=1$ and $a=\exp(\frac{1}{2}\log(1+t^2)H)$ by a standard $\SL(2,\RR)$-computation (or alternatively, using Lemma~\ref{lem:IwasawaAProjectionOnNbar}). Bounding $|f(k)|$ by $C=\sup_{x\in K}|f(x)|$ shows that
	$$ |f(\overline{n}_{(0,t)})| \leq C(1+t^2)^{-\frac{\nu+\rho}{2}} \to 0 \qquad \mbox{as }t\to\infty, $$
	by our assumpion on $\Re\nu$. Hence, $f$ cannot be a non-zero polynomial in $t$.
\end{proof}

\begin{remark}
We note that
$$ \calS'(\Lambda\times\Lambda)\simeq\Hom(\calS(\Lambda),\calS'(\Lambda))=\Hom(\calH^\infty,\calH^{-\infty}) $$
since the space $\calH^\infty$\index{H2infty@$\calH^\infty$} of smooth vectors in $\calH=L^2(\Lambda)$ is given by the space $\calS(\Lambda)$ of Schwarz functions.
\end{remark}

The previous observation allows us to define a representation $\widehat{\pi}_{\zeta,\nu}$ of $G$ on $\widehat{I}(\zeta,\nu)=\calF(I(\zeta,\nu))\subseteq\calD'(\RR^\times)\otimeshat\calS'(\Lambda\times\Lambda)\otimeshat V_\zeta$ by
$$ \widehat{\pi}_{\zeta,\nu}(g) = \calF\circ\pi_{\zeta,\nu}(g)\circ\calF^{-1}, \qquad g\in G.\index{1pi2hatzetanu@$\widehat{\pi}_{\zeta,\nu}$} $$
We call this realization the \emph{Fourier transformed picture}. In this picture, the action of the opposite parabolic subgroup $\overline{P}$ is expressed in terms of the representation $\sigma_\lambda$ and the metaplectic representation $\omega_{\met,\lambda}$:

\begin{proposition}\label{prop:ActionFTpicture}
For $f\in\widehat{I}(\zeta,\nu)\subseteq\calD'(\RR^\times)\otimeshat\Hom(\calH^\infty,\calH^{-\infty})\otimeshat V_\zeta$ we have
\begin{align*}
 \widehat{\pi}_{\zeta,\nu}(\overline{n}_{(z,t)})f(\lambda) &= \sigma_\lambda(z,t)\circ f(\lambda), && \overline{n}_{(z,t)}\in\overline{N},\\
 \widehat{\pi}_{\zeta,\nu}(m)f(\lambda) &= \zeta(m)\cdot\omega_{\met,\lambda}(m)\circ f(\chi(m)\lambda)\circ\omega_{\met,\lambda}(m)^{-1}, && m\in M,\\
 \widehat{\pi}_{\zeta,\nu}(e^{tH})f(\lambda) &= e^{(\nu-\rho)t}\delta_{e^t}\circ f(e^{-2t}\lambda)\circ\delta_{e^{-t}}, && e^{tH}\in A,
\end{align*}
where $\delta_s\varphi(x)=\varphi(sx)$\index{1deltas@$\delta_s$} ($s>0$). Alternatively, viewing $f$ as a $V_\zeta$-valued distribution in $(\lambda,x,y)\in\RR^\times\times\Lambda\times\Lambda$, we have
\begin{align*}
 \widehat{\pi}_{\zeta,\nu}(\overline{n}_{(z,t)})f(\lambda,x,y) &= e^{i\lambda t}e^{i\lambda(\omega(z'',y)+\frac{1}{2}\omega(z',z''))}f(\lambda,x,y-z'), && \overline{n}_{(z,t)}\in\overline{N},\\
 \widehat{\pi}_{\zeta,\nu}(m)f(\lambda,x,y) &= \zeta(m)(\id_{\RR^\times}^*\otimes\omega_{\met,-\lambda}(m)\otimes\omega_{\met,\lambda}(m))f(\chi(m)\lambda,x,y), && m\in M,\\
 \widehat{\pi}_{\zeta,\nu}(e^{tH})f(\lambda,x,y) &= e^{(\nu-1)t}f(e^{-2t}\lambda,e^tx,e^ty), && e^{tH}\in A,
\end{align*}
\end{proposition}

\begin{proof}
	The first three identities are derived from Proposition~\ref{prop:GroupActionNonCptPicture} by composing with the Fourier transform given in \eqref{eq:DefSigmaOfF} and \eqref{eq:DefFT}. For instance, the first identity follows from the definitions using left-invariance of the Haar measure $d\overline{n}$ on $\overline{N}$:
	\begin{align*}
		\widehat{\pi}_{\zeta,\nu}(\overline{n}_{(z,t)})\circ\calF u(\lambda) &= \calF\circ\pi_{\zeta,\nu}(\overline{n}_{(z,t)})u(\lambda) = \int_{\overline{N}}\pi_{\zeta,\nu}(\overline{n}_{(z,t)})u(\overline{n})\sigma_\lambda(\overline{n})\,d\overline{n}\\
		&= \int_{\overline{N}}u(\overline{n}_{(z,t)}^{-1}\overline{n})\sigma_\lambda(\overline{n})\,d\overline{n} = \int_{\overline{N}}u(\overline{n})\sigma_\lambda(\overline{n}_{(z,t)}\overline{n})\,d\overline{n}\\
		&= \int_{\overline{N}}u(\overline{n})\Big[\sigma_\lambda(\overline{n}_{(z,t)})\circ\sigma_\lambda(\overline{n})\Big]\,d\overline{n} = \sigma_\lambda(\overline{n}_{(z,t)})\circ\int_{\overline{N}}u(\overline{n})\sigma_\lambda(\overline{n})\,d\overline{n}\\
		&= \sigma_\lambda(z,t)\circ\calF u(\lambda),
	\end{align*}
	where we use the identification $V\times\RR\simeq\overline{N},\,(z,t)\mapsto\overline{n}_{(z,t)}$. The last three identities follow from the first ones by identifying $f(\lambda)\in\Hom(\calS(\Lambda),\calS'(\Lambda))$ with its Schwartz kernel $f(\lambda,x,y)$ in the sense that
	$$ f(\lambda)\varphi(y) = \int_\Lambda f(\lambda,x,y)\varphi(x)\,dx, $$
	and by making use of the fact that
	$$ \sigma_\lambda(\overline{n})^\top=\sigma_{-\lambda}(\overline{n}^{-1}) \qquad \mbox{and} \qquad \omega_{\met,\lambda}(m)^\top=\omega_{\met,-\lambda}(m^{-1}) $$
	as operators $\calS(\Lambda)\to\calS(\Lambda)$, $\calS(\Lambda)\to\calS'(\Lambda)$ or $\calS'(\Lambda)\to\calS'(\Lambda)$. For instance, the previous computation shows that $\widehat{\pi}_{\zeta,\nu}(\overline{n}_{(z,t)})$ acts on $f(\lambda)$ by the composition $\sigma_\lambda(z,t)\circ f(\lambda)$, which in turn is given by \eqref{eq:DefSchroedingerModel}. Hence, for every test function $\varphi\in\calS(\Lambda)$:
	\begin{align*}
		\Big[\sigma_\lambda(z,t)\circ f(\lambda)\Big]\varphi(y) &= e^{i\lambda t}e^{i\lambda(\omega(z'',y)+\frac{1}{2}\omega(z',z''))}f(\lambda)\varphi(y-z')\\
		&= e^{i\lambda t}e^{i\lambda(\omega(z'',y)+\frac{1}{2}\omega(z',z''))}\int_\Lambda f(\lambda,x,y-z')\varphi(x)\,dx\\
		&= \int_\Lambda\Big[e^{i\lambda t}e^{i\lambda(\omega(z'',y)+\frac{1}{2}\omega(z',z''))}f(\lambda,x,y-z')\Big]\varphi(x)\,dx,
	\end{align*}
	so the Schwartz kernel of $\sigma_\lambda(z,t)\circ f(\lambda)$ is $e^{i\lambda t}e^{i\lambda(\omega(z'',y)+\frac{1}{2}\omega(z',z''))}f(\lambda,x,y-z')$.
\end{proof}

It seems difficult to express the action of $N$ or $w_0$ in the Fourier transformed picture. More accessible is the action of the Lie algebra $\frakg$ in the differentiated representation $d\widehat{\pi}_{\zeta,\nu}$ which can be obtained using Corollary~\ref{cor:LieAlgActionNonCptPicture} and the formulas in the following lemma. We will not carry out the computation of the Lie algebra action on the whole principal series representation, but rather restrict to a certain subrepresentation in Section~\ref{sec:FTpictureMinRep}.

For the following statement, denote by $\partial_{v,x}$ resp. $\partial_{v,y}$ the directional derivative in the variable $x$ resp. $y$ in the direction $v$. Moreover, we use the coordinates $(z,t)$ on $V\times\RR$.

\begin{lemma}\label{lem:FTMultDiff}
Let $u\in\calS'(V\times\RR)$.
\begin{enumerate}[(1)]
	\item\label{lem:FTMultDiff1} For $v\in\Lambda$ we have
	\begin{align*}
	\widehat{\omega(v,z)u}(\lambda,x,y) &= -\frac{1}{i\lambda}(\partial_{v,x}+\partial_{v,y})\widehat{u}(\lambda,x,y),\\
	\widehat{\partial_vu}(\lambda,x,y) &= -\frac{1}{2}(\partial_{v,x}-\partial_{v,y})\widehat{u}(\lambda,x,y).
	\end{align*}
	\item\label{lem:FTMultDiff2} For $w\in\Lambda^*$ we have
	\begin{align*}
	\widehat{\omega(z,w)u}(\lambda,x,y) &= \omega(y-x,w)\widehat{u}(\lambda,x,y),\\
	\widehat{\partial_wu}(\lambda,x,y) &= \frac{i\lambda}{2}\omega(x+y,w)\widehat{u}(\lambda,x,y).
	\end{align*}
	\item\label{lem:FTMultDiff3} For differentiation and multiplication with respect to the central variable $t$ we have
	\begin{align*}
	\widehat{\partial_tu}(\lambda,x,y) &= -i\lambda \widehat{u}(\lambda,x,y),\\
	\widehat{tu}(\lambda,x,y) &= -i\partial_\lambda \widehat{u}(\lambda,x,y) - \frac{1}{2i\lambda}(\partial_{x+y,x}+\partial_{x+y,y})\widehat{u}(\lambda,x,y).
	\end{align*}
\end{enumerate}
\end{lemma}

\begin{proof}
We only show the last formula, the rest is standard. For this let $(e_\alpha)\subseteq\Lambda$ be a basis of $\Lambda$ with dual basis $(\widehat{e}_\alpha)\subseteq\Lambda^*$. Then, using \eqref{lem:FTMultDiff1} we find
\begin{align*}
\widehat{tu}(\lambda,x,y) &= -i\int_{\Lambda^*}\int_\RR u(y-x,z'',t)\partial_\lambda\left[e^{i\lambda t}\right]e^{-\frac{i\lambda}{2}\omega(x+y,z'')}\,dt\,dz''\\
&= -i\partial_\lambda \widehat{u}(x,y,\lambda) + \frac{1}{2}\int_{\Lambda^*}\int_\RR\omega(x+y,z'')u(y-x,z'',t)e^{i\lambda t}e^{-\frac{i\lambda}{2}\omega(x+y,z'')}\,dt\,dz''\\
&= -i\partial_\lambda \widehat{u}(\lambda,x,y) + \frac{1}{2}\sum_\alpha\omega(x+y,\widehat{e}_\alpha)\widehat{\omega(e_\alpha,z'')u}(\lambda,x,y)\\
&= -i\partial_\lambda \widehat{u}(\lambda,x,y) - \frac{1}{2i\lambda}\sum_\alpha\omega(x+y,\widehat{e}_\alpha)(\partial_{e_\alpha,x}+\partial_{e_\alpha,y})\widehat{u}(\lambda,x,y)
\end{align*}
and the claimed formula follows.
\end{proof}

\chapter[Conformally invariant systems]{Conformally invariant systems and their Fourier transform}

We recall the construction of conformally invariant systems on Heisenberg nilradicals due to \cite{BKZ08} in Section~\ref{sec:QuantizationSymplecticInvariants}, discuss their conformal invariance in Section~\ref{sec:ConformalInvariance}, and compute their action in the Fourier transformed picture in Sections~\ref{sec:FTOmegaOmega}, \ref{sec:FTOmegaMu} and \ref{sec:FTOmegaPsiQ}.

\section{Quantization of the symplectic invariants}\label{sec:QuantizationSymplecticInvariants}

In \cite[Sections 5 and 6]{BKZ08} four conformally invariant systems of differential operators on $\overline{N}$ are constructed. We briefly recall their construction and properties. For this, let $(e_\alpha)\subseteq V$ be a basis and $\widehat{e}_\alpha$ be the dual basis with respect to the symplectic form, i.e. $\omega(e_\alpha,\widehat{e}_\beta)=\delta_{\alpha\beta}$. Denote by $X_\alpha$ the left-invariant vector field on $\overline{N}$ corresponding to $e_\alpha\in\overline{\frakn}$, i.e.
$$ X_\alpha f(\overline{n}) = \left.\frac{d}{dt}\right|_{t=0}f(\overline{n}e^{te_\alpha}). $$
In the coordinates $(x,t)\in V\times\RR\simeq\overline{N}$ this operator takes by \eqref{eq:MultiplicationHeisenbergGroup} the form
\begin{equation}
 X_\alpha = \partial_\alpha+\tfrac{1}{2}\omega(x,e_\alpha)\partial_t,\label{eq:LeftInvVectorFields}
\end{equation}
where $\partial_\alpha=\partial_{e_\alpha}$.

\subsection{Quantization of $\omega$}

For $v\in V$ we let
$$ \Omega_\omega(v) := \sum_\alpha\omega(v,\widehat{e}_\alpha)X_\alpha = \partial_v+\frac{1}{2}\omega(x,v)\partial_t.\index{1ZOmega1omegav@$\Omega_\omega(v)$} $$

\subsection{Quantization of $\mu$}

For $T\in\frakm$ we let
$$ \Omega_\mu(T) = \sum_{\alpha,\beta} \omega(T\widehat{e}_\alpha,\widehat{e}_\beta)X_\alpha X_\beta.\index{1ZOmega2muT@$\Omega_\mu(T)$} $$
Note that, by Corollary~\ref{cor:TraceTMuProportionalSympForm}, the linear form $T\mapsto\omega(T\widehat{e}_\alpha,\widehat{e}_\beta)$ on $\frakm$ is proportional to $T\mapsto\tr(TB_\mu(\widehat{e}_\alpha,\widehat{e}_\beta))$ on each simple or abelian factor of $\frakm$, which, in turn, is proportional to $T\mapsto\kappa(T,B_\mu(\widehat{e}_\alpha,\widehat{e}_\beta))$ by Lemma~\ref{lem:KillingForm}, $\kappa$ being the Killing form of $\frakg$. In this sense, $\Omega_\mu$ can be understood as a quantization of $\mu$. Using the explicit expression \eqref{eq:LeftInvVectorFields} of $X_\alpha$ in the coordinates $(x,t)\in V\times\RR$ we find
\begin{align}
 \Omega_\mu(T) ={}& \sum_{\alpha,\beta} \omega(T\widehat{e}_\alpha,\widehat{e}_\beta)(\partial_\alpha+\tfrac{1}{2}\omega(x,e_\alpha)\partial_t)(\partial_\beta+\tfrac{1}{2}\omega(x,e_\beta)\partial_t)\notag\\
 ={}& \sum_{\alpha,\beta} \omega(T\widehat{e}_\alpha,\widehat{e}_\beta)\Big[\partial_\alpha\partial_\beta+\tfrac{1}{2}\omega(x,e_\alpha)\partial_\beta\partial_t+\tfrac{1}{2}\omega(x,e_\beta)\partial_\alpha\partial_t+\tfrac{1}{4}\omega(x,e_\alpha)\omega(x,e_\beta)\partial_t^2\Big]\notag\\
 ={}& \sum_{\alpha,\beta}\omega(T\widehat{e}_\alpha,\widehat{e}_\beta)\partial_\alpha\partial_\beta-\partial_{Tx}\partial_t+\tfrac{1}{4}\omega(Tx,x)\partial_t^2.\label{eq:ExplicitDmuT}
\end{align}

\subsection{Quantization of $\Psi$ and $Q$}

For $v\in V$ we let
\begin{align*}
 \Omega_\Psi(v) &:= \sum_{\alpha,\beta,\gamma} \omega(v,B_\Psi(\widehat{e}_\alpha,\widehat{e}_\beta,\widehat{e}_\gamma))X_\alpha X_\beta X_\gamma,\index{1ZOmega3PsiX@$\Omega_\Psi(v)$}\\
 \Omega_Q &:= \sum_{\alpha,\beta,\gamma,\delta} B_Q(\widehat{e}_\alpha,\widehat{e}_\beta,\widehat{e}_\gamma,\widehat{e}_\delta)X_\alpha X_\beta X_\gamma X_\delta.\index{1ZOmega4Q@$\Omega_Q$}
\end{align*}

\section{Conformal invariance}\label{sec:ConformalInvariance}

In \cite[Sections 5 and 6]{BKZ08} it is shown that all four systems are conformally invariant for certain special parameters $\nu$ in the case where the representation $(\zeta,V_\zeta)$ is trivial on the identity component of $M$, i.e. the differentiated representation $d\zeta$ is zero. Since we also need to involve non-trivial representations $d\zeta$ in Sections~\ref{sec:FTpictureMinRepSLn} and \ref{sec:FTpictureMinRepSOpq}, we give a self-contained proof of conformal invariance for the systems $\Omega_\omega$ and $\Omega_\mu$ to have all relevant formulas available. The following computations use the coordinates $(x,s)\in V\times\RR$.

\begin{theorem}[{see \cite[Theorem 5.1]{BKZ08}}]\label{thm:ConfInvOmegaOmega}
For every $v\in V$ we have
\begin{align*}
	[\Omega_\omega(v),d\pi_{\zeta,\nu}(X)] &= 0 && (X\in\overline{\frakn}),\\
	[\Omega_\omega(v),d\pi_{\zeta,\nu}(H)] &= \Omega_\omega(v),\\
	[\Omega_\omega(v),d\pi_{\zeta,\nu}(S)] &= -\Omega_\omega(Sv) && (S\in\frakm).
\end{align*}
Moreover,
$$ [\Omega_\omega(v),d\pi_{\zeta,\nu}(E)] = s\Omega_\omega(v)-\Omega_\omega(\mu(x)v)+\frac{\nu+\rho}{2}\omega(x,v)+2d\zeta(B_\mu(x,v)). $$
In particular, for $d\zeta=0$ and $\nu=-\rho$ the space
$$ I(\zeta,\nu)^{\Omega_\omega(V)} = \{f\in I(\zeta,\nu):\Omega_\omega(v)f=0\mbox{ for all }v\in V\}\index{IzetanuOmegaomegaV@$I(\zeta,\nu)^{\Omega_\omega(V)}$} $$
is a subrepresentation of $(\pi_{\zeta,\nu},I(\zeta,\nu))$.
\end{theorem}

\begin{proof}
	This is a straightforward verification using the formulas in Corollary~\ref{cor:LieAlgActionNonCptPicture}.
\end{proof}

\begin{remark}
Since the left-invariant vector fields generate the $C^\infty(\overline{N})$-module of all vector fields, $I(\zeta,\nu)^{\Omega_\omega(V)}$ only consists of constant functions and is therefore the trivial representation. To obtain a non-trivial representation we try to reduce the conformally invariant system $\Omega_\omega(V)$ to $\Omega_\omega(W)$ for a subspace $W\subseteq V$. However, in all cases except $\frakg\simeq\sl(n,\RR)$ the Lie algebra $\frakm$ acts irreducibly on $V$ so that $[\Omega_\omega(w),d\pi_{\zeta,\nu}(S)]f=0$ for some $w\in V$ implies $\Omega_\omega(v)f=0$ for all $v\in V$. For $\frakg=\sl(n,\RR)$, the adjoint representation of $\frakm$ on $V$ splits into two irreducible subspaces. In Section~\ref{sec:FTpictureMinRepSLn} we show that restricting $\Omega_\omega$ to one of those subspaces yields a conformally invariant subsystem for certain parameters $(\zeta,\nu)$.
\end{remark}

\begin{theorem}[{\cite[Theorem 5.2]{BKZ08}}]\label{thm:ConfInvOmegaMu}
For every $T\in\frakm$ we have
\begin{align*}
	[\Omega_\mu(T),d\pi_{\zeta,\nu}(X)] &= 0 && (X\in\overline{\frakn}),\\
	[\Omega_\mu(T),d\pi_{\zeta,\nu}(H)] &= 2\Omega_\mu(T),\\
	[\Omega_\mu(T),d\pi_{\zeta,\nu}(S)] &= \Omega_\mu([T,S]) && (S\in\frakm).
\end{align*}
Further, if $T\in\frakm'$ where $\frakm'$ is any simple or abelian factor of $\frakm$, then
\begin{multline*}
 [\Omega_\mu(T),d\pi_{\zeta,\nu}(E)] = 2s\Omega_\mu(T)+\Omega_\mu([T,\mu(x)])+(2\,\calC(\frakm')-2-(\nu+\rho))\Omega_\omega(Tx)\\
 +4\sum_\alpha d\zeta(B_\mu(x,e_\alpha))\Omega_\omega(T\widehat{e}_\alpha)-2\,\calC(\frakm')d\zeta(T).
\end{multline*}
In particular, for any simple or abelian factor $\frakm'$ of $\frakm$, $d\zeta=0$ and $\nu=2\,\calC(\frakm')-\rho-2$ the space
$$ I(\zeta,\nu)^{\Omega_\mu(\frakm')} = \{f\in I(\zeta,\nu):\Omega_\mu(T)f=0\mbox{ for all }T\in\frakm'\} $$
is a subrepresentation of $(\pi_{\zeta,\nu},I(\zeta,\nu))$.\index{IzetanuOmegamumprime@$I(\zeta,\nu)^{\Omega_\mu(\frakm')}$}
\end{theorem}

\begin{proof}
The first three identities are easy to verify. To calculate $[\Omega_\mu(T),d\pi_{\zeta,\nu}(E)]$ we compute all commutators between the different summands of
$$ \Omega_\mu(T) = \sum_{\alpha,\beta}\omega(T\widehat{e}_\alpha,\widehat{e}_\beta)\partial_\alpha\partial_\beta-\partial_{Tx}\partial_s+\tfrac{1}{4}\omega(Tx,x)\partial_s^2 $$
and
$$ d\pi_{\zeta,\nu}(E) = \partial_{sx}+\partial_{\Psi(x)}+(s^2+Q(x))\partial_s+(\nu+\rho)s+d\zeta(\mu(x)) $$
separately. First,
$$ \Big[\sum_{\alpha,\beta}\omega(T\widehat{e}_\alpha,\widehat{e}_\beta)\partial_\alpha\partial_\beta,\partial_{sx}\Big] = 2s\sum_{\alpha,\beta}\omega(T\widehat{e}_\alpha,\widehat{e}_\beta)\partial_\alpha\partial_\beta. $$
Next, using Lemma~\ref{lem:SymmetrizationsOfSymplecticCovariants}~\eqref{lem:SymmetrizationsOfSymplecticCovariants2} we have
\begin{align*}
 \Big[\sum_{\alpha,\beta}\omega(T\widehat{e}_\alpha,\widehat{e}_\beta)\partial_\alpha\partial_\beta,\partial_{\Psi(x)}\Big] ={}& 6\sum_{\alpha,\beta}\omega(T\widehat{e}_\alpha,\widehat{e}_\beta)\partial_{B_\Psi(e_\alpha,x,x)}\partial_\beta+6\sum_{\alpha,\beta}\omega(T\widehat{e}_\alpha,\widehat{e}_\beta)\partial_{B_\Psi(e_\alpha,e_\beta,x)}\\
 ={}& -\sum_{\alpha,\beta}\Big(2\omega(T\widehat{e}_\alpha,\widehat{e}_\beta)\partial_{\mu(x)e_\alpha}\partial_\beta+\omega(T\widehat{e}_\alpha,\widehat{e}_\beta)\omega(x,e_\alpha)\partial_x\partial_\beta\Big)\\
 & -\sum_{\alpha,\beta}\Big(2\omega(T\widehat{e}_\alpha,\widehat{e}_\beta)\partial_{B_\mu(e_\alpha,e_\beta)x}+\omega(T\widehat{e}_\alpha,\widehat{e}_\beta)\partial_{\omega(e_\alpha,x)e_\beta}\Big)\\
 ={}& \sum_{\alpha,\beta}\omega([T,\mu(x)]\widehat{e}_\alpha,\widehat{e}_\beta)\partial_\alpha\partial_\beta+\sum_{\alpha,\beta}\omega(x,\widehat{e}_\alpha)\omega(Tx,\widehat{e}_\beta)\partial_\alpha\partial_\beta\\
 & -\sum_{\alpha,\beta}2\partial_{B_\mu(e_\alpha,T\widehat{e}_\alpha)x}-\partial_{Tx}.
\intertext{By Lemma~\ref{lem:BezoutianSum}, the sum in the third term evaluates to $\sum_\alpha B_\mu(e_\alpha,T\widehat{e}_\alpha)=-\calC(\frakm')T$ and together we obtain}
 ={}& \sum_{\alpha,\beta}\omega([T,\mu(x)]\widehat{e}_\alpha,\widehat{e}_\beta)\partial_\alpha\partial_\beta+\sum_{\alpha,\beta}\omega(x,\widehat{e}_\alpha)\omega(Tx,\widehat{e}_\beta)\partial_\alpha\partial_\beta\\
 & +(2\,\calC(\frakm')-1)\partial_{Tx}.
\end{align*}
Next, by Lemma~\ref{lem:SymmetrizationsOfSymplecticCovariants}~\eqref{lem:SymmetrizationsOfSymplecticCovariants2} and \eqref{lem:SymmetrizationsOfSymplecticCovariants3}
\begin{align*}
 & \Big[\sum_{\alpha,\beta}\omega(T\widehat{e}_\alpha,\widehat{e}_\beta)\partial_\alpha\partial_\beta,(s^2+Q(x))\partial_s\Big]\\
 ={}& 8\sum_{\alpha,\beta}\omega(T\widehat{e}_\alpha,\widehat{e}_\beta)B_Q(e_\alpha,x,x,x)\partial_\beta\partial_s+12\sum_{\alpha,\beta}\omega(T\widehat{e}_\alpha,\widehat{e}_\beta)B_Q(e_\alpha,e_\beta,x,x)\partial_s\\
 ={}& 2\partial_{T\Psi(x)}\partial_s-\sum_{\alpha,\beta}\omega(T\widehat{e}_\alpha,\widehat{e}_\beta)\Big(\omega(x,B_\mu(e_\alpha,e_\beta)x)+\frac{1}{2}\omega(x,e_\beta)\omega(e_\alpha,x)\Big)\partial_s\\
 ={}& -\frac{2}{3}\partial_{T\mu(x)x}\partial_s-\sum_\alpha\omega(x,B_\mu(e_\alpha,T\widehat{e}_\alpha)x)\partial_s+\frac{1}{2}\omega(Tx,x)\partial_s,
\intertext{which is, again by Lemma~\ref{lem:BezoutianSum}, equal to}
 ={}& -\frac{2}{3}\partial_{T\mu(x)x}\partial_s+(\frac{1}{2}-\calC(\frakm'))\omega(Tx,x)\partial_s.
\end{align*}
Finally, the last commutator of this type is
$$ \Big[\sum_{\alpha,\beta}\omega(T\widehat{e}_\alpha,\widehat{e}_\beta)\partial_\alpha\partial_\beta,(\nu+\rho)s\Big] = 0. $$
Next, we have
$$ [\partial_{Tx}\partial_s,\partial_{sx}] = \partial_{Tx} + \sum_{\alpha,\beta}\omega(x,\widehat{e}_\alpha)\omega(Tx,\widehat{e}_\beta)\partial_\alpha\partial_\beta. $$
Further, again by Lemma~\ref{lem:SymmetrizationsOfSymplecticCovariants}
\begin{align*}
 [\partial_{Tx}\partial_s,\partial_{\Psi(x)}] &= \sum_\alpha\Big(3\omega(B_\Psi(Tx,x,x),\widehat{e}_\alpha)-\omega(T\Psi(x),\widehat{e}_\alpha)\Big)\partial_\alpha\partial_s\\
 &= -\partial_{\mu(x)Tx}\partial_s+\frac{1}{2}\omega(Tx,x)\partial_x\partial_s+\frac{1}{3}\partial_{T\mu(x)x}\partial_s.
\end{align*}
Next,
\begin{align*}
 [\partial_{Tx}\partial_s,(s^2+Q(x))\partial_s] &= 2s\partial_{Tx}\partial_s + 4B_Q(Tx,x,x,x)\partial_s^2\\
 &= 2s\partial_{Tx}\partial_s + \omega(Tx,\Psi(x))\partial_s^2\\
 &= 2s\partial_{Tx}\partial_s - \frac{1}{6}\omega([T,\mu(x)]x,x)\partial_s^2.
\end{align*}
And the last commutator of this type is
\begin{align*}
 [\partial_{Tx}\partial_s,(\nu+\rho)s] &= (\nu+\rho)\partial_{Tx}.
\end{align*}
Next, we have
\begin{align*}
 [\omega(Tx,x)\partial_s^2,\partial_{sx}] &= 2\omega(Tx,x)\partial_x\partial_s-2s\omega(Tx,x)\partial_s^2.
\end{align*}
Further,
\begin{align*}
 [\omega(Tx,x)\partial_s^2,\partial_{\Psi(x)}] &= -\Big(\omega(T\Psi(x),x)+\omega(Tx,\Psi(x))\Big)\partial_s^2 = \frac{1}{3}\omega([T,\mu(x)]x,x)\partial_s^2.
\end{align*}
Next,
\begin{align*}
 [\omega(Tx,x)\partial_s^2,(s^2+Q(x))\partial_s] &= 2\omega(Tx,x)\partial_s+4s\omega(Tx,x)\partial_s^2
\end{align*}
and
\begin{align*}
 [\omega(Tx,x)\partial_s^2,(\nu+\rho)s] &= 2(\nu+\rho)\omega(Tx,x)\partial_s.
\end{align*}
Finally, by Lemma~\ref{lem:BezoutianSum}:
\begin{align*}
	[\Omega_\mu(T),d\zeta(\mu(x))] &= 2\sum_\alpha d\zeta(B_\mu(e_\alpha,T\widehat{e}_\alpha))+4\sum_\alpha d\zeta(B_\mu(x,e_\alpha))\partial_{T\widehat{e}_\alpha}-2d\zeta(B_\mu(Tx,x))\partial_t\\
	&= -2\,\calC(\frakm')d\zeta(T)+4\sum_\alpha d\zeta(B_\mu(x,e_\alpha))\Omega_\omega(T\widehat{e}_\alpha)
\end{align*}
Collecting all terms shows the claimed formula.
\end{proof}

\section{The Fourier transform of $\Omega_\omega$}\label{sec:FTOmegaOmega}

Since $\Omega_\omega(v)$ is acting by the left-invariant vector field corresponding to $v\in V$, it follows immediately from \eqref{eq:FTofVectorField} that
\begin{equation}
	\sigma_\lambda(\Omega_\omega(v)u) = -\sigma_\lambda(u)d\sigma_\lambda(v).\label{eq:FTofOmegaOmega}
\end{equation}

\section{The Fourier transform of $\Omega_\mu$}\label{sec:FTOmegaMu}

We show that in the Fourier transformed picture the conformally invariant system $\Omega_\mu(T)$, $T\in\frakm$, is nothing else but the metaplectic representation of $\sp(V,\omega)$ restricted to $\frakm$ as defined in Section~\ref{sec:HeisFT}. For this, we use that $d\omega_{\met,\lambda}$ is the unique representation of $\frakm$ on $\calS(\Lambda)$ by differential operators such that (cf. \eqref{eq:DefMetaplecticRep})
$$ d\sigma_\lambda([T,X]) = [d\omega_{\met,\lambda}(T),d\sigma_\lambda(X)] \qquad \mbox{for all }T\in\frakm,X\in\overline{\frakn}. $$

\begin{theorem}\label{thm:FTofOmegaMu}
	For every $\lambda\in\RR^\times$ and $T\in\frakm$ we have
	$$ d\sigma_\lambda(\Omega_\mu(T)) = 2i\lambda d\omega_{\met,\lambda}(T). $$
\end{theorem}

\begin{proof}
	It suffices to show that
	$$ [d\sigma_\lambda(\Omega_\mu(T)),d\sigma_\lambda(X)] = 2i\lambda d\sigma_\lambda([T,X]) \qquad \mbox{for all }X\in V. $$
	We compute
	\begin{align*}
	& [d\sigma_\lambda(\Omega_\mu(T)),d\sigma_\lambda(X)] = \sum_{\alpha,\beta}\omega([T,\widehat{e}_\alpha],\widehat{e}_\beta)[d\sigma_\lambda(X_\alpha)\sigma_\lambda(X_\beta),d\sigma_\lambda(X)]\\
	={}& \sum_{\alpha,\beta}\omega([T,\widehat{e}_\alpha],\widehat{e}_\beta)\Big(d\sigma_\lambda(X_\alpha)[d\sigma_\lambda(X_\beta),d\sigma_\lambda(X)]+[d\sigma_\lambda(X_\alpha),d\sigma_\lambda(X)]d\sigma_\lambda(X_\beta)\Big)\\
	={}& \sum_{\alpha,\beta}\omega([T,\widehat{e}_\alpha],\widehat{e}_\beta)\Big(d\sigma_\lambda(X_\alpha)d\sigma_\lambda([X_\beta,X])+d\sigma_\lambda([X_\alpha,X])d\sigma_\lambda(X_\beta)\Big)\\
	={}& i\lambda\sum_{\alpha,\beta}\omega([T,\widehat{e}_\alpha],\widehat{e}_\beta)\left(\omega(X_\beta,X)d\sigma_\lambda(X_\alpha)+\omega(X_\alpha,X)d\sigma_\lambda(X_\beta)\right)\\
	={}& 2i\lambda\sum_\beta\omega([T,X],\widehat{e}_\beta)d\sigma_\lambda(X_\beta) = 2i\lambda d\sigma_\lambda([T,X]).\qedhere
	\end{align*}
\end{proof}

\section{The Fourier transform of $\Omega_\Psi$ and $\Omega_Q$}\label{sec:FTOmegaPsiQ}

Although we do not investigate the kernel of the conformally invariant systems $\Omega_\Psi$ and $\Omega_Q$ any further in this work, we provide the Fourier transform of these systems for completeness.

\begin{corollary}\label{cor:FTofOmegaPsi}
For every $\lambda\in\RR^\times$ and $v\in V$ we have
\begin{align*}
 d\sigma_\lambda(\Omega_\Psi(v)) &= \frac{2i\lambda}{3}\sum_\alpha\sigma_\lambda(e_\alpha)d\omega_{\met,\lambda}(B_\mu(v,\widehat{e}_\alpha)) + \frac{i\lambda}{12}(\dim V+1)\sigma_\lambda(v)\\
 &= \frac{i\lambda\kappa_0}{3}\sum_i\sigma_\lambda(T_i'v)d\omega_{\met,\lambda}(T_i) + \frac{i\lambda}{12}(\dim V+1)\sigma_\lambda(v),
\end{align*}
where $(e_\alpha)$ is a basis of $V$, $(\widehat{e}_\alpha)$ the dual basis with respect the symplectic form $\omega$, $(T_i)$ a basis of $\frakm$ and $(T_i')$ the dual basis with respect to the Killing form $\kappa$ of $\frakg$.
\end{corollary}

\begin{proof}
By Lemma~\ref{lem:SymmetrizationsOfSymplecticCovariants} and Theorem~\ref{thm:FTofOmegaMu} we have
\begin{align*}
 \sigma_\lambda(\Omega_\Psi(v)) ={}& \sum_{\alpha,\beta,\gamma}\omega(\widehat{e}_\beta,B_\Psi(v,\widehat{e}_\alpha,\widehat{e}_\gamma))\sigma_\lambda(e_\alpha)\sigma_\lambda(e_\beta)\sigma_\lambda(e_\gamma)\\
 ={}& -\frac{1}{3}\sum_{\alpha,\beta,\gamma}\omega(\widehat{e}_\beta,B_\mu(v,\widehat{e}_\alpha)\widehat{e}_\gamma)\sigma_\lambda(e_\alpha)\sigma_\lambda(e_\beta)\sigma_\lambda(e_\gamma)\\
 & \qquad -\frac{1}{12}\sum_{\alpha,\beta,\gamma}\omega(\widehat{e}_\beta,\omega(v,\widehat{e}_\gamma)\widehat{e}_\alpha+\omega(\widehat{e}_\alpha,\widehat{e}_\gamma)X)\sigma_\lambda(e_\alpha)\sigma_\lambda(e_\beta)\sigma_\lambda(e_\gamma)\\
 ={}& \frac{1}{3}\sum_\alpha\sigma_\lambda(e_\alpha)\sigma_\lambda(\Omega_\mu(B_\mu(v,\widehat{e}_\alpha))\\
 & \qquad + \frac{1}{12}\sum_\alpha\Big(\sigma_\lambda(e_\alpha)\sigma_\lambda(\widehat{e}_\alpha)\sigma_\lambda(v) + \sigma_\lambda(e_\alpha)\sigma_\lambda(v)\sigma_\lambda(\widehat{e}_\alpha)\Big)\\
 ={}& \frac{2i\lambda}{3}\sum_\alpha\sigma_\lambda(e_\alpha)d\omega_{\met,\lambda}(B_\mu(v,\widehat{e}_\alpha))\\
 & \qquad + \frac{1}{6}\sum_\alpha\sigma_\lambda(e_\alpha)\sigma_\lambda(\widehat{e}_\alpha)\sigma_\lambda(v) + \frac{1}{12}\sum_\alpha\sigma_\lambda(e_\alpha)\sigma_\lambda([v,\widehat{e}_\alpha]).
\end{align*}
Using the independence of the chosen basis we find
\begin{equation}
 \sum_\alpha\sigma_\lambda(e_\alpha)\sigma_\lambda(\widehat{e}_\alpha) = \frac{1}{2}\sum_\alpha\big(\sigma_\lambda(e_\alpha)\sigma_\lambda(\widehat{e}_\alpha)-\sigma_\lambda(\widehat{e}_\alpha)\sigma_\lambda(e_\alpha)\big) = \frac{1}{2}\sum_\alpha\sigma_\lambda([e_\alpha,\widehat{e}_\alpha]) = \frac{i\lambda}{2}\dim V\label{eq:SumSigmaSquared}
\end{equation}
and $\sigma_\lambda([v,\widehat{e}_\alpha])=i\lambda\omega(v,\widehat{e}_\alpha)$. This shows
$$ \sigma_\lambda(\Omega_\Psi(v)) = \frac{2i\lambda}{3}\sum_\alpha\sigma_\lambda(e_\alpha)d\omega_{\met,\lambda}(B_\mu(v,\widehat{e}_\alpha)) + \frac{i\lambda}{12}(\dim V+1)\sigma_\lambda(v). $$
The second identity follows by expanding $B_\mu(v,\widehat{e}_\alpha)=\sum_i\kappa(B_\mu(v,\widehat{e}_\alpha),T_i')T_i$ into the basis $(T_i)$ and using $\kappa(B_\mu(x,y),T)=\frac{\kappa_0}{2}\omega(Tx,y)$ (see Remark~\ref{rem:SphVectorPositive}).
\end{proof}

The Fourier transform of the conformally invariant differential operator $\Omega_Q$ is essentially the Casimir element in the restriction of the metaplectic representation of $\sp(V,\omega)$ to $\frakm$.

\begin{corollary}\label{cor:FTofOmegaQ}
For every $\lambda\in\RR^\times$ we have
$$ \sigma_\lambda(\Omega_Q) = \frac{\kappa_0}{3}\lambda^2 d\omega_{\met,\lambda}(\Cas_\frakm) - \frac{(\dim V)^2}{96}\lambda^2, $$
where $\Cas_\frakm\in U(\frakm)$\index{Casm@$\Cas_\frakm$} denotes the Casimir element of $\frakm$ with respect to the Killing form $\kappa$.
\end{corollary}

\begin{proof}
Using Lemma~\ref{lem:SymmetrizationsOfSymplecticCovariants}, we find
\begin{align*}
 \sigma_\lambda(\Omega_Q) ={}& -\frac{1}{12}\sum_{\alpha,\beta,\gamma,\delta}\omega(\widehat{e}_\gamma,B_\mu(\widehat{e}_\alpha,\widehat{e}_\beta)\widehat{e}_\delta)\sigma_\lambda(X_\alpha)\sigma_\lambda(X_\beta)\sigma_\lambda(X_\gamma)\sigma_\lambda(X_\delta)\\
 & \qquad -\frac{1}{24}\sum_{\alpha,\beta,\gamma,\delta}\omega(\widehat{e}_\gamma,B_\tau(\widehat{e}_\alpha,\widehat{e}_\beta)\widehat{e}_\delta)\sigma_\lambda(X_\alpha)\sigma_\lambda(X_\beta)\sigma_\lambda(X_\gamma)\sigma_\lambda(X_\delta)\\
 ={}& \frac{1}{12}\sum_{\alpha,\beta}\sigma_\lambda(X_\alpha)\sigma_\lambda(X_\beta)\sigma_\lambda(\Omega_\mu(B_\mu(\widehat{e}_\alpha,\widehat{e}_\beta))+\frac{1}{24}\sum_{\alpha,\beta}\sigma_\lambda(e_\alpha)\sigma_\lambda(e_\beta)\sigma_\lambda(\widehat{e}_\beta)\sigma_\lambda(\widehat{e}_\alpha).
\end{align*}
Applying \eqref{eq:SumSigmaSquared} to the latter sum, first for the summation over $\beta$ and then for the summation over $\alpha$, we obtain $\frac{1}{24}(\frac{i\lambda}{2}\dim V)^2=-\frac{1}{96}\lambda^2(\dim V)^2$. For the first sum let $(T_i)$ be a basis of $\frakm$ and $(T_i')$ its dual basis with respect to the Killing form $\kappa$ of $\frakg$. Then
\begin{align*}
 \sum_{\alpha,\beta}\sigma_\lambda(X_\alpha)\sigma_\lambda(X_\beta)\sigma_\lambda(\Omega_\mu(B_\mu(\widehat{e}_\alpha,\widehat{e}_\beta)) &= \sum_i\sum_{\alpha,\beta}\kappa(B_\mu(\widehat{e}_\alpha,\widehat{e}_\beta),T_i')\sigma_\lambda(X_\alpha)\sigma_\lambda(X_\beta)\sigma_\lambda(\Omega_\mu(T_i))\\
 &= -\kappa_0\sum_i\sigma_\lambda(\Omega_\mu(T_i'))\sigma_\lambda(\Omega_\mu(T_i))\\
 &= 4\kappa_0\lambda^2\sum_i d\omega_{\met,\lambda}(T_i')d\omega_{\met,\lambda}(T_i)\\
 &= 4\kappa_0\lambda^2 d\omega_{\met,\lambda}(\Cas_\frakm),
\end{align*}
where $\Cas_\frakm=\sum_i T_i'T_i\in U(\frakm)$ is the Casimir element of $\frakm$.
\end{proof}

\chapter{Analysis of the Fourier transform of $\Omega_\mu$}\label{ch:AnalysisFTofOmegaMu}

We study the subrepresentation $I(\zeta,\nu)^{\Omega_\mu(\frakm)}$ and its image in $\calD'(\RR^\times)\otimeshat\calS'(\Lambda\times\Lambda)\otimeshat V_\zeta$ under the Fourier transform in the case where $G$ is non-Hermitian. For this, we first recall some more structure theory following \cite{SS,SS15}. More precisely, the $5$-grading of $\frakg$ with respect to the grading element $H\in\frakg_0$ is refined to a bigrading with respect to a two-dimensional abelian subalgebra of $\frakg_0$ containing $H$ (see Section~\ref{sec:LagrangianDecomposition}, \ref{sec:Bigrading} and \ref{sec:IdentitiesBigrading}). The bigrading gives rise to natural complementary Lagrangian subspaces $\Lambda,\Lambda^*\subseteq V=\frakg_{-1}$ which are used in the Heisenberg group Fourier transform. Using this structure, we solve the differential equation $\Omega_\mu(T)u=0$ on the Fourier transformed side (see Section~\ref{sec:InvariantDistributionVectors}) and in this way obtain a model of the representation $I(\zeta,\nu)^{\Omega_\mu(\frakm)}$ on a subspace of $\calD'(\RR^\times)\otimeshat\calS'(\Lambda)$, the \emph{Fourier transformed picture} (see Section~\ref{sec:FTpictureMinRep}). The cases $\frakg\simeq\sl(n,\RR)$ and $\so(p,q)$ have to be treated separately in Sections~\ref{sec:FTpictureMinRepSLn} and \ref{sec:FTpictureMinRepSOpq}. Finally, we compare the formulas for the Lie algebra action in this model with the literature in Section~\ref{sec:LAactionLiterature}.

\section{The Lagrangian decomposition}\label{sec:LagrangianDecomposition}

By Theorem~\ref{thm:CharacterizationHermitian}, the group $G$ is non-Hermitian if and only if there exists $O\in V$\index{O@$O$} such that $Q(O)>0$. We renormalize $O$ such that $Q(O)=1$. Any such $O\in V$ has by \cite[Main Theorem]{SS15} a Lagrangian decomposition
$$ O=A+B\index{A1@$A$}\index{B@$B$} $$
where $\mu(A)=\mu(B)=0$ and $\omega(A,B)=2$. This decomposition is unique and $A$ and $B$ are given by
$$ A = \tfrac{1}{2}(O-\Psi(O)) \qquad \mbox{and} \qquad B = \tfrac{1}{2}(O+\Psi(O)). $$
Further, the tangent spaces of $Z=\mu^{-1}(0)$\index{Z@$Z$} at $A$ and $B$,
$$ \Lambda := T_AZ \qquad \mbox{and} \qquad \Lambda^* := T_BZ,\index{1Lambda@$\Lambda$}\index{1LambdaStar@$\Lambda^*$}$$
are complementary Lagrangian subspaces. We use this particular Lagrangian decomposition $V=\Lambda\oplus\Lambda^*$ for the Schr\"{o}dinger model of the representation $\sigma_\lambda$ of $\overline{N}$.

Note that we use the same letter $A$ for the element $A\in V$ and the one-dimensional subgroup $A=\exp(\RR H)\subseteq G$. It should be clear from the context which object is meant.

\section{The bigrading}\label{sec:Bigrading}

Let
$$ H_\alpha = \tfrac{1}{2}(H+\mu(O))=\tfrac{1}{2}(H+2B_\mu(A,B)), \qquad H_\beta = \tfrac{1}{2}(H-\mu(O))=\tfrac{1}{2}(H-2B_\mu(A,B)).\index{H1alpha@$H_\alpha$}\index{H1beta@$H_\beta$} $$
Then $[H_\alpha,H_\beta]=0$ and $\ad(H_\alpha)$ and $\ad(H_\beta)$ are simultaneously diagonalizable. Write
$$ \frakg_{(i,j)} := \{X\in\frakg:\ad(H_\alpha)X=iX,\ad(H_\beta)X=jX\},\index{g3ij@$\frakg_{(i,j)}$} $$
then we have the bigrading
$$ \frakg = \bigoplus_{i,j}\frakg_{(i,j)}. $$
More precisely:
\begin{align*}
  \frakg_2 &= \frakg_{(1,1)},\\
  \frakg_1 &= \frakg_{(2,-1)}+\frakg_{(1,0)}+\frakg_{(0,1)}+\frakg_{(-1,2)},\\
  \frakg_0 &= \frakg_{(1,-1)}+\frakg_{(0,0)}+\frakg_{(-1,1)},\\
  \frakg_{-1} &= \frakg_{(1,-2)}+\frakg_{(0,-1)}+\frakg_{(-1,0)}+\frakg_{(-2,1)},\\
  \frakg_{-2} &= \frakg_{(-1,-1)}.
\end{align*}
Note that $\Ad(w_0)\frakg_{(i,j)}=\frakg_{(-j,-i)}$, i.e. $w_0$ flips the star diagram along the axis $i+j=0$.
$$
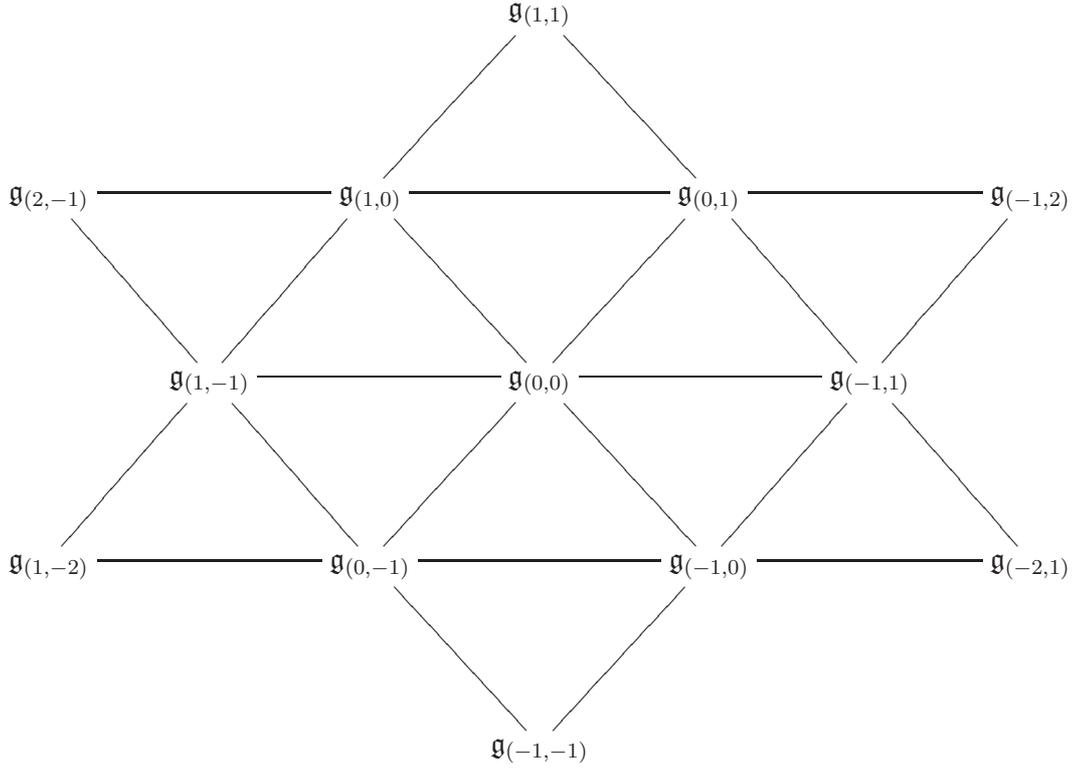
\begin{xy}
\xymatrix{
 & & & \frakg_{(1,1)} \ar@{-}[ddl] \ar@{-}[ddr]\\ \\
 \frakg_{(2,-1)} \ar@{-}[ddr] \ar@{-}[rr] & & \frakg_{(1,0)} \ar@{-}[ddl] \ar@{-}[ddr] \ar@{-}[rr] & & \frakg_{(0,1)} \ar@{-}[ddl] \ar@{-}[ddr] \ar@{-}[rr] & & \frakg_{(-1,2)} \ar@{-}[ddl]\\ \\
 & \frakg_{(1,-1)} \ar@{-}[ddr] \ar@{-}[ddl] \ar@{-}[rr] & & \frakg_{(0,0)} \ar@{-}[ddr] \ar@{-}[ddl] \ar@{-}[rr] & & \frakg_{(-1,1)} \ar@{-}[ddr] \ar@{-}[ddl]\\ \\
 \frakg_{(1,-2)} \ar@{-}[rr] & & \frakg_{(0,-1)} \ar@{-}[ddr] \ar@{-}[rr] & & \frakg_{(-1,0)} \ar@{-}[ddl] \ar@{-}[rr] & & \frakg_{(-2,1)}\\ \\
 & & & \frakg_{(-1,-1)}
}
\end{xy}
$$

Here $\frakm=\frakg_{(1,-1)}\oplus(\frakm\cap\frakg_{(0,0)})\oplus\frakg_{(-1,1)}$ and $\frakg_{(0,0)}=\RR H\oplus(\frakm\cap\frakg_{(0,0)})$. Further, $\frakm\cap\frakg_{(0,0)}=\RR B_\mu(A,B)\oplus\frakm^O$, where
$$ \frakm^O=\{T\in\frakm:TO=0\}=\{T\in\frakm:TA=TB=0\}.\index{m3O@$\frakm^O$} $$
Moreover,
$$ \frakg_{(1,-2)} = \RR A \qquad \mbox{and} \qquad \frakg_{(-2,1)} = \RR B, $$
and the Lagrangians $\Lambda$ and $\Lambda^*$ are given by
$$ \Lambda = \RR A+\frakg_{(0,-1)} \qquad \mbox{and} \qquad \Lambda^* = \RR B+\frakg_{(-1,0)}. $$
Note that $\omega(A,\frakg_{(-1,0)})=0$ and $\omega(B,\frakg_{(0,-1)})=0$. Moreover, the maps
$$ \frakg_{(0,-1)}\to\frakg_{(-1,1)}, \, v\mapsto B_\mu(v,B) \qquad \mbox{and} \qquad \frakg_{(-1,0)}\to\frakg_{(1,-1)}, \, w\mapsto B_\mu(A,w) $$
are $\frakg_{(0,0)}$-equivariant isomorphisms. Note that $B_\mu(A,B)$ acts on $\frakg_{-1}$ by
\begin{equation}
 B_\mu(A,B) = \begin{cases}\frac{3}{2} & \mbox{on $\frakg_{(1,-2)}$,}\\\frac{1}{2} & \mbox{on $\frakg_{(0,-1)}$,}\\-\frac{1}{2} & \mbox{on $\frakg_{(-1,0)}$,}\\-\frac{3}{2} & \mbox{on $\frakg_{(-2,1)}$.}\end{cases}\label{eq:BmuABonG-1}
\end{equation}
Further, note that
$$ \frakg_{(2,-1)} = \RR\overline{A}, \qquad \frakg_{(-1,2)} = \RR\overline{B},  $$
and
$$ [\overline{A},B] = -2H_\alpha, \qquad [\overline{B},A] = 2H_\beta. $$


For $z\in\frakg_{(0,-1)}$ we have $\mu(z)\in\frakg_{(1,-1)}$ and hence $\Psi(z)\in\frakg_{(1,-2)}=\RR A$. We define $n:\frakg_{(0,-1)}\to\RR$ by
$$ \Psi(z) = n(z)A, \qquad z\in\frakg_{(0,-1)}.\index{n2z@$n(z)$} $$
Then the function $n(z)$ is a polynomial of degree $3$ which vanishes identically if and only if $\frakg\simeq\sl(n,\RR)$ (see \cite[Proposition 7.9]{SS}). In all other cases,
$$ \calJ=\frakg_{(0,-1)}\index{J3@$\calJ$} $$
carries the structure of a rank $3$ Jordan algebra with Jordan determinant $n(z)$. (Strictly speaking, one also has to exclude the case $\frakg\simeq\frakg_{2(2)}$ where $\frakg_{(0,-1)}\simeq\RR$ with $n(z)=z^3$.) Note that $\Psi^{-1}(0)\cap\calJ$ resp. $\mu^{-1}(0)\cap\calJ$ is the subvariety of elements of rank $\leq2$ resp. $\leq1$. We write $\calJ^*=\frakg_{(-1,0)}$\index{J3@$\calJ^*$} which can be identified with the dual of $\calJ$ using the symplectic form.

\begin{remark}
A slightly different and more natural point of view is to endow $(V^+,V^-)=(\calJ,\calJ^*)$ with the structure of a \emph{Jordan pair}. This structure consists of trilinear maps
$$ \{\cdot,\cdot,\cdot\}_\pm:V^\pm\times V^\mp\times V^\pm\to V^\pm, $$
such that
\begin{enumerate}[(1)]
\item\label{enum:JordanTripleAxioms1} $\{u,v,w\}_\pm=\{w,v,u\}_\pm$ for all $u,w\in V^\pm$ and $v\in V^\mp$,
\item\label{enum:JordanTripleAxioms2} $\{x,y,\{u,v,w\}_\pm\}_\pm=\{\{x,y,u\}_\pm,v,w\}_\pm-\{u,\{v,x,y\}_\mp,w\}_\pm+\{u,v,\{x,y,w\}_\pm\}_\pm$ for all $x,u,w\in V^\pm$ and $y,v\in V^\mp$.
\end{enumerate}
If we define
$$ \{u,v,w\}_\pm := \pm B_\Psi(u,v,w), $$
property \eqref{enum:JordanTripleAxioms1} follows immediately from the symmetry of $B_\Psi$, and Lemma~\ref{lem:SymmetrizationsOfSymplecticCovariants} implies that property \eqref{enum:JordanTripleAxioms2} is equivalent to the $\frakm$-equivariance of $B_\mu$.
\end{remark}

In addition to $w_0$ we define the group elements
\begin{equation}
	w_1 = \exp\left(\frac{\pi}{2\sqrt{2}}(A-\overline{B})\right),  \qquad w_2 = \exp\left(\frac{\pi}{2\sqrt{2}}(B+\overline{A})\right).\index{w21@$w_1$}\index{w22@$w_2$}\label{eq:DefW1W2}
\end{equation}

\begin{lemma}\label{lem:W1W2}
\begin{enumerate}[(1)]
\item\label{lem:W1W2-1} The elements $w_1$ and $w_2$ have the following mapping properties:
$$ \Ad(w_1)\frakg_{(i,j)} = \frakg_{(i+j,-j)}, \qquad \Ad(w_2)\frakg_{(i,j)} = \frakg_{(-i,i+j)}. $$
\item\label{lem:W1W2-2} For $v\in\calJ$, $w\in\calJ^*$ and $T\in\frakm^O$ we have
\begin{align*}
 \Ad(w_1)F &= -\frac{1}{\sqrt{2}}B, & \Ad(w_1)E &= \frac{1}{\sqrt{2}}\overline{A},\\
 \Ad(w_1)A &= -\overline{B}, & \Ad(w_1)\overline{A} &= -\sqrt{2}E,\\
 \Ad(w_1)v &= \sqrt{2}B_\mu(v,B), &  \Ad(w_1)\overline{v} &= \overline{v},\\
 \Ad(w_1)w &= w, & \Ad(w_1)\overline{w} &= \sqrt{2}B_\mu(A,w),\\
 \Ad(w_1)B &= \sqrt{2}F, & \Ad(w_1)\overline{B} &= -A,\\
 \Ad(w_1)B_\mu(v,B) &= -\frac{1}{\sqrt{2}}v, & \Ad(w_1)B_\mu(A,w) &= -\frac{1}{\sqrt{2}}\overline{w},\\
 \Ad(w_1)H &= B_\mu(A,B)+\frac{1}{2}H, & \Ad(w_1)B_\mu(A,B) &= -\frac{1}{2}B_\mu(A,B)+\frac{3}{4}H,\\
 \Ad(w_1)T &= T,
\end{align*}
and similar for $w_2$ by substituting $(A,B)\mapsto(B,-A)$.
\item\label{lem:W1W2-3} We have $w_1^2,w_2^2\in M$ with $\chi(w_1^2)=\chi(w_2^2)=-1$ and
\begin{equation*}
\begin{split}
 \Ad(w_1^2)(aA+v+w+bB) &= aA-v+w-bB\\
 \Ad(w_2^2)(aA+v+w+bB) &= -aA+v-w+bB
\end{split}
\qquad(a,b\in\RR,v\in\calJ,w\in\calJ^*).
\end{equation*}
\item\label{lem:W1W2-4} The following relations hold:
\begin{align*}
 w_0w_1w_0^{-1} &= w_2^{-1}, & w_1w_0w_1^{-1} &= w_2, & w_2w_0w_2^{-1} &= w_1^{-1}\\
 w_0w_2w_0^{-1} &= w_1, & w_1w_2w_1^{-1} &= w_0^{-1}, & w_2w_1w_2^{-1} &= w_0.
\end{align*}
\end{enumerate}
\end{lemma}

\begin{proof}
\eqref{lem:W1W2-2} is an easy though longish computation using the definitions in Section~\ref{sec:W0} as well as Lemma~\ref{lem:G1bracketG-1}. The formulas for $\Ad(w_1)$ and $\Ad(w_2)$ then imply \eqref{lem:W1W2-1} and \eqref{lem:W1W2-3}. Finally, \eqref{lem:W1W2-4} follows with the identity
$$ w\exp(X)w^{-1} = \exp(\Ad(w)X) $$
and the definitions in Section~\ref{sec:W0}.
\end{proof}

\section{Identities for the moment map $\mu$ in the bigrading}\label{sec:IdentitiesBigrading}

We show several identities for the moment map $\mu$ and its symmetrization $B_\mu$, acting on different parts of the decomposition
$$ V = \RR A\oplus\calJ\oplus\calJ^*\oplus\RR B. $$

\begin{lemma}\label{lem:TraceOnG0-1}
Assume $\frakg\not\simeq\sl(n,\RR),\so(p,q)$. For $z\in\calJ$ and $w\in\calJ^*$ we have
$$ \tr(B_\mu(A,w)\circ B_\mu(z,B)|_\calJ) = \left(\frac{1}{2}+\frac{1}{6}\dim\calJ\right)\omega(z,w). $$
\end{lemma}

We remark that the formula also holds for $\frakg=\sl(3,\RR)$ and $\so(4,4)$, but we do not need this.

\begin{proof}
First note that, since $B_\mu$ is $\frakm$-equivariant,
$$ [B_\mu(A,w),B_\mu(z,B)] = B_\mu(B_\mu(A,w)z,B)+B_\mu(z,B_\mu(A,w)B). $$
By Lemma~\ref{lem:RewriteBmu}, we have $B_\mu(A,w)z=\frac{1}{2}\omega(z,w)A$ and $B_\mu(A,w)B=-w$, and hence
$$ [B_\mu(A,w),B_\mu(z,B)] = \frac{1}{2}\omega(z,w)B_\mu(A,B)-B_\mu(z,w). $$
Choose a basis $(e_\alpha)$ of $\frakg_{(0,-1)}$ and let $(\widehat{e}_\alpha)$ be the basis of $\frakg_{(-1,0)}$ such that $\omega(e_\alpha,\widehat{e}_\beta)=\delta_{\alpha\beta}$. Then
\begin{align*}
 & \tr(B_\mu(A,w)\circ B_\mu(z,B)|_{\frakg_{(0,-1)}}) = \sum_\alpha \omega(B_\mu(A,w)B_\mu(z,B)e_\alpha,\widehat{e}_\alpha)\\
 ={}& \sum_\alpha \Big(\omega(B_\mu(z,B)B_\mu(A,w)e_\alpha,\widehat{e}_\alpha)+\frac{1}{2}\omega(z,w)\omega(B_\mu(A,B)e_\alpha,\widehat{e}_\alpha)-\omega(B_\mu(z,w)e_\alpha,\widehat{e}_\alpha)\Big).
\end{align*}
Again by Lemma~\ref{lem:RewriteBmu}, we find that $B_\mu(A,w)e_\alpha=\frac{1}{2}\omega(e_\alpha,w)A$ and $B_\mu(z,B)A=z$, and further $B_\mu(A,B)e_\alpha=\frac{1}{2}e_\alpha$ so that
$$ \tr(B_\mu(A,w)\circ B_\mu(z,B)|_{\frakg_{(0,-1)}}) = \left(\frac{1}{2}+\frac{1}{4}\dim\frakg_{(0,-1)}\right)\omega(z,w)-\tr(B_\mu(z,w)|_{\frakg_{(0,-1)}}). $$
We have $B_\mu(z,w)\in\frakg_{(0,0)}=\RR H_\alpha\oplus\RR H_\beta\oplus\frakm^O$, where $\frakm^O=\{T\in\frakm:[T,O]=0\}$. Write $B_\mu(z,w)=aH_\alpha+bH_\beta+T$ with $T\in\frakm^O$. Then $\tr(T|_{\frakg_{(0,-1)}})=0$ (since $\frakg\not\simeq\sl(n,\RR),\so(p,q)$ and hence $\frakm^O$ is semisimple) and $\tr(H_\alpha|_{\frakg_{(0,-1)}})=0$, $\tr(H_\beta|_{\frakg_{(0,-1)}})=-\dim\frakg_{(0,-1)}$. We determine $a$ and $b$ which will complete the proof. Since $[T,A]=[T,B]=0$ we have
$$ B_\mu(z,w)A = (a-2b)A \qquad \mbox{and} \qquad B_\mu(z,w)B = (b-2a)B. $$
On the other hand, by Lemma~\ref{lem:RewriteBmu} we find
\begin{align*}
 B_\mu(z,w)A &= \frac{1}{2}\omega(B_\mu(z,w)A,B)A = \frac{1}{2}\omega(B_\mu(A,B)z,w)A = \frac{1}{4}\omega(z,w)A,\\
 B_\mu(z,w)B &= \frac{1}{2}\omega(A,B_\mu(z,w)B)B = \frac{1}{2}\omega(z,B_\mu(A,B)w)B = -\frac{1}{4}\omega(z,w)B.
\end{align*}
Thus, $a=-b=\frac{1}{12}\omega(z,w)$ and the claimed formula follows.
\end{proof}

\begin{lemma}\label{lem:MuSquared}
For $z\in\calJ$ we have
$$ \mu(z)^2B = -4n(z)z. $$
\end{lemma}

\begin{proof}
Using Lemma~\ref{lem:RewriteBmu} and $\mu(z)z=-3\Psi(z)=-3n(z)A$ we find
\begin{align*}
 \mu(z)^2B &= \mu(z)B_\mu(z,z)B = \mu(z)B_\mu(z,B)z = [\mu(z),B_\mu(z,B)]z + B_\mu(z,B)\mu(z)z\\
 &= [\mu(z),B_\mu(z,B)]z - 3n(z)B_\mu(z,B)A = [\mu(z),B_\mu(z,B)]z - 3n(z)z.
\end{align*}
Now, by the $\frakm$-equivariance of $B_\mu$:
\begin{align*}
 [\mu(z),B_\mu(z,B)]z &= B_\mu(\mu(z)z,B)z + B_\mu(z,\mu(z)B)z\\
 &= -3n(z)B_\mu(A,B)z + B_\mu(z,B_\mu(z,B)z)z\\
 &= -\frac{3}{2}n(z)z + \frac{1}{2}[B_\mu(z,B),\mu(z)]z,
\end{align*}
and hence $[\mu(z),B_\mu(z,B)]z=-n(z)z$, and the claim follows.
\end{proof}

\begin{lemma}\label{lem:CmForSimpleM}
If $\frakg$ is non-Hermitian and $\frakm$ is simple, the number $\calC(\frakm)$ in Lemma~\ref{lem:BezoutianSum} is given by
$$ \calC(\frakm) = \frac{3}{2}+\frac{\dim\calJ}{6}.\index{Cmprime@$\calC(\frakm')$} $$
\end{lemma}

\begin{proof}
Let $(e_\alpha)$ be a basis of $\frakg_{-1}$ and $(\widehat{e}_\alpha)$ its dual basis with respect to $\omega$, then by Lemma~\ref{lem:BezoutianSum}
\begin{equation}
 \sum_\alpha B_\mu(Te_\alpha,\widehat{e}_\alpha) = \calC(\frakm)\cdot T \qquad \mbox{for all }T\in\frakm.\label{eq:CofM}
\end{equation}
We may choose $e_\alpha\in\{A,B\}\cup\frakg_{(0,-1)}\cup\frakg_{(-1,0)}$, then $\widehat{e}_\alpha\in\{-\frac{1}{2}A,\frac{1}{2}B\}\cup\frakg_{(0,-1)}\cup\frakg_{(-1,0)}$. Now put $T=B_\mu(A,B)\in\frakm$, then the left hand side of \eqref{eq:CofM} becomes
$$ \frac{1}{2}B_\mu(TA,B) - \frac{1}{2}B_\mu(TB,A) + \sum_{e_\alpha\in\frakg_{(0,-1)}} B_\mu(Te_\alpha,\widehat{e}_\alpha) + \sum_{e_\alpha\in\frakg_{(-1,0)}} B_\mu(Te_\alpha,\widehat{e}_\alpha), $$
and by \eqref{eq:BmuABonG-1} this is
$$ \frac{3}{2}B_\mu(A,B) + \sum_{e_\alpha\in\frakg_{(0,-1)}} B_\mu(e_\alpha,\widehat{e}_\alpha). $$
Hence
$$ \sum_{e_\alpha\in\frakg_{(0,-1)}} B_\mu(e_\alpha,\widehat{e}_\alpha) = \left(\calC(\frakm)-\frac{3}{2}\right) B_\mu(A,B). $$
We apply both sides to $A$ and find
$$ \sum_{e_\alpha\in\frakg_{(0,-1)}}B_\mu(e_\alpha,\widehat{e}_\alpha)A = \left(\calC(\frakm)-\frac{3}{2}\right)B_\mu(A,B)A = \frac{3}{2}\left(\calC(\frakm)-\frac{3}{2}\right)A. $$
As in the proof of Lemma~\ref{lem:TraceOnG0-1}, we have $B_\mu(e_\alpha,\widehat{e}_\alpha)A=\frac{1}{4}\omega(e_\alpha,\widehat{e}_\alpha)A=\frac{1}{4}A$ and the claim follows.
\end{proof}

\begin{lemma}\label{lem:DecompBigradingMuPsiQ}
Let $x=aA+z+w+bB$ with $a,b\in\RR$, $z\in\frakg_{(0,-1)}$ and $w\in\frakg_{(-1,0)}$. Then
\begin{align*}
 \mu(x) ={}& \underbrace{(\mu(z)+2aB_\mu(A,w))}_{\in\frakg_{(1,-1)}} + \underbrace{(2abB_\mu(A,B)+2B_\mu(z,w))}_{\in\frakg_{(0,0)}} + \underbrace{(\mu(w)+2bB_\mu(z,B))}_{\in\frakg_{(-1,1)}},\\
 \Psi(x) ={}& \underbrace{(-a^2b-\tfrac{1}{2}a\omega(z,w)+n(z))A}_{\in\frakg_{(1,-2)}} + \underbrace{\Big[(-ab-\tfrac{1}{2}\omega(z,w))z-a\mu(w)A-\mu(z)w\Big]}_{\in\frakg_{(0,-1)}}\\
 &\hspace{1.5cm}+\underbrace{\Big[(ab+\tfrac{1}{2}\omega(z,w))w-b\mu(z)B-\mu(w)z\Big]}_{\in\frakg_{(-1,0)}} + \underbrace{(ab^2+\tfrac{1}{2}b\omega(z,w)+n^*(w))B}_{\in\frakg_{(-2,1)}},\\
 Q(x) ={}& a^2b^2+ab\omega(z,w)-2bn(z)+2an^*(w) + \frac{1}{4}\omega(z,w)^2+\frac{1}{2}\omega(\mu(z)w,w),
\end{align*}
where we write $\Psi(z)=n(z)A$ and $\Psi(w)=n^*(w)B$. Moreover, we have $ \mu(z) = -B_\mu(A,\mu(z)B)$, $\mu(w) = B_\mu(\mu(w)A,B) $ and
$$ B_\mu(z,w) \in \tfrac{1}{6}\omega(z,w)B_\mu(A,B) + \frakm^O. $$
\end{lemma}

\begin{proof}
This is a direct computation.
\end{proof}

\section[$\frakm$-invariant distribution vectors]{$\frakm$-invariant distribution vectors in the metaplectic representation}\label{sec:InvariantDistributionVectors}

Using Theorem~\ref{thm:FTofOmegaMu}, we compute $d\omega_{\met,\lambda}(T)$ explicitly for $T\in\frakm$. In view of the decomposition $\Lambda=\RR A+\calJ$, we write $x\in\Lambda$ as $x=aA+z$ with $a\in\RR$ and $z\in\calJ$.

\begin{proposition}
For every $\lambda\in\RR^\times$ the representation $d\omega_{\met,\lambda}$ of $\frakm$ on $\calS'(\Lambda)$ is given by
$$ d\omega_{\met,\lambda}(T) = \begin{cases}\frac{1}{2i\lambda}\sum_{e_\alpha,e_\beta\in\frakg_{(0,-1)}}\omega(T\widehat{e}_\alpha,\widehat{e}_\beta)\partial_\alpha\partial_\beta-\frac{1}{2}\omega(TB,z)\partial_A & T\in\frakg_{(1,-1)},\\-\frac{1}{2}\omega(TA,B)a\partial_A-\partial_{Tz}-\frac{1}{2}\tr(T|_\Lambda) & T\in\frakg_{(0,0)}\cap\frakm,\\-a\partial_{TA}+\frac{1}{2}i\lambda\omega(Tz,z) & T\in\frakg_{(-1,1)}.\end{cases} $$
\end{proposition}

\begin{proof}
By Theorem~\ref{thm:FTofOmegaMu}:
$$ d\omega_{\met,\lambda}(T) = \frac{1}{2i\lambda}\sum_{\alpha,\beta}\omega(T\widehat{e}_\alpha,\widehat{e}_\beta)d\sigma_\lambda(X_\alpha)d\sigma_\lambda(X_\beta). $$
Since this expression is independent of the choice of the basis $(e_\alpha)$ we may choose $e_\alpha\in\{A,B\}\cup\frakg_{(0,-1)}\cup\frakg_{(-1,0)}$. Then $\widehat{e}_\alpha\in\{-\tfrac{1}{2}A,\tfrac{1}{2}B\}\cup\frakg_{(0,-1)}\cup\frakg_{(-1,0)}$. Further, the representation $d\sigma_\lambda$ is in the coordinates $(a,z)$ given by
$$ d\sigma_\lambda(A) = -\partial_A, \qquad d\sigma_\lambda(v) = -\partial_v \qquad (v\in\frakg_{(0,-1)}) $$
and
$$ d\sigma_\lambda(B) = -2i\lambda a, \qquad d\sigma_\lambda(w) = i\lambda\omega(w,z) \qquad (w\in\frakg_{(-1,0)}). $$

First, let $T\in\frakg_{(1,-1)}$, then
\begin{align*}
 d\omega_{\met,\lambda}(T) ={}& \frac{1}{2i\lambda}\sum_{\alpha,\beta}\omega(T\widehat{e}_\alpha,\widehat{e}_\beta)d\sigma_\lambda(X_\alpha)d\sigma_\lambda(X_\beta)\\
 ={}& \frac{1}{2i\lambda}\Bigg(-\frac{1}{2}i\lambda\sum_{e_\beta\in\frakg_{(-1,0)}}\omega(TB,\widehat{e}_\beta)\omega(e_\beta,z)\partial_A+\sum_{e_\alpha,e_\beta\in\frakg_{(0,-1)}}\omega(T\widehat{e}_\alpha,\widehat{e}_\beta)\partial_\alpha\partial_\beta\\
 & -\frac{1}{2}i\lambda\sum_{e_\alpha\in\frakg_{(-1,0)}}\omega(T\widehat{e}_\alpha,B)\omega(e_\alpha,z)\partial_A\Bigg)\\
 ={}& \frac{1}{2i\lambda}\sum_{e_\alpha,e_\beta\in\frakg_{(0,-1)}}\omega(T\widehat{e}_\alpha,\widehat{e}_\beta)\partial_\alpha\partial_\beta-\frac{1}{2}\omega(TB,z)\partial_A.
\end{align*}

Next, let $T\in\frakg_{(0,0)}$, then $\ad(T)$ preserves each $\frakg_{(i,j)}$ and we find
\begin{align*}
 d\omega_{\met,\lambda}(T) ={}& \frac{1}{2i\lambda}\sum_{\alpha,\beta}\omega(T\widehat{e}_\alpha,\widehat{e}_\beta)d\sigma_\lambda(X_\alpha)d\sigma_\lambda(X_\beta)\\
 ={}& \frac{1}{2i\lambda}\Bigg(-\frac{1}{2}i\lambda\omega(TB,A)\partial_Aa-i\lambda\sum_{\substack{e_\alpha\in\frakg_{(0,-1)}\\e_\beta\in\frakg_{(-1,0)}}}\omega([T,\widehat{e}_\alpha],\widehat{e}_\beta)\partial_\alpha\omega(e_\beta,z)\\
 & -i\lambda\sum_{\substack{e_\alpha\in\frakg_{(-1,0)}\\e_\beta\in\frakg_{(0,-1)}}}\omega([T,\widehat{e}_\alpha],\widehat{e}_\beta)\omega(e_\alpha,z)\partial_\beta-\frac{1}{2}i\lambda\omega(TA,B)a\partial_A\Bigg)\\
 ={}& -\frac{1}{2}\Big(\omega(TA,B)a\partial_A+2\partial_{Tz}+\tr(T|_\Lambda)\Big).
\end{align*}

Finally, let $T\in\frakg_{(-1,1)}$, then
\begin{align*}
 d\omega_{\met,\lambda}(T) ={}& \frac{1}{2i\lambda}\sum_{\alpha,\beta}\omega(T\widehat{e}_\alpha,\widehat{e}_\beta)d\sigma_\lambda(X_\alpha)d\sigma_\lambda(X_\beta)\\
 ={}& \frac{1}{2i\lambda}\Bigg(-i\lambda\sum_{e_\alpha\in\frakg_{(0,-1)}}\omega(T\widehat{e}_\alpha,A)a\partial_\alpha-\lambda^2\sum_{e_\alpha,e_\beta\in\frakg_{(-1,0)}}\omega(T\widehat{e}_\alpha,\widehat{e}_\beta)\omega(e_\alpha,z)\omega(e_\beta,z)\\
 & -i\lambda\sum_{e_\beta\in\frakg_{(0,-1)}}\omega(TA,\widehat{e}_\beta)a\partial_\beta\Bigg)\\
 ={}& -a\partial_{TA}+\frac{1}{2}i\lambda\omega(Tz,z).\qedhere
\end{align*}
\end{proof}

Recall that $\calJ=\frakg_{(0,-1)}$ is a rank $3$ Jordan algebra with norm function $n(z)=\frac{1}{2}\omega(\Psi(z),B)$, except in the case $\frakg\simeq\sl(n,\RR)$ where $n(z)=0$ and in the case $\frakg\simeq\frakg_{2(2)}$ where $\calJ\simeq\RR$ is strictly speaking of rank one. We note that for any $w\in\calJ$ we have, by Lemma~\ref{lem:SymmetrizationsOfSymplecticCovariants},
\begin{equation}
 \partial_wn(z) = \frac{1}{2}\omega(3B_\Psi(z,z,w),B) = -\frac{1}{2}\omega(\mu(z)w+\frac{1}{2}\tau(z)w,B) = -\frac{1}{2}\omega(\mu(z)w,B).\label{eq:DerivativeOfN}
\end{equation}

\begin{theorem}\label{thm:InvDistributionVector}
Assume $\frakg\not\simeq\sl(n,\RR),\so(p,q)$. For every $\lambda\in\RR^\times$ the space $L^2(\Lambda)^{-\infty,\frakm}=\calS'(\Lambda)^{\frakm}$ of $\frakm$-invariant distribution vectors in $\omega_{\met,\lambda}$ is two-dimensional. More precisely, $\calS'(\Lambda)^{\frakm}=\CC\xi_{\lambda,0}\oplus\CC\xi_{\lambda,1}$, where
$$ \xi_{\lambda,\varepsilon}(a,z) = \sgn(a)^\varepsilon|a|^{s_\min}e^{-i\lambda\frac{n(z)}{a}} \qquad (a\in\RR,z\in\calJ),\index{1oxilambdaepsilon@$\xi_{\lambda,\varepsilon}$} $$
where $s_\min=-\tfrac{1}{6}(\dim\Lambda+2)$\index{smin@$s_\min$} and $\varepsilon\in\ZZ/2\ZZ$.
\end{theorem}

\begin{remark}
	For $s_\min\leq-1$ the definition of $\xi_{\lambda,\varepsilon}(a,z)$ does not define a locally integrable function on $\Lambda$, but has to be interpreted as a distribution. In Appendix~\ref{app:MeromFamily} we show that $\xi_{\lambda,\varepsilon}$ is indeed the special value of a meromorphic family of distributions on $\Lambda$ at a regular point.
\end{remark}

\begin{proof}[Proof of Theorem~\ref{thm:InvDistributionVector}]
Let $\xi\in\calS'(\Lambda)$ such that $d\omega_{\met,\lambda}(T)\xi=0$ for all $T\in\frakm$. By the assumptions on $\frakg$, the subalgebra $\frakm$ is generated by $\frakg_{(1,-1)}$ and $\frakg_{(-1,1)}$, so the above condition is equivalent to $d\omega_{\met,\lambda}(T)\xi=0$ for $T\in\frakg_{(1,-1)}$ and $T\in\frakg_{(-1,1)}$. First, let $T=B_\mu(w,B)\in\frakg_{(-1,1)}$, $w\in\frakg_{(0,-1)}$, then by Lemma~\ref{lem:RewriteBmu}:
$$ TA = B_\mu(B,w)A = B_\mu(B,A)w +\frac{1}{4}\omega(B,w)A-\frac{1}{4}\omega(B,A)w-\frac{1}{2}\omega(w,A)B = w, $$
and for $z\in\calJ$:
$$ \omega(Tz,z) = \omega(B_\mu(w,B)z,z) = \omega(B_\mu(w,z)B,z) = \omega(B_\mu(z,w)z,B) = \omega(\mu(z)w,B). $$
Hence, $d\omega_{\met,\lambda}(T)\xi=0$ implies
$$ a\partial_w\xi = \frac{1}{2}i\lambda\omega(\mu(z)w,B)\xi = -i\lambda(\partial_wn(z))\xi, $$
which is equivalent to
$$ a\partial_w(\xi\cdot e^{i\lambda\frac{n(z)}{a}}) = 0. $$
Since $w\in\frakg_{(0,-1)}$ was arbitrary, we have
$$ \xi(a,z) = \xi_0(a)e^{-i\lambda\frac{n(z)}{a}} + \xi_1(a,z), $$
where $\xi_0(a)$ is independent of $z$ and $\xi_1(a,z)$ has support on $\{a=0\}$. Next, let $T=B_\mu(A,w)\in\frakg_{(1,-1)}$, $w\in\frakg_{(-1,0)}$, then again by Lemma~\ref{lem:RewriteBmu}:
$$ TB = B_\mu(A,w)B = B_\mu(A,B)w+\frac{1}{4}\omega(A,w)B-\frac{1}{4}\omega(A,B)w-\frac{1}{2}\omega(w,B)A = -w. $$
Therefore, $d\omega_{\met,\lambda}(T)\xi=0$ implies
$$ \sum_{e_\alpha,e_\beta\in\frakg_{(0,-1)}} \omega(T\widehat{e}_\alpha,\widehat{e}_\beta)\partial_\alpha\partial_\beta\xi = i\lambda\omega(z,w)\partial_A\xi. $$
Let us first assume $a\neq0$, then $\xi(a,z)=\xi_0(a)e^{-i\lambda\frac{n(z)}{a}}$ and hence
\begin{align*}
 \partial_A\xi(a,z) &= \xi_0'(a)e^{i\lambda\frac{n(z)}{a}}+i\lambda a^{-2}n(z)\xi(a,z), \\
 \partial_\alpha\xi(a,z) &= \frac{1}{2}i\lambda a^{-1}\omega(\mu(z)e_\alpha,B)\xi(a,z),\\
 \partial_\alpha\partial_\beta\xi(a,z) &= i\lambda a^{-1}\omega(B_\mu(z,e_\beta)e_\alpha,B)\xi(a,z)-\frac{1}{4}\lambda^2a^{-2}\omega(\mu(z)e_\alpha,B)\omega(\mu(z)e_\beta,B)\xi(a,z).
\end{align*}
We sum the two terms for $\partial_\alpha\partial_\beta\xi$ over $\alpha$ and $\beta$ separately. For the first term we obtain, using Lemma~\ref{lem:RewriteBmu} and \ref{lem:TraceOnG0-1}:
\begin{align*}
 & \sum_{e_\alpha,e_\beta\in\frakg_{(0,-1)}} \omega(T\widehat{e}_\alpha,\widehat{e}_\beta)\omega(B_\mu(z,e_\beta)e_\alpha,B) = \sum_{e_\alpha,e_\beta\in\frakg_{(0,-1)}} \omega(T\widehat{e}_\beta,\widehat{e}_\alpha)\omega(B_\mu(z,e_\beta)B,e_\alpha)\\
 ={}& \sum_{e_\beta\in\frakg_{(0,-1)}} \omega(B_\mu(z,e_\beta)B,T\widehat{e}_\beta) = -\sum_{e_\beta\in\frakg_{(0,-1)}} \omega(TB_\mu(z,B)e_\beta,\widehat{e}_\beta)\\
 ={}& -\tr(B_\mu(A,w)\circ B_\mu(z,B)|_{\frakg_{(0,-1)}}) = -\left(\frac{1}{2}+\frac{1}{6}\dim\frakg_{(0,-1)}\right)\omega(z,w).
\end{align*}
For the second term we have
$$ \sum_{e_\alpha,e_\beta\in\frakg_{(0,-1)}} \omega(T\widehat{e}_\alpha,\widehat{e}_\beta)\omega(\mu(z)e_\alpha,B)\omega(\mu(z)e_\beta,B) = -\omega(\mu(z)T\mu(z)B,B). $$
But $[T,\mu(z)]=2B_\mu(Tz,z)=0$ since $T\in\frakg_{(1,-1)}$ implies $Tz\in\frakg_{(1,-2)}$ and $B_\mu(Tz,z)\in\frakg_{(2,-2)}=0$. Therefore, by Lemma~\ref{lem:MuSquared}:
$$ \omega(\mu(z)T\mu(z)B,B) = -\omega(\mu(z)^2B,TB) = -4n(z)\omega(z,w). $$
This implies
$$ \sum_{e_\alpha,e_\beta\in\frakg_{(0,-1)}} \omega(T\widehat{e}_\alpha,\widehat{e}_\beta)\partial_\alpha\partial_\beta\xi = -i\lambda a^{-1}\left(\frac{1}{2}+\frac{1}{6}\dim\frakg_{(0,-1)}\right)\omega(z,w)\xi - \lambda^2a^{-2}n(z)\omega(z,w)\xi, $$
and hence $d\omega_{\met,\lambda}(T)\xi=0$ becomes
$$ a\xi_0'(a) = s_\min \xi_0(a). $$
It follows that $\xi_0(a)=c_1|a|^{s_\min}+c_2\sgn(a)|a|^{s_\min}$. Now assume $\xi_1(a,z)$ has support in $\{a=0\}$ and solves $d\omega_{\met,\lambda}(T)\xi_1=0$. Then there exists $m\in\NN$ and $\xi_{1,k}\in\calS'(\calJ)$, $k=0,\ldots,m$, such that $\xi_1(a,z)=\sum_{k=0}^m\xi_{1,k}(z)\delta^{(k)}(a)$, where $\delta^{(k)}(a)$ denotes the $k$-th derivative of the Dirac distribution $\delta(a)$. The differential equation $d\omega_{\met,\lambda}(T)\xi_1=0$ then reads
$$ \sum_{k=0}^m\sum_{e_\alpha,e_\beta\in\frakg_{(0,-1)}} \omega(T\widehat{e}_\alpha,\widehat{e}_\beta)\partial_\alpha\partial_\beta\xi_{1,k}(z)\delta^{(k)}(a) = -i\lambda\omega(z,w)\sum_{k=0}^m\xi_{1,k}(z)\delta^{(k+1)}(a). $$
Comparing coefficients of $\delta^{(k)}(a)$ we find inductively that $\xi_{1,k}=0$, so that $\xi_1=0$. This finishes the proof.
\end{proof}

\section{The Fourier transformed picture}\label{sec:FTpictureMinRep}

Assume $\frakg\not\simeq\sl(n,\RR),\so(p,q)$. Then $\frakm$ is simple and $\calC:=\calC(\frakm)=\frac{3}{2}+\frac{\dim\calJ}{6}$. It follows from Theorem~\ref{thm:ConfInvOmegaMu} that for a representation $(\zeta,V_\zeta)$ with $d\zeta=0$ and $\nu=2\,\calC-\rho-2=-\frac{2}{3}\dim\calJ-1$, the space
$$ I(\zeta,\nu)^{\Omega_\mu(\frakm)} = \{u\in I(\zeta,\nu):\Omega_\mu(T)u=0\mbox{ for all }T\in\frakm\} \subseteq I(\zeta,\nu)\index{IzetanuOmegamumprime@$I(\zeta,\nu)^{\Omega_\mu(\frakm')}$} $$
is a subrepresentation of $(\pi_{\zeta,\nu},I(\zeta,\nu))$. We study this subrepresentation in the Fourier transformed picture. Let $u\in I(\zeta,\nu)^{\Omega_\mu(\frakm)}$ and recall the representations $\sigma_\lambda$ ($\lambda\in\RR^\times$) of the Heisenberg group. By \eqref{eq:FTofVectorField}, for every $\lambda\in\RR^\times$:
$$ 0 = \sigma_\lambda(\Omega_\mu(T)u) = \sigma_\lambda(u)\circ d\sigma_\lambda(\Omega_\mu(T)) \qquad \mbox{for all }T\in\frakm. $$

\begin{corollary}\label{cor:FTImageOfSubrep}
Let $A:\calS(\Lambda)\to\calS'(\Lambda)$ be a continuous linear operator such that $A\circ d\sigma_\lambda(\Omega_\mu(T))=0$ for all $T\in\frakm$. Then there exist $u_0,u_1\in\calS'(\Lambda)$ such that
$$ A\varphi = \langle\varphi,\xi_{-\lambda,0}\rangle u_0 + \langle\varphi,\xi_{-\lambda,1}\rangle u_1 \qquad \mbox{for all }\varphi\in\calS(\Lambda). $$
\end{corollary}

\begin{proof}
Since $d\sigma_\lambda(\Omega_\mu(T))=2i\lambda d\omega_{\met,\lambda}(T)$ by Theorem~\ref{thm:FTofOmegaMu}, we have $A\circ d\omega_{\met,\lambda}(T)=0$ for all $T\in\frakm$. Let $A^\top:\calS(\Lambda)\to\calS'(\Lambda)$ denote the transpose of $A$ and note that $d\omega_{\met,\lambda}(T)^\top=-d\omega_{\met,-\lambda}(T)$ for $T\in\frakm$. Hence $d\omega_{\met,-\lambda}(T)\circ A^\top=0$ for all $T\in\frakm$. Theorem~\ref{thm:InvDistributionVector} implies that the image of $A^\top$ is contained in $\calS'(\Lambda)^\frakm=\CC\xi_{-\lambda,0}\oplus\CC\xi_{-\lambda,1}$, so there exist unique $u_0,u_1\in\calS'(\Lambda)$ such that $A^\top\varphi=\langle\varphi,u_0\rangle\xi_{-\lambda,0}+\langle\varphi,u_1\rangle\xi_{-\lambda,1}$ for $\varphi\in\calS(\Lambda)$. Passing to the transposed operator once more shows the claimed formula for $A$.
\end{proof}

By Corollary~\ref{cor:FTImageOfSubrep}, we can write
\begin{equation}
	\sigma_\lambda(u)\varphi(y) = \langle\varphi,\xi_{-\lambda,0}\rangle u_0(\lambda,y) + \langle\varphi,\xi_{-\lambda,1}\rangle u_1(\lambda,y)\label{eq:IntegralFormulaIntertwiner0}
\end{equation}
for unique $u_0,u_1\in\calD'(\RR^\times)\otimeshat\calS'(\Lambda)\otimeshat V_\zeta$\index{u0lambday@$u_0(\lambda,y)$}\index{u1lambday@$u_1(\lambda,y)$}. In terms of the integral kernel $\widehat{u}(\lambda,x,y)$ of $\sigma_\lambda(u):\calS(\Lambda)\to\calS'(\Lambda)\otimeshat V_\zeta$ this can be written as
\begin{equation}
 \widehat{u}(x,y,\lambda) = \xi_{-\lambda,0}(x)u_0(\lambda,y) + \xi_{-\lambda,1}(x)u_1(\lambda,y).\label{eq:IntegralFormulaIntertwiner}
\end{equation}
The map
$$ I(\zeta,\nu)^{\Omega_\mu(\frakm)} \to (\calD'(\RR^\times)\otimeshat\calS'(\Lambda)\otimeshat V_\zeta)\oplus(\calD'(\RR^\times)\otimeshat\calS'(\Lambda)\otimeshat V_\zeta), \quad u\mapsto(u_0,u_1) $$
is injective and we denote its image by $J_\min$\index{J1min@$J_\min$}. Let $\rho_\min$\index{1rhomin@$\rho_\min$} denote the representation of $G$ on $J_\min$ which turns the map $u\mapsto(u_0,u_1)$ into an isomorphism of $G$-representations.

We remark that $J_\min$ could be trivial, which is equivalent to $I(\zeta,\nu)^{\Omega_\mu(\frakm)}=\{0\}$. In order to show that there exists some representation $\zeta$ of $M$ such that $J_\min\neq\{0\}$ and to extend $\rho_\min$ to an irreducible unitary representation of $G$, we compute the Lie algebra action $d\rho_\min(\frakg)$. For this, we first state the action of the identity component $\overline{P}_0$ of $\overline{P}$ which is derived from Proposition~\ref{prop:ActionFTpicture} as well as \eqref{eq:IntegralFormulaIntertwiner0} and \eqref{eq:IntegralFormulaIntertwiner}. Recall that we assume $d\zeta=0$, i.e. $\zeta$ is trivial on the identity component $M_0$ of $M$.

\begin{proposition}\label{prop:FTActionPbar}
The representation $\rho_\min$ is for $g\in\overline{P}_0$ given by
$$ \rho_\min(g)(u_0,u_1)=(\rho_{\min,0}(g)u_0,\rho_{\min,1}(g)u_1), $$
where $\rho_{\min,\varepsilon}$ is the representation of $\overline{P}_0$ on $\calD'(\RR^\times)\otimeshat\calS'(\Lambda)$ given by
\begin{align*}
 \rho_{\min,\varepsilon}(\overline{n}_{(z,t)})f(\lambda,y) &= e^{i\lambda t}e^{i\lambda(\omega(z'',y)+\frac{1}{2}\omega(z',z''))}f(\lambda,y-z') && \overline{n}_{(z,t)}\in\overline{N},\\
 \rho_{\min,\varepsilon}(m)f(\lambda,y) &= (\id_{\RR^\times}^*\otimes\,\omega_{\met,\lambda}(m))f(\lambda,y) && m\in M_0,\\
 \rho_{\min,\varepsilon}(e^{tH})f(\lambda,y) &= e^{(\nu+s_\min-1)t}f(e^{-2t}\lambda,e^ty) && e^{tH}\in A.
\end{align*}
Since $\rho_{\min,\varepsilon}$ is independent of $\varepsilon\in\ZZ/2\ZZ$, we abuse notation and write $\rho_\min=\rho_{\min,0}=\rho_{\min,1}$.
\end{proposition}

To state the Lie algebra action, we write $y\in\Lambda$ as $y=aA+y'$.

\begin{proposition}\label{prop:drhomin}
The Lie algebra representation $d\rho_\min$\index{drhomin@$d\rho_\min$} of $\frakg$ is given by
$$ d\rho_\min(X)(u_0,u_1) = (d\rho_{\min,0}(X)u_0,d\rho_{\min,1}(X)u_1)), $$
where $d\rho_{\min,\varepsilon}$ is the representation of $\frakg$ on $\calD'(\RR^\times)\otimeshat\calS'(\Lambda)\otimeshat V_\zeta$ given by
\begin{align*}
 d\rho_{\min,\varepsilon}(F) ={}& i\lambda\\
 d\rho_{\min,\varepsilon}(v) ={}& -\partial_v && (v\in\Lambda)\\
 d\rho_{\min,\varepsilon}(w) ={}& -i\lambda\omega(y,w) && (w\in\Lambda^*)\\
 d\rho_{\min,\varepsilon}(T) ={}& d\omega_{\met,\lambda}(T) && (T\in\frakm)\\
 d\rho_{\min,\varepsilon}(H) ={}& \partial_y-2\lambda\partial_\lambda+2s_\min-\frac{\dim\Lambda}{2}-1\\
 d\rho_{\min,\varepsilon}(\overline{A}) ={}& i\partial_\lambda\partial_A+\frac{s_\min-\frac{\dim\Lambda}{2}}{i\lambda}\partial_A-\frac{2}{\lambda^2}n(\partial')\\
 d\rho_{\min,\varepsilon}(\overline{v}) ={}& i\partial_\lambda\partial_v+\frac{3s_\min+1}{i\lambda}\partial_v+\frac{1}{2}\omega(\mu(y')v,B)\partial_A\\
 & -\frac{1}{i\lambda}\sum_{\alpha,\beta}\omega(B_\mu(y',v)\widehat{e}_\alpha,\widehat{e}_\beta)\partial_{e_\alpha}\partial_{e_\beta} && (v\in\frakg_{(0,-1)})\\
 d\rho_{\min,\varepsilon}(\overline{w}) ={}& -\omega(y,w)\lambda\partial_\lambda+\omega(y,w)\partial_y+2s_\min\omega(y,w)\\
 & +\partial_{\mu(y')w}-\frac{1}{2i\lambda}\omega(y,B)\sum_{\alpha,\beta}\omega(B_\mu(A,w)\widehat{e}_\alpha,\widehat{e}_\beta)\partial_{e_\alpha}\partial_{e_\beta} && (w\in\frakg_{(-1,0)})\\
 d\rho_{\min,\varepsilon}(\overline{B}) ={}& -\omega(y,B)\lambda\partial_\lambda+\omega(y,B)\partial_y+s_\min\omega(y,B)+2i\lambda n(y')\\
 d\rho_{\min,\varepsilon}(E) ={}& i\lambda\partial_\lambda^2-ia\partial_\lambda\partial_A-i\partial_\lambda\partial_{y'}-5is_\min\partial_\lambda-\frac{4s_\min+1}{i\lambda}a\partial_A+n(y')\partial_A\\
 & +\frac{2}{\lambda^2}an(\partial')+\frac{s_\min}{i\lambda}\dim\Lambda-\frac{3s_\min+1}{i\lambda}\partial_{y'}\\
 & +\frac{1}{2i\lambda}\sum_{\alpha,\beta}\omega(\mu(y')\widehat{e}_\alpha,\widehat{e}_\beta)\partial_{e_\alpha}\partial_{e_\beta},
\end{align*}
where $(e_\alpha)$ is a basis of $\calJ=\frakg_{(0,-1)}$, $(\widehat{e}_\alpha)$ the corresponding dual basis of $\frakg_{(-1,0)}$ with respect to the symplectic form, and
$$ n(\partial')=\frac{1}{2}\sum_{\alpha,\beta,\gamma}\omega(A,B_\Psi(\widehat{e}_\alpha,\widehat{e}_\beta,\widehat{e}_\gamma))\partial_{e_\alpha}\partial_{e_\beta}\partial_{e_\gamma}. $$\index{n2zz@$n(\partial')$}
Since $d\rho_{\min,\varepsilon}$ is independent of $\varepsilon\in\ZZ/2\ZZ$ and of $\zeta$, we abuse notation and write $d\rho_\min=d\rho_{\min,0}=d\rho_{\min,1}$ for the obvious extension of both representations to $\calD'(\RR^\times)\otimeshat\calS'(\Lambda)$.
\end{proposition}

\begin{proof}
The formulas for $\frakm$, $\fraka$ and $\overline{\frakn}$ follow by differentiating the formulas in Proposition~\ref{prop:FTActionPbar}. We next compute $d\rho_\min(\overline{B})$. Writing $z=aA+v+w+bB$ with $a,b\in\RR$ and $v\in\calJ$, $w\in\calJ^*$, we have, by Lemma~\ref{lem:DecompBigradingMuPsiQ}:
\begin{multline*}
 d\pi_{\zeta,\nu}(\overline{B}) = 2a^2\partial_A+2a\partial_v+\partial_{\mu(v)B}-(ab+\tfrac{1}{2}\omega(v,w)+t)\partial_B\\
 +(at+n(v)-\tfrac{1}{2}a\omega(v,w)-a^2b)\partial_t+(\nu+\rho)a.
\end{multline*}
A careful application of Lemma~\ref{lem:FTMultDiff} yields
\begin{align}
\begin{split}
 d\widehat{\pi}_{\zeta,\nu}(\overline{B}) ={}& -\omega(y,B)\lambda\partial_\lambda+(\omega(y,B)-\tfrac{1}{2}\omega(x,B))\omega(x,B)\partial_{A,x}+\omega(y-x,B)\partial_{x',x}\\
 &+\omega(x,B)\partial_{y',x}+\tfrac{1}{2}\omega(y,B)^2\partial_{A,y}+\omega(y,B)\partial_{y',y}\\
 &+4i\lambda n(x')+i\lambda\omega(\mu(x')y',B)+2i\lambda n(y')+\tfrac{\nu+\rho}{2}\omega(y-x,B),
\end{split}\label{eq:FTpiBbar}
\end{align}
where $\partial_{v,x}$ resp. $\partial_{w,y}$ means differentiation in the direction $v$ resp. $w$ with respect to the variable $x$ resp. $y$ of $\widehat{u}(\lambda,x,y)$, and $x\in\RR A+x'$, $y\in\RR A+y'$. In view of \eqref{eq:IntegralFormulaIntertwiner}, we compute with $x=aA+x'$:
\begin{align*}
 \lambda\partial_\lambda\xi_{-\lambda,\varepsilon}(x) &= \tfrac{i\lambda n(x')}{a}\xi_{-\lambda,\varepsilon}(a,x'),\\
 \omega(x,B)\partial_{A,x}\xi_{-\lambda,\varepsilon}(x) &= 2a\partial_A\xi_{-\lambda,\varepsilon}(a,x') = 2s_\min\xi_{-\lambda,\varepsilon}(a,x')-\tfrac{2i\lambda n(x')}{a}\xi_{-\lambda,\varepsilon}(a,x'),\\
 \partial_{x',x}\xi_{-\lambda,\varepsilon}(x) &= \tfrac{3i\lambda n(x')}{a}\xi_{-\lambda,\varepsilon}(a,x'),\\
 \partial_{y',x}\xi_{-\lambda,\varepsilon}(x) &= -\frac{i\lambda}{2a}\omega(\mu(x')y',B)\xi_{-\lambda,\varepsilon}(a,x').
\end{align*}
A short computation using $\frac{\nu+\rho}{2}=-s_\min$ then shows
\begin{multline*}
 \widehat{(d\pi_{\zeta,\nu}(\overline{B})u)}(\lambda,x,y) = \sum_{\varepsilon\in\ZZ/2\ZZ}\xi_{-\lambda,\varepsilon}(x)\Big[-\omega(y,B)\lambda\partial_\lambda+\tfrac{1}{2}\omega(y,B)^2\partial_A+\omega(y,B)\partial_{y'}\\
 +s_\min\omega(y,B)+2i\lambda n(y')\Big]u_\varepsilon(\lambda,y).
\end{multline*}
This proves the claimed formula for $d\rho_{\min,\varepsilon}(\overline{B})$. Since $\frakg$ is generated by $\frakm$, $\fraka$, $\overline{\frakn}$ and $\overline{B}$, the remaining formulas can be obtained as commutators. More precisely, by \eqref{eq:BarMInvariant}, \eqref{eq:CommutatorG1} and Lemma~\ref{lem:RewriteBmu} the following identities hold for $v\in\calJ$, $w\in\calJ^*$ and $T\in\frakg_{(1,-1)}$ (so that $Tw\in\calJ^*$):
\begin{equation*}
 [\overline{B},B_\mu(A,w)] = \overline{w}, \qquad [T,\overline{w}] = \overline{Tw}, \qquad [B_\mu(A,w),\overline{v}] = \tfrac{1}{2}\omega(v,w)\overline{A}, \qquad [\overline{B},\overline{A}]=2E.\qedhere
\end{equation*}
\end{proof}

\begin{remark}\label{rem:MotivationL2}
We give a formal heuristic argument, why the representation $\rho_\min$ should extend to a unitary representation on $L^2(\RR^\times\times\Lambda,|\lambda|^{\dim\Lambda-2s_\min}d\lambda)$. Note that several steps of the argument need a certain regularization to make sense. However, since we prove that $\rho_\min$ extends to a unitary representation by different means, we do not justify all steps. For simplicity we assume that $\zeta$ is the trivial representation.\\
By Proposition~\ref{prop:KnappSteinIntegralFormula}, the invariant Hermitian form on $I(\zeta,\nu)$ is given by a regularization of
\begin{align*}
 \langle u,\overline{\Delta^{\frac{\nu-\rho}{2}}*v}\rangle &= \int_{\RR^\times} \tr\left(\sigma_\lambda(\Delta^{\frac{\nu-\rho}{2}}*v)^*\sigma_\lambda(u)\right) |\lambda|^{\dim\Lambda}\,d\lambda\\
 &= \int_{\RR^\times} \tr\left(\sigma_\lambda(\Delta^{\frac{\nu-\rho}{2}})\sigma_\lambda(v)^*\sigma_\lambda(u)\right) |\lambda|^{\dim\Lambda}\,d\lambda,
\end{align*}
where $\Delta(z,t)=t^2-Q(z)$\index{1Deltazt@\par\indexspace$\Delta(z,t)$}. Assume that $u,v\in I(\zeta,\nu)^{\Omega_\mu(\frakm)}$ satisfy $\sigma_\lambda(u)\varphi=\langle\xi_{-\lambda,\varepsilon},\varphi\rangle u_\varepsilon(\lambda)$ and $\sigma_\lambda(v)^*\psi=\langle\overline{v_\varepsilon(\lambda)},\psi\rangle\overline{\xi_{-\lambda,\varepsilon}}$, and hence
$$ \sigma_\lambda(v)^*\sigma_\lambda(u)\varphi = \langle\xi_{-\lambda,\varepsilon},\varphi\rangle\langle u_\varepsilon(\lambda),\overline{v_\varepsilon(\lambda)}\rangle\overline{\xi_{-\lambda,\varepsilon}}, $$
so that
$$ \tr\left(\sigma_\lambda(\Delta^{\frac{\nu-\rho}{2}})\sigma_\lambda(v)^*\sigma_\lambda(u)\right) = \langle u_\varepsilon(\lambda),\overline{v_\varepsilon(\lambda)}\rangle\langle\sigma_\lambda(\Delta^{\frac{\nu-\rho}{2}})\overline{\xi_{-\lambda,\varepsilon}},\xi_{-\lambda,\varepsilon}\rangle. $$
Now, $\Delta(-z,-t)=\Delta(z,t)$ and $\sigma_\lambda(z,t)^\top=\sigma_{-\lambda}(-z,-t)$, so that
$$ \langle\sigma_\lambda(\Delta^{\frac{\nu-\rho}{2}})\overline{\xi_{-\lambda,\varepsilon}},\xi_{-\lambda,\varepsilon}\rangle = \langle\sigma_{-\lambda}(\Delta^{\frac{\nu-\rho}{2}})\xi_{-\lambda,\varepsilon},\overline{\xi_{-\lambda,\varepsilon}}\rangle. $$
A short computation shows that
$$ \sigma_\lambda(\Delta^{\frac{\nu-\rho}{2}})=|\lambda|^{-\nu}\delta_{|\lambda|^{\frac{1}{2}}}\circ\sigma_{\sgn\lambda}(\Delta^{\frac{\nu-\rho}{2}})\circ\delta_{|\lambda|^{-\frac{1}{2}}} \qquad \mbox{and} \qquad \xi_{\lambda,\varepsilon} = |\lambda|^{-\frac{s_\min}{2}}\delta_{|\lambda|^{\frac{1}{2}}}\xi_{\sgn\lambda,\varepsilon}, $$
where $\delta_s\varphi(x)=\varphi(sx)$\index{1deltas@$\delta_s$} ($s>0$), and hence
\begin{align*}
 \langle\sigma_{-\lambda}(\Delta^{\frac{\nu-\rho}{2}})\xi_{-\lambda,\varepsilon},\overline{\xi_{-\lambda,\varepsilon}}\rangle &= |\lambda|^{-\nu-s_\min}\langle\delta_{|\lambda|^{\frac{1}{2}}}\circ\sigma_{-\sgn\lambda}(\Delta^{\frac{\nu-\rho}{2}})\xi_{-\sgn\lambda,\varepsilon},\delta_{|\lambda|^{\frac{1}{2}}}\overline{\xi_{-\sgn\lambda,\varepsilon}}\rangle\\
 &= |\lambda|^{-\nu-s_\min-\frac{1}{2}\dim\Lambda}\langle\sigma_{-\sgn\lambda}(\Delta^{\frac{\nu-\rho}{2}})\xi_{-\sgn\lambda,\varepsilon},\overline{\xi_{-\sgn\lambda,\varepsilon}}\rangle.
\end{align*}
Since $\Delta(z,t)$ is $M_0$-invariant, the operator $\sigma_\lambda(\Delta^{\frac{\nu-\rho}{2}})$ is $\omega_{\met,\lambda}(M_0)$-invariant. The subspace of $M_0$-invariant vectors in $\omega_{\met,\lambda}$ is spanned $\xi_{\lambda,0}$ and $\xi_{\lambda,1}$, so $\sigma_\lambda(\Delta^{\frac{\nu-\rho}{2}})\xi_{\lambda,\varepsilon}$ is a linear combination of $\xi_{\lambda,0}$ and $\xi_{\lambda,1}$. This shows that
\begin{align*}
 \langle u,\overline{\Delta^{\frac{\nu-\rho}{2}}*v}\rangle &= \const\cdot\int_{\RR^\times} \langle u_\varepsilon(\lambda),\overline{v_\varepsilon(\lambda)}\rangle |\lambda|^{\frac{1}{2}\dim\Lambda-\nu-s_\min}\,d\lambda\\
 &= \const\cdot\int_{\RR^\times}\int_\Lambda u_\varepsilon(\lambda,x)\overline{v_\varepsilon(\lambda,x)} |\lambda|^{\dim\Lambda-2s_\min}\,dx\,d\lambda.
\end{align*}
This suggests that the representation $\rho_\min$ should be unitary with respect to the inner product on $L^2(\RR^\times\times\Lambda,|\lambda|^{\dim\Lambda-2s_\min}\,d\lambda\,dx)$.
\end{remark}

We renormalize $\rho_\min$ to obtain a unitary representation on $L^2(\RR^\times\times\Lambda)$. For $\delta\in\ZZ/2\ZZ$ let

\begin{equation}
	\Phi_\delta:\calD'(\RR^\times)\otimeshat\calS'(\Lambda)\to\calD'(\RR^\times)\otimeshat\calS'(\Lambda), \quad \Phi_\delta u(\lambda,x) = \sgn(\lambda)^\delta|\lambda|^{-s_\min}u(\lambda,\tfrac{x}{\lambda}),\label{eq:DefPhi}\index{1Phidelta@$\Phi_\delta$}
\end{equation}
then $\Phi_\delta$ restricts to an isometric isomorphism
$$ L^2(\RR^\times\times\Lambda,|\lambda|^{\dim\Lambda-2s_\min}\,d\lambda\,dy) \to L^2(\RR^\times\times\Lambda). $$
With $\delta\in\ZZ/2\ZZ$ to determined later, we define
$$ \pi_\min(g) := \Phi_\delta\circ\rho_\min(g)\circ\Phi_\delta^{-1},\index{1pimin@$\pi_\min$} $$
then
\begin{align*}
 \Phi_\delta\circ\partial_v\circ\Phi_\delta^{-1} &= \lambda\partial_v, &
 \Phi_\delta\circ\partial_\lambda\circ\Phi_\delta^{-1} &= (\partial_\lambda+\lambda^{-1}\partial_x+s\lambda^{-1}),\\
 \Phi_\delta\circ\omega(y,w)\circ\Phi_\delta^{-1} &= \lambda^{-1}\omega(x,w), & \Phi_\delta\circ\lambda\circ\Phi_\delta^{-1} &= \lambda.
\end{align*}
Using the coordinates $(\lambda,x)=(\lambda,aA+x')\in\RR^\times\times\Lambda$ with $a\in\RR$ and $x'\in\calJ$ we find:

\begin{proposition}\label{prop:dpimin}
The representation $d\pi_\min$\index{dpimin@$d\pi_\min$} of $\frakg$ on $\calD'(\RR^\times)\otimeshat\calS'(\Lambda)$ is given by
\begin{align*}
 d\pi_\min(F) ={}& i\lambda\\
 d\pi_\min(v) ={}& -\lambda\partial_v && (v\in\Lambda)\\
 d\pi_\min(w) ={}& -i\omega(x,w) && (w\in\Lambda^*)\\
 d\pi_\min(T) ={}& -\frac{i\lambda}{2}\sum_{\alpha,\beta}\omega(T\widehat{e}_\alpha,\widehat{e}_\beta)\partial_\alpha\partial_\beta-\frac{1}{2}\omega(TB,x')\partial_A && (T\in\frakg_{(1,-1)})\\
 d\pi_\min(T) ={}& -\partial_{Tx}-\frac{1}{2}\tr(T|_\Lambda) && (T\in\frakg_{(0,0)}\cap\frakm)\\
 d\pi_\min(T) ={}& -a\partial_{TA}-\frac{1}{2i\lambda}\omega(Tx',x') && (T\in\frakg_{(-1,1)})\\
 d\pi_\min(H) ={}& -\partial_x-2\lambda\partial_\lambda-\frac{\dim\Lambda+2}{2}\\
 d\pi_\min(\overline{A}) ={}& i\lambda\partial_\lambda\partial_A+i\partial_x\partial_A+i\frac{\dim\Lambda+2}{2}\partial_A-2\lambda n(\partial')\\
 d\pi_\min(\overline{v}) ={}& i\lambda\partial_\lambda\partial_v+i\partial_x\partial_v-2is_\min\partial_v+\frac{1}{2}\lambda^{-1}\omega(\mu(x')v,B)\partial_A\\
 & -\frac{1}{i}\sum_{\alpha,\beta}\omega(B_\mu(x',v)\widehat{e}_\alpha,\widehat{e}_\beta)\partial_{e_\alpha}\partial_{e_\beta} && (v\in\frakg_{(0,-1)})\\
 d\pi_\min(\overline{w}) ={}& -\omega(x,w)\partial_\lambda+s_\min\lambda^{-1}\omega(x,w)\\
 & +\lambda^{-1}\partial_{\mu(x')w}-\frac{1}{2i}\omega(x,B)\sum_{\alpha,\beta}\omega(B_\mu(A,w)\widehat{e}_\alpha,\widehat{e}_\beta)\partial_{e_\alpha}\partial_{e_\beta} && (w\in\frakg_{(-1,0)})\\
 d\pi_\min(\overline{B}) ={}& -\omega(x,B)\partial_\lambda+2i\lambda^{-2}n(x')\\
 d\pi_\min(E) ={}& i\lambda\partial_\lambda^2 + i\partial_\lambda\partial_x-3is_\min\partial_\lambda+\lambda^{-2}n(x')\partial_A+2an(\partial')\\
 & -\frac{is_\min}{3\lambda}(\dim\Lambda-1)+\frac{s_\min}{i\lambda}\partial_{x'} +\frac{1}{2i\lambda}\sum_{\alpha,\beta}\omega(\mu(x')\widehat{e}_\alpha,\widehat{e}_\beta)\partial_{e_\alpha}\partial_{e_\beta}.
\end{align*}\index{n2zz@$n(\partial')$}
\end{proposition}

\begin{proposition}\label{prop:AnnihilatorJosephIdeal}
	The annihilator of $d\pi_\min$ in $U(\frakg_\CC)$ is a completely prime ideal whose associated variety is equal to the closure of the minimal nilpotent coadjoint orbit $\calO_{\min,\CC}^*\subseteq\frakg_\CC^*$. In particular, for $\frakg_\CC$ not of type $A$, the annihilator is the Joseph ideal.
\end{proposition}

\begin{proof}
	For the first claim, the same argument as in \cite[Proof of Theorem 2.18]{HKM14} applies. The key point is that the (complex) dimension of the minimal nilpotent coadjoint orbit in $\frakg_\CC^*$ equals twice the (real) dimension of $\RR^\times\times\Lambda$ by Corollary~\ref{cor:DimMinOrbit}. The rest follows along the same lines as in \cite{HKM14}. For the second claim, we use the uniqueness result for the Joseph ideal \cite[Theorem 3.1]{GS04}.	
\end{proof}

\section{The case $\frakg=\sl(n,\RR)$}\label{sec:FTpictureMinRepSLn}

For $\frakg=\sl(n,\RR)$, the previous arguments cannot be applied in the same way since here $\frakm$ is not simple and the value $\nu\in\fraka_\CC^*$ for which $\Omega_\mu(\frakm')$ is conformally invariant is different for the two factors $\frakm'$ of $\frakm$. We discuss how to use the first order system $\Omega_\omega$ instead to obtain a small subrepresentation of $I(\zeta,\nu)$ for some $\nu$. It turns out that this subrepresentation has a Fourier transformed picture similar to the other cases.

The subalgebra $\frakm$ decomposes into the direct sum of two ideals
$$ \frakm = \frakm_0\oplus\frakm_1 $$
with $\frakm_0=\RR T_0$\index{T0@$T_0$} and $\frakm_1\simeq\sl(n-2,\RR)$. We can normalize $T_0$ such that it has eigenvalues $\pm1$ on $V$. Then, the eigenspaces $\Lambda=\ker(T_0-\id_V)$ and $\Lambda^*=\ker(T_0+\id_V)$ are dual Lagrangian subspaces with $V=\Lambda\oplus\Lambda^*$. Note that both $\Lambda$ and $\Lambda^*$ are invariant under $\frakm_1$. The following lemma can be verified using the explicit realization of $\frakg$ given in Appendix~\ref{app:SLn}:

\begin{lemma}\label{lem:SLnIdentities}
	\begin{enumerate}[(1)]
		\item\label{lem:SLnIdentities1} For $v\in\Lambda$ and $w\in\Lambda^*$ we have $\mu(v)=\mu(w)=0$.
		\item\label{lem:SLnIdentities3} $B_\mu(x,y)\equiv\frac{n}{4(n-2)}\omega(T_0x,y)T_0\mod\frakm_1$ for all $x,y\in V$.
	\end{enumerate}
\end{lemma}

For $r\in\CC$ let $\zeta_r$\index{1fzetar@$\zeta_r$} denote a character of $M$ for which $d\zeta_r(\frakm_1)=0$ and $d\zeta_r(T_0)=\frac{n-2}{2}+\frac{n-2}{n}r$. In particular, $V_\zeta=\CC$.

\begin{theorem}
	For any $r\in\CC$ and $\zeta=\zeta_r$, the system of differential operators $\Omega_\omega(v)$ ($v\in\Lambda$) is conformally invariant for $\pi_{\zeta,\nu}$ with $\nu+\rho=\frac{n}{2}+r$.
\end{theorem}

\begin{proof}
	Since $[\frakm,\Lambda]\subseteq\Lambda$, Theorem~\ref{thm:ConfInvOmegaOmega} implies that $\Omega_\omega(v)$ ($v\in\Lambda$) is conformally invariant if and only if
	$$ \frac{\nu+\rho}{2}\omega(x,v)+2d\zeta(B_\mu(x,v)) = 0 $$
	for all $v\in\Lambda$. This now follows from Lemma~\ref{lem:SLnIdentities} \eqref{lem:SLnIdentities3}.
\end{proof}

\begin{remark}
	$\Omega_\omega(v)u=0$ for all $v\in\Lambda$ implies $\Omega_\mu(T)u=0$ for all $T\in\frakm_1$. In fact, since $T\Lambda\subseteq\Lambda$ and $T\Lambda^*\subseteq\Lambda^*$ we have
	\begin{align*}
		\Omega_\mu(T) &= \sum_{e_\alpha\in\Lambda,e_\beta\in\Lambda^*}\omega(T\widehat{e}_\alpha,\widehat{e}_\beta)(X_\alpha X_\beta+X_\beta X_\alpha)\\
		&=  \sum_{e_\alpha\in\Lambda,e_\beta\in\Lambda^*}\omega(T\widehat{e}_\alpha,\widehat{e}_\beta)([X_\alpha,X_\beta]+2X_\beta X_\alpha)
	\end{align*}
	for a basis $(e_\alpha)$ of $V$ with $e_\alpha\in\Lambda\cup\Lambda^*$. If $\Omega_\omega(v)u=0$ for all $v\in\Lambda$, then $X_\alpha u=0$ for $e_\alpha\in\Lambda$. Further, $[X_\alpha,X_\beta]=\omega(e_\alpha,e_\beta)\partial_t$, so that $\sum\omega(T\widehat{e}_\alpha,\widehat{e}_\beta)[X_\alpha,X_\beta]=\tr(T|_\Lambda)\partial_t$, which vanishes for $T\in\frakm_1\simeq\sl(n-2,\RR)$.
\end{remark}

By \eqref{eq:FTofOmegaOmega}, the Fourier transform of $\Omega_\omega(v)$ is given by composition with $\sigma_\lambda(v)$. In terms of the distribution kernel $\widehat{u}(\lambda,x,y)$ of $\sigma_\lambda(u)$ this means
$$ \widehat{\Omega_\omega(v)u}(\lambda,x,y) = \sigma_{-\lambda}(v)_x\widehat{u}(\lambda,x,y). $$
This implies that, for every $u\in I(\zeta_r,\nu)^{\Omega_\omega(\Lambda)}$, the distribution $\widehat{u}(\lambda,x,y)$ is in the $x$-variable a distribution vector in $L^2(\Lambda)^{-\infty}=\calS'(\Lambda)$ which is invariant under $\sigma_{-\lambda}(v)=-\partial_v$ for all $v\in\Lambda$. These are obviously only the constant functions:

\begin{proposition}
	For every $\lambda\in\RR^\times$ the space $L^2(\Lambda)^{-\infty,\Lambda}=\calS'(\Lambda)^{\Lambda}$ of $\Lambda$-invariant distribution vectors in $\sigma_\lambda$ is one-dimensional and spanned by the constant function $\xi_\lambda$ given by
	$$ \xi_\lambda(x) = 1 \qquad (x\in\Lambda).\index{1oxilambda@$\xi_\lambda$} $$
\end{proposition}

It follows that, for $u\in I(\zeta_r,\nu)^{\Omega_\omega(\Lambda)}$, we can write
$$ \widehat{u}(\lambda,x,y) = \xi_{-\lambda}(x)u_0(\lambda,y) = u_0(\lambda,y)\index{u0lambday@$u_0(\lambda,y)$} $$
for some $u_0\in\calD'(\RR^\times)\otimeshat\calS'(\Lambda)$. Let $J_{\min,r}\subseteq\calD'(\RR^\times)\otimeshat\calS'(\Lambda)$\index{J1minr@$J_{\min,r}$} denote the image of the map
$$ I(\zeta_r,\nu)^{\Omega_\omega(\Lambda)} \to \calD'(\RR^\times)\otimeshat\calS'(\Lambda), \quad u\mapsto u_0, $$
and write $\rho_{\min,r}$\index{1rhominr@$\rho_{\min,r}$} for the representation of $G$ on $J_{\min,r}$ which makes this map $G$-equivariant.

\begin{proposition}\label{prop:LAactionMinRepSLn}
	The representation $d\rho_{\min,r}$\index{drhominr@$d\rho_{\min,r}$} of $\frakg$ on $J_{\min,r}\subseteq\calD'(\RR^\times)\otimeshat\calS'(\Lambda)$ is given in coordinates $(\lambda,y)\in\RR^\times\times\Lambda$ by
	\begin{align*}
	d\rho_{\min,r}(F) ={}& i\lambda,\\
	d\rho_{\min,r}(v) ={}& -\partial_v && (v\in\Lambda),\\
	d\rho_{\min,r}(w) ={}& -i\lambda\omega(y,w) && (w\in\Lambda^*),\\
	d\rho_{\min,r}(T) ={}& -\partial_{Ty}-\tfrac{n-2r}{2n}\tr(T|_\Lambda) && (T\in\frakm),\\
	d\rho_{\min,r}(H) ={}& \partial_y-2\lambda\partial_\lambda-\tfrac{n-2r}{2},\\
	d\rho_{\min,r}(\overline{v}) ={}& i(\partial_\lambda+\tfrac{n-2r-2}{2\lambda})\partial_v && (v\in\Lambda),\\
	d\rho_{\min,r}(\overline{w}) ={}& \omega(y,w)(\partial_y-\lambda\partial_\lambda) && (w\in\Lambda^*),\\
	d\rho_{\min,r}(E) ={}& i(\lambda\partial_\lambda^2-\partial_\lambda\partial_y-\tfrac{n-2r-2}{2\lambda}\partial_y+\tfrac{n-2r}{2}\partial_\lambda).
	\end{align*}
\end{proposition}

\begin{proof}
	We proceed as in the proof of Proposition~\ref{prop:drhomin}. The formulas for $\frakm$, $\fraka$ and $\overline{\frakn}$ follow from Proposition~\ref{prop:ActionFTpicture}, and for $w\in\Lambda^*$ we find, using Lemma~\ref{lem:FTMultDiff} and Lemma~\ref{lem:SLnIdentities}:
	\begin{multline*}
		d\widehat{\pi}_{\zeta,\nu}(\overline{w}) = \omega(x,w)\left[\partial_{y-x,x}-\frac{\nu+\rho}{2}+\frac{n}{2(n-2)}d\zeta(T_0)\right]\\
		+\omega(y,w)\left[\partial_{x,x}+\partial_{y,y}-\lambda\partial_\lambda+\frac{\nu+\rho}{2}-\frac{n}{2(n-2)}d\zeta(T_0)\right].
	\end{multline*}
	Since $\nu+\rho=\frac{n}{n-2}d\zeta_r(T_0)$ and $\partial_{v,x}\xi_\lambda(x)=0$ for all $v\in\Lambda$, it follows that for $u(\lambda,x,y)=\xi_{-\lambda}(x)u_0(\lambda,y)$:
	\begin{equation*}
		d\widehat{\pi}_{\zeta,\nu}(\overline{w})u(\lambda,x,y) = \xi_{-\lambda}(x)\cdot\omega(y,w)(\partial_y-\lambda\partial_\lambda)u_0(\lambda,y).
	\end{equation*}
	This shows the formula for $d\rho_{\min,r}(\overline{w})$. The formula for $d\rho_{\min,r}(\overline{v})$ is obtained by a similar computation, and for $d\rho_{\min,r}(E)$ we use that $[\overline{v},\overline{w}]=-\omega(v,w)E$.
\end{proof}

The change of coordinates $x=\lambda y$ finally yields a representation $d\pi_{\min,r}$\index{dpiminr@$d\pi_{\min,r}$} of $\frakg$ on $\calD'(\RR^\times)\otimeshat\calS'(\Lambda)$ given by

\begin{align*}
d\pi_{\min,r}(F) ={}& i\lambda,\\
d\pi_{\min,r}(v) ={}& -\lambda\partial_v && (v\in\Lambda),\\
d\pi_{\min,r}(w) ={}& -i\omega(x,w) && (w\in\Lambda^*),\\
d\pi_{\min,r}(T) ={}& -\partial_{Tx}-\tfrac{n-2r}{2n}\tr(T|_\Lambda) && (T\in\frakm),\\
d\pi_{\min,r}(H) ={}& -\partial_x-2\lambda\partial_\lambda-\tfrac{n-2r}{2},\\
d\pi_{\min,r}(\overline{v}) ={}& i(\lambda\partial_\lambda+\partial_x+\tfrac{n-2r}{2})\partial_v && (v\in\Lambda^*),\\
d\pi_{\min,r}(\overline{w}) ={}& -\omega(x,w)\partial_\lambda && (w\in\Lambda^*),\\
d\pi_{\min,r}(E) ={}& i(\lambda\partial_\lambda+\partial_x+\tfrac{n-2r}{2})\partial_\lambda.
\end{align*}

\begin{remark}
It can be shown that for $\zeta=\zeta_r$, $r\in\CC$, and $\nu+\rho=n-2$, the second order differential operator
$$ \Omega_\mu^\zeta(T_0) = \Omega_\mu(T_0)+\frac{n}{n-2}d\zeta(T_0)\partial_t $$
is conformally invariant for $\pi_{\zeta,\nu}$. For $n>3$, the single equation $\Omega_\mu^\zeta(T_0)u=0$ is not sufficient to give a small representation similar to the previous cases; only for $n=3$ this is the case, since here $\frakm=\frakm_0=\RR T_0$. In fact, for $n=3$ the same arguments as before identify $I(\zeta,\nu)^{\Omega_\mu^\zeta(T_0)}$ with a subspace of $\calD'(\RR^\times)\otimeshat\calS'(\Lambda)$ and the corresponding Lie algebra action agrees with the one obtained in Proposition~\ref{prop:LAactionMinRepSLn}. This is due to the fact that for $n=3$ the parameter families $(\nu,r)=(-\frac{1}{2}+r,r)$ and $(\nu,r)=(-1,r)$ are related by the Weyl group element
$$ w(\diag(H_1,H_2,H_3)) = \diag(H_1,H_3,H_2). $$
It is likely that the corresponding standard intertwining operator identifies the two subrepresentations $I(\zeta_r,-\frac{1}{2}+r)^{\Omega_\omega(\Lambda)}$ and $I(\zeta_r,-1)^{\Omega_\mu^\zeta(T_0)}$.
\end{remark}

\section{The case $\frakg=\so(p,q)$}\label{sec:FTpictureMinRepSOpq}

We also treat the case $\frakg=\so(p,q)$ separately since here $\frakm$ is not simple either and the values $\nu\in\fraka_\CC^*$ for which $\Omega_\mu(\frakm')$ is conformally invariant are different for the two factors $\frakm'$ of $\frakm$. Instead, we combine variations of the first order system $\Omega_\omega$ and the second order system $\Omega_\mu$ to the case of vector-valued principal series in order to obtain a subrepresentation of $I(\zeta,\nu)$ which has a Fourier transformed picture similar to the other cases.

For $\frakg=\so(p,q)$ the lack of simplicity of $\frakm$ stems from the fact that $\calJ=\frakg_{(0,-1)}$ is not a simple Jordan algebra but the sum of two simple Jordan algebras, the one-dimensional Jordan algebra which is of rank one and a $(p+q-6)$-dimensional Jordan algebra of rank two. Write $\calJ=\calJ_0\oplus\overline{\calJ}$\index{J3@$\calJ$} with $\calJ_0=\RR P$\index{J30@$\calJ_0$}\index{P1@$P$} and $\overline{\calJ}\simeq\RR^{p-3,q-3}\simeq\RR^{p-3}\times\RR^{q-3}$\index{J3@$\overline{\calJ}$}\index{Rpq@$\RR^{p-3,q-3}$} and similarly $\calJ^*=\calJ_0^*\oplus\overline{\calJ}^*$\index{J3@$\calJ^*$}\index{J30star@$\calJ^*_0$}\index{J3@$\overline{\calJ}^*$} with $\calJ_0^*=\RR Q$\index{Q@$Q$} such that $\omega(P,Q)=1$ and $\omega(\calJ_0,\overline{\calJ}^*)=0=\omega(\overline{\calJ},\calJ_0^*)$. We decompose $v\in\calJ$ into $v=v_0+\overline{v}$ with $v_0\in\calJ_0$ and $\overline{v}\in\overline{\calJ}$ and similar for $w\in\calJ^*$.

Note that we use the same letter $P$ for the element $P\in\calJ$ and the parabolic subgroup $P=MAN\subseteq G$. It should be clear from the context which object is meant.

The following statement is the analog of Lemma~\ref{lem:TraceOnG0-1}:

\begin{lemma}\label{lem:SOpqTrace}
For $v\in\calJ$ and $w\in\calJ^*$ we have
$$ \tr(B_\mu(A,w)\circ B_\mu(v,B)|_\calJ) = \frac{p+q-6}{2}\omega(v_0,w_0)+\omega(\overline{v},\overline{w}). $$
\end{lemma}

\begin{proof}
This is a straightforward computation using Appendix~\ref{app:SOpq}.
\end{proof}

According to the decomposition $\calJ=\calJ_0\oplus\overline{\calJ}$ the Lie algebra $\frakm$ splits into
$$ \frakm=\frakm_0\oplus\overline{\frakm} \qquad \mbox{with} \qquad \frakm_0\simeq\sl(2,\RR) \quad \mbox{and} \quad \overline{\frakm}\simeq\so(p-2,q-2).\index{m30@$\frakm_0$}\index{m3@$\overline{\frakm}$} $$
We collect a few more identities related to the bigrading, all which can be verified by direct computations using the explicit realization of $\frakg$ given in Appendix~\ref{app:SOpq}.

\begin{lemma}\label{lem:SOpqIdentities}
\begin{enumerate}[(1)]
	\item $\frakm_0\simeq\sl(2,\RR)$ is spanned by the $\sl(2)$-triple
	$$ (e,f,h)=(\sqrt{2}B_\mu(A,Q),\sqrt{2}B_\mu(P,B),B_\mu(A,B)-2B_\mu(P,Q)).\index{e@$e$}\index{f@$f$}\index{h3@$h$} $$
	\item $B_\mu(A,Q):\overline{\calJ}\to\overline{\calJ}^*$ and $B_\mu(P,B):\overline{\calJ}^*\to\overline{\calJ}$ are isomorphisms satisfying
	$$ B_\mu(A,Q)\circ B_\mu(P,B)|_{\overline{\calJ}^*}=\frac{1}{2}\id_{\overline{\calJ}^*} \qquad \mbox{and} \qquad B_\mu(P,B)\circ B_\mu(A,Q)|_{\overline{\calJ}}=\frac{1}{2}\id_{\overline{\calJ}}. $$
	\item $B_\mu(v,B)\in\overline{\frakm}$ if and only if $v\in\overline{\calJ}$.
	\item $\mu(P)=\mu(Q)=0$ and $\mu(\calJ)\subseteq\RR B_\mu(P,B)$, $\mu(\overline{\calJ})\subseteq\RR B_\mu(A,Q)$
	\item For $v\in\calJ$:
	\begin{align*}
		\mu(v)Q &= -\omega(v_0,Q)\overline{v},\\
		\mu(v)w &= -\omega(\overline{v},w)\omega(v,Q)P+\omega(\mu(v)P,B)B_\mu(A,Q)w && (w\in\overline{\calJ}^*),\\
		\mu(v)B &= -\omega(\mu(\overline{v})P,B)Q+2\omega(v_0,Q)B_\mu(P,B)\overline{v}.
	\end{align*}
	\item $B_\mu(P,Q)$ acts on $V$ as follows:
	\begin{align*}
	 B_\mu(P,Q)A &= \frac{1}{4}A, & B_\mu(P,Q)P &= \frac{3}{4}P, & B_\mu(P,Q)v &= -\frac{1}{4}v &&(v\in\overline{\calJ}),\\
	 B_\mu(P,Q)B &= -\frac{1}{4}B, & B_\mu(P,Q)Q &= -\frac{3}{4}Q, & B_\mu(P,Q)w &= \frac{1}{4}w &&(v\in\overline{\calJ}^*).
	\end{align*}
\end{enumerate}
\end{lemma}

Let $\zeta$ be a representation of $M$ such that $d\zeta$ is trivial on $\overline{\frakm}$.

\begin{proposition}
	The system of differential operators
	$$ \Omega_\omega^\zeta(v) = \sum_\alpha \Big((\nu+\rho-2)\omega(v,\widehat{e}_\alpha)+4d\zeta(B_\mu(v,\widehat{e}_\alpha))\Big)\Omega_\omega(e_\alpha)\index{1ZOmega1omegazetav@$\Omega_\omega^\zeta(v)$} $$
	is conformally invariant for $\pi_{\zeta,\nu}$ if and only if $d\zeta(\Cas_{\frakm_0})=(\nu+\rho)(\nu+\rho-2)$, where $\Cas_{\frakm_0}=h^2+2ef+2fe$\index{Casm0@$\Cas_{\frakm_0}$}. In this case, the joint kernel
	$$ I(\zeta,\nu)^{\Omega_\omega^\zeta(\Lambda)} = \{u\in I(\zeta,\nu):\Omega_\omega^\zeta(v)u=0\mbox{ for all }v\in\Lambda\}\index{IzetanuOmegaomegazetaLambda@$I(\zeta,\nu)^{\Omega_\omega^\zeta(\Lambda)}$} $$
	is a subrepresentation of $I(\zeta,\nu)$.
\end{proposition}

\begin{proof}
	Using Theorem~\ref{thm:ConfInvOmegaMu}, we find
	\begin{align*}
		[\Omega_\omega^\zeta(v),d\pi_{\zeta,\nu}(X)] &= 0 && (X\in\overline{\frakn}),\\
		[\Omega_\omega^\zeta(v),d\pi_{\zeta,\nu}(H)] &= \Omega_\omega^\zeta(v),\\
		[\Omega_\omega^\zeta(v),d\pi_{\zeta,\nu}(S)] &= -\Omega_\omega^\zeta(Sv) && (S\in\frakm).
	\end{align*}
	Moreover, we have
	\begin{align*}
		& [\Omega_\omega^\zeta(v),d\pi_{\zeta,\nu}(E)]\\
		&= \sum_\alpha\Big((\nu+\rho-2)\omega(v,\widehat{e}_\alpha)+4d\zeta(B_\mu(v,\widehat{e}_\alpha))\Big)[\Omega_\omega(e_\alpha),d\pi_{\zeta,\nu}(E)]\\
		&\hspace{9cm}+ 4\sum_\alpha d\zeta([B_\mu(v,\widehat{e}_\alpha),\mu(x)])\Omega_\omega(e_\alpha)\\
		&= \sum_\alpha\Big((\nu+\rho-2)\omega(v,\widehat{e}_\alpha)+4d\zeta(B_\mu(v,\widehat{e}_\alpha))\Big)\Big(t\Omega_\omega(e_\alpha)-\Omega_\omega(\mu(x)e_\alpha)\\
		&\hspace{9.5cm}+\tfrac{\nu+\rho}{2}\omega(x,e_\alpha)+2d\zeta(B_\mu(x,e_\alpha))\Big)\\
		&\qquad\qquad- 4\sum_\alpha\Big(d\zeta(B_\mu(\mu(x)v,\widehat{e}_\alpha)+d\zeta(B_\mu(v,\mu(x)\widehat{e}_\alpha)])\Big)\Omega_\omega(e_\alpha)\\
		&= t\Omega_\omega^\zeta(v)-\Omega_\omega^\zeta(\mu(x)v)+\tfrac{(\nu+\rho)(\nu+\rho-2)}{2}\omega(x,v)-4d\zeta(B_\mu(x,v))\\
		&\hspace{9.2cm}+8\sum_\alpha d\zeta(B_\mu(v,\widehat{e}_\alpha))d\zeta(B_\mu(x,e_\alpha)).
	\end{align*}
	The last term is evaluated in the next lemma and the claim follows.
\end{proof}

\begin{lemma}
	For $v,w\in V$ we have
	$$ \sum_\alpha B_\mu(v,\widehat{e}_\alpha)B_\mu(e_\alpha,w) \equiv \frac{1}{16}\omega(v,w)(h^2+2ef+2fe)+\frac{1}{2}B_\mu(v,w) \mod \overline{\frakm}\,\calU(\frakm_0) $$
\end{lemma}

\begin{proof}
	We first note that
	$$ B_\mu(v,w) \equiv \frac{1}{4}\omega(hv,w)h+\frac{1}{2}\omega(fv,w)e+\frac{1}{2}\omega(ev,w)f \mod\overline{\frakm}. $$
	This can be shown by applying both sides to $A$, $B$, $P$ and $Q$ and pairing with another element in this list with respect to the symplectic form. Plugging this into the sum and using the following identities on $V$ (which is a direct sum of $p+q-4$ copies of the standard representation $\RR^2$ of $\frakm_0\simeq\sl(2,\RR)$)
	\begin{equation}\label{eq:SL2FormulasSOpq}
	\begin{aligned}
		\ad(h)^2 &= 1, & \ad(e)^2 &= \ad(f)^2=0,\\
		\ad(e)\ad(f) &= \frac{1}{2}(1+\ad(h)), & \ad(f)\ad(e) &= \frac{1}{2}(1-\ad(h)),\\
		\ad(h)\ad(e) &= -\ad(e)\ad(h)=\ad(e), & \ad(h)\ad(f) &= -\ad(f)\ad(h)=-\ad(f),
	\end{aligned}
	\end{equation}
	shows the desired formula.
\end{proof}

Now let $G$ be a connected Lie group with Lie algebra $\frakg=\so(p,q)$ such that the analytic subgroup $\langle\exp\frakm_0\rangle$ of $M$ corresponding to $\frakm_0\simeq\sl(2,\RR)$ is the non-trivial double cover of $SL(2,\RR)$. For $k\in\ZZ/4\ZZ$ and $s\in\CC$ there exists a principal series representation $(\zeta_{k,s},V_{k,s})$\index{1fzetaks@$\zeta_{k,s}$}\index{Vks@$V_{k,s}$} of $\langle\exp\frakm_0\rangle$ with $K$-types $v_n$, $n\in2\ZZ+\frac{k}{2}$, on which the basis
$$ \kappa = f-e, \qquad x_\pm = h\mp i(e+f)\index{1kappa@$\kappa$}\index{xpm@$x_\pm$} $$
of $\frakm_0$ acts by
\begin{equation}
	d\zeta_{k,s}(\kappa)v_n = inv_n, \qquad d\zeta_{k,s}(x_\pm)v_n = (s\pm n+1)v_{n\pm2}.\label{eq:ActionSL2ZetaSOpq}
\end{equation}
Note that for $k=0,2$ these representations factor through $SL(2,\RR)$ and become the usual even and odd principal series for $\SL(2,\RR)$, while for $k=1,3$ the representations are genuine. We also write $(\zeta,V_\zeta)=(\zeta_{k,s},V_{k,s})$ for short and denote by $\zeta$ any extension to $M$ which is trivial on $\overline{\frakm}$. Note that for $\zeta=\zeta_{k,s}$, we have $d\zeta(\Cas_{\frakm_0})=s^2-1$, so that $\Omega_\omega^\zeta$ is conformally invariant for $\pi_{\zeta,\nu}$ if and only if $s=\pm(\nu+\rho-1)$. We therefore let $s=-(\nu+\rho-1)$.

The Fourier transform of the equation $\Omega_\omega^\zeta(v)u=0$ is by \eqref{eq:FTofOmegaOmega}:
$$ 0 = \widehat{\Omega_\omega^\zeta(v)u}(\lambda,x,y) = \Big((\nu+\rho-2)d\sigma_{-\lambda}(v)+4\sum_\alpha d\zeta(B_\mu(v,\widehat{e}_\alpha))d\sigma_{-\lambda}(e_\alpha)\Big)_x\widehat{u}(\lambda,x,y) $$
This motivates the following:

\begin{proposition}\label{prop:InvDistVectSOpq1}
	For every $\lambda\in\RR^\times$, the space of all $\xi\in\calS'(\Lambda)\otimes V_\zeta$ satisfying
	$$ \Big((\nu+\rho-2)d\sigma_\lambda(v)+4\sum_\alpha d\zeta(B_\mu(v,\widehat{e}_\alpha))d\sigma_\lambda(e_\alpha)\Big)\xi=0 $$
	consists of all distributions of the form $\xi(a,z)=\xi_0(a,p)e^{-i\lambda\frac{n(z)}{a}}$ ($a\in\RR$, $z\in\calJ$, $p=\omega(z,Q)$) with
	$$ \xi_0=\sum_{n\equiv\frac{k}{2}\mod2}\xi_{0,n}\otimes v_n\in\calS'(\RR^2)\otimes V_\zeta $$
	and each $\xi_{0,n}\in\calS'(\RR^2)$ homogeneous of degree $-1$ satisfying the recurrence relation
	$$ (s+n+1)(p-i\sqrt{2}a)\xi_{0,n} = (s-n-1)(p+i\sqrt{a})\xi_{0,n+2}. $$
\end{proposition}

\begin{proof}
	A short computation shows that the above equation is equivalent to
	\begin{align}
		& ((s+1)-d\zeta(h))\partial_A\xi-2\sqrt{2}d\zeta(e)\partial_P\xi = 0,\label{eq:InvDistVectSOpq1-1}\\
		& ((s+1)+d\zeta(h))a\xi+\sqrt{2}d\zeta(f)p\xi = 0,\label{eq:InvDistVectSOpq1-2}\\
		& ((s+1)+d\zeta(h))\partial_P\xi-\sqrt{2}d\zeta(f)\partial_A\xi = 0,\label{eq:InvDistVectSOpq1-3}\\
		& ((s+1)-d\zeta(h))p\xi+2\sqrt{2}d\zeta(e)a\xi = 0,\label{eq:InvDistVectSOpq1-4}
	\end{align}
	and for $v\in\overline{\calJ}$
	\begin{align}
		& ((s+1)-d\zeta(h))\partial_v\xi+2\sqrt{2}i\lambda d\zeta(e)\omega(B_\mu(x,v)P,B)\xi = 0,\label{eq:InvDistVectSOpq1-5}\\
		& i\lambda((s+1)+d\zeta(h))\omega(B_\mu(x,v)P,B)\xi+\sqrt{2}d\zeta(f)\partial_v\xi = 0.\label{eq:InvDistVectSOpq1-6}
	\end{align}
	Combining \eqref{eq:InvDistVectSOpq1-4} and \eqref{eq:InvDistVectSOpq1-5} resp. \eqref{eq:InvDistVectSOpq1-2} and \eqref{eq:InvDistVectSOpq1-6} we find
	$$ d\zeta(e)\big(a\partial_v-i\lambda p\omega(B_\mu(x,v)P,B)\big)\xi = 0 = d\zeta(f)\big(a\partial_v-i\lambda p\omega(B_\mu(x,v)P,B)\big)\xi, $$
	hence $(a\partial_v-i\lambda p\omega(B_\mu(x,v)P,B))\xi=0$. Note that
	$$ \partial_v n(x) = -\frac{1}{2}\omega(\mu(x)v,B) = -p\omega(B_\mu(x,v)P,B), $$
	so that this equation is equivalent to $\partial_v(\xi\cdot e^{i\lambda\frac{n(x)}{a}})=0$. It follows that $\xi(a,x)=\xi_0(a,p)e^{-i\lambda\frac{n(x)}{a}}$. Combining \eqref{eq:InvDistVectSOpq1-1} and \eqref{eq:InvDistVectSOpq1-4} resp. \eqref{eq:InvDistVectSOpq1-2} and \eqref{eq:InvDistVectSOpq1-3} yields
	$$ d\zeta(e)\big(a\partial_A+p\partial_P+1\big)\xi_0 = 0 = d\zeta(f)\big(a\partial_A+p\partial_P+1\big)\xi_0, $$
	hence $(a\partial_A+p\partial_P+1)\xi_0=0$ and $\xi_0$ is homogeneous of degree $-1$. Write $\xi_0=\sum_n\xi_{0,n}\otimes v_n$, then using \eqref{eq:ActionSL2ZetaSOpq}, we find that \eqref{eq:InvDistVectSOpq1-1} and \eqref{eq:InvDistVectSOpq1-3} are equivalent to
	$$ (s+n+1)(\partial_A+i\sqrt{2}\partial_P)\xi_{0,n} = (s-n-1)(\partial_A-i\sqrt{2}\partial_P)\xi_{0,n+2}, $$
	and that \eqref{eq:InvDistVectSOpq1-2} and \eqref{eq:InvDistVectSOpq1-4} are equivalent to
	$$ (s+n+1)(p-i\sqrt{2}a)\xi_{0,n} = (s-n-1)(p+i\sqrt{2}a)\xi_{0,n+2}. $$
	It is easy to see that the latter identity implies the first one whenever $\xi_{0,n}$ and $\xi_{0,n+2}$ are homogeneous of degree $-1$.
\end{proof}

In Proposition~\ref{prop:InvDistVectSOpq1}, the space of invariant distribution vectors $\xi$ is still infinite-dimensional. This indicates that the kernel of the system $\Omega_\omega^\zeta(v)$, $v\in\Lambda$, is not small enough to yield a representation in the same way as in Section~\ref{sec:FTpictureMinRep}. We therefore also construct a vector-valued version of the second order system $\Omega_\mu$:

\begin{proposition}
	For $\nu=-\frac{p+q-2}{2}$ the system of differential operators
	$$ \Omega_\mu^\zeta(T) = \Omega_\mu(T) + 2d\zeta(T)\partial_t \qquad (T\in\frakm)\index{1ZOmega2muzetaT@$\Omega_\mu^\zeta(T)$} $$
	is conformally invariant on the kernel of the system $\Omega_\omega^\zeta$, i.e. the joint kernel
	$$ I(\zeta,\nu)^{\Omega_\omega^\zeta(\Lambda),\Omega_\mu^\zeta(\frakm)} = \{u\in I(\zeta,\nu)^{\Omega_\omega^\zeta(\Lambda)}:\Omega_\mu^\zeta(T)u=0\mbox{ for all }T\in\frakm\}\index{IzetanuOmegaomegazetaLambdaOmegamuzetam@$I(\zeta,\nu)^{\Omega_\omega^\zeta(\Lambda),\Omega_\mu^\zeta(\frakm)}$} $$
	is a subrepresentation of $I(\zeta,\nu)$.
\end{proposition}

\begin{proof}
	Using Theorem~\ref{thm:ConfInvOmegaMu} we find that
	\begin{align*}
		[\Omega_\mu^\zeta(T),d\pi_{\zeta,\nu}(\overline{\frakn})] &= 0,\\
		[\Omega_\mu^\zeta(T),d\pi_{\zeta,\nu}(H)] &= 2\Omega_\mu^\zeta(T),\\
		[\Omega_\mu^\zeta(T),d\pi_{\zeta,\nu}(S)] &= \Omega_\mu^\zeta([T,S]) && (S\in\frakm).
	\end{align*}
	We further show that $[\Omega_\mu^\zeta(T),d\pi_{\zeta,\nu}(E)]$ can be expressed as a $C^\infty(\overline{\frakn})$-linear combination of operators in $\Omega_\mu^\zeta(\frakm)$ and $\Omega_\omega^\zeta(\Lambda)$. First note that
	$$ [d\zeta(T)\partial_t,d\pi_{\zeta,\nu}(E)] = d\zeta(T)(\partial_x+2t\partial_t)+(\nu+\rho)d\zeta(T)+d\zeta([T,\mu(x)])\partial_t. $$
	Together with the formula for $[\Omega_\mu(T),d\pi_{\zeta,\nu}(E)]$ in Theorem~\ref{thm:ConfInvOmegaMu}, this yields
	\begin{multline}
		[\Omega_\mu^\zeta(T),d\pi_{\zeta,\nu}(E)] = 2t\Omega_\mu^\zeta(T)+\Omega_\mu^\zeta([T,\mu(x)])+(2\,\calC(\frakm')-2-(\nu+\rho))\Omega_\omega(Tx)\\
		+4\sum_\alpha d\zeta(B_\mu(x,T\widehat{e}_\alpha))\Omega_\omega(e_\alpha)+2d\zeta(T)\partial_x+2(\nu+\rho-\calC(\frakm'))d\zeta(T).\label{eq:SOpqCommutatorOmegaMuAndE}
	\end{multline}
	By Table~\ref{tab:Cvalues} in Appendix~\ref{app:Tables}, $\calC(\frakm_0)=\frac{p+q-4}{2}$ and $\calC(\overline{\frakm})=2$ and $\nu+\rho=\frac{p+q-4}{2}$. Let us first assume that $\frakm'=\overline{\frakm}$, then $d\zeta(T)=0$. Further, since $[T,B_\mu(x,\widehat{e}_\alpha)]\in\overline{\frakm}$, we have $d\zeta([T,B_\mu(x,\widehat{e}_\alpha)])=0$, and hence
	\begin{multline*}
		(2\,\calC(\frakm')-2-(\nu+\rho))\Omega_\omega(Tx) + 4\sum_\alpha d\zeta(B_\mu(x,T\widehat{e}_\alpha))\Omega_\omega(e_\alpha)\\
		= -(\nu+\rho-2)\Omega_\omega(Tx)-4\sum_\alpha d\zeta(B_\mu(Tx,\widehat{e}_\alpha))\Omega_\omega(e_\alpha) = -\Omega_\omega^\zeta(Tx).
	\end{multline*}
	Now let $\frakm'=\frakm_0$, then $\nu+\rho-\calC(\frakm')=0$, so the last term in \eqref{eq:SOpqCommutatorOmegaMuAndE} vanishes. Using $\Omega_\omega^\zeta(Tx)=0$, the first two terms combine to
	$$ 4\sum_\alpha d\zeta\big(B_\mu(x,T\widehat{e}_\alpha)-B_\mu(Tx,\widehat{e}_\alpha)\big)\Omega_\omega(e_\alpha), $$
	which, by the following lemma, equals
	\begin{equation*}
		-2d\zeta(T)\Omega_\omega(x) = -2d\zeta(T)\partial_x.\qedhere
	\end{equation*}
\end{proof}

\begin{lemma}
	For $T\in\frakm_0$ and $x,y\in V$ we have
	$$ B_\mu(Tx,y)-B_\mu(x,Ty) \equiv \frac{1}{2}\omega(x,y)T \mod\overline{\frakm}. $$
\end{lemma}

\begin{proof}
	Modulo $\overline{\frakm}$ we have
	\begin{align*}
		B_\mu(Tx,y)-B_\mu(x,Ty) \equiv{}& \frac{1}{4}\big(\omega(hTx,y)-\omega(hx,Ty)\big)h + \frac{1}{2}\big(\omega(fTx,y)-\omega(fx,Ty)\big)e\\
		& \hspace{5cm}+ \frac{1}{2}\big(\omega(eTx,y)-\omega(ex,Ty)\big)f\\
		={}& \frac{1}{4}\omega((hT+Th)x,y)h + \frac{1}{2}\omega((fT+Tf)x,y)e + \frac{1}{2}\omega((eT+Te)x,y)f.
	\end{align*}
	Using \eqref{eq:SL2FormulasSOpq}, one shows that for $T=T_hh+T_ee+T_ff$ we have
	$$ (hT+Th)x=2T_hx, \qquad (eT+Te)x=T_fx, \qquad (fT+Tf)x=T_ex, $$
	so the result follows.
\end{proof}

We finally fix $\nu=-\frac{p+q-2}{2}$, $k=p+q$ and $s=-(\nu+\rho-1)=-\frac{p+q-6}{2}$. Note that $\zeta=\zeta_{k,s}$ is reducible and has a unique irreducible subrepresentation. This subrepresentation is finite-dimensional for $p+q$ even and spanned by
$$ v_{\frac{p+q-8}{2}},v_{\frac{p+q-8}{2}-2},\ldots,v_{-\frac{p+q-8}{2}}, $$
and it is infinite-dimensional for $p+q$ odd and spanned by
$$ v_{\frac{p+q-8}{2}},v_{\frac{p+q-8}{2}-2},\ldots $$

By Theorem~\ref{thm:FTofOmegaMu} and Lemma~\ref{lem:FTMultDiff}, the Fourier transform of $\Omega_\mu^\zeta(T)$ takes the form
$$ \widehat{\Omega_\mu^\zeta(T)u}(\lambda,x,y) = -2i\lambda\left( d\omega_{\met,-\lambda}(T)_x+d\zeta(T)\right)\widehat{u}(\lambda,x,y). $$
We therefore study $\frakm$-invariant distribution vectors in $L^2(\Lambda)^{-\infty}\otimes V_\zeta=\calS'(\Lambda)\otimes V_\zeta$.

\begin{proposition}\label{prop:InvDistVectSOpq2}
For every $\lambda\in\RR^\times$, the space $(\calS'(\Lambda)\otimes V_\zeta)^{\frakm}$ of $\frakm$-invariant distribution vectors in $\omega_{\met,-\lambda}\otimes\zeta$ is two-dimensional and spanned by the distributions ($a\in\RR$, $z\in\calJ$, $p=\omega(z,Q)$)
$$ \xi_{\lambda,\varepsilon}(a,z)=\xi_{0,\varepsilon}(a,p)e^{-i\lambda\frac{n(z)}{a}} \qquad \mbox{with} \qquad \xi_{0,\varepsilon} = \sum_{n\equiv\frac{k}{2}\mod2}\xi_{0,\varepsilon,n}\otimes v_n\index{1oxilambdaepsilon@$\xi_{\lambda,\varepsilon}$} $$
and
$$ \xi_{0,\varepsilon,n}(a,p) = c_n\sgn(a)^\varepsilon|a|^{-\frac{p+q-6}{2}}(\sqrt{2}|a|+i\sgn(a)p)^n(2a^2+p^2)^{\frac{p+q-2n-8}{4}} $$
with $(c_n)_n$ satisfying
$$ (n+2+\tfrac{p+q-8}{2})c_{n+2} = (n-\tfrac{p+q-8}{2})c_n. $$
\end{proposition}

\begin{proof}
We first study invariance under $T\in\overline{\frakm}$ in the same way as in Theorem~\ref{thm:InvDistributionVector}. For $v\in\calJ$ we have
\begin{equation}
	d\omega_{\met,\lambda}(B_\mu(v,B))\xi = -a\partial_v(\xi\cdot e^{i\lambda\frac{n(z)}{a}})\cdot e^{-i\lambda\frac{n(z)}{a}},\label{eq:InvDistribVectSOpq1}
\end{equation}
so that invariance under $B_\mu(\overline{\calJ},B)\subseteq\overline{\frakm}$ implies $\xi(a,z)=\xi_0(a,p)e^{-i\lambda\frac{n(z)}{a}}$ with $\xi_0\in\calS'(\RR^2)\otimes V_\zeta$ and $p=\omega(z,Q)$. For $w\in\calJ^*$ we further have
$$ d\omega_{\met,\lambda}(B_\mu(A,w))\xi = \frac{1}{2i\lambda}\sum_{\alpha,\beta}\omega(B_\mu(A,w)\widehat{e}_\alpha,\widehat{e}_\beta)\partial_\alpha\partial_\beta\xi-\frac{1}{2}\omega(z,w)\partial_A\xi. $$
For $\xi$ as above we find
\begin{align*}
 \partial_A\xi ={}& \partial_A\xi_0\cdot e^{-i\lambda\frac{n(z)}{a}}+\frac{i\lambda n(z)}{a^2}\xi,\\
 \partial_\alpha\xi ={}& \omega(e_\alpha,Q)\partial_P\xi_0\cdot e^{-i\lambda\frac{n(z)}{a}}+\frac{i\lambda}{2a}\omega(\mu(z)e_\alpha,B)\xi,\\
 \partial_\alpha\partial_\beta\xi ={}& \omega(e_\alpha,Q)\omega(e_\beta,Q)\partial_P^2\xi_0\cdot e^{-i\lambda\frac{n(z)}{a}}+\frac{i\lambda}{a}\omega(e_\alpha,Q)\omega(\mu(z)e_\beta,B)\partial_P\xi_0\cdot e^{-i\lambda\frac{n(z)}{a}},\\
 & -\frac{\lambda^2}{4a^2}\omega(\mu(z)e_\alpha,B)\omega(\mu(z)e_\beta,B)\xi + \frac{i\lambda}{a}\omega(B_\mu(z,e_\alpha)e_\beta,B)\xi.
\end{align*}
Combined with Lemma~\ref{lem:SOpqTrace} and \ref{lem:SOpqIdentities}, this gives
\begin{multline}
 d\omega_{\met,\lambda}(B_\mu(A,w))\xi = -\frac{\omega(z_0,w_0)}{2a}\Big[a\partial_A\xi_0+\frac{p+q-6}{2}\xi_0\Big]e^{-i\lambda\frac{n(z)}{a}}\\
 -\frac{\omega(z_1,w_1)}{2a}\Big[a\partial_A\xi_0+p\partial_P\xi_0+\xi_0\Big]e^{-i\lambda\frac{n(z)}{a}},\label{eq:InvDistribVectSOpq2}
\end{multline}
so that invariance under $B_\mu(A,\overline{\calJ}^*)$ implies that $\xi_0$ is homogeneous of degree $-1$. Next, we consider the action of $\frakm_0$ on $\xi$ of this form. By \eqref{eq:InvDistribVectSOpq1} and \eqref{eq:InvDistribVectSOpq2}, a distribution $\xi(a,z)=\xi(a,p)e^{-i\lambda\frac{n(z)}{a}}$ is invariant under $\frakm_0$ if and only if
$$ -\frac{p}{a}\left(a\partial_A+\frac{p+q-6}{2}\right)\xi_0+\sqrt{2}d\zeta(e)\xi_0 = 0 \qquad \mbox{and} \qquad -2a\partial_P\xi_0+\sqrt{2}d\zeta(f)\xi_0 = 0. $$
Writing $\xi_0=\sum_n\xi_{0,n}\otimes v_n$ and using \eqref{eq:ActionSL2ZetaSOpq} shows that this is equivalent to
\begin{align*}
	\left(2a\partial_P-p\partial_A-\frac{p+q-6}{2}\frac{p}{a}\right)\xi_{0,n} &= in\sqrt{2}\xi_{0,n},\\
	\left(2a\partial_P+p\partial_A+\frac{p+q-6}{2}\frac{p}{a}\right)\xi_{0,n} &= i\frac{\sqrt{2}}{2}\big((s+n-1)\xi_{0,n-2}-(s-n-1)\xi_{0,n+2}\big).
\end{align*}
The first equation has the solutions
$$ \xi_{0,n}(a,p)=c_n\sgn(a)^\varepsilon|a|^{-\frac{p+q-6}{2}}(\sqrt{2}|a|+i\sgn(a)p)^n(2a^2+p^2)^{\frac{p+q-2n-8}{4}} \qquad (\varepsilon\in\ZZ/2\ZZ), $$
and for this choice of $\xi_{0,n}$ the second equation is equivalent to
$$ (n-s+1)c_{n+2} = (n+s+1)c_n. $$
It follows that $c_n=0$ for $n>-(s+1)=\frac{p+q-8}{2}$, and for $n\leq\frac{p+q-8}{2}$ the sequence $(c_n)$ is uniquely determined by $c_{\frac{p+q-8}{2}}$. The result follows.
\end{proof}

\begin{remark}
	Comparing the invariant distribution vectors in Proposition~\ref{prop:InvDistVectSOpq1} and Proposition~\ref{prop:InvDistVectSOpq2} suggests that $\Omega_\mu^\zeta(\frakm)u=0$ implies $\Omega_\omega^\zeta(V)u=0$. However, we were not able to show this only using the differential operators $\Omega_\mu^\zeta(T)$ and $\Omega_\omega^\zeta(v)$.
\end{remark}

By the same arguments as in the other cases, $u\in I(\zeta,\nu)^{\Omega_\omega^\zeta(V),\Omega_\mu^\zeta(\frakm)}$ implies
$$ \widehat{u}(\lambda,x,y) = \xi_{-\lambda,0}(x)u_0(\lambda,y)+\xi_{-\lambda,1}(x)u_1(\lambda,y)\index{u0lambday@$u_0(\lambda,y)$}\index{u1lambday@$u_1(\lambda,y)$} $$
and we obtain a representation $\rho_\min=(\rho_{\min,0},\rho_{\min,1})$\index{1rhomin@$\rho_\min$} of $G$ on a subspace $J_\min\subseteq(\calD'(\RR^\times)\otimeshat\calS'(\Lambda))\oplus(\calD'(\RR^\times)\otimeshat\calS'(\Lambda))$\index{J1min@$J_\min$} which makes the map $u\mapsto(u_0,u_1)$ equivariant. Also here, $d\rho_{\min,\varepsilon}$\index{drhomin@$d\rho_\min$} is independent of $\varepsilon$ and we simply write $d\rho_\min=d\rho_{\min,0}=d\rho_{\min,1}$ and extend $d\rho_\min$ to $\calD'(\RR^\times)\otimeshat\calS'(\Lambda)$.

\begin{proposition}
	The representation $d\rho_\min$ of $\frakg$ on $\calD'(\RR^\times)\otimeshat\calS'(\Lambda)$ is given by the same formulas as in Proposition~\ref{prop:drhomin} with $s_\min=-1$, except for the following:
	\begin{align*}
		d\rho_\min(\overline{v}) ={}& i\partial_\lambda\partial_v+\frac{1}{2}\omega(\mu(y')v,B)\partial_A-\frac{1}{i\lambda}\sum_{\alpha,\beta}\omega(B_\mu(y',v)\widehat{e}_\alpha,\widehat{e}_\beta)\partial_{e_\alpha}\partial_{e_\beta}\\
		& \hspace{7.35cm}+\begin{cases}(s_\min-1)\frac{1}{i\lambda}\partial_v&(v\in\calJ_0),\\(s_\min-\frac{\dim\Lambda}{2}+1)\frac{1}{i\lambda}\partial_v&(v\in\overline{\calJ}),\end{cases}\\
		d\rho_\min(\overline{w}) ={}& -\omega(y,w)\lambda\partial_\lambda+\omega(y,w)\partial_y+\partial_{\mu(y')w}-\frac{1}{2i\lambda}\omega(y,B)\sum_{\alpha,\beta}\omega(B_\mu(A,w)\widehat{e}_\alpha,\widehat{e}_\beta)\partial_{e_\alpha}\partial_{e_\beta}\\
		& \hspace{6.7cm}+\begin{cases}(s_\min-\frac{\dim\Lambda}{2}+1)\omega(y,w)&(w\in\calJ_0^*),\\(s_\min-1)\omega(y,w)&(w\in\overline{\calJ}^*),\end{cases}\\
		d\rho_\min(E) ={}& i\lambda\partial_\lambda^2-ia\partial_\lambda\partial_A-i\partial_\lambda\partial_{y'}-i(2s-\tfrac{\dim\Lambda}{2}-1)\partial_\lambda-\frac{s-\frac{\dim\Lambda}{2}}{i\lambda}a\partial_A+n(y')\partial_A\\
		& +\frac{2}{\lambda^2}an(\partial')-\frac{(s-1)(s-\frac{\dim\Lambda}{2}+1)}{i\lambda}-\frac{s-1}{i\lambda}\partial_{y_0'}-\frac{s-\frac{\dim\Lambda}{2}+1}{i\lambda}\partial_{y_1'}\\
		& +\frac{1}{2i\lambda}\sum_{\alpha,\beta}\omega(\mu(y')\widehat{e}_\alpha,\widehat{e}_\beta)\partial_{e_\alpha}\partial_{e_\beta}.
	\end{align*}
\end{proposition}

\begin{proof}
	The proof is similar to the one of Proposition~\ref{prop:drhomin}, the crucial computation being the one for $d\rho_\min(\overline{B})$ which is obtained by taking the Fourier transform of $d\pi_{\zeta,\nu}(\overline{B})u$. Since
	$$ d\pi_{\zeta,\nu}(\overline{B})=d\pi_{\1,\nu}(\overline{B})-2d\zeta(B_\mu(x,B)), $$
	we can consider the two terms $d\pi_{\1,\nu}(\overline{B})$ and $d\zeta(B_\mu(x,B))$ separately. For the first term, the Fourier transform was computed in \eqref{eq:FTpiBbar}, and for the second term we use
	$$ B_\mu(x,B)=\frac{1}{2}\omega(x,B)B_\mu(A,B)+\omega(x,Q)B_\mu(P,B), $$
	so that
	$$ \widehat{d\zeta(B_\mu(x,B))u} = \frac{1}{2}\omega(y-x,B)d\zeta(B_\mu(A,B))\widehat{u}+\omega(y-x,Q)d\zeta(B_\mu(P,B)). $$
	Applying the result to
	$$ \widehat{u}(\lambda,x,y)=\sum_\varepsilon\xi_{-\lambda,\varepsilon}(x)u_\varepsilon(\lambda,y), $$
	using
	\begin{align*}
	\lambda\partial_\lambda\xi_{-\lambda,\varepsilon}(x) &= \tfrac{i\lambda n(x')}{a}\xi_{-\lambda,\varepsilon}(a,x'),\\
	(a\partial_A+p\partial_P)\xi_{-\lambda,\varepsilon}(x) &= -\xi_{-\lambda,\varepsilon}(x),\\
	\partial_{\overline{x}}\xi_{-\lambda,\varepsilon}(x) &= \tfrac{2i\lambda n(x')}{a}\xi_{-\lambda,\varepsilon}(a,x'),\\
	\partial_{\overline{y}}\xi_{-\lambda,\varepsilon}(x) &= -\frac{i\lambda}{2a}\omega(\mu(x')\overline{y},B)\xi_{-\lambda,\varepsilon}(a,x'),
	\end{align*}
	gives
	\begin{multline*}
		d\widehat{\pi}_{\zeta,\nu}(\overline{B})\widehat{u}(\lambda,x,y) = \sum_\varepsilon\xi_{-\lambda,\varepsilon}(x)\Bigg[-\omega(y,B)\lambda\partial_\lambda+\omega(y,B)\partial_y-\omega(y,B)+2i\lambda n(y')\Bigg]u_\varepsilon(\lambda,y)\\
		+\omega(y-x,B)\sum_\varepsilon u_\varepsilon(\lambda,y)\left[\frac{1}{2}(a\partial_A-p\partial_P+\partial_{\overline{x}})+\frac{\nu+\rho-1}{2}-d\zeta(B_\mu(A,B))\right]\xi_{-\lambda,\varepsilon}(x)\\
		+\omega(y-x,Q)\sum_\varepsilon u_\varepsilon(\lambda,y)\left[2a\partial_P+i\lambda\omega(\mu(z)P,B)-2d\zeta(B_\mu(P,B))\right]\xi_{-\lambda,\varepsilon}(x).
	\end{multline*}
	The $\frakm$-invariance of $\xi_{-\lambda,\varepsilon}$ further implies that
	\begin{align*}
		d\zeta(B_\mu(P,B))\xi_{-\lambda,\varepsilon} &= \left(a\partial_P+\frac{1}{2}i\lambda\omega(\mu(z)P,B)\right)\xi_{-\lambda,\varepsilon}\\
		d\zeta(B_\mu(A,B))\xi_{-\lambda,\varepsilon} &= \frac{1}{2}\left(a\partial_A-p\partial_P+\partial_{\overline{x}}+\frac{p+q-6}{2}\right)\xi_{-\lambda,\varepsilon},
	\end{align*}
	so that the claimed formula follows. The formulas for $d\rho_\min(\overline{v})$, $d\rho_\min(\overline{w})$ and $d\rho_\min(E)$ are now obtained by taking commutators as in the proof of Proposition~\ref{prop:drhomin}.
\end{proof}

As before, we change coordinates using the map $\Phi_\delta$ in \eqref{eq:DefPhi} with $s_\min=-1$ and obtain a representation $d\pi_\min$\index{dpimin@$d\pi_\min$} of $\frakg$ on $\calD'(\RR^\times)\otimeshat\calS'(\Lambda)$ which is given by the same formulas as in Proposition~\ref{prop:dpimin}, except for

\begin{align*}
	d\rho_\min(\overline{v}) ={}& i\lambda\partial_\lambda\partial_v+i\partial_x\partial_v+\frac{1}{2\lambda}\omega(\mu(x')v,B)\partial_A+i\sum_{\alpha,\beta}\omega(B_\mu(x',v)\widehat{e}_\alpha,\widehat{e}_\beta)\partial_{e_\alpha}\partial_{e_\beta}\\
	& \hspace{8.2cm}+\begin{cases}2i\partial_v&(v\in\calJ_0),\\i\frac{p+q-4}{2}\partial_v&(v\in\overline{\calJ}),\end{cases}\\
	d\rho_\min(\overline{w}) ={}& -\omega(x,w)\partial_\lambda +\frac{1}{\lambda}\partial_{\mu(x')w}+\frac{1}{2}i\omega(x,B)\sum_{\alpha,\beta}\omega(B_\mu(A,w)\widehat{e}_\alpha,\widehat{e}_\beta)\partial_{e_\alpha}\partial_{e_\beta}\\
	& \hspace{7.35cm}-\begin{cases}\frac{p+q-6}{2\lambda}\omega(x,w)&(w\in\calJ_0^*),\\\frac{1}{\lambda}\omega(x,w)&(w\in\overline{\calJ}^*),\end{cases}\\
d\rho_\min(E) ={}& i\lambda\partial_\lambda^2+i\partial_\lambda\partial_x+i\frac{p+q-2}{2}\partial_\lambda+\frac{n(x')}{\lambda^2}\partial_A+2an(\partial')\\
& -\frac{p+q-6}{2i\lambda}-\frac{p+q-6}{2i\lambda}\partial_{x_0'}-\frac{1}{i\lambda}\partial_{x_1'}+\frac{1}{2i\lambda}\sum_{\alpha,\beta}\omega(\mu(x')\widehat{e}_\alpha,\widehat{e}_\beta)\partial_{e_\alpha}\partial_{e_\beta}.
\end{align*}

\section{Matching the Lie algebra action with the literature}\label{sec:LAactionLiterature}

For some cases, the Lie algebra representation $d\pi_\min$ can be found in the existing literature.

\subsection{The split cases $\frakg=\so(n,n),\frake_{6(6)},\frake_{7(7)},\frake_{8(8)}$}

For the split cases $\frakg=\so(n,n)$, $\frake_{6(6)}$, $\frake_{7(7)}$ and $\frake_{8(8)}$, our formulas for the representation $d\pi_\min$ agree with the formulas in \cite[Appendix]{KPW02}.

\subsection{The case $\frakg=\frakg_{2(2)}$}

Let $\frakg=\frakg_{2(2)}$, the split real form of $\frakg_2(\CC)$. Then the subspace $\frakh=\RR H_\alpha\oplus\RR H_\beta$ is a Cartan subalgebra of $\frakg$. The roots $\lambda=\alpha-2\beta$ and $\mu=-\alpha+\beta$ form a system of simple roots with $\lambda$ a long root and $\mu$ a short root. A Chevalley basis of $\frakg$ is given by
\begin{align*}
	X_\mu &= -2B_\mu(B,C), & X_{-\mu} &= -2B_\mu(A,D), & X_{2\lambda+3\mu} &= F, & X_{-2\lambda-3\mu} &= E,\\
	X_\lambda &= -\frac{1}{\sqrt{2}}A, & X_{\lambda+\mu} &= -\sqrt{2}C, & X_{\lambda+2\mu} &= -\sqrt{2}D, & X_{\lambda+3\mu} &= -\frac{1}{\sqrt{2}}B,\\
	X_{-\lambda} &= -\frac{1}{\sqrt{2}}\overline{B}, & X_{-\lambda-\mu} &= -\sqrt{2}\overline{D}, & X_{-\lambda-2\mu} &= \sqrt{2}\overline{C}, & X_{-\lambda-3\mu} &= \frac{1}{\sqrt{2}}\overline{A}.
\end{align*}

Using the the coordinates
$$ (\lambda,x) = \left(z,\frac{x}{\sqrt{2}}A+\sqrt{2}yC\right), $$
the Lie algebra action $d\pi_\min$ equals the one given in \cite[pages 124--125]{Sav93}, which is due to Gelfand \cite{Gel80}. (In \cite{Sav93} the simple roots are denoted by $\alpha$ and $\beta$ instead of $\lambda$ and $\mu$. Further note that in \cite{Sav93} the term $-\frac{iz}{27}D_y^3$ in the formula for $T(X_{-\alpha-3\beta})$ has to be replaced by $-\frac{z}{27}D_y^3$, cf. \cite{Gel80}.)

\subsection{The case $\frakg=\sl(n,\RR)$}

For $\frakg=\sl(n,\RR)$, the Lie algebra action $d\pi_{\min,r}$ agrees with the action of $\frakg$ on the Fourier transformed picture of a different degenerate principal series, namely one corresponding to a maximal parabolic subgroup. Let $Q=L_QN_Q\subseteq G$ be a parabolic subgroup with $L_Q\simeq\GL(n-1,\RR)$. The characters $\chi_{r,\varepsilon}$ of $L_Q$ are parameterized by $r\in\CC$ and $\varepsilon\in\ZZ/2\ZZ$, and we form the degenerate principal series $\Ind_Q^G(\chi_{r,\varepsilon})$. In \cite[Proposition 4.2]{MS17}, this representation is realized in the non-compact picture on $\overline{N}_Q\simeq\RR^{n-1}$, and the Euclidean Fourier transform on $\RR^{n-1}$ is applied. Surprisingly, this results in the same formulas as the ones obtained for $d\pi_{\min,r}$ in Section~\ref{sec:FTpictureMinRepSLn}, if we identify the tuple $(\lambda,x)\in\RR^\times\times\Lambda$ with a vector in $\RR^{n-1}\simeq\overline{N}_Q$.

In Section~\ref{sec:Int(g,K)Module} we integrate $d\pi_{\min,r}$ to irreducible unitary representations of $\SL(n,\RR)$ which are equivalent to the unitary degenerate principal series $\Ind_Q^G(\chi_{r,\varepsilon})$ with $r\in i\RR$ and $\varepsilon\in\ZZ/2\ZZ$.

\subsection{The case $\frakg=\so(4,3)$}\label{sec:TheCaseSO43}

For $\frakg=\so(4,3)$, Sabourin~\cite{Sab96} constructed an explicit $L^2$-realization of the minimal representation. His formulas in~\cite[Proposition 3.6.2 and 3.6.3]{Sab96} have a lot in common with our realization, but the major difference is that in his model the Lie algebra acts by differential operators of order $\leq2$, while we need order $3$ in general. We believe that Sabourin's realization can be obtained with our methods by choosing a different Lagrangian subspace $\Lambda\subseteq V$. More precisely, for $\frakg=\so(p,q)$ the Lie algebra $\frakm\simeq\sl(2,\RR)\otimes\so(p-2,q-2)$ acts on $V\simeq\RR^2\otimes\RR^{p+q-4}$ by the tensor product of the two standard representations. We believe that choosing $\Lambda$ and $\Lambda^*$ to be $\so(p-2,q-2)$-invariant, one obtains Sabourin's formulas.

\chapter{Lowest $K$-types}\label{ch:LKT}

To show that the subrepresentation $I(\zeta,\nu)^{\Omega_\mu(\frakm)}$ is non-trivial for some choice of a representation $\zeta$ of $M$, we find in this section the lowest $K$-type in the Fourier transformed picture explicitly. For this, we first construct in Section~\ref{sec:CartanInvolutions} a Cartan involution on $\frakg$ with corresponding maximal compact subalgebra $\frakk\subseteq\frakg$ which is compatible with the bigrading on $\frakg$ constructed in Section~\ref{sec:Bigrading}. The rest of the chapter is devoted to explicit case-by-case computations exhibiting the lowest $K$-type. The nature of the lowest $K$-type is quite different in the various cases, so a classification-free description seems out of reach.

\section{Cartan involutions}\label{sec:CartanInvolutions}

We study Cartan involutions in the case where $\frakm$ is simple. This excludes the cases $\frakg=\sl(n,\RR)$ and $\frakg=\so(p,q)$ for which we separately discuss Cartan involutions in Sections~\ref{sec:LKTsln} and \ref{sec:LKTSOpq}.

To construct a Cartan involution on $\frakg$, we make use of Lemma~\ref{lem:CartanInvFromJ}, i.e. we construct a map $J\in\End(V)$ which satisfies the conditions of Lemma~\ref{lem:CartanInvFromJ}. For this, we first choose a unit in the Jordan algebra $\calJ$. Let $C\in\calJ=\frakg_{(0,-1)}$\index{C@$C$} with $\Psi(C)\neq0$ and put $D:=\mu(C)B\in\frakg_{(-1,0)}$\index{D@$D$}. Then, by Lemma~\ref{lem:DecompBigradingMuPsiQ} we have $\mu(C)=-B_\mu(A,\mu(C)B)=-B_\mu(A,D)$ and $\mu(D)=B_\mu(\mu(D)A,B)$. On the other hand, by the $\frakm$-equivariance of $B_\mu$ and Lemma~\ref{lem:MuSquared}:
\begin{align*}
\mu(D) &= B_\mu(\mu(C)B,\mu(C)B) = [\mu(C),\underbrace{B_\mu(B,\mu(C)B)}_{\in\frakg_{(-2,2)}=\{0\}}] - B_\mu(B,\mu(C)^2B) = 4n(C)B_\mu(C,B).
\end{align*}
It follows that $\mu(D)A=4n(C)C$ and hence $\mu(D)=4n(C)B_\mu(C,B)$. We therefore renormalize $C$ such that $n(C)=\frac{1}{4}$, then
$$ \mu(C) = -B_\mu(A,D) \qquad \mbox{and} \qquad \mu(D) = B_\mu(C,B), $$
as well as
$$ \mu(C)D = -C \qquad \mbox{and} \qquad \mu(D)C = D. $$
It further follows that $\Psi(C)=\frac{1}{4}A$ and
\begin{align*}
\Psi(D) &= \frac{1}{2}\omega(A,\Psi(D))B = -\frac{1}{6}\omega(A,\mu(D)D)B = \frac{1}{6}\omega(\mu(D)A,D)B = \frac{1}{6}\omega(C,D)B\\
&= \frac{1}{6}\omega(C,\mu(C)B)B = -\frac{1}{6}\omega(\mu(C)C,B)B = \frac{1}{2}\omega(\Psi(C),B)B = \frac{1}{4}B.
\end{align*}
Note that this computation also shows that $\omega(C,D)=\frac{3}{2}$.

\begin{lemma}\label{lem:BmuCD}
$B_\mu(C,D)=\frac{1}{4}B_\mu(A,B)$.
\end{lemma}

\begin{proof}
	We have
	$$ [\mu(C),\mu(D)] = 2B_\mu(\mu(C)D,D) = -2B_\mu(C,D). $$
	On the other hand, $\mu(D)=B_\mu(C,B)$, so that
	$$ [\mu(C),\mu(D)] = B_\mu(\mu(C)C,B)+B_\mu(C,\mu(C)B) = -\frac{3}{4}B_\mu(A,B)+B_\mu(C,D).\qedhere $$
\end{proof}

\begin{lemma}
	The elements $\{2\mu(C),2B_\mu(A,B),-2\mu(D)\}$ form an $\sl(2)$-triple.
\end{lemma}

\begin{proof}
	By the proof of Lemma~\ref{lem:BmuCD}, we have $[\mu(C),\mu(D)]=-\frac{1}{2}B_\mu(A,B)$. Further,
	\begin{align*}
	[B_\mu(A,B),\mu(C)] &= 2B_\mu(B_\mu(A,B)C,C) = \mu(C),\\
	[B_\mu(A,B),\mu(D)] &= 2B_\mu(B_\mu(A,B)D,D) = -\mu(D).\qedhere
	\end{align*}
\end{proof}

Let $\calJ_0=\{v\in\calJ:\omega(v,D)=0\}$\index{J30@$\calJ_0$}, so that $\calJ=\RR C\oplus\calJ_0$. Similarly, let $\calJ^*_0=\{w\in\calJ^*:\omega(C,w)=0\}$\index{J30star@$\calJ^*_0$}, so that $\calJ^*=\RR D\oplus\calJ^*_0$.

\begin{lemma}\label{lem:muCmuD}
	$\mu(C)\mu(D)|_{\calJ_0}=-\frac{1}{4}\id_{\calJ_0}$ and $\mu(D)\mu(C)|_{\calJ_0^*}=-\frac{1}{4}\id_{\calJ_0^*}$.
\end{lemma}

\begin{proof}
	It suffices to show the first statement, the second one is proven similarly. Let $v\in\calJ_0$, then
	\begin{align*}
		\mu(C)\mu(D)v &= [\mu(C),\mu(D)]v + \mu(D)\mu(C)v = -\frac{1}{4}v + \mu(D)\mu(C)v.
	\end{align*}
	Since
	$$ \mu(C)v = \frac{1}{2}\omega(\mu(C)v,B) = -\frac{1}{2}\omega(v,\mu(C)B) = -\frac{1}{2}\omega(v,D) = 0, $$
	the claim follows.
\end{proof}

In \cite[Theorem 7.32]{SS} it is shown that $\calJ$ can be endowed with a natural Jordan algebra structure with unit element $C$ and norm function $N(v)=4n(v)$. The corresponding trace form is given by
$$ T(u,v) = \partial_uN(C)\partial_vN(C)-\partial_u\partial_vN(C) = 4\omega(u,D)\omega(v,D)+4\omega(\mu(D)u,v), \qquad u,v\in\calJ.\index{Tuv@$T(u,v)$} $$
Note that, since $T(C,v)=2\omega(v,D)$, the $T$-orthogonal complement of $C$ in $\calJ$ equals $\calJ_0$. Write $u=u_0C+u'$ and $v=v_0C+v'$ with $u',v'\in\calJ_0$, then
\begin{equation}
T(u,v) = 3u_0v_0 + 4\omega(\mu(D)u',v').\label{eq:FormulaTraceForm}
\end{equation}

\begin{definition}
	A \emph{Cartan involution} of a Jordan algebra $\calJ$ with trace form $T$ is an involutive algebra automorphism $\vartheta$\index{1htheta@$\vartheta$} of $\calJ$ such that $T(\vartheta v,v)>0$ for all $v\in\calJ\setminus\{0\}$.
\end{definition}

Using a Cartan involution of the Jordan algebra $\calJ$, we obtain a natural Cartan involution compatible with the bigrading on $\frakg$ constructed in Section~\ref{sec:Bigrading}:

\begin{proposition}\label{prop:JFromJordanCartanInv}
	For every Cartan involution $\vartheta$ of the Jordan algebra $\calJ$, the map $J\in\End(V)$\index{J1@$J$} given by
	\begin{align*}
	JA &= -B, & JC &= -D, & Jv &= 2\mu(D)\vartheta v && (x\in\calJ_0),\\
	JB &= A, & JD &= C, & Jw &= 2\vartheta\mu(C)w && (w\in\frakg_{(-1,0)}\mbox{ with }\omega(C,w)=0),
	\end{align*}
	satisfies the conditions of Lemma~\ref{lem:CartanInvFromJ}.
\end{proposition}

In what follows we will fix such a Cartan involution $\vartheta$ of $\calJ$ and let $J$ be the corresponding linear map on $V$ and $\theta$ be the corresponding Cartan involution on $\frakg$. Further, let $\frakk=\frakg^\theta$ denote the corresponding maximal compact subalgebra and $\frakp=\frakg^{-\theta}$ its Cartan complement.

\begin{proof}
	Condition \eqref{lem:CartanInvFromJ1} follows from $\vartheta^2=\id$ and Lemma~\ref{lem:muCmuD}. Condition \eqref{lem:CartanInvFromJ2} follows from \eqref{eq:FormulaTraceForm} and Lemma~\ref{lem:muCmuD}. The only non-trivial computation for condition \eqref{lem:CartanInvFromJ3} is
	\begin{align*}
	\omega(Jv,Jw) &= 4\omega(\mu(D)\vartheta v,\vartheta\mu(C)w) = T(\vartheta v,\vartheta\mu(C)w) = T(v,\mu(C)w)\\
	&= 4\omega(\mu(D)v,\mu(C)w) = -4\omega(\mu(C)\mu(D)v,w) = \omega(v,w) \qquad\qquad (v,w\in V),
	\end{align*}
	by Lemma~\ref{lem:muCmuD} and the fact that $\vartheta$ is a Jordan algebra automorphism and therefore leaves the trace form invariant. Condition \eqref{lem:CartanInvFromJ4} is equivalent to
	$$ JB_\mu(x,y)J^{-1}=B_\mu(Jx,Jy) \qquad \mbox{for all $x,y\in V$,} $$
	which is checked by a lengthy case-by-case computation.
\end{proof}

The Jordan algebra $\calJ$ is called \emph{Euclidean} if its trace form is positive definite. In this case, we can and will choose $\vartheta=\id_\calJ$ as the Cartan involution. Let
$$ T_0 = B_\mu(B,C)-B_\mu(A,D) = \mu(C)+\mu(D) \in \frakk\cap\frakm,\index{T0@$T_0$} $$
noting that $T_0\in\frakk$ since $\theta(\mu(C))=\mu(JC)=\mu(D)$ by Lemma~\ref{lem:CartanInvFromJ} and Proposition~\ref{prop:JFromJordanCartanInv}.

\begin{proposition}\label{prop:SU2Ideal}
	If $\calJ$ is a Euclidean Jordan algebra and $\vartheta=\id_\calJ$, then the elements
	$$ T_1=2T_0-(E-F), \quad T_2=A-2D+\theta(A-2D), \quad T_3=B+2C+\theta(B+2C)\index{T1@$T_1$}\index{T2@$T_2$}\index{T3@$T_3$} $$
	span an ideal $\frakk_1\subseteq\frakk$\index{k31@$\frakk_1$} which is isomorphic to $\su(2)$.
\end{proposition}

\begin{proof}
	We first note the following commutator formulas in $\frakk$ (see \eqref{eq:DefBar1}, \eqref{eq:DefBar2}, \eqref{eq:CommutatorG1} and Lemma~\ref{lem:G1bracketG-1}):
	\begin{align*}
	[E-F,x+\theta(x)] &= Jx+\theta(Jx) && \mbox{for all }x\in V,\\
	[S,x+\theta(x)] &= Sx+\theta(Sx) && \mbox{for all }S\in\frakk\cap\frakm,x\in V,\\
	[E-F,S] &= 0 && \mbox{for all }S\in\frakk\cap\frakm,\\
	[x+\theta(x),y+\theta(y)] &= -\omega(x,y)(E-F)-2(B_\mu(Jx,y)-B_\mu(x,Jy)) && \mbox{for all }x,y\in V.
	\end{align*}
	Further, we have
	$$ T_0A = C, \quad T_0B=D, \quad T_0C=-\frac{3}{4}A+D, \quad T_0D=-\frac{3}{4}B-C. $$
	We first show that $\frakk_1$ is a subalgebra. For this we compute
	\begin{align*}
	[T_1,T_2] ={}& (2T_0-J)(A-2D)+\theta((2T_0-J)(A-2D)) = 4T_3,\\
	[T_1,T_3] ={}& (2T_0-J)(B+2C)+\theta((2T_0-J)(B+2C)) = -4T_2,\\
	[T_2,T_3] ={}& -\omega(A-2D,B+2C)(E-F)\\
	& -2(B_\mu(J(A-2D),B+2C)-B_\mu(A-2D,J(B+2C)))\\
	={}& 8T_1.
	\end{align*}
	It remains to show that $\frakk_1$ is an ideal, i.e. $[\frakk,\frakk_1]\subseteq\frakk_1$. First, by similar computations as above, $T_1$, $T_2$ and $T_3$ commute with
	$$ 2T_0+3(E-F), \quad 3A+2D+\theta(3A+2D) \quad \mbox{and} \quad 3B-2C+\theta(3B-2C).  $$
	Finally, similar computations show that $T_1$, $T_2$ and $T_3$ commute with
	\begin{itemize}
		\item $v+\theta v$, $v\in\calJ_0$ or $v\in\calJ^*_0$,
		\item $B_\mu(v,B)+B_\mu(A,Jv)$, $v\in\calJ_0$,
		\item $S\in\frakg_{(0,0)}\cap\frakk$.\qedhere
	\end{itemize}
\end{proof}

\begin{remark}\label{rem:SU2Ideal}
	The renormalized generators
	$$ \widetilde{T}_1 = \frac{1}{2}T_1, \quad \widetilde{T}_2 = \frac{1}{2\sqrt{2}}T_2, \quad \widetilde{T}_3 = \frac{1}{2\sqrt{2}}T_3 $$
	satisfy the standard $\su(2)$-relations
	$$ [\widetilde{T}_1,\widetilde{T}_2]=2\widetilde{T}_3, \quad [\widetilde{T}_2,\widetilde{T}_3]=2\widetilde{T}_1, \quad [\widetilde{T}_3,\widetilde{T}_1]=2\widetilde{T}_2. $$
\end{remark}

\section{The quaternionic cases $\frakg=\frake_{6(2)},\frake_{7(-5)},\frake_{8(-24)}$}

Assume that the Jordan algebra $\calJ$ is simple and Euclidean, i.e. the identity $\vartheta=\id_{\calJ}$ is a Cartan involution. Then $\calJ$ is isomorphic to $\Herm(3,\FF)$ with $\FF\in\{\RR,\CC,\HH,\OO\}$, the Jordan algebra of $3\times3$ Hermitian matrices over the real numbers $\RR$, the complex numbers $\CC$, the quaternions $\HH$ or the octonions $\OO$. By Proposition~\ref{prop:SU2Ideal} the group $G$ is of quaternionic type, and by the classification, we have $\frakg\simeq\frakf_{4(4)},\frake_{6(2)},\frake_{7(-5)},\frake_{8(-24)}$ with $s_\min=-\frac{3}{2},-2,-3,-5$, respectively. We decompose $\frakk$ into simple ideals:
$$ \frakk=\frakk_1\oplus\frakk_2\index{k31@$\frakk_1$}\index{k32@$\frakk_2$} $$
with $\frakk_1=\RR T_1\oplus\RR T_2\oplus\RR T_3\simeq\su(2)$. We further abbreviate $n=-s_\min-1\in\{\frac{1}{2},1,2,4\}$\index{n2@$n$}. For the following statement we use the coordinates $(\lambda,aA+x)\in\RR^\times\times\Lambda$, where $a\in\RR$ and $x\in\calJ$.

\begin{theorem}\label{thm:LKTQuat}
	For $\frakg=\frake_{6(2)},\frake_{7(-5)},\frake_{8(-24)}$ the space $W=\bigoplus_{k=-n}^n\CC f_k$\index{W1@$W$} with
	$$ f_k(\lambda,a,x) = (\lambda-i\sqrt{2}a)^k(\lambda^2+2a^2)^{\frac{s_\min-k}{2}}\exp\left(-\frac{2ian(x)}{\lambda(\lambda^2+2a^2)}\right) \sum_{m=-n}^nh_{k,m}K_m(r)e^{im\theta},\index{fk@$f_k$} $$
	where $K_m(z)$ denotes the classical $K$-Bessel function (see Appendix~\ref{app:KBessel}),
	$$ (r\cos\theta,r\sin\theta) = \left(\frac{2(\lambda^2+2a^2)I_1-I_3}{\sqrt{2}(\lambda^2+2a^2)},\frac{I_2+\lambda^2+2a^2}{(\lambda^2+2a^2)^{\frac{1}{2}}}\right) $$
	with
	$$ I_1=\omega(x,D), \qquad I_2=\omega(\mu(x)C,B), \qquad I_3=\omega(\Psi(x),B)\index{I1@$I_1$}\index{I2@$I_2$}\index{I3@$I_3$} $$
	and $(h_{k,m})_{m=-n,\ldots,n}$\index{h3km@$h_{k,m}$} is given by \eqref{eq:QuatRecurrenceSolution}, is a $\frakk$-subrepresentation of $(d\pi_\min,\calD'(\RR^\times)\otimeshat\calS'(\Lambda))$ isomorphic to the representation $S^{2n}(\CC^2)\boxtimes\CC$ of $\frakk\simeq\su(2)\oplus\frakk_2$.
\end{theorem}

\begin{remark}\label{rem:LKTF4}
	We exclude the case $\frakg=\frakf_{4(4)}$ in the theorem, because here $n=\frac{1}{2}$ and therefore the summation would have to be over $m=\pm\frac{1}{2}$. This is not immediately possible since it would require the use of $e^{\pm\frac{i\theta}{2}}=(\cos\theta+i\sin\theta)^{\frac{1}{2}}$ which cannot be defined as a smooth function on $\RR_+\times\Lambda$ or $\RR_-\times\Lambda$, because the image of both $2(\lambda^2+2a^2)I_1-I_3$ and $I_2+\lambda^2+2a^2$ is $\RR$. However, it might be possible to find the lowest $K$-type $S^1(\CC^2)\boxtimes\CC$ of the minimal representation in a space of vector-valued functions as a subrepresentation of a vector-valued degenerate principal series (cf. \cite[Section 12]{GW96}).
\end{remark}

We prove this result in several steps. For this we decompose $\frakk_2$ as follows:
\begin{align*}
\frakk_2 ={}& \RR(2T_0+3(E-F))\oplus\RR(3A+2D+\theta(3A+2D))\oplus\RR(3B-2C+\theta(3B-2C))\\
&\oplus\{v+\theta v:v\in\calJ_0\}\oplus\{\overline{v}+\theta\overline{v}:v\in\calJ_0\}\\
&\oplus\{B_\mu(v,B)+B_\mu(A,Jv):v\in\calJ_0\} \oplus(\frakg_{(0,0)}\cap\frakk).
\end{align*}

\begin{lemma}\label{lem:QuatStep0}
	The Lie algebra $\frakk_2$ is generated by
	$$ (\frakg_{(0,0)}\cap\frakk)\oplus\{v+\theta v:v\in\calJ_0\}\oplus\{B_\mu(v,B)+B_\mu(A,Jv):v\in\calJ_0\}. $$
\end{lemma}

\begin{proof}
	Let $\frakh\subseteq\frakk$ denote the subalgebra generated by the above elements. For $v,w\in\calJ_0$ we have
	\begin{align*}
	[B_\mu(v,B)+B_\mu(A,Jv),w+\theta w] &= B_\mu(v,B)w+B_\mu(A,Jv)w+\theta(B_\mu(v,B)w+B_\mu(A,Jv)w)\\
	&= -\frac{1}{6}\omega(Jv,w)(3A+2D+\theta(3A+2D)) + u+\theta(u)
	\end{align*}
	with $u=B_\mu(v,w)B+\frac{1}{3}\omega(Jv,w)D\in\calJ^*_0$. Note that, since the Jordan algebra $\calJ$ is simple, its trace form is non-degenerate, and hence we can always find $v,w\in\calJ_0$ such that $u\neq0$. If we further act by $S\in\frakg_{(0,0)}\cap\frakk$, using $SA=SD=0$, we obtain
	$$ [S,[B_\mu(v,B)+B_\mu(A,Jv),w+\theta w]] = Su+\theta(Su). $$
	Now, $\frakg_{(0,0)}\cap\frakk$ is the Lie algebra of the automorphism group of the Jordan algebra $\calJ$ which acts irreducibly on $\calJ_0$. It follows that $\{u+\theta u:u\in\calJ^*_0\}\subseteq\frakh$. Further, by choosing $v,w\in\calJ_0$ above such that $\omega(Jv,w)\neq0$ we obtain $3A+2D+\theta(3A+2D)\in\frakh$. A similar argument with
	$$ [B_\mu(v,B)+B_\mu(A,Jv),\overline{w}+\theta(\overline{w})] $$
	shows $3B-2C+\theta(3B-2C)\in\frakh$. Finally,
	\begin{equation*}
		[3A+2D+\theta(3A+2D),3B-2C+\theta(3B-2C)] = -8(2T_0+3(E-F))\in\frakh.\qedhere
	\end{equation*}
\end{proof}

\begin{lemma}\label{lem:QuatStep1}
	$f\in\calD'(\RR^\times)\otimeshat\calS'(\Lambda)$ is $(\frakg_{(0,0)}\cap\frakk)$-invariant if and only if it is of the form
	$$ f(\lambda,a,x) = f_1(\lambda,a,I_1,I_2,I_3), $$
	where
	$$ I_1=\omega(x,D), \qquad I_2=\omega(\mu(x)C,B), \qquad I_3=\omega(\Psi(x),B). $$
\end{lemma}

\begin{proof}
	The Lie algebra $\frakg_{(0,0)}$ decomposes as $\RR H\oplus(\frakm\cap\frakg_{(0,0)})$ and $\frakm\cap\frakg_{(0,0)}$ is the Lie algebra of the structure group of the Jordan algebra $\calJ$ (see \cite{FK94} for details on Jordan algebras). Since $\calJ$ is Euclidean, the maximal compact subalgebra $\frakk\cap\frakg_{(0,0)}$ is the Lie algebra of the automorphism group of $\calJ$. Its invariants are the coefficients $a_1(x),a_2(x),a_3(x)$ of the minimal polynomial $X^3-a_1(x)X^2+a_2(x)X-a_3(x)$ of a generic element $x\in\calJ$ (see \cite[Chapter II.2]{FK94}). These are given by $a_1(x)=\tr(x)=T(x,C)=2\omega(x,D)$, $a_3(x)=\det(x)=4n(x)=2\omega(\Psi(x),B)$ and
	\begin{equation*}
	a_2(x)=\partial_C\det(x)=6\omega(B_\Psi(x,x,C),B)=-2\omega(\mu(x)B,C).\qedhere
	\end{equation*}
\end{proof}

\begin{lemma}\label{lem:QuatStep2}
	$f\in\calD'(\RR^\times)\otimeshat\calS'(\Lambda)$ is additionally an eigenfunction of $d\pi_\min(A-\overline{B})$ to the eigenvalue $ik\sqrt{2}$ if and only if it is, for $\lambda>0$ resp. $\lambda<0$, of the form
	$$ f(\lambda,a,x) = (\lambda-i\sqrt{2}a)^k\exp\left(-\frac{iaI_3}{\lambda R}\right)f_2(R,I_1,I_2,I_3), $$
	where
	$$ R=\lambda^2+2a^2.\index{R@$R$} $$
\end{lemma}

In Lemma~\ref{lem:QuatStep2}, $f_2(R,I_1,I_2,I_3)$ could be different for $\lambda>0$ and $\lambda<0$. However, it later turns out that choosing the same $f_2(R,I_1,I_2,I_3)$ for all $\lambda\in\RR^\times$ yields the $K$-finite vectors in a unitary representation $\pi_\min$ of $\widetilde{G}$ on $L^2(\RR^\times\times\Lambda)$. This is due to the fact that the derived representation $d\pi_\min$ has to be infinitesimally unitary on the $(\frakg,K)$-module generated by $f$, and an integration by parts argument on $\RR^\times$ shows that this is only the case if $f_2(R,I_1,I_2,I_3)$ is the same for $\lambda\to0^+$ and $\lambda\to0^-$.

\begin{proof}
	The method of characteristics applied to the first order equation
	$$ d\pi_\min(A-\overline{B})f = \left(-\lambda\partial_A+2a\partial_\lambda-\frac{iI_3}{\lambda^2}\right)f=ik\sqrt{2}f $$
	shows the claim.
\end{proof}

\begin{lemma}\label{lem:QuatStep3}
	$f\in\calD'(\RR^\times)\otimeshat\calS'(\Lambda)$ is additionally annihilated by the operators
	$\{\lambda\,d\pi_\min(v+\theta v)+2a\,d\pi_\min(B_\mu(v,B)+B_\mu(A,Jv)):v\in\calJ_0\}$
	if and only if it is of the form
	$$ f(\lambda,a,x) = (\lambda-i\sqrt{2}a)^k(\lambda^2+2a^2)^{\frac{s_\min-k}{2}}\exp\left(-\frac{iaI_3}{\lambda R}\right) f_3(S,T) $$
	with
	$$ S=\frac{2RI_1-I_3}{\sqrt{2}R}, \quad T=\frac{I_2+R}{R^{\frac{1}{2}}}.\index{S@$S$}\index{T@$T$} $$
\end{lemma}

\begin{proof}
	Applying
	\begin{multline*}
	\lambda\,d\pi_\min(v+\theta v)+2a\,d\pi_\min(B_\mu(v,B)+B_\mu(A,Jv)) = -\omega(x,Jv)(\lambda\partial_\lambda+a\partial_A-s_\min)\\
	-(2a^2+\lambda^2)\partial_v+\partial_{\mu(x)Jv}+i\frac{a}{\lambda}\omega(\mu(x)v,B)
	\end{multline*}
	to $f(\lambda,a,x)=(\lambda-i\sqrt{2}a)^k\exp\left(-\frac{iaI_3}{\lambda R}\right)f_2(R,I_1,I_2,I_3)$ gives
	\begin{multline*}
	(\lambda-i\sqrt{2}a)^k\exp\left(-\frac{iaI_3}{\lambda R}\right)\Bigg[\omega(x,Jv)\Big(-2R\partial_R+(R-I_2)\partial_2-2I_3\partial_3+s_\min-k\Big)f_2\\
	+\omega(\mu(x)v,B)\Big(R\partial_3+\frac{1}{2}\partial_1\Big)f_2\Bigg] = 0.
	\end{multline*}
	Here we have used
	\begin{align*}
	\omega(B_\mu(x,v)C,B) &= -\frac{1}{2}\omega(x,Jv), & \omega(B_\mu(x,\mu(x)Jv)C,B) &= -\frac{1}{2}I_2\omega(x,Jv),\\
	\omega(\mu(x)Jv,D) &= \frac{1}{2}\omega(\mu(x)v,B), & \omega(B_\Psi(x,x,\mu(x)Jv),B) &= -\frac{2}{3}I_3\omega(x,Jv).
	\end{align*}
	This yields two first order partial differential equations:
	\begin{align}
	& -2R\partial_Rf_2+(R-I_2)\partial_2f_2-2I_3\partial_3f_2 = (k-s_\min)f_2,\label{eq:QuatStep1Eq1}\\
	& R\partial_3f_2+\frac{1}{2}\partial_1f_2 = 0.\label{eq:QuatStep1Eq2}
	\end{align}
	Solving \eqref{eq:QuatStep1Eq2} using the method of characteristics gives
	$$ f_2(R,I_1,I_2,I_3) = \widetilde{f_2}(R,I_2,U) \qquad \mbox{with }U=2RI_1-I_3. $$
	Applying \eqref{eq:QuatStep1Eq1} to this expression and using again the method of characteristics yields
	\begin{equation*}
	\widetilde{f_2}(\lambda,a,I_2,U) = R^{\frac{s_\min-k}{2}} f_3(S,T).\qedhere
	\end{equation*}
\end{proof}

\begin{lemma}\label{lem:QuatStep4}
	$f\in\calD'(\RR^\times)\otimeshat\calS'(\Lambda)$ is additionally invariant under $\{B_\mu(v,B)+B_\mu(A,Jv):v\in\calJ_0\}$ if and only if it is of the form
	$$ f(\lambda,a,x) = (\lambda-i\sqrt{2}a)^k(\lambda^2+2a^2)^{\frac{s_\min-k}{2}}\exp\left(-\frac{iaI_3}{\lambda R}\right) \sum_{m\in\ZZ}h_{k,m}K_m(r)e^{im\theta}, $$
	where $(S,T)=(r\cos\theta,r\sin\theta)$ and $(h_{k,m})_{m\in\ZZ}$ is a sequence satisfying
	\begin{equation}
	\frac{m+s_\min}{2}h_{k,m-1}-\frac{m-s_\min}{2}h_{k,m+1}+kh_{k,m} = 0.\qedhere\label{eq:QuatStep2Recurrence}
	\end{equation}
\end{lemma}

\begin{proof}
	Using the identities
	\begin{align*}
	\sum_{\alpha,\beta} \omega(e_\alpha,D)\omega(e_\beta,D)\omega(B_\mu(A,Jv)\widehat{e}_\alpha,\widehat{e}_\beta) &= 0,\\
	\sum_{\alpha,\beta} \omega(e_\alpha,D)\omega(B_\mu(x,e_\beta)C,B)\omega(B_\mu(A,Jv)\widehat{e}_\alpha,\widehat{e}_\beta) &= -\frac{1}{4}\omega(x,Jv),\\
	\sum_{\alpha,\beta} \omega(e_\alpha,D)\omega(\mu(x)e_\beta,B)\omega(B_\mu(A,Jv)\widehat{e}_\alpha,\widehat{e}_\beta) &= \frac{1}{2}\omega(\mu(x)v,B),\\
	\sum_{\alpha,\beta} \omega(B_\mu(e_\alpha,e_\beta)x,B)\omega(B_\mu(A,Jv)\widehat{e}_\alpha,\widehat{e}_\beta) &= s_\min\omega(x,Jv),\\
	\sum_{\alpha,\beta} \omega(\mu(x)e_\alpha,B)\omega(\mu(x)e_\beta,B)\omega(B_\mu(A,Jv)\widehat{e}_\alpha,\widehat{e}_\beta) &= 2I_3\omega(x,Jv),\\
	\sum_{\alpha,\beta} \omega(\mu(x)e_\alpha,B)\omega(B_\mu(x,e_\beta)C,B)\omega(B_\mu(A,Jv)\widehat{e}_\alpha,\widehat{e}_\beta) &= -\frac{1}{2}I_2\omega(x,Jv),\\
	\sum_{\alpha,\beta} \omega(B_\mu(x,e_\alpha)C,B)\omega(B_\mu(x,e_\beta)C,B)\omega(B_\mu(A,Jv)\widehat{e}_\alpha,\widehat{e}_\beta) &= \frac{1}{2}I_1\omega(x,Jv) + \frac{1}{4}\omega(\mu(x)v,B),
	\end{align*}
	the equation $d\pi_\min(B_\mu(v,B)+B_\mu(A,Jv))f=0$ for $f$ as in Lemma~\ref{lem:QuatStep3} becomes
	\begin{multline*}
	\omega(x,Jv)\Big(\frac{I_3}{R}\partial_S^2-\sqrt{2}T\partial_S\partial_T+2I_1\partial_T^2+\sqrt{2}s_\min\partial_S-\frac{I_3}{R}-k\sqrt{2}\Big)f_3\\
	+\omega(\mu(x)v,B)\Big(\partial_S^2+\partial_T^2-1\Big)f_3 = 0.
	\end{multline*}
	Again, this gives rise to two partial differential equations, this time of second order:
	\begin{align}
		& \frac{I_3}{R}\partial_S^2f_3-\sqrt{2}T\partial_S\partial_Tf_3+2I_1\partial_T^2f_3+\sqrt{2}s_\min\partial_Sf_3-\frac{I_3}{R}f_3-k\sqrt{2}f_3 = 0,\label{eq:QuatStep2Eq1}\\
		& \partial_S^2f_3+\partial_T^2f_3 = f_3.\label{eq:QuatStep2Eq2}
	\end{align}
	While \eqref{eq:QuatStep2Eq2} only contains the variables $S$ and $T$, \eqref{eq:QuatStep2Eq1} also contains $I_1$ and $I_3$. We therefore subtract $2I_1$ times \eqref{eq:QuatStep2Eq2} from \eqref{eq:QuatStep2Eq1} to obtain the equivalent equation
	\begin{equation}
		(S\partial_S+T\partial_T-s_\min)\partial_Sf_3-Sf_3+kf_3 = 0.\label{eq:QuatStep2Eq3}
	\end{equation}	
	Using polar coordinates $(S,T)=(r\cos\theta,r\sin\theta)$, we expand $f_3$ into a Fourier series 
	$$ f_3(S,T) = \sum_{m\in\ZZ} g_m(r)e^{im\theta}. $$
	Then \eqref{eq:QuatStep2Eq2} becomes
	$$ \partial_r^2g_m+\frac{1}{r}\partial_rg_m-\Big(1+\frac{m^2}{r^2}\Big)g_m = 0. $$
	The two solutions to this ordinary differential equation are the Bessel functions $I_m(r)$ and $K_m(r)$. The $I$-Bessel function grows exponentially as $r\to\infty$, while the $K$-Bessel function decays exponentially. Since we are only interested in tempered distributions (in fact, only $L^2$-functions) we write
	$$ g_m(r) = h_{k,m}K_m(r) $$
	for some scalars $h_{k,m}\in\CC$. Applying \eqref{eq:QuatStep2Eq3} to the Fourier expansion finally yields the relation \eqref{eq:QuatStep2Recurrence}.
\end{proof}

\begin{lemma}\label{lem:QuatStep5}
	For $-n\leq k\leq n$ there is a unique (up to scalar multiples) sequence $(h_{k,m})_{m\in\ZZ}$ satisfying \eqref{eq:QuatStep2Recurrence}. It satisfies $h_{k,m}=0$ for $|m|>n$ and, normalizing $h_{k,m}=1$ for $m=-n$, it is given by
	\begin{equation}\label{eq:QuatRecurrenceSolution}
		h_{k,m} = \sum_{j=0}^{m+n}(-1)^{m+n-j}{n+k\choose j}{n-k\choose m+n-j}
	\end{equation}
	for $m=-n,\ldots,n$.
\end{lemma}

\begin{proof}
	Let $(h_{k,m})_{m\in\ZZ}$ be a sequence satisfying \eqref{eq:QuatStep2Recurrence} and form the generating function
	$$ h(t) = \sum_{m\in\ZZ}h_{k,m}t^m. $$
	Then \eqref{eq:QuatStep2Recurrence} becomes the differential equation
	$$ h'(t) = \frac{-nt^2+2kt-n}{t(1-t^2)}h(t). $$
	Writing
	$$ \frac{-nt^2+2kt-n}{t(1-t^2)} = -\frac{n}{t} + \frac{k-n}{1-t} + \frac{k+n}{1+t} $$
	reveals the solution
	$$ h(t) = t^{-n}(1-t)^{n-k}(1+t)^{n+k}. $$
	Expanding both $(1-t)^{n-k}$ and $(1+t)^{n+k}$ into Taylor series around $t=0$ shows the claim.
\end{proof}

Recall the elements $\widetilde{T}_1,\widetilde{T}_2,\widetilde{T}_3\in\frakk$ from Proposition~\ref{prop:SU2Ideal} and Remark~\ref{rem:SU2Ideal} which span the ideal $\frakk_1\simeq\su(2)$. The element $\widetilde{T}_2$ spans a maximal torus in $\frakk_1$ and $\ad(\widetilde{T}_2)$ has eigenvalues $0,\pm2i$. By the representation theory of $\su(2)$, the action of $\widetilde{T}_2$ in every finite-dimensional representation of $\frakk_1$ is diagonalizable with eigenvalues $2ik$, $k\in\ZZ\cup(\ZZ+\frac{1}{2})$, and the elements $\widetilde{T}_3\pm i\widetilde{T}_1$ step between the eigenspaces. Lemma~\ref{lem:QuatStep2} shows that $f_k$ is in fact an eigenfunction of $d\pi_\min(\widetilde{T}_2)$ to the eigenvalue $2ik$. We therefore compute the action of $d\pi_\min(\widetilde{T}_3\pm i\widetilde{T}_1)$ on $f_k$.

\begin{lemma}\label{lem:QuatStep6}
	For $-n\leq k\leq n$ we have
	$$ d\pi_\min(2T_0\pm i\sqrt{2}(C-\overline{D}))f_k = -3(k\mp n)f_{k\pm1}. $$
\end{lemma}

\begin{proof}
	For $f\in\calD'(\RR^\times)\otimeshat\calS'(\Lambda)$ of the form
	$$ f(\lambda,a,x) = (\lambda-i\sqrt{2}a)^k(\lambda^2+2a^2)^{\frac{s_\min-k}{2}}\exp\left(-\frac{iaI_3}{\lambda R}\right) f_3(S,T) $$
	with $f_3$ satisfying \eqref{eq:QuatStep2Eq2} and \eqref{eq:QuatStep2Eq3}, a lengthy computation shows that
	\begin{multline*}
		d\pi_\min(2T_0\pm i\sqrt{2}(C-\overline{D}))f(\lambda,a,x) = (\lambda-i\sqrt{2}a)^{k\pm1}(\lambda^2+2a^2)^{\frac{s_\min-(k\pm1)}{2}}\exp\left(-\frac{iaI_3}{\lambda R}\right)\\
		\times3i\Big[T\partial_S^2-S\partial_S\partial_T\mp T\partial_S\pm S\partial_T+s_\min\partial_T\Big]f_3(S,T).
	\end{multline*}
	If now $f_3(S,T)=\sum_m h_{k,m}K_m(r)e^{im\theta}$ with $(S,T)=(r\cos\theta,r\sin\theta)$ it can further be shown that
	\begin{multline*}
		\Big[T\partial_S^2-S\partial_S\partial_T\mp T\partial_S\pm S\partial_T+s_\min\partial_T\Big]f_3(S,T)\\
		= i\sum_m\left[\frac{m+s_\min}{2}h_{k,m-1}+\frac{m-s_\min}{2}h_{k,m+1}\pm mh_{k,m}\right]K_m(r)e^{im\theta}.
	\end{multline*}
	Finally, it can be shown using Lemma~\ref{lem:QuatStep5} that
	\begin{equation*}
		\frac{m+s_\min}{2}h_{k,m-1}+\frac{m-s_\min}{2}h_{k,m+1}\pm mh_{k,m} = (k\pm(s_\min+1))h_{k+1,m}.\qedhere
	\end{equation*}
\end{proof}

Combining the various lemmas, we obtain Theorem~\ref{thm:LKTQuat}.

Using the explicit formulas for the action of $\frakk$ on $W$, we can compute the action of some group elements on $W$. Recall the elements $w_0,w_1,w_2\in K$ from \eqref{eq:DefW0} and \eqref{eq:DefW1W2}.

\begin{corollary}\label{cor:ActionWeylSquaresLKTquat}
	The elements $w_0^2,w_1^2,w_2^2\in K$ act on $W$ in the following way:
	\begin{align*}
		\pi_\min(w_0^2)f_k &= (-1)^{n-k}f_{-k}, & \pi_\min(w_1^2)f_k &= (-1)^kf_k, & \pi_\min(w_2^2)f_k &= (-1)^nf_{-k}.
	\end{align*}
\end{corollary}

\begin{proof}
	From Lemma~\ref{lem:QuatStep2} and \ref{lem:QuatStep6} it follows that the $\su(2)$-triple $\widetilde{T}_1,\widetilde{T}_2,\widetilde{T}_3$ acts on $W$ by
	\begin{equation}
		d\pi_\min(\widetilde{T}_2)f_k = 2ikf_k, \qquad d\pi_\min(\widetilde{T}_3\pm i\widetilde{T}_1)f_k = 2i(s_\min+1\mp k)f_{k\mp1}.\label{eq:SU2TripleActionQuat}
	\end{equation}
	Since $\frakk_2$ acts trivially on $W$, we find
	\begin{align*}
		\pi_\min(w_0^2)f_k &= \pi_\min(\exp(-\tfrac{\pi}{2}\widetilde{T}_1))f_k,\\
		\pi_\min(w_1^2)f_k &= \pi_\min(\exp(\tfrac{\pi}{2}\widetilde{T}_2))f_k,\\
		\pi_\min(w_2^2)f_k &= \pi_\min(\exp(\tfrac{\pi}{2}\widetilde{T}_3))f_k.
	\end{align*}
	The formulas now follow from the representation theory of $\SU(2)$.
\end{proof}

\section{The split cases $\frakg=\frake_{6(6)},\frake_{7(7)},\frake_{8(8)}$}

Assume that the Jordan algebra $\calJ$ is simple, non-Euclidean and split, i.e. $\calJ$ is isomorphic to $\Herm(3,\FF_s)$ with $\FF\in\{\CC,\HH,\OO\}$, the Jordan algebra of $3\times3$ Hermitian matrices over the split complex numbers $\CC_s$, the split quaternions $\HH_s$ or the split octionions $\OO_s$. In this case, the group $G$ is split, and by the classification we have $\frakg\simeq\frake_{6(6)}$, $\frake_{7(7)}$ or $\frake_{8(8)}$.

The lowest $K$-type in this case turns out to be the trivial representation. It is spanned by a vector which is most easily described using a renormalization $\overline{K}_\alpha(x)=x^{-\frac{\alpha}{2}}K_\alpha(\sqrt{x})$ of the $K$-Bessel function (see Appendix~\ref{app:KBessel} for details).

\begin{theorem}\label{thm:LKTSplit}
The space $W=\CC f_0$\index{W1@$W$} with
$$ f_0(\lambda,a,x) = (\lambda^2+2a^2)^{\frac{s_\min}{2}}\exp\left(-\frac{2ian(x)}{\lambda(\lambda^2+2a^2)}\right)\overline{K}_{-\frac{s_\min+1}{2}}\Bigg(\frac{2R^2I_2+I_3^2-RI_4+2R^3}{2R^2}\Bigg),\index{fk@$f_k$} $$
where
$$ R=\lambda^2+2a^2, \qquad I_2=\omega(Jx,x), \qquad I_3=\omega(\Psi(x),B), \qquad I_4=\omega(\mu(x)Jx,Jx),\index{R@$R$}\index{I2@$I_2$}\index{I3@$I_3$}\index{I4@$I_4$} $$
is a $\frakk$-subrepresentation of $(d\pi_\min,\calD'(\RR^\times)\otimeshat\calS'(\Lambda))$ isomorphic to the trivial representation.
\end{theorem}

\begin{remark}
The spherical vector $f_0$ has previously been found in \cite[equation (4.72)]{KPW02} using case-by-case computations. Note that $-\frac{s_\min+1}{2}$ equals $\frac{1}{2},1,2$ for $\frakg=\frake_{6(6)},\frake_{7(7)},\frake_{8(8)}$.
\end{remark}

We prove this result in several steps. The following lemma is proven in a similar way as Lemma~\ref{lem:QuatStep0}:

\begin{lemma}\label{lem:SplitStep0}
The Lie algebra $\frakk$ is generated by $\frakg_0\cap\frakk$ and $\{v+\theta v:v\in\Lambda\}$.
\end{lemma}

\begin{lemma}\label{lem:SplitStep1}
$f\in\calD'(\RR^\times)\otimeshat\calS'(\Lambda)$ is $(\frakg_{(0,0)}\cap\frakk)$-invariant if and only if it is of the form
$$ f(\lambda,a,x) = f_1(\lambda,a,I_2,I_3,I_4), $$
where
$$ I_2=\omega(Jx,x), \qquad I_3=\omega(\Psi(x),B), \qquad I_4=\omega(\mu(x)Jx,Jx). $$
\end{lemma}

\begin{proof}
In a non-Euclidean split Jordan algebra $\calJ$ of degree $3$, the invariants under the maximal compact subgroup of the structure group corresponding to $\vartheta$ are generated by the polynomials $\tr(\vartheta(x)x)$, $\det(x)$ and $\tr(\vartheta(x)x\vartheta(x)x)$. These are essentially $I_2$, $I_3$ and $I_4$.
\end{proof}

\begin{lemma}\label{lem:SplitStep2}
$f\in\calD'(\RR^\times)\otimeshat\calS'(\Lambda)$ is additionally invariant under $A-\overline{B}$ if and only if it is, for $\lambda>0$ resp. $\lambda<0$, of the form
$$ f(\lambda,a,x) = \exp\left(-\frac{2ian(x)}{\lambda(\lambda^2+2a^2)}\right)f_2(R,I_2,I_3,I_4) $$
with
$$ R = \lambda^2+2a^2. $$
\end{lemma}

\begin{proof}
We have
$$ d\pi_\min(A-\overline{B}) = -\lambda\partial_A+2a\partial_\lambda-\frac{iI_3}{\lambda^2}. $$
As in the quaternionic case, the method of characteristics shows the claim.
\end{proof}

\begin{lemma}
$f\in\calD'(\RR^\times)\otimeshat\calS'(\Lambda)$ is additionally annihilated by the operators $\{\lambda\,d\pi_\min(v+\theta v)+2a\,d\pi_\min(B_\mu(v,B)+B_\mu(A,Jv)):v\in\calJ\}$ if and only if it is of the form
$$ f(\lambda,a,x) = (\lambda^2+2a^2)^{\frac{s_\min}{2}}\exp\left(-\frac{2ian(x)}{\lambda(\lambda^2+2a^2)}\right)f_3(Z) $$
with
$$ Z=\frac{2R^2I_2+I_3^2-RI_4+2R^3}{2R^2}. $$
\end{lemma}

\begin{proof}
We have
\begin{multline*}
 \lambda\,d\pi_\min(v+\theta v)+2a\,d\pi_\min(B_\mu(v,B)+B_\mu(A,Jv))\\
 = -\omega(x,Jv)(\lambda\partial_\lambda+a\partial_A-s_\min)-(\lambda^2+2a^2)\partial_v+\partial_{\mu(x)Jv}+\frac{ia}{\lambda}\omega(\mu(x)v,B).
\end{multline*}
Using
\begin{align*}
 \omega(\mu(x)Jx,J\mu(x)Jv) &= \frac{1}{2}I_3\omega(\mu(x)v,B)-\frac{1}{2}I_4\omega(x,Jv),\\
 \omega(\mu(x)\mu(x)Jv,B) &= 2I_3\omega(x,Jv),\\
 \omega(x,J\mu(x)Jv) &= \omega(\mu(x)Jx,Jv),
\end{align*}
we find that this equals
\begin{align*}
 \exp\left(-\frac{2ian(x)}{\lambda(\lambda^2+2a^2)}\right)\Bigg[
 &\omega(x,Jv)\Big(-2R\partial_Rf_2+2R\partial_2f_2-2I_3\partial_3f_2-2I_4\partial_4f_2\Big)\\
 &+\omega(\mu(x)v,B)\Big(R\partial_3f_2+2I_3\partial_4f_2\Big)\\
 &+\omega(\mu(x)Jx,Jv)\Big(-4R\partial_4f_2-2\partial_2f_2\Big)\Bigg],
\end{align*}
resulting in three first order differential equations for $f_2$. Solving all three using the method of characteristics yields
\begin{equation*}
 f_2(R,I_2,I_3,I_4) = R^{\frac{s_\min}{2}}f_3(Z).\qedhere
\end{equation*}
\end{proof}

\begin{lemma}
$f\in\calD'(\RR^\times)\otimeshat\calS'(\Lambda)$ is additionally invariant under $\{v+\theta v:v\in\calJ\}$ if and only if the function $f_3(Z)$ solves the differential equation
$$ Zf_3''(Z)+\frac{1-s_\min}{2}f_3'(Z)-\frac{1}{4}f_3(Z)=0. $$
\end{lemma}

\begin{proof}
We have
\begin{multline*}
 \lambda d\pi_\min(v+\theta v) = -\lambda^2\partial_v-\omega(x,Jv)\lambda\partial_\lambda+s_\min\omega(x,Jv)+\partial_{\mu(x)Jv}\\
 +ia\lambda\sum_{\alpha,\beta}\omega(B_\mu(A,Jv)\widehat{e}_\alpha,\widehat{e}_\beta)\partial_\alpha\partial_\beta.
\end{multline*}
Using
\begin{align*}
 \sum_{\alpha,\beta}\omega(B_\mu(A,Jv)\widehat{e}_\alpha,\widehat{e}_\beta)\omega(B_\mu(e_\alpha,e_\beta)x,B) ={}& s_\min\omega(x,Jv),\\
 \sum_{\alpha,\beta}\omega(B_\mu(A,Jv)\widehat{e}_\alpha,\widehat{e}_\beta)\omega(\mu(x)e_\alpha,B)\omega(\mu(x)e_\beta,B) ={}& 2I_3\omega(x,Jv),\\
 \sum_{\alpha,\beta}\omega(B_\mu(A,Jv)\widehat{e}_\alpha,\widehat{e}_\beta)\omega(\mu(x)e_\alpha,B)\omega(x,Je_\beta) ={}& \omega(\mu(x)Jx,Jv),\\
 \sum_{\alpha,\beta}\omega(B_\mu(A,Jv)\widehat{e}_\alpha,\widehat{e}_\beta)\omega(\mu(x)e_\alpha,B)\omega(\mu(x)Jx,Je_\beta) ={}& \frac{1}{2}I_3\omega(\mu(x)v,B)-\frac{1}{2}I_4\omega(x,Jv),\\
 \sum_{\alpha,\beta}\omega(B_\mu(A,Jv)\widehat{e}_\alpha,\widehat{e}_\beta)\omega(e_\beta,Je_\alpha) ={}& 0,\\
 \sum_{\alpha,\beta}\omega(B_\mu(A,Jv)\widehat{e}_\alpha,\widehat{e}_\beta)\omega(x,Je_\alpha)\omega(x,Je_\beta) ={}& \omega(\mu(x)v,B),\\
 \sum_{\alpha,\beta}\omega(B_\mu(A,Jv)\widehat{e}_\alpha,\widehat{e}_\beta)\omega(x,Je_\alpha)\omega(\mu(x)Jx,Je_\beta) ={}& \frac{1}{2}I_2\omega(\mu(x)v,B)+\frac{1}{2}I_3\omega(x,Jv),\\
 \sum_{\alpha,\beta}\omega(B_\mu(A,Jv)\widehat{e}_\alpha,\widehat{e}_\beta)\omega(\mu(x)Je_\alpha,Je_\beta) ={}& s_\min\omega(\mu(x)v,B),\\
 \sum_{\alpha,\beta}\omega(B_\mu(A,Jv)\widehat{e}_\alpha,\widehat{e}_\beta)\omega(B_\mu(x,e_\beta)Jx,Je_\alpha) ={}& -\frac{1}{2}\omega(\mu(x)v,B),\\
 \sum_{\alpha,\beta}\omega(B_\mu(A,Jv)\widehat{e}_\alpha,\widehat{e}_\beta)\omega(\mu(x)Jx,Je_\alpha)\omega(\mu(x)Jx,Je_\beta) ={}& I_2I_3\omega(x,Jv) - \frac{1}{2}I_4\omega(\mu(x)v,B)\\
 & - I_3\omega(\mu(x)Jx,Jv),
\end{align*}
we find that
\begin{equation*}
 \lambda d\pi_\min(v+\theta v)f = \frac{ia\lambda}{R^2}\Big(I_3\omega(x,Jv)+R\omega(\mu(x)v,B)\Big) \Big(4Z\partial_Z^2f_3+2(1-s_\min)\partial_Zf_3-f_3\Big).\qedhere
\end{equation*}
\end{proof}

By Appendix~\ref{app:KBessel}, the unique tempered solution to this ordinary differential equation is the renormalized $K$-Bessel function $\overline{K}_\alpha(Z)$ with $\alpha=-\frac{s_\min+1}{2}$. This shows Theorem~\ref{thm:LKTSplit}.

\begin{corollary}\label{cor:ActionWeylSquaresLKTsplit}
	The elements $w_0^2,w_1^2,w_2^2\in K$ act on $W$ in the following way:
	$$ \pi_\min(w_0^2)f_0 = \pi_\min(w_1^2)f_0 = \pi_\min(w_2^2)f_0 = f_0. $$
\end{corollary}

\section{The case $\frakg=\frakg_{2(2)}$}

Let $\frakg=\frakg_{2(2)}$, then $\calJ=\RR C$ is one-dimensional and therefore, strictly speaking, not of rank three. We treat this case separately.

First note that the Lie algebra $\frakk$ splits into the direct sum of two ideals
$$ \frakk = \frakk_1\oplus\frakk_2 $$
with $\frakk_1,\frakk_2\simeq\su(2)$ given by
$$ \frakk_1 = \RR T_1\oplus\RR T_2\oplus\RR T_3, \qquad \frakk_2 = \RR S_1\oplus\RR S_2\oplus\RR S_3,\index{k31@$\frakk_1$}\index{k32@$\frakk_2$} $$
where $T_1,T_2,T_3$ are as in Proposition~\ref{prop:SU2Ideal} and
$$ S_1 = 2T_0+3(E-F), \qquad S_2 = 3A+2D+\theta(3A+2D), \qquad S_3 = 3B-2C+\theta(3B-2C).\index{S1@$S_1$}\index{S2@$S_2$}\index{S3@$S_3$} $$

For simplicity, we write $x\in\calJ$ as $x=cC$ and $f(\lambda,a,x)=f(\lambda,a,c)$.

\begin{theorem}\label{thm:LKTG2}
The space $W=\RR f_{-1}\oplus\RR f_0\oplus\RR f_1$\index{W1@$W$} with
\begin{align*}
 f_0(\lambda,a,c) ={}& (\lambda^2+2a^2)^{-\frac{1}{6}}\exp\left(-\frac{iac^3}{2\lambda(\lambda^2+2a^2)}\right)\overline{K}_{-\frac{1}{3}}(S),\\
 f_{\pm1}(\lambda,a,c) ={}& (\lambda\mp i\sqrt{2}a)(\lambda^2+2a^2)^{-\frac{7}{6}}\exp\left(-\frac{iac^3}{2\lambda(\lambda^2+2a^2)}\right)\\
 & \hspace{5cm}\times\left[c\overline{K}_{-\frac{1}{3}}(S)\mp\frac{\sqrt{2}(2\lambda^2+4a^2+c^2)^2}{4(\lambda^2+2a^2)}\overline{K}_{\frac{2}{3}}(S)\right],\index{fk@$f_k$}
\end{align*}
where
$$ S = \frac{(2\lambda^2+4a^2+c^2)^3}{8(\lambda^2+2a^2)^2},\index{S@$S$} $$
is a $\frakk$-subrepresentation of $(d\pi_\min,\calD'(\RR^\times)\otimes\calS'(\Lambda))$ isomorphic to the representation $\CC\boxtimes S^2(\CC^2)$ of $\frakk\simeq\su(2)\oplus\su(2)$.
\end{theorem}

\begin{remark}
In \cite[equation (3.119)]{GNPP08}, one vector in the $K$-type $W$ is obtained, but the formula differs slightly from ours. Comparing to the other cases, our formula looks more natural than the one in \cite{GNPP08} which contains an additional transcendental function.
\end{remark}

To find the $K$-type $\CC\boxtimes S^2(\CC^2)$ on which $\frakk_1$ acts trivially and $\frakk_2\simeq\su(2)$ acts by the three-dimensional representation $S^2(\CC^2)$, we write $S^2(\CC^2)$ as the direct sum of weight spaces relative to the maximal torus $\RR S_2\subseteq\frakk_2$. Since $0$ is a weight for $S^2(\CC^2)$, we try to find a $\frakk_1$-invariant vector which is additionally $S_2$-invariant.

\begin{lemma}\label{lem:G2Step1}
$f\in\calD'(\RR^\times)\otimeshat\calS'(\Lambda)$ is invariant under $S_2+T_2$ if and only if it is, for $\lambda>0$ resp. $\lambda<0$, of the form
$$ f(\lambda,a,c) = \exp\left(-\frac{iac^3}{2\lambda(\lambda^2+2a^2)}\right)f_1(R,c), $$
where $R=\lambda^2+2a^2$.
\end{lemma}

\begin{proof}
We have $S_2+T_2=4(A-\overline{B})$ and
$$ d\pi_\min(A-\overline{B}) = -\lambda\partial_A+2a\partial_\lambda-\frac{iI_3}{\lambda^2}, $$
where $I_3=2n(x)=\frac{1}{2}c^3$. Applying the method of characteristics shows the claim.
\end{proof}

\begin{lemma}\label{lem:G2Step2}
$f\in\calD'(\RR^\times)\otimeshat\calS'(\Lambda)$ is additionally annihilated by $d\pi_\min(S_2-3T_2)$ and $\lambda d\pi_\min(T_1)-a\,d\pi_\min(T_3)$ if and only if the function $f_1(R,c)$ satisfies
$$ \left[2R\partial_R\partial_C+\frac{1}{3}c\partial_C^2+\frac{4}{3}\partial_C+\frac{3c}{8R^2}(c^2+2R)(c^2-2R)\right]f_1 = 0 $$
and
$$ \left[\frac{2R}{3}\partial_C^2-4R^2\partial_R^2-6R\partial_R+\frac{2}{3}c\partial_C-\frac{2}{9}+\frac{8R^3-24c^2R^2-6c^4R+4c^6}{8R^2}\right]f_1 = 0. $$
\end{lemma}

\begin{proof}
This is an elementary computation.
\end{proof}

Inspired by the previous cases we make the Ansatz $f_1(R,c)=R^{-\frac{1}{6}}f_2(S)$ with
$$ S=\frac{(2R+c^2)^3}{8R^2}. $$

\begin{lemma}\label{lem:G2Step3}
	A function of the form $f_1(R,c)=R^{-\frac{1}{6}}f_2(S)$ satisfies the differential equations in Lemma~\ref{lem:G2Step2} if and only if $f_2$ satisfies
	$$ Tf_2''(S) +\frac{2}{3}f_2'(S)-\frac{1}{4}f_2(S) = 0. $$
\end{lemma}

\begin{proof}
	Another elementary computation.
\end{proof}

From Appendix~\ref{app:KBessel} we know that $f_2(S)=\overline{K}_{-\frac{1}{3}}(S)$ is the unique tempered solution to the above differential equation. This leads to the function $f_0(\lambda,a,x)$. We now apply $S_3\pm i\sqrt{2}S_1$ to $f_0$.

\begin{lemma}\label{lem:G2Step4}
Let
$$ f(\lambda,a,c) = (\lambda^2+2a^2)^{-\frac{1}{6}}\exp\left(-\frac{iac^3}{2\lambda(\lambda^2+2a^2)}\right)f_2(S) $$
be invariant under $T_1,T_2,T_3$ and $S_2$. Then
\begin{multline*}
 d\pi_\min(S_3\pm i\sqrt{2}S_1)f = -8(\lambda\mp i\sqrt{2}a)(\lambda^2+2a^2)^{-\frac{7}{6}}\exp\left(-\frac{iac^3}{2\lambda(\lambda^2+2a^2)}\right)\\
 \times\left(\frac{c}{2}f_2(S)\pm\frac{\sqrt{2}(2R+c^2)^2}{4R}f_2'(S)\right).
\end{multline*}
\end{lemma}

\begin{proof}
Since $f$ is invariant under $T_1$ and $T_3$, we find that
$$ d\pi_\min(S_1)f = 8\,d\pi_\min(T)f \qquad \mbox{and} \qquad d\pi_\min(S_3)f = -8\,d\pi_\min(C-\overline{D}). $$
The rest is a simple computation.
\end{proof}

Since $\overline{K}_{-\frac{1}{3}}'(x)=-\frac{1}{2}\overline{K}_{\frac{2}{3}}(x)$ by \eqref{eq:BesselDerivative1}, this leads to the functions $f_1$ and $f_{-1}$.

\begin{proposition}\label{prop:G2Step5}
	The functions $f_{-1}$, $f_0$ and $f_1$ are $\frakk_1$-invariant and transform in the following way under the action of $\frakk_2$:
	\begin{align*}
		d\pi_\min(S_2)f_k &= 4\sqrt{2}ikf_k & d\pi_\min(S_3\pm i\sqrt{2}S_1)f_0 &= -4f_{\pm1},\\
		d\pi_\min(S_3\pm i\sqrt{2}S_1)f_{\pm1} &= 0, & d\pi_\min(S_3\mp i\sqrt{2}S_1)f_{\pm1} &= 16f_0.
	\end{align*}
\end{proposition}

\begin{proof}
	The first formula follows as in the proof of Lemma~\ref{lem:G2Step1}. The second formula is essentially Lemma~\ref{lem:G2Step4}. The third and the fourth formula are easily computed using
	$$ S_3\pm i\sqrt{2}S_1 \equiv 8(-C+\overline{D}\pm i\sqrt{2}T_0) \mod \frakk_1 $$
	and \eqref{eq:BesselDerivative1} and \eqref{eq:BesselDerivative2}.
\end{proof}

\begin{proof}[Proof of Theorem~\ref{thm:LKTG2}]
Note that the elements
$$ \widetilde{S}_1 = \frac{1}{2}S_1, \qquad \widetilde{S}_2 = \frac{1}{2\sqrt{2}}S_2 \qquad \mbox{and} \qquad \widetilde{S}_3 = -\frac{1}{2\sqrt{2}}S_3 $$
form an $\su(2)$-triple, then the statement follows from Proposition~\ref{prop:G2Step5}.
\end{proof}

\begin{corollary}\label{cor:ActionWeylSquaresLKTG2}
The elements $w_0^2,w_1^2,w_2^2\in K$ act on $W$ in the following way:
\begin{align*}
	\pi_\min(w_0^2)f_k &= (-1)^{k-1}f_{-k}, & \pi_\min(w_1^2)f_k &= (-1)^kf_k, & \pi_\min(w_2^2)f_k &= -f_{-k}.
\end{align*}
\end{corollary}

\begin{proof}
Since $\frakk_1$ acts trivially on $W$, we find
\begin{align*}
	\pi_\min(w_0^2)f_k &= \pi_\min(\exp(\tfrac{\pi}{2}\widetilde{S}_1))f_k,\\
	\pi_\min(w_1^2)f_k &= \pi_\min(\exp(\tfrac{\pi}{2}\widetilde{S}_2))f_k,\\
	\pi_\min(w_2^2)f_k &= \pi_\min(\exp(-\tfrac{\pi}{2}\widetilde{S}_3))f_k.
\end{align*}
The rest is an $\SU(2)$ computation.
\end{proof}

\section{The case $\frakg=\sl(n,\RR)$}\label{sec:LKTsln}

Let $\frakg=\sl(n,\RR)$, then $\calJ\simeq\RR^{n-3}$ does not carry the structure of a Jordan algebra. In this case, it is not necessary to use the decomposition $\frakg_{-1}=\RR A\oplus\calJ\oplus\calJ^*\oplus\RR B$. We choose any Cartan involution $\theta$ on $\frakg$ associated to a map $J:V\to V$ as in Lemma~\ref{lem:CartanInvFromJ}, and let $\frakk=\frakg^\theta\simeq\so(n)$ denote the maximal compact subalgebra. We choose complementary Lagrangian subspaces $\Lambda$ and $\Lambda^*$ of $V$ as in Section~\ref{sec:FTpictureMinRepSLn}.

In contrast to the previous cases, we have a continuous family $(d\pi_{\min,r},\calD'(\RR^\times)\otimeshat\calS'(\Lambda))$ ($r\in\CC$) of representations, and for each $r\in\CC$ we find two non-equivalent $(\frakg,K)$-submodules of $\calD'(\RR^\times)\otimeshat\calS'(\Lambda)$ which, for $r\in i\RR$, integrate to non-equivalent irreducible unitary representations of $L^2(\RR^\times\times\Lambda)$. In this section we determine their lowest $K$-types.

\begin{theorem}\label{thm:LKTSLn}
	Let $r\in\CC$.
	\begin{enumerate}[(1)]
		\item\label{thm:LKTSLn1} The space $W_{0,r}=\CC f_{0,r}$\index{W10r@$W_{0,r}$} with
		$$ f_{0,r}(\lambda,x) = \overline{K}_{\frac{n-2r-2}{4}}(\lambda^2+4|x|^2)\index{f0r@$f_{0,r}$} $$
		is a $\frakk$-subrepresentation of $(d\pi_{\min,r},\calD'(\RR^\times)\otimeshat\calS'(\Lambda))$ isomorphic to the trivial representation.
		\item\label{thm:LKTSLn2} The space $W_{1,r}=\CC g_{0,r}\oplus\CC g_{1,r}\oplus\{g_{w,r}:w\in\Lambda^*\}$\index{W11r@$W_{1,r}$} with
		\begin{align*}
			 g_{0,r}(\lambda,x) &= \lambda\overline{K}_{\frac{n-2r}{4}}(\lambda^2+4|x|^2)\index{g20r@$g_{0,r}$}\\
			 g_{1,r}(\lambda,x) &= \overline{K}_{\frac{n-2r-4}{4}}(\lambda^2+4|x|^2)\index{g21r@$g_{1,r}$}\\
			 g_{w,r}(\lambda,x) &= \omega(x,w)\overline{K}_{\frac{n-2r}{4}}(\lambda^2+4|x|^2) && (w\in\Lambda^*)\index{g2wr@$g_{w,r}$}
		\end{align*}
		is a $\frakk$-subrepresentation of $(d\pi_\min,\calD'(\RR^\times)\otimeshat\calS'(\Lambda))$ isomorphic to the standard representation of $\frakk\simeq\so(n)$ on $\CC^n$.
	\end{enumerate}
\end{theorem}

\begin{proof}
We first find the spherical vector $f_{0,r}$. If $f_{0,r}$ is invariant under $\frakk\cap\frakg_0=\so(n-2)$, it has to be of the form
$$ f_{0,r}(\lambda,x) = f_1(\lambda,I_2) \qquad \mbox{with} \qquad I_2=|x|^2=\frac{1}{4}\omega(Jx,x). $$
For such $f_{0,r}$ the equation $d\pi_{\min,r}(v+\theta v)f_{0,r}=0$ takes the form
$$ \omega(x,Jv)\left(\frac{1}{2}\lambda\partial_2-\partial_\lambda\right)f_1 = 0, $$
so that $f_1(\lambda,I_2)=f_2(\lambda^2+4I_2)$. Finally, the equation $d\pi_{\min,r}(w+\theta w)=0$ reduces to
$$ 4zf_2''+(n-2r+2)f_2'-f_2=0, $$
which has, by Appendix~\ref{app:KBessel}, the unique tempered solution $f_2(z)=\overline{K}_{\frac{n-2r-2}{4}}(z)$.\\
Now, let us show that $g_{0,r}$, $g_{1,r}$ and the $g_{w,r}$'s span a finite-dimensional $\frakk$-representation. Similar to \eqref{thm:LKTSLn1}, one verifies for $T\in\frakk\cap\frakg_0\simeq\so(p-2)$:
$$ d\pi_\min(T)g_{w,r} = g_{Tw,r}, \qquad d\pi_\min(T)g_{0,r} = d\pi_\min(T)g_{1,r}=0, $$
for $v\in\Lambda$:
$$ d\pi_\min(v+\theta v)g_{0,r} = -g_{Jv,r}, \quad d\pi_\min(v+\theta v)g_{1,r} = 0, \quad d\pi_\min(v+\theta v)g_{w,r} = -\omega(v,w)g_{0,r}, $$
and for $w\in\Lambda^*$:
$$ d\pi_\min(w+\theta w)g_{0,r} = 0, \quad d\pi_\min(w+\theta w)g_{1,r} = -ig_{w,r}, \quad d\pi_\min(w+\theta w)g_{w',r} = i\omega(w',Jw)g_{1,r}, $$
where we have used \eqref{eq:BesselDerivative1} and \eqref{eq:BesselDerivative2}.
\end{proof}

\section{The case $\frakg=\sl(3,\RR)$}

For $\frakg=\sl(3,\RR)$ there is an additional $(\frakg,K)$-module contained in $(d\pi_{\min,r},\calD'(\RR^\times)\otimeshat\calS'(\Lambda))$ for $r=0$, giving rise to a genuine irreducible unitary representation of the double cover $\widetilde{SL}(3,\RR)$ of $\SL(3,\RR)$. For this we write $\Lambda=\RR A$ and $\Lambda^*=\RR B$ such that $JA=-B$ and $JB=A$ and use coordinates $x=aA\in\Lambda$.

\begin{theorem}\label{thm:LKTSL3}
The space $W_{\frac{1}{2}}=\CC h_{-\frac{1}{2}}\oplus\CC h_{\frac{1}{2}}$\index{W12@$W_{\frac{1}{2}}$} with
\begin{align*}
 h_{\frac{1}{2}}(\lambda,a) &= (|\lambda|-i\sqrt{2}\sgn(\lambda)a)^{\frac{1}{2}}\overline{K}_{\frac{1}{2}}(\lambda^2+2a^2),\\
 h_{-\frac{1}{2}}(\lambda,a) &= \sgn(\lambda)(|\lambda|+i\sqrt{2}\sgn(\lambda)a)^{\frac{1}{2}}\overline{K}_{\frac{1}{2}}(\lambda^2+2a^2)\index{h312@$h_{\pm\frac{1}{2}}$}
\end{align*}
is a $\frakk$-subrepresentation of $(d\pi_{\min,0},\calD'(\RR^\times)\otimeshat\calS'(\Lambda))$ isomorphic to the representation $S^1(\CC^2)$ of $\frakk=\so(3)\simeq\su(2)$.
\end{theorem}

\begin{proof}
The elements $U_1=\sqrt{2}(A-\overline{B})$, $U_2=\sqrt{2}(B+\overline{A})$ and $U_3=-2(E-F)$ form an $\su(2)$-triple, i.e. $[U_1,U_2]=2U_3$, $[U_2,U_3]=2U_1$ and $[U_3,U_1]=2U_2$. With respect to the maximal torus $\RR U_1$ in $\frakk$, the vectors $U_2\mp iU_3$ are root vectors:
$$ [U_1,U_2\mp iU_3] = \pm2i(U_2\mp iU_3). $$
The highest weight vector $h$ of a $\frakk$-type isomorphic to $S^1(\CC^2)$ solves the weight equation
$$ d\pi_{\min,0}(U_1)h = \sqrt{2}(2a\partial_\lambda-\lambda\partial_A)h = ih. $$
Using the method of characteristics, we find that
$$ h(\lambda,a) = (|\lambda|-i\sqrt{2}\sgn(\lambda)a)^{\frac{1}{2}}\cdot u(\lambda^2+2a^2) $$
is a solution. The highest weight equation
$$ d\pi_{\min,0}(U_2-iU_3)h = 0 $$
then gives
$$ 4zu''(z)+6u'(z)-u(z) = 0. $$
By Appendix~\ref{app:KBessel}, the unique tempered solution is $u(z)=\overline{K}_{\frac{1}{2}}(z)$, which leads to the highest weight vector $h_{\frac{1}{2}}$. A straightforward computation using $\overline{K}_{\frac{1}{2}}(x)=\sqrt{\frac{\pi}{2}}x^{-\frac{1}{2}}e^{-\sqrt{x}}$ (see Appendix~\ref{app:KBessel}) shows that
\begin{equation*}
 d\pi_{\min,0}(U_2+iU_3)h_{\frac{1}{2}} = -2h_{-\frac{1}{2}} \qquad \mbox{and} \qquad d\pi_{\min,0}(U_2+iU_3)h_{-\frac{1}{2}}=0.\qedhere
\end{equation*}
\end{proof}

\begin{remark}
We observe that these functions together with the ones from Theorem~\ref{thm:LKTSLn} in the case $n=3$ agree (up to a change of coordinates) with the ones found by Torasso \cite[Proposition 14, 15 \& 16]{Tor83}.
\end{remark}

\section{The case $\frakg=\so(p,q)$}\label{sec:LKTSOpq}

The construction of a Cartan involution in Section~\ref{sec:CartanInvolutions} does not apply in the case $\frakg\simeq\so(p,q)$, since here $\calJ$ is not simple but the direct sum of a rank one Jordan algebra $\calJ_0\simeq\RR$ and a rank two Jordan algebra $\overline{\calJ}\simeq\RR^{p-3,q-3}$. We therefore give a separate construction.

According to the decomposition $\calJ=\calJ_0\oplus\overline{\calJ}$, the norm function decomposes into
$$ n(x) = -\frac{1}{2}\omega(x_0,Q)\omega(\mu(\overline{x})P,B) $$
where $-\omega(\mu(\overline{x})P,B)$ is a quadratic form on $\overline{\calJ}$ of signature $(p-3,q-3)$, the norm function of the quadratic Jordan algebra $\overline{\calJ}\simeq\RR^{p-3,q-3}$. The following result can be proven using the explicit decompositions in Appendix~\ref{app:SOpq}:

\begin{proposition}\label{prop:CartanInvSOpq}
Let $\vartheta:\overline{\calJ}\to\overline{\calJ}$\index{1htheta@$\vartheta$} be a Jordan algebra automorphism such that the symmetric bilinear form $(v_1,v_2)\mapsto-\omega(B_\mu(v_1,\vartheta v_2)P,B)$ is positive definite. Then the map $J:V\to V$ given by
\begin{align*}
	JA &= -B, & JP &= -Q, & Jv &= -\sqrt{2}B_\mu(P,B)\vartheta v && (v\in\overline{\calJ}),\\
	JB &= A, & JQ &= P, & Jw &= \sqrt{2}\vartheta B_\mu(A,Q)w && (w\in\overline{\calJ}^*),
\end{align*}
satisfies the conditions of Lemma~\ref{lem:CartanInvFromJ}.
\end{proposition}

We fix $\vartheta$ as in the proposition and let $\calJ_1$\index{J31@$\calJ_1$} denote the $+1$ eigenspace and $\calJ_2$\index{J32@$\calJ_2$} the $-1$ eigenspace of $\vartheta$ in $\overline{\calJ}$. Then $\calJ=\calJ_0\oplus\calJ_1\oplus\calJ_2$, and using the symplectic form we obtain a dual decomposition $\calJ^*=\calJ_0^*\oplus\calJ_1^*\oplus\calJ_2^*$\index{J31star@$\calJ_1^*$}\index{J32star@$\calJ_2^*$}. Again, using the explicit decompositions in Appendix~\ref{app:SOpq}, one verifies:

\begin{lemma}
	\begin{enumerate}[(1)]
		\item For each $i=1,2$, $B_\mu(P,B)$ resp. $B_\mu(A,Q)$ restricts to a linear isomorphism $\calJ_i\to\calJ_i^*$ resp. $\calJ_i^*\to\calJ_i$.
		\item $B_\mu(\calJ_1,\calJ_2)=0$ and $\mu(\calJ_1),\mu(\calJ_2)\in\RR B_\mu(A,Q)$.
	\end{enumerate}
\end{lemma}

Denote by $\theta$ the corresponding Cartan involution of $\frakg$ and by $\frakk=\frakg^\theta$ the corresponding maximal compact subalgebra of $\frakg$. Then
$$ T_0=B_\mu(A,Q)-B_\mu(P,B) \in \frakk.\index{T0@$T_0$} $$

\begin{proposition}\label{prop:SOpqKstructure}
The Lie algebra $\frakk$ decomposes into the sum of two ideals $\frakk=\frakk_1\oplus\frakk_2$ with $\frakk_1\simeq\so(p)$ and $\frakk_2\simeq\so(q)$ given by
	\begin{multline*}
	\frakk_1 = \RR(2T_0+\sqrt{2}(E-F))\oplus\RR(A-\sqrt{2}Q+\theta(A-\sqrt{2}Q))+\RR(B+\sqrt{2}P+\theta(B+\sqrt{2}P))\\
	\oplus\{B_\mu(v,B)+B_\mu(A,Jv):v\in\calJ_1\}\oplus\{x+\theta x:x\in\calJ_1\oplus\calJ_1^*\}\oplus\so(p-3)\index{k31@$\frakk_1$}
	\end{multline*}
	and
	\begin{multline*}
	\frakk_2 = \RR(2T_0-\sqrt{2}(E-F))\oplus\RR(A+\sqrt{2}Q+\theta(A+\sqrt{2}Q))+\RR(B-\sqrt{2}P+\theta(B-\sqrt{2}P))\\
	\oplus\{B_\mu(v,B)+B_\mu(A,Jv):v\in\calJ_2\}\oplus\{x+\theta x:x\in\calJ_2\oplus\calJ_2^*\}\oplus\so(q-3),\index{k32@$\frakk_2$}
	\end{multline*}
	where $\so(p-3)$ resp. $\so(q-3)$ denotes the ideal of $\frakk\cap\frakg_{(0,0)}\simeq\so(p-3)\oplus\so(q-3)$ which acts trivially on $\calJ_2\simeq\RR^{q-3}$ resp. $\calJ_1\simeq\RR^{p-3}$.
\end{proposition}

\begin{proof}
Note that
$$ T_0A=-P, \qquad T_0P=\frac{1}{2}A, \qquad T_0Q=\frac{1}{2}B, \qquad T_0B=-Q. $$
The rest is along the same lines as the proof of Proposition~\ref{prop:SU2Ideal}.
\end{proof}

To state explicit formulas for $K$-finite vectors we first need the following result:

\begin{lemma}\label{lem:LKTSOpqUniquePolys}
For $j\in\NN$ and $k\in\ZZ$ there exists a unique family of polynomials $(p_{j,k,m})_{m=0,\ldots,j}\subseteq\CC[S,T]$\index{p2jkmST@$p_{j,k,m}(S,T)$} in two variables satisfying
\begin{enumerate}[(1)]
	\item\label{lem:LKTSOpqUniquePolys1} $(2\partial_S\partial_T+k\sqrt{2})p_{j,k,m}-T\partial_Sp_{j,k,m-1}-2S\partial_Tp_{j,k,m+1}=0$,
	\item\label{lem:LKTSOpqUniquePolys2} $\partial_T^2p_{j,k,m+1}+(m-T\partial_T)p_{j,k,m}=0$,
	\item\label{lem:LKTSOpqUniquePolys3} $p_{j,k,0}(S,T)=S^j$.
\end{enumerate}
For the family of polynomials $(p_{j,k,m})_{j,k,m}$ the following identities hold:
\begin{enumerate}[(a)]
	\item\label{lem:LKTSOpqUniquePolysA} $\partial_Sp_{j,k,m}\mp\sqrt{2}\partial_Tp_{j,k,m+1}=(j\mp k)p_{j-1,k\pm1,m}$,
	\item\label{lem:LKTSOpqUniquePolysB} $\pm\frac{\sqrt{2}}{2}Tp_{j,k,m-1}+Sp_{j,k,m}\mp\sqrt{2}\partial_Tp_{j,k,m}=p_{j+1,k\pm1,m}$,
	\item\label{lem:LKTSOpqUniquePolysC} $(S\partial_S-T\partial_T)p_{j,k,m}=(j-2m)p_{j,k,m}$,
	\item\label{lem:LKTSOpqUniquePolysD} $-\frac{T^2}{2}p_{j,k,m-2}+(2m-1)p_{j,k,m-1}+S^2p_{j,k,m}=p_{j+2,k,m}$,
	\item\label{lem:LKTSOpqUniquePolysE} $p_{j,k,m}(-S,T)=(-1)^jp_{j,-k,m}(S,T)$.
\end{enumerate}
\end{lemma}

\begin{proof}
We first show uniqueness. For this we write $p_m=p_{j,k,m}$ for short. Every polynomial $p_m$ can be written as the sum of homogeneous polynomials
$$ p_m = \sum_{\alpha\geq0} p_m^\alpha $$
with $p_m^\alpha$ homogeneous of degree $\alpha$. Let $\alpha$ be maximal with $p_m^\alpha\neq0$ for some $m$. We claim that $\alpha=j$. By \eqref{lem:LKTSOpqUniquePolys3} we have $\alpha\geq j$, so we assume $\alpha>j$. Then \eqref{lem:LKTSOpqUniquePolys2} would imply $p_m^\alpha=c_m^\alpha S^{\alpha-m}T^m$, and \eqref{lem:LKTSOpqUniquePolys1} would imply
$$ k\sqrt{2}c_m^\alpha-(\alpha-m+1)c_{m-1}^\alpha-2(m+1)c_{m+1}^\alpha=0. $$
From \eqref{lem:LKTSOpqUniquePolys3} we know that $c_0^\alpha=0$, and recursively we find $c_m^\alpha=0$ for all $m$, which is a contradiction, so $\alpha=j$ is maximal with the property that $p_m^\alpha\neq0$ for some $m$. The previous argument also shows that $p_m^j=c_m^jS^{j-m}T^m$ with $c_m^j$ uniquely determined by $c_0^j=1$. For the lower order terms we observe that $p_0^\alpha=0$ for $\alpha<j$. For fixed $\alpha<j$, equation \eqref{lem:LKTSOpqUniquePolys1} determines $\partial_Tp_{m+1}^\alpha$ from $p_{m-1}^\alpha$, $p_m^\alpha$ and $p_m^{\alpha+2}$, so $p_{m+1}^\alpha$ is unique modulo polynomials in $S$ independent of $T$. This disambiguity is removed by \eqref{lem:LKTSOpqUniquePolys2} for $m>0$.\\
Now, let us prove existence by induction on $j$. For $j=0$ we also have $k=m=0$ and $p_{0,0,0}=1$ by \eqref{lem:LKTSOpqUniquePolys3}, which also satisfies \eqref{lem:LKTSOpqUniquePolys1} and \eqref{lem:LKTSOpqUniquePolys2}. Next, we note that the left hand side of \eqref{lem:LKTSOpqUniquePolysB} satisfies \eqref{lem:LKTSOpqUniquePolys1}, \eqref{lem:LKTSOpqUniquePolys2} and \eqref{lem:LKTSOpqUniquePolys3} for $j$ replaced by $j+1$ and $k$ replaced by $k\pm1$. Therefore, \eqref{lem:LKTSOpqUniquePolysB} can be used to recursively define the family $(p_{j+1,k,m})_{k,m}$ using $(p_{j,k,m})_{k,m}$, which establishes the existence part of the proof.\\
Finally, the identities \eqref{lem:LKTSOpqUniquePolysA}, \eqref{lem:LKTSOpqUniquePolysB}, \eqref{lem:LKTSOpqUniquePolysC}, \eqref{lem:LKTSOpqUniquePolysD} and \eqref{lem:LKTSOpqUniquePolysE} are proven by showing that the left hand side satisfies \eqref{lem:LKTSOpqUniquePolys1} and \eqref{lem:LKTSOpqUniquePolys2} for certain values of $j$ and $k$ (for \eqref{lem:LKTSOpqUniquePolysC} the term $2mp_{j,k,m}$ has to be moved to the left hand side first), and then using the previously established uniqueness result.
\end{proof}

\begin{remark}
	It is easy to see that
	$$ p_{j,k,m}(S,T) = \const\times S^{j-m}T^m + \const\times S^{j-m-1}T^{m-1} + \cdots. $$
	For $k=\pm j$ it is possible to find the coefficient of $S^{j-m}T^m$:
	$$ p_{j,\pm j,m}(S,T) = (\pm1)^m2^{-\frac{m}{2}}{j\choose m}S^{j-m}T^m + \mbox{lower order terms}. $$
	However, we were not able to find a closed formula in general.
\end{remark}

We assume from now on that $p\geq q\geq3$. The lowest $K$-type turns out to be isomorphic to $\CC\boxtimes\calH^{\frac{p-q}{2}}(\RR^q)$ as a representation of $\frakk\simeq\so(p)\oplus\so(q)$, where $\calH^\alpha(\RR^n)$\index{H2alphaRn@$\calH^\alpha(\RR^n)$} denotes the space of homogeneous polynomials on $\RR^n$ of degree $\alpha$ which are harmonic. Since only the subalgebra $\frakk_2\cap\frakg_{(0,0)}\simeq\so(q-3)$ acts geometrically in the representation $d\pi_\min$, it is helpful to use the following multiplicity-free branching rule:
$$ \calH^{\frac{p-q}{2}}(\RR^q)|_{\so(2)\oplus\so(q-2)} \simeq \bigoplus_{j=0}^{\frac{p-q}{2}}\bigoplus_{\substack{k=-j\\k\equiv j\mod2}}^j\CC_k\boxtimes\calH^{\frac{p-q}{2}-j}(\RR^{q-2}), $$
where $\CC_k$\index{Ck@$\CC_k$} denotes the obvious character of $\so(2)$, and further decompose
$$ \calH^{\frac{p-q}{2}-j}(\RR^{q-2})|_{\so(q-3)} \simeq \bigoplus_{\ell=0}^{\frac{p-q}{2}-j}\calH^\ell(\RR^{q-3}). $$
Together we find that
\begin{equation*}
 \calH^{\frac{p-q}{2}}(\RR^q)|_{\so(2)\oplus\so(q-3)} \simeq \bigoplus_{j=0}^{\frac{p-q}{2}}\Bigg(\bigoplus_{\substack{k=-j\\k\equiv j\mod2}}^j\CC_k\Bigg)\boxtimes\Bigg(\bigoplus_{\ell=0}^{\frac{p-q}{2}-j}\calH^\ell(\RR^{q-3})\Bigg).
\end{equation*}
For a distribution $f\in\calD'(\RR^\times)\otimeshat\calS'(\Lambda)$ we write
$f\otimes\calH^\ell(\RR^{q-3})$ for the space of distributions $f\otimes\varphi$, $\varphi\in\calH^\ell(\RR^{q-3})$, given by
$$ (f\otimes\varphi)(\lambda,a,x) = f(\lambda,a,x)\varphi(x_2). $$
Here, we define spherical harmonics with respect to the positive definite quadratic form $x_2\mapsto\omega(\mu(x_2)P,B)$ on $\calJ_2\simeq\RR^{q-3}$.

We further recall the renormalized $K$-Bessel function $\overline{K}_\alpha(z)$ from Appendix~\ref{app:KBessel}.

\begin{theorem}\label{thm:LKTSOpq}
Let $\frakg=\so(p,q)$ with $p\geq q\geq3$ and $p+q$ even. Then
$$ W = \bigoplus_{j=0}^{\frac{p-q}{2}}\bigoplus_{\substack{k=-j\\k\equiv j\mod2}}^j\bigoplus_{\ell=0}^{\frac{p-q}{2}-j}f_{j,k,\ell}\otimes\calH^\ell(\RR^{q-3})\index{W1@$W$} $$
with
$$ f_{j,k,\ell}(\lambda,a,x) = (\lambda-i\sqrt{2}a)^k(\lambda^2+2a^2)^{-\frac{k+1}{2}}\exp\left(-\frac{2ian(x)}{\lambda(\lambda^2+a^2)}\right)h_{j,k,\ell}(S,T,U)\index{fjkl@$f_{j,k,\ell}$} $$
and
$$ h_{j,k,\ell}(S,T,U) = \sum_{m=0}^jp_{j,k,m}(S,T)(1+S^2)^{\frac{q+2\ell+2m-4}{2}}\overline{K}_{\frac{q+2\ell+2m-4}{2}}(U)\index{h3jkl@$h_{j,k,\ell}$} $$
with
$$ S = \frac{I_1}{\sqrt{\lambda^2+2a^2}} \qquad \mbox{and} \qquad T = \frac{I_2^p+I_2^q-\sqrt{2}(\lambda^2+2a^2)}{\sqrt{\lambda^2+2a^2}}\index{S@$S$}\index{T@$T$} $$
and
$$ I_1 = \omega(x_0,Q), \qquad I_2^p = -2\sqrt{2}|x_1|^2, \qquad I_2^q = 2\sqrt{2}|x_2|^2\index{I1@$I_1$}\index{I2p@$I_2^p$}\index{I2q@$I_2^q$} $$
is a $\frakk$-subrepresentation of $(d\pi_\min,\calD'(\RR^\times)\otimeshat\calS'(\Lambda))$ isomorphic to the representation $\CC\boxtimes\calH^{\frac{p-q}{2}}(\RR^q)$ of $\frakk\simeq\so(p)\oplus\so(q)$.
\end{theorem}

We begin with the action of $\frakk\cap\frakg_{(0,0)}\simeq\so(p-3)\oplus\so(q-3)$.

\begin{lemma}\label{lem:LKTSOpqStep1}
If $f\in\calD'(\RR^\times)\otimeshat\calS'(\Lambda)$ generates under the action of $\frakk\cap\frakg_{(0,0)}\simeq\so(p-3)\oplus\so(q-3)$ a subrepresentation isomorphic to $\CC\boxtimes\calH^\ell(\RR^{q-3})$, it has to be a linear combination of distributions of the form
$$ f(\lambda,a,x) = f_1(\lambda,a,I_1,I_2^p,I_2^q)\varphi(x_2), $$
where
$$ I_1 = \omega(x,Q), \qquad I_2^p = \omega(\mu(x_1)P,B)=-2\sqrt{2}|x_1|^2, \qquad I_2^q = \omega(\mu(x_2)P,B)=2\sqrt{2}|x_2|^2 $$
and $\varphi\in\calH^\ell(\RR^{q-3})$.
\end{lemma}

\begin{proof}
The subalgebra $\frakk\cap\frakg_{(0,0)}\simeq\so(p-3)\oplus\so(q-3)$ acts on $\calJ_1\oplus\calJ_2\simeq\RR^{p-3}\oplus\RR^{q-3}$ by the direct sum of the standard representations of $\so(p-3)$ and $\so(q-3)$. An $\so(p-3)$-invariant distribution on $\calJ_1\simeq\RR^{p-3}$ only depends on $I_2^p=-2\sqrt{2}|x_1|^2$, and a distribution on $\calJ_2\simeq\RR^{q-3}$ that belongs to the isotypic component of $\calH^\ell(\RR^{q-3})$ has to be a linear combination of products of $\varphi\in\calH^\ell(\RR^{q-3})$ and distributions only depending on $I_2^q=2\sqrt{2}|x_2|^2$.
\end{proof}

We note the following identities for derivatives of the invariants $I_1$, $I_2^p$, $I_2^q$ in the directions $P$, $v\in\calJ_1$ and $w\in\calJ_2$:
\begin{align*}
	\partial_PI_1 &= 1, & \partial_PI_2^p &= 0, & \partial_PI_2^q &= 0,\\
	\partial_vI_1 &= 0, & \partial_vI_2^p &= \sqrt{2}\omega(x,Jv), & \partial_vI_2^q &= 0,\\
	\partial_wI_1 &= 0, & \partial_wI_2^p &= 0, & \partial_wI_2^q &= -\sqrt{2}\omega(x,Jv).
\end{align*}
Further, we have
$$ I_3 = \omega(\Psi(x),B) = 2n(x) = -I_1(I_2^p+I_2^q).\index{I3@$I_3$} $$
In what follows we write $\partial_1$, $\partial_{2p}$ and $\partial_{2q}$ for the derivatives of $f_1(\lambda,a,I_1,I_2^p,I_2^q)$ with respect to the variables $I_1$, $I_2^p$ and $I_2^q$.

\begin{lemma}\label{lem:LKTSOpqStep2}
$f$ is additionally an eigenfunction of $d\pi_\min(A-\overline{B})$ to the eigenvalue $ik\sqrt{2}$ if and only if it is, for $\lambda>0$ resp. $\lambda<0$, of the form
$$ f(\lambda,a,x) = (\lambda-i\sqrt{2}a)^k\exp\left(-\frac{iaI_3}{\lambda R}\right)f_2(R,I_1,I_2^p,I_2^q)\varphi(x_2), $$
where
$$ R = \lambda^2+2a^2. $$
\end{lemma}

\begin{proof}
The method of characteristics applied to the first order equation
$$ d\pi_\min(A-\overline{B})f = \left(-\lambda\partial_A+2a\partial_\lambda-\frac{2in(x)}{\lambda^2}\right)f = ik\sqrt{2}f $$
shows the claim.
\end{proof}

\begin{lemma}\label{lem:LKTSOpqStep3}
$f$ is additionally annihilated by the operators $\{\lambda\,d\pi_\min(v+\theta v)+2a(B_\mu(v,B)+B_\mu(A,Jv)):v\in\calJ_1\}\subseteq\frakk_1$ if and only if it is of the form
$$ f(\lambda,a,x) = (\lambda-i\sqrt{2}a)^kR^{-\frac{k+1}{2}}\exp\left(-\frac{iaI_3}{\lambda R}\right)f_3\left(S,T,I_2^q\right)\varphi(x_2), $$
where
$$ S = \frac{I_1}{R^{\frac{1}{2}}} \qquad \mbox{and} \qquad T = \frac{I_2^p+I_2^q-\sqrt{2}R}{R^{\frac{1}{2}}}. $$
\end{lemma}

\begin{proof}
Applying $\lambda\,d\pi_\min(v+\theta v)+2a(B_\mu(v,B)+B_\mu(A,Jv))$ to $f(\lambda,a,x)$ as in Lemma~\ref{lem:LKTSOpqStep2} leads to the differential equation
$$ 2R\partial_R+I_1\partial_1+(I_2^p+I_2^q+\sqrt{2}R)\partial_{2p})f_2=-(k+1)f_2, $$
which can be solved using the method of characteristics.
\end{proof}

\begin{lemma}\label{lem:LKTSOpqStep4}
$f$ is additionally invariant under $\{Jv-\overline{v}:v\in\calJ_1\}\subseteq\frakk_1$ if and only if the function $f_3(S,T,I_2^q)$ solves the following two partial differential equations:
\begin{align}
 \Big(-2\partial_T^2-\sqrt{2}I_2^q\partial_{2q}^2-\tfrac{\sqrt{2}}{2}(2\ell+q-3)\partial_{2q}+(1+S^2)\Big)f_3 &= 0,\label{eq:LKTSOpqPDE1}\\
 \Big(2\partial_S\partial_T-ST+k\sqrt{2}\Big)f_3 &= 0.\label{eq:LKTSOpqPDE2}
\end{align}
\end{lemma}

\begin{proof}
A lengthy computation involving Lemma~\ref{lem:SOpqTrace} and Lemma~\ref{lem:SOpqIdentities} shows that
\begin{multline*}
	id\pi_\min(Jv-\overline{v})f = \omega(x,Jv)(\lambda-i\sqrt{2}a)^kR^{-\frac{k+1}{2}}\exp\left(-\frac{iaI_3}{\lambda R}\right)\varphi(x_2)\\
	\times\Bigg[\Big(-2\partial_T^2-\sqrt{2}I_2^q\partial_{2q}^2-\tfrac{\sqrt{2}}{2}(2\ell+q-3)\partial_{2q}+(1+S^2)\Big)f_3\\
	+R^{-\frac{1}{2}}\Big(-\sqrt{2}S\partial_S\partial_T+\tfrac{\sqrt{2}}{2}S^2T-kS\Big)f_3\Bigg].
\end{multline*}
Since $R^{-\frac{1}{2}}$ is independent of $S$, $T$ and $I_2^q$, this implies the two equations.
\end{proof}

We remark at this point that
$$ \frac{1}{2}T^2+2\sqrt{2}I_2^q = \frac{1}{2R}(I_2^p+I_2^q)^2-\sqrt{2}(I_2^p-I_2^q)+R > 0. $$
Together with a deeper analysis of the equations in Lemma~\ref{lem:LKTSOpqStep4} this leads us to introducing a new variable
$$ U = (1+S^2)\Big(\frac{1}{2}T^2+2\sqrt{2}I_2^q\Big). $$
and making the Ansatz
$$ f_3(S,T,I_2^q) = f_4(U)g_4(S). $$
Equation \eqref{eq:LKTSOpqPDE1} applied to this gives
$$ Uf_4''(U)+\tfrac{q+2\ell-2}{2}f_4'(U)-\tfrac{1}{4}f_4(U) = 0, $$
which has the solution $f_4(U)=\overline{K}_{\frac{q+2\ell-4}{2}}(U)$. Plugging this into \eqref{eq:LKTSOpqPDE2} gives
$$ 2Tf_4'(U)\Big((1+S^2)g_4'(S)-(q+2\ell-4)Sg_4(S)\Big)+k\sqrt{2}f_4(U)g_4(S) = 0. $$
Since by \eqref{eq:BesselDerivative1} the functions $f_4(U)$ and $f_4'(U)$ are linearly independent, this equation can only have a non-trivial solution for $k=0$. In this case
$$ g_4(S)=(1+S^2)^{\frac{q+2\ell-4}{2}} $$
solves the equation and we obtain the functions $f_{0,0,\ell}(\lambda,a,x)$, $\ell\geq0$. To investigate whether these functions are $K$-finite, we apply $\frakk_2\simeq\so(q)$ to $f_{0,0,\ell}$. For this we decompose $\frakk_2$ according to Proposition~\ref{prop:SOpqKstructure} and compute the action of each part on a distribution $f(\lambda,a,x)$ of the form given in Lemma~\ref{lem:LKTSOpqStep3}.

\begin{lemma}\label{lem:LKTSOpqStep5}
For a distribution $f\in\calD'(\RR^\times)\otimeshat\calS'(\Lambda)$ of the form
$$ f(\lambda,a,x) = (\lambda-i\sqrt{2}a)^kR^{-\frac{k+1}{2}}\exp\left(-\frac{iaI_3}{\lambda R}\right)f_3\left(S,T,I_2^q\right)\varphi(x_2) $$
with $f_3(S,T,I_2^q)$ satisfying \eqref{eq:LKTSOpqPDE1} and \eqref{eq:LKTSOpqPDE2} we have for $v\in\calJ_2$:
\begin{multline*}
	id\pi_\min(Jv-\overline{v})f = (\lambda-i\sqrt{2}a)^kR^{-\frac{k+1}{2}}\exp\left(-\frac{iaI_3}{\lambda R}\right)\\
	\times\Bigg[\omega(x,Jv)\varphi\Big(4\partial_T^2+\sqrt{2}S\partial_S\partial_{2q}-\sqrt{2}T\partial_T\partial_{2q}+\frac{\sqrt{2}}{2}(2\ell+q-p-2)\partial_{2q}-2S^2\Big)f_3\\
	+\partial_v\varphi\Big(-S\partial_S+T\partial_T+2I_2^q\partial_{2q}+\frac{p+q+2\ell-8}{2}\Big)f_3\Bigg],
\end{multline*}
\begin{multline*}
	d\pi_\min(2i(B_\mu(A,Jv)+B_\mu(v,B))\mp\sqrt{2}(v+\overline{Jv}))f = (\lambda-i\sqrt{2}a)^{k\pm1}R^{-\frac{(k\pm1)+1}{2}}\exp\left(-\frac{iaI_3}{\lambda R}\right)\\
	\times\Bigg[\omega(x,Jv)\varphi\Big(-2\partial_S\partial_{2q}\mp4\partial_T\pm\sqrt{2}T\partial_{2q}+2\sqrt{2}S\Big)f_3\\
	+\partial_v\varphi\Big(\sqrt{2}\partial_S\mp T\Big)f_3\Bigg],
\end{multline*}
and
\begin{multline*}
	d\pi_\min(-P+\overline{Q}\mp i\sqrt{2}T_0)f = (\lambda-i\sqrt{2}a)^{k\pm1}R^{-\frac{(k\pm1)+1}{2}}\exp\left(-\frac{iaI_3}{\lambda R}\right)\varphi(x_2)\\
	\times\Big(\pm\sqrt{2}T\partial_T^2\pm2\sqrt{2}I_2^q\partial_T\partial_{2q}+(1+S^2)\partial_S+(-ST\pm\sqrt{2}\tfrac{p+q+2\ell-6}{2})\partial_T\\
	-2SI_2^q\partial_{2q}\mp\tfrac{\sqrt{2}}{2}(1+S^2)T-\tfrac{p+q+2\ell\mp2k-8}{2}S\Big)f_3.
\end{multline*}
\end{lemma}

Applying these operators to $f_{0,0,\ell}$ suggests that the $\frakk$-representation generated by $f_{0,0,0}$ consists of functions $f(\lambda,a,x)$ as in Lemma~\ref{lem:LKTSOpqStep3} with
\begin{equation}
 f_3(S,T,I_2^q) = \sum_m p_m(S,T)(1+S^2)^{\frac{q+2\ell+2m-4}{2}}\overline{K}_{\frac{q+2\ell+2m-4}{2}}(U)\label{eq:LKTSOpqAnsatzSumPolyBessel}
\end{equation}
for some polynomials $p_m(S,T)$.

\begin{lemma}\label{lem:LKTSOpqStep6}
A function $f_3(S,T,I_2^q)$ of the form \eqref{eq:LKTSOpqAnsatzSumPolyBessel} solves the equations \eqref{eq:LKTSOpqPDE1} and \eqref{eq:LKTSOpqPDE2} in Lemma~\ref{lem:LKTSOpqStep4} if and only if the family of polynomials $p_m(S,T)$ satisfies \eqref{lem:LKTSOpqUniquePolys1} and \eqref{lem:LKTSOpqUniquePolys2} in Lemma~\ref{lem:LKTSOpqUniquePolys}.
\end{lemma}

\begin{proof}
From \eqref{eq:BesselDerivative1} and \eqref{eq:BesselDerivative2} it follows that
\begin{align*}
	\partial_S(1+S^2)^\alpha\overline{K}_\alpha(U) &= -S(1+S^2)^{\alpha-1}\overline{K}_{\alpha-1}(U),\\
	\partial_T(1+S^2)^\alpha\overline{K}_\alpha(U) &= -\frac{T}{2}(1+S^2)^{\alpha+1}\overline{K}_{\alpha+1}(U).
\end{align*}
Using these identities the proof is a direct computation.
\end{proof}

This motivates the definition of the functions $f_{j,k,\ell}$ in Theorem~\ref{thm:LKTSOpq}. We finally calculate how the Lie algebra $\frakk_2\simeq\so(q)$ acts on $f_{j,k,\ell}$. For this note that
$$ \omega(x,Jv)\varphi = \varphi_v^++I_2^q\varphi_v^- $$
with
$$ \varphi_v^+ = \omega(x,Jv)\varphi+\frac{\sqrt{2}I_2^q\partial_v\varphi}{q+2\ell-5} \in \calH^{\ell+1}(\RR^{q-3}), \qquad \varphi_v^- = -\frac{\sqrt{2}\partial_v\varphi}{q+2\ell-5} \in \calH^{\ell-1}(\RR^{q-3}).\index{1wphivpm@$\varphi_v^\pm$} $$

\begin{proposition}\label{prop:LKTSOpqStep7}
For all $j,\ell\geq0$, $k\in\ZZ$ and $\varphi\in\calH^\ell(\RR^{q-3})$, the function $f_{j,k,\ell}\otimes\varphi$ is $\frakk_1$-invariant and satisfies
\begin{align*}
	d\pi_\min(T)f_{j,k,\ell}\otimes\varphi &= f_{j,k,\ell}\otimes(-\partial_{Tx}\varphi) && (T\in\frakk_2\cap\frakg_{(0,0)}\simeq\so(q-3)),\\
	d\pi_\min(A-\overline{B})f_{j,k,\ell}\otimes\varphi &= ik\sqrt{2}f_{j,k,\ell}\otimes\varphi,
\end{align*}
and
\begin{align*}
	&d\pi_\min(2i(B_\mu(A,Jv)+B_\mu(v,B))\mp\sqrt{2}(v+\overline{Jv}))f_{j,k,\ell}\otimes\varphi = 2\sqrt{2}(j\mp k)f_{j-1,k\pm1,\ell+1}\otimes\varphi_v^+\\
	&\hspace{4cm}+ \Big[(q+j\mp k+2\ell-5)f_{j+1,k\pm1,\ell-1}+(j\mp k)f_{j-1,k\pm1,\ell-1}\Big]\otimes\varphi_v^-,\\
	&id\pi_\min(Jv-\overline{v})f_{j,k,\ell}\otimes\varphi = (p-q-2j-2\ell)f_{j,k,\ell+1}\otimes\varphi_v^+\\
	&\hspace{2.3cm}+ \tfrac{1}{2\sqrt{2}}\Big[(p-q-2\ell-2j)f_{j+2,k,\ell-1}+(p+q+2\ell-2j-10)f_{j,k,\ell-1}\Big]\otimes\varphi_v^-,\\
	& d\pi_\min(-P+\overline{Q}\mp i\sqrt{2}T_0)f_{j,k,\ell}\otimes\varphi = \Big[(j\mp k)f_{j-1,k\pm1,\ell}-\tfrac{p-q-2j-2\ell}{2}f_{j+1,k\pm1,\ell}\Big]\otimes\varphi.
\end{align*}
\end{proposition}

\begin{proof}
With Lemma~\ref{lem:LKTSOpqUniquePolys} and Lemma~\ref{lem:LKTSOpqStep5} this is now an easy, though longish, computation using
\begin{equation*}
	\partial_{2q}(1+S^2)^\alpha\overline{K}_\alpha(U) = -\sqrt{2}(1+S^2)^{\alpha+1}\overline{K}_{\alpha+1}(U).\qedhere
\end{equation*}
\end{proof}

This proves Theorem~\ref{thm:LKTSOpq}.

\section{The case $\frakg=\so(p,3)$}\label{sec:LKTSOp3}

For $q=3$, we note that $\calH^\ell(\RR^{q-3})=\{0\}$ for $\ell>0$, so that the lowest $K$-type is spanned by $f_{j,k,0}$ ($0\leq j\leq\frac{p-3}{2}$, $-j\leq k\leq j$, $k\equiv j\mod2$). However, these functions cannot form a basis of $W\simeq\calH^{\frac{p-3}{2}}(\RR^3)\simeq S^{p-3}(\CC^2)$ since $\dim W=p-2$, so the functions $f_{j,k,0}$ for fixed $k$ have to be linearly dependent.

\begin{lemma}
	For $q=3$ and all $0\leq j\leq\frac{p-3}{2}$, $-j\leq k\leq j$, $k\equiv j\mod2$, we have
	$$ h_{j,k,0}(S,T) = i^{j+k}\sqrt{\frac{\pi}{2}}(1+S^2)^{-\frac{1}{2}}(\sqrt{1+S^2}-S)^k. $$
\end{lemma}

\begin{proof}
Since $q=3$, we have $I_2^q=0$ and therefore equation \eqref{eq:LKTSOpqPDE1} becomes
$$ \Big(-2\partial_T^2+(1+S^2)\Big)f_3 = 0, $$
which has the unique tempered solution $f_3(S,T)=g_4(S)e^{\sqrt{\frac{1+S^2}{2}}T}$. Plugging this into \eqref{eq:LKTSOpqPDE2} gives
$$ g_4'(S)+\left(\frac{S}{1+S^2}+\frac{k}{\sqrt{1+S^2}}\right)g_4(S)=0 $$
which has the solution
$$ g_4(S)=\const\times(1+S^2)^{-\frac{1}{2}}(\sqrt{1+S^2}-S)^k. $$
To find the constant, we multiply both $h_{j,k,0}(S,T)$ and $f_3(S,T)$ by $(1+S^2)^{\frac{1}{2}}$ and let $S\to\pm i$ (for the unique analytic extensions of the functions). Using the explicit formulas for the $K$-Bessel function at half-integer parameters, \eqref{eq:KBessel1/2} and \eqref{eq:KBesselHalfInt}, shows that for $h_{j,k,0}(S,T)=\sum_{m=0}^jp_{j,k,m}(S,T)(1+S^2)^{m-\frac{1}{2}}\overline{K}_{m-\frac{1}{2}}(U)$ only the summand for $m=0$ survives. Comparing this with $(\sqrt{1+S^2}-S)^ke^{-\sqrt{U}}$ at $S=\pm i$ shows the claim.
\end{proof}

\begin{remark}
	It should be possible to obtain this expression for $h_{j,k,0}(S,T)$ from the explicit formulas for the $K$-Bessel function of half-integer parameters \eqref{eq:KBessel1/2} and \eqref{eq:KBesselHalfInt}, once a closed formula for the polynomials $p_{j,k,m}(S,T)$ in Lemma~\ref{lem:LKTSOpqUniquePolys} is known.
\end{remark}

If we let
\begin{multline}
	f_k(\lambda,a,x) = (\lambda-i\sqrt{2}a)^kR^{-\frac{k+1}{2}}\exp\left(-\frac{iaI_3}{\lambda R}\right)\\
\times(1+S^2)^{-\frac{1}{2}}(\sqrt{1+S^2}+S)^{-k}\exp\left(\sqrt{\frac{1+S^2}{2}}T\right),\label{eq:LKTSOp3fkInt}\index{fk@$f_k$}
\end{multline}
then Proposition~\ref{prop:LKTSOpqStep7} can be reformulated as
\begin{equation*}
	d\pi_\min(-P+\overline{Q}\mp i\sqrt{2}T_0)f_k = \left(\pm\frac{p-3}{2}-k\right)f_{k\pm1}.
\end{equation*}

We note that $(\lambda-i\sqrt{2}a)^k=\sgn(\lambda)^k(|\lambda|-i\sqrt{2}\sgn(\lambda)a)^k$ and, in view of the case $\frakg=\sl(3,\RR)$, we put for even $p\geq4$, and $k\in\ZZ+\frac{1}{2}$:
\begin{multline}
 f_k(\lambda,a,x) = \sgn(\lambda)^{k-\frac{1}{2}}(|\lambda|-i\sqrt{2}\sgn(\lambda)a)^kR^{-\frac{k+1}{2}}\exp\left(-\frac{iaI_3}{\lambda R}\right)\\
 \times(1+S^2)^{-\frac{1}{2}}(\sqrt{1+S^2}+S)^{-k}\exp\left(\sqrt{\frac{1+S^2}{2}}T\right).\label{eq:LKTSOp3fkHalfInt}\index{fk@$f_k$}
\end{multline}
Then it is easy to see that $f_k$ is still $\frakk_1$-invariant. The subalgebra $\frakk_2\simeq\so(3)\simeq\su(2)$ is spanned by the $\su(2)$-triple
$$ T_1=\tfrac{\sqrt{2}}{2}A+Q+\theta(\tfrac{\sqrt{2}}{2}A+Q), \qquad T_2=\tfrac{\sqrt{2}}{2}B-P+\theta(\tfrac{\sqrt{2}}{2}B-P), \qquad T_3=\sqrt{2}T_0-(E-F),\index{T1@$T_1$}\index{T2@$T_2$}\index{T3@$T_3$} $$
and it follows from the computations in Section~\ref{sec:LKTSOpq} that
\begin{equation*}
	d\pi_\min(T_1)f_k = 2ikf_k, \qquad d\pi_\min(T_2\mp iT_3)f_k = \left(\pm(p-3)-2k\right)f_{k\pm1}.
\end{equation*}
This shows:

\begin{theorem}\label{thm:LKTSOp3}
Let $p\geq3$ be arbitrary. Then the space
$$ W={\mathrm span}\left\{f_k:k=-\frac{p-3}{2},-\frac{p-3}{2}+1,\ldots,\frac{p-3}{2}\right\}\index{W1@$W$} $$
with $f_k$ as in \eqref{eq:LKTSOp3fkInt} resp. \eqref{eq:LKTSOp3fkHalfInt} is a $\frakk$-subrepresentation of $(d\pi_\min,\calD'(\RR^\times)\otimeshat\calS'(\Lambda))$ isomorphic to the representation $\CC\boxtimes S^{p-3}(\CC^2)$ of $\frakk\simeq\so(p)\oplus\su(2)$.
\end{theorem}

\chapter{$L^2$-models for minimal representations}

After having exhibited explicit $K$-finite vectors in the representation $d\pi_\min$ of $\frakg$ on $\calD'(\RR^\times)\otimeshat\calS'(\Lambda)$, we show in this section that these vectors generate an irreducible $(\frakg,K)$-module which integrates to an irreducible unitary representation $\pi_\min$ of the universal covering group of $G$ on $L^2(\RR^\times\times\Lambda)$ (see Section~\ref{sec:Int(g,K)Module}). Moreover, we find the action of certain Weyl group elements in the $L^2$-model in Section~\ref{sec:ActionWeylGroupElts}. Together with the Heisenberg parabolic subgroup these elements generate the whole group $G$, so that we have a complete description of the representation $\pi_\min$ on the group level.

\section{Integration of the $(\frakg,K)$-module}\label{sec:Int(g,K)Module}

Let $W$ be one of the irreducible $\frakk$-subrepresentations of $(d\pi_\min,\calD'(\RR^\times)\otimeshat\calS'(\Lambda))$ constructed in Chapter~\ref{ch:LKT}. For $\frakg=\sl(n,\RR)$, $W$ may be one of the $K$-types $W_{0,r}$ or $W_{1,r}$, where we now assume that $r\in i\RR$ (in order for the Lie algebra representation $d\pi_{\min,r}$ to be infinitesimally unitary on $L^2(\RR^\times\times\Lambda)$). For $\frakg=\sl(3,\RR)$ we additionally allow $W_{\frac{1}{2}}$. Consider the $\frakg$-subrepresentation generated by $W$:
$$ \overline{W}=d\pi_\min(U(\frakg))W\subseteq\calD'(\RR^\times)\otimeshat\calS'(\Lambda).\index{W1@$\overline{W}$} $$
Then, by standard arguments, $\overline{W}$ is a $(\frakg,K)$-module (see e.g. \cite[Lemma 2.23]{HKM14}).

\begin{proposition}\label{prop:WinL2}
$\overline{W}\subseteq L^2(\RR^\times\times\Lambda)$.
\end{proposition}

\begin{proof}
We first observe that $W\subseteq L^2(\RR^\times\times\Lambda)$. In fact, this follows from the asymptotic behavior of the $K$-Bessel function (see Section~\ref{sec:BesselDiffEqAsymptotics}). For the more general statement $\overline{W}\subseteq L^2(\RR^\times\times\Lambda)$ first note that, for the Lie algebra $\frakq$ of any standard parabolic subgroup of $G$, we have $\frakg=\frakk\oplus\frakq$, which implies $U(\frakg)=U(\frakq)U(\frakk)$ by the Poincare--Birkhoff--Witt Theorem. For $\frakq$ we may for instance choose
$$ \frakq=\frakg_{-2}\oplus\frakg_{-1}\oplus\frakg_{(-1,1)}\oplus\frakg_{(0,0)}. $$
Then $\overline{W}=d\pi_\min(U(\frakq))W$ can be computed using the identities \eqref{eq:BesselDerivative1} and \eqref{eq:BesselDerivative2}, the rest is technicality.
\end{proof}

\begin{lemma}\label{lem:WinfUnit}
The action of $\frakg$ on $\overline{W}$ is infinitesimally unitary with respect to the inner product on $L^2(\RR^\times\times\Lambda)$.
\end{lemma}

\begin{proof}
A heuristic proof is given by the observations in Remark~\ref{rem:MotivationL2}. A rigorous proof can be obtained using the formulas for $d\pi_\min(X)$ ($X\in\frakg$) and integrating by parts, showing directly that $d\pi_\min(X)$ is symmetric on $L^2(\RR^\times\times\Lambda)$.
\end{proof}

To finally integrate $\overline{W}$ to a group representation, we make use of the following statement about its restriction to $\overline{P}_0$:

\begin{lemma}\label{lem:RepOfPbar}
The representation $\varpi$ of $\overline{P}_0=M_0A\overline{N}$ on $L^2(\RR^\times\times\Lambda)$ given by
\begin{align*}
	\varpi(\overline{n}_{(z,t)})f(\lambda,x) &= e^{i\lambda t}e^{i(\omega(z'',x)+\frac{1}{2}\lambda\omega(z',z''))}f(\lambda,x-\lambda z') && (\overline{n}_{(z,t)}\in\overline{N}),\\
	\varpi(m)f(\lambda,x) &= (\id_{\RR^\times}^*\otimes\,\omega_{\met,\lambda^{-1}}(m))f(\lambda,x) && (m\in M_0),\\
	\varpi(\exp(sH))f(\lambda,x) &= e^{-\frac{\dim\Lambda+2}{2}s}f(e^{-2s}\lambda,e^{-s}x) && (\exp(sH)\in A),\index{1pi@$\varpi$}
\end{align*}
is unitary and decomposes into the direct sum of irreducible subrepresentations $L^2(\RR_+\times\Lambda)\oplus L^2(\RR_-\times\Lambda)$. Moreover, if $\varpi$ extends to a unitary representation of $\overline{P}=MA\overline{N}$ on $L^2(\RR^\times\times\Lambda)$, then this extension is irreducible.
\end{lemma}

\begin{proof}
Using the isomorphism
$$ L^2(\RR^\times\times\Lambda) \simeq L^2(\RR^\times,L^2(\Lambda)) $$
we observe that $\varpi|_{\overline{N}}$ is given by
$$ (\varpi(\overline{n})f)(\lambda) = \widetilde{\sigma}_\lambda(\overline{n})[f(\lambda)] \qquad (f\in L^2(\RR^\times,L^2(\Lambda))),$$
where $\widetilde{\sigma}_\lambda$ is the representation of $\overline{N}$ on $L^2(\Lambda)$ given by $\widetilde{\sigma}_\lambda(\overline{n}_{(z,t)})=\sigma_{\lambda^{-1}}(\overline{n}_{(\lambda z,\lambda^2t)})$. Since $\sigma_\lambda$ and hence $\widetilde{\sigma}_\lambda$ is irreducible for every $\lambda\in\RR^\times$, it follows from Schur's Lemma that any intertwining operator $T:L^2(\RR^\times,L^2(\Lambda))\to L^2(\RR^\times,L^2(\Lambda))$ is of the form $Tf(\lambda,x)=t(\lambda)f(\lambda,x)$ for some measurable function $t$ on $\RR^\times$. Now, $T$ also commutes with $A$ which implies $t(e^{-2s}\lambda)=t(\lambda)$ for all $s\in\RR$, whence $t(\lambda)$ is constant on $\RR_+$ and $\RR_-$, respectively. This shows that $L^2(\RR_\pm\times\Lambda)$ are invariant subspaces on which $\overline{P}_0$ acts irreducibly.\\
Now assume that $\varpi$ extends to $\overline{P}$ and let $m_0\in M$ with $\chi(m_0)=-1$ (which exists by Theorem~\ref{thm:CharacterizationHermitian} since $G$ is non-Hermitian). The operator $\varpi(m_0)$ satisfies
$$ \varpi(m_0)\circ\varpi(\overline{n}_{(0,t)}) = \varpi(\overline{n}_{(0,-t)})\circ\varpi(m_0) \qquad \mbox{for all }t\in\RR, $$
where $\varpi(\overline{n}_{(0,t)})f(\lambda,x)=e^{i\lambda t}f(\lambda,x)$. It follows that $\varpi(m_0)f(\lambda,x)=m(\lambda)U_xf(-\lambda,x)$ for some function $m(\lambda)$ and a unitary operator $U$ on $L^2(\Lambda)$. If now $T$ also commutes with $\varpi(m_0)$, it follows that $t(-\lambda)=t(\lambda)$ which implies that $t$ is constant on $\RR^\times$ and hence $T$ is a scalar multiple of the identity.
\end{proof}

\begin{theorem}\label{thm:IntMinRep}
The $(\frakg,K)$-module $\overline{W}$ integrates to an irreducible unitary representation $\pi_\min$ of the universal cover $\widetilde{G}$ of $G$ on $L^2(\RR^\times\times\Lambda)$ which is minimal in the sense that its annihilator in $U(\frakg_\CC)$ is a completely prime ideal with associated variety equal to the minimal nilpotent orbit. For $\frakg_\CC$ not of type $A$, the annihilator is the Joseph ideal.
\end{theorem}

\begin{proof}
Using Proposition~\ref{prop:WinL2}, Lemma~\ref{lem:WinfUnit} and Lemma~\ref{lem:RepOfPbar}, it follows along the same lines as in \cite[Proposition 2.27]{HKM14} that $\overline{W}$ is admissible. It therefore integrates to a representation $\pi_\min$ of $G$. This representation is unitary on a Hilbert space $\calH\subseteq L^2(\RR^\times\times\Lambda)$ by Proposition~\ref{prop:WinL2} and Lemma~\ref{lem:WinfUnit}. On the other hand, its restriction to $\overline{P}$ is given by the action in Lemma~\ref{lem:RepOfPbar} which is irreducible on $L^2(\RR^\times\times\Lambda)$. This implies $\calH=L^2(\RR^\times\times\Lambda)$ and $\pi_\min$ is irreducible. That the representation is minimal follows from Proposition~\ref{prop:AnnihilatorJosephIdeal}.
\end{proof}

For $\frakg=\sl(n,\RR)$ we write $\pi_{\min,\varepsilon,r}$\index{1piminepsilonr@$\pi_{\min,\varepsilon,r}$} for the representation with underlying $(\frakg,K)$-module $W_{\varepsilon,r}$, $r\in i\RR$, and in the case $n=3$ we write $\pi_{\min,\frac{1}{2}}$\index{1pimin12@$\pi_{\min,\frac{1}{2}}$} for the representation with underlying $(\frakg,K)$-module $W_{\frac{1}{2}}$.

\section{Action of Weyl group elements}\label{sec:ActionWeylGroupElts}

The results from the previous section can also be phrased in a different way. The parabolic subgroup $\overline{P}$ acts unitarily and irreducibly on $L^2(\RR^\times\times\Lambda)$ by the representation $\varpi$ (see Lemma~\ref{lem:RepOfPbar}). Theorem~\ref{thm:IntMinRep} shows that this representation extends to some covering group of $G$. This point of view was used in \cite[Theorem 2]{KS90} and \cite[Proposition 4.2]{Sav93} in order to construct the above $L^2$-models for the split groups $\SO(n,n)$, $E_{6(6)}$, $E_{7(7)}$, $E_{8(8)}$ and $G_{2(2)}$. There, it is shown that $\varpi$ can be extended to an irreducible unitary representation $\pi_\min$ of $G$ by defining $\pi_\min$ on the representative of a certain Weyl group element $w_1$ and checking the Chevalley relations (see Section~\ref{sec:Bigrading} for the definition of $w_1$). This technique does not easily generalize to the case of non-split groups. However, after having constructed the $L^2$-model in a different way, we can obtain the action of $w_1$ as a corollary.

The answer depends on the eigenvalues of $d\pi_\min(A-\overline{B})$ on the lowest $K$-type $W$. Note that in all cases, $W$ is spanned by functions of the form
$$ f(\lambda,a,x) = (\lambda-i\sqrt{2}a)^kg(\lambda^2+2a^2,x) \qquad (\lambda>0), $$
with either $k\in\ZZ$ or $k\in\ZZ+\frac{1}{2}$. On such functions, $d\pi_\min(A-\overline{B})$ acts by $ik\sqrt{2}$. We refer to the \emph{integer case} if $k\in\ZZ$ and to the \emph{half-integer case} if $k\in\ZZ+\frac{1}{2}$. From the constructions in Chapter~\ref{ch:LKT}, it follows that:
\begin{itemize}
	\item For $\frakg=\frake_{6(2)},\frake_{6(6)},\frake_{7(-5)},\frake_{7(7)},\frake_{8(-24)},\frake_{8(8)},\frakg_{2(2)}$ the representation $\pi_\min$ belongs to the integer case.
	\item For $\frakg=\sl(n,\RR)$ the representations $\pi_{\min,0,r}$ and $\pi_{\min,1,r}$ ($r\in i\RR$) belong to the integer case, and for $\frakg=\sl(3,\RR)$ the representation $\pi_{\min,\frac{1}{2}}$ belongs to the half-integer case.
	\item For $\frakg=\so(p,q)$ the representation $\pi_\min$ belongs to the integer case if $p,q\geq3$, $p+q$ even, and it belongs to the half-integer case if $p\geq q=3$, $p$ even.
\end{itemize}

\begin{theorem}\label{thm:ActionW1}
	The element $w_1$ acts in the $L^2$-model of the minimal representation by
	\begin{equation}
		\pi_\min(w_1)f(\lambda,a,x) = e^{-i\frac{n(x)}{\lambda a}}f(\sqrt{2}a,-\tfrac{\lambda}{\sqrt{2}},x)\times\begin{cases}1&\mbox{in the integer case,}\\\varepsilon(a\lambda)&\mbox{in the half-integer case,}\end{cases}\label{eq:ActionW1}
	\end{equation}
	where
	$$ \varepsilon(x) = \begin{cases}1&\mbox{for $x>0$,}\\i&\mbox{for $x<0$.}\end{cases}\index{1epsilonx@$\varepsilon(x)$} $$
\end{theorem}

\begin{proof}
	Let $A$ denote the unitary operator on $L^2(\RR^\times\times\Lambda)$ given by \eqref{eq:ActionW1}. A direct computation shows that
	$$ A\circ d\pi_\min(X)\circ A^{-1} = d\pi_\min(\Ad(w_1)X) \qquad \mbox{for all }X\in\frakg. $$
	It follows that $A\circ\pi_\min(w_1)^{-1}$ is a $\frakg$-intertwining unitary operator on $L^2(\RR^\times\times\Lambda)$, and therefore has to be a scalar multiple of the identity by Schur's Lemma. To find the scalar, we apply both $A$ and $\pi_\min(w_1)$ to a vector in the lowest $K$-type $W$. Both can be computed using the explicit description of $W$ in Chapter~\ref{ch:LKT}.
\end{proof}

\begin{remark}
	As mentioned above, the formula for $\pi_\min(w_1)$ can be found in \cite[Theorem 2]{KS90} and \cite[Proposition 4.2]{Sav93} for the cases $G=\SO(n,n)$, $E_{6(6)}$, $E_{7(7)}$, $E_{8(8)}$ and $G_{2(2)}$. Note that these are all integer cases. In the half-integer case $G=\widetilde{\SL}(3,\RR)$, Torasso obtained the formula in \cite[Lemme 16]{Tor83}. In fact, he even obtained the action of the whole one-parameter subgroup $\exp(\RR(A-\overline{B}))$ which should also be possible in general using the same methods as in Theorem~\ref{thm:ActionW1}.
\end{remark}

\begin{remark}
	The restriction $\varpi$ of $\pi_\min$ to $\overline{P}$ together with the action $\pi_\min(w_1)$ of $w_1$ determines the representation $\pi_\min$ uniquely since $\overline{P}$ and $w_1$ generate $G$. This philosophy was advocated in \cite{KM11} where a different $L^2$-model for the minimal representation of ${\mathrm O}(p,q)$ was explicitly determined on a maximal parabolic subgroup and the representative of a non-trivial Weyl group element. Theorem~\ref{thm:ActionW1} can be seen as an analogue of their result for our $L^2$-models.
\end{remark}

\begin{remark}
	It would be interesting to also find explicit formulas for the action of the Weyl group element $w_0$. One possible way to achieve this is by the help of an additional element
	$$ w_2=\exp\left(\frac{\pi}{2\sqrt{2}}(B+\overline{A})\right). $$
	We have the identity
	$$ w_2 = w_1w_0w_1^{-1}, $$
	so that $\pi_\min(w_0)$ and $\pi_\min(w_2)$ can be computed from each other using the previously obtained formula for $\pi_\min(w_1)$. Moreover, in all cases except $\frakg\simeq\sl(n,\RR)$, the elements $w_1$ and $w_2$ are conjugate via $M_0$ which acts in $\pi_\min$ via the metaplectic representation. It should be possible to use this in order to obtain a formula for $\pi_\min(w_2)$, and then also for $\pi_\min(w_0)$.
\end{remark}

Using the same technique as in Theorem~\ref{thm:ActionW1} we can obtain the action of $w_0^2$, $w_1^2$ and $w_2^2$, which are all contained in $M$, but lie in different connected components of $M$.

\begin{proposition}\label{prop:ActionWeylGroupSquares}
	The elements $w_0^2,w_1^2,w_2^2\in M$ act in the $L^2$-model of the minimal representation in the following way:
	\begin{enumerate}[(1)]
		\item In the quaternionic cases $\frakg=\frake_{6(2)},\frake_{7(-5)},\frake_{8(-24)}$ we have:
		\begin{align*}
			\pi_\min(w_0^2)f(\lambda,a,x) &= (-1)^nf(\lambda,-a,-x),\\
			\pi_\min(w_1^2)f(\lambda,a,x) &= f(-\lambda,-a,x),\\
			\pi_\min(w_2^2)f(\lambda,a,x) &= (-1)^nf(-\lambda,a,-x),
		\end{align*}
		where $n=-s_\min-1=1,2,4$, i.e. the lowest $K$-type has dimension $2n+1$.
		\item In the split cases $\frakg=\frake_{6(6)},\frake_{7(7)},\frake_{8(8)}$ we have:
		\begin{align*}
			\pi_\min(w_0^2)f(\lambda,a,x) &= f(\lambda,-a,-x),\\
			\pi_\min(w_1^2)f(\lambda,a,x) &= f(-\lambda,-a,x),\\
			\pi_\min(w_2^2)f(\lambda,a,x) &= f(-\lambda,a,-x).
		\end{align*}
		\item In the case $\frakg=\frakg_{2(2)}$ we have:
		\begin{align*}
			\pi_\min(w_0^2)f(\lambda,a,x) &= -f(\lambda,-a,-x),\\
			\pi_\min(w_1^2)f(\lambda,a,x) &= f(-\lambda,-a,x),\\
			\pi_\min(w_2^2)f(\lambda,a,x) &= -f(-\lambda,a,-x).
		\end{align*}
		\item In the case $\frakg=\sl(n,\RR)$ we have for $\varepsilon=0,1$, $r\in i\RR$:
		\begin{align*}
			\pi_{\min,\varepsilon,r}(w_0^2)f(\lambda,a,x) &= (-1)^\varepsilon f(\lambda,-a,-x),\\
			\pi_{\min,\varepsilon,r}(w_1^2)f(\lambda,a,x) &= f(-\lambda,-a,x),\\
			\pi_{\min,\varepsilon,r}(w_2^2)f(\lambda,a,x) &= (-1)^\varepsilon f(-\lambda,a,-x).
		\end{align*}
		\item In the case $\frakg=\sl(3,\RR)$ we have:
		\begin{align*}
			\pi_{\min,\frac{1}{2}}(w_0^2)f(\lambda,a) &= -i\sgn(\lambda)f(\lambda,-a),\\
			\pi_{\min,\frac{1}{2}}(w_1^2)f(\lambda,a) &= if(-\lambda,-a),\\
			\pi_{\min,\frac{1}{2}}(w_2^2)f(\lambda,a) &= -\sgn(\lambda)f(-\lambda,a).
		\end{align*}
		\item In the case $\frakg=\so(p,q)$ with either $p\geq q\geq4$ and $p+q$ even or $p\geq q=3$ we have:
		\begin{align*}
			\pi_\min(w_0^2)f(\lambda,a,x) &= (-i)^{p-q}\sgn(\lambda)^{p-q}f(\lambda,-a,-x),\\
			\pi_\min(w_1^2)f(\lambda,a,x) &= f(-\lambda,-a,x)\times\begin{cases}1&\mbox{for $p+q$ even,}\\i&\mbox{for $p+q$ odd,}\end{cases}\\
			\pi_\min(w_2^2)f(\lambda,a,x) &= (-i)^{p-q}\sgn(\lambda)^{p-q}f(-\lambda,a,-x)\times\begin{cases}1&\mbox{for $p+q$ even,}\\i&\mbox{for $p+q$ odd.}\end{cases}
		\end{align*}
	\end{enumerate}
\end{proposition}

\begin{proof}
	As in the proof of Theorem~\ref{thm:ActionW1}, we define a unitary operator $A$ on $L^2(\RR^\times\times\Lambda)$ by the right hand side for one of the elements $m=w_0^2,w_1^2,w_2^2$ and show that it satisfies
	$$ A\circ d\pi_\min(X)\circ A^{-1} = d\pi_\min(\Ad(m)X) \qquad \mbox{for all }X\in\frakg $$
	using the adjoint action of $m$, which was computed in Section~\ref{sec:W0} and Lemma~\ref{lem:W1W2}. Then, thanks to Schur's Lemma, $\pi_\min(m)=\const\times A$. To show that $\pi_\min(m)=A$, we apply $\pi_\min(m)$ and $A$ to a vector in the lowest $K$-type. Using the explicit formulas for vectors in the lowest $K$-type $W$ from Chapter~\ref{ch:LKT}, one can compute $A$ on $W$, and the identification of $W$ with a finite-dimensional $\frakk$-representation allows to compute $\pi_\min(m)=\exp(d\pi_\min(X))$ where $X=\pi(E-F)$ for $m=w_0^2$, $X=\frac{\pi}{\sqrt{2}}(A-\overline{B})$ for $m=w_1^2$ and $X=\frac{\pi}{\sqrt{2}}(B+\overline{A})$ for $m=w_2^2$. The latter only requires the representation theory of $\su(2)$, more precisely, if $U_1,U_2,U_3\in\su(2)$ form an $\su(2)$-triple and $V_n$ is an irreducible representation of $\su(2)$ of dimension $n$ with basis $v_0,v_1,\ldots,v_n$ such that
	$$ U_1\cdot v_k = i(n-2k)v_k, \qquad (U_2+iU_3)\cdot v_k = -2i(n-k)v_{k+1}, \quad (U_2-iU_3)\cdot v_k = -2ikv_{k-1}, $$
	then
	$$ \exp(\tfrac{\pi}{2}U_1)\cdot v_k = i^{n-2k}v_k, \quad \exp(\tfrac{\pi}{2}U_2)\cdot v_k = i^{-n}v_{n-k}, \quad \exp(\tfrac{\pi}{2}U_3)\cdot v_k = (-1)^{n-k}v_{n-k}. $$
	In all cases except $\frakg\simeq\sl(n,\RR)$ and $\frakg\simeq\so(p,q)$, $p\geq q\geq4$, $p+q$ even, the lowest $K$-type is an irreducible representation of $\su(2)$. The case $\frakg\simeq\sl(n,\RR)$ can be dealt with explicitly using the representation theory of $\so(n)$, and in the remaining case $\frakg\simeq\so(p,q)$, it is sufficient to consider the vectors $f=f_{0,0,\frac{p-q}{2}}\otimes\varphi$, $\varphi\in\calH^{\frac{p-q}{2}}(\RR^{q-3})$ which are invariant under $\so(p)\oplus\so(3)\subseteq\frakk$ with $\so(3)\subseteq\so(q)$ spanned by
	\begin{align*}
		T_1 &= \tfrac{\sqrt{2}}{2}A+Q+\theta(\tfrac{\sqrt{2}}{2}A+Q) && \equiv \sqrt{2}(A-\overline{B}) \mod\so(p),\\
		T_2 &= \tfrac{\sqrt{2}}{2}B-P+\theta(\tfrac{\sqrt{2}}{2}B-P) && \equiv \sqrt{2}(B+\overline{A}) \mod\so(p),\\
		T_3 &= \sqrt{2}T_0-(E-F) && \equiv -2(E-F) \mod\so(p).\qedhere
	\end{align*}
\end{proof}

The knowledge of $\pi_\min(m)$ for $m=w_0^2,w_1^2,w_2^2\in M$ allows us to obtain information about the induction parameter $\zeta$ for the corresponding degenerate principal series representation $I(\zeta,\nu)$ that contains $\pi_\min$ as a subrepresentation. We exclude the case $\frakg=\so(p,q)$ since here $\zeta$ is infinite-dimensional (see Section~\ref{sec:FTpictureMinRepSOpq} for details).

\begin{corollary}\label{cor:RelationZetaImageFT}
	Assume that $\pi_\min$ is a subrepresentation of the degenerate principal series $I(\zeta,\nu)$.
	\begin{enumerate}[(1)]
		\item In the quaternionic cases $\frakg=\frake_{6(2)},\frake_{7(-5)},\frake_{8(-24)}$, the character $\zeta$ of $M/M_0$ satisfies
		$$ \zeta(w_0^2)=1, \qquad \zeta(w_1^2)=\zeta(w_2^2)=(-1)^n, \qquad \mbox{where }n=-s_\min-1. $$
		\item In the split cases $\frakg=\frake_{6(6)},\frake_{7(7)},\frake_{8(8)}$, the character $\zeta$ of $M/M_0$ satisfies
		$$ \zeta(w_0^2)=\zeta(w_1^2)=\zeta(w_2^2)=1. $$
		\item In the case $\frakg=\frakg_{2(2)}$, the character $\zeta$ of $M/M_0$ satisfies
		$$ \zeta(w_0^2)=1, \qquad \zeta(w_1^2)=\zeta(w_2^2)=-1. $$
		\item In the case $\frakg=\sl(n,\RR)$ with $\pi_\min=\pi_{\min,0,r}$, $r\in i\RR$, the character $\zeta$ of $M$ satisfies
		$$ \zeta(w_0^2)=\zeta(w_1^2)=\zeta(w_2^2)=1, $$
		and for $\pi_\min=\pi_{\min,1,r}$, $r\in i\RR$, it satisfies
		$$ \zeta(w_0^2)=-1, \qquad \mbox{and either }\begin{cases}\zeta(w_1^2)=1\quad\mbox{and}\quad\zeta(w_2^2)=-1&\mbox{or}\\\zeta(w_1^2)=-1\quad\mbox{and}\quad\zeta(w_2^2)=1.&\end{cases} $$
		\item In the case $\frakg=\sl(3,\RR)$ with $\pi_\min=\pi_{\min,\frac{1}{2}}$ the representation $\zeta$ is the unique irreducible two-dimensional representation of the quaternion group $M\simeq\{\pm1,\pm i,\pm j,\pm k\}$.
	\end{enumerate}
\end{corollary}

\begin{proof}
	By Frobenius reciprocity, the lowest $K$-type $W$ as determined in Chapter~\ref{ch:LKT} is contained in the degenerate principal series $\pi_{\zeta,\nu}$ if and only if
	\begin{equation}
		\Hom_{M\cap K}(W|_{M\cap K},\zeta|_{M\cap K}) \neq \{0\}.\label{eq:FrobeniusReciprocity}
	\end{equation}
	\begin{enumerate}[(1)]
		\item In the quaternionic cases $\frakg=\frake_{6(2)},\frake_{7(-5)},\frake_{8(-24)}$, the lowest $K$-type has to contain a non-zero vector $f\in W$ such that $d\pi_\min(\frakm\cap\frakk)f=0$ and $\pi_\min(w_i^2)f=\zeta(w_i^2)f$, $i=0,1,2$. The first condition implies that $d\pi_\min(T_1)f=0$, and with \eqref{eq:SU2TripleActionQuat} it follows that
		$$ f=\const\times\sum_{\substack{k=-n\\k\equiv n\mod2}}^n{n\choose\frac{k+n}{2}}f_k. $$
		Acting by $\pi_\min(w_i^2)$, using Corollary~\ref{cor:ActionWeylSquaresLKTquat}, and comparing with $\zeta(w_i^2)$ shows the claim.
		\item In the split cases $\frakg=\frake_{6(6)},\frake_{7(7)},\frake_{8(8)}$, the lowest $K$-type $W$ is the trivial representation of $K$, hence the character $\zeta$ has to be trivial on $M\cap K$ by \eqref{eq:FrobeniusReciprocity}.
		\item In the case $\frakg=\frakg_{2(2)}$, the lowest $K$-type has to contain a non-zero vector $f\in W$ such that $d\pi_\min(\frakm\cap\frakk)f=0$ and $\pi_\min(w_i^2)f=\zeta(w_i^2)f$, $i=0,1,2$. The first condition implies that $d\pi_\min(S_1)f=0$, and with Proposition~\ref{prop:G2Step5} it follows that $f=\const\times(f_1+f_{-1})$. Acting by $\pi_\min(w_i^2)$, using Corollary~\ref{cor:ActionWeylSquaresLKTG2}, and comparing with $\zeta(w_i^2)$ shows the claim.
		\item In the case $\frakg=\sl(n,\RR)$ with $\pi_\min=\pi_{\min,0,r}$ resp. $\pi_{\min,1,r}$, the lowest $K$-type is the trivial representation $\CC$ of $K=\SO(n)$ resp. the standard representation $\CC^n$ of $K=\SO(n)$. In the first case, it is clear that $\zeta|_{M\cap K}$ must be the trivial representation. In the second case, the lowest $K$-type must contain a non-zero vector $f$ such that $d\pi_\min(\frakm\cap\frakk)f=0$ and $\pi_\min(w_i^2)f=\zeta(w_i^2)f$, $i=0,1,2$. The first condition implies that $f=c_0f_0+c_1f_1$. By Proposition~\ref{prop:ActionWeylGroupSquares} we find
		\begin{align*}
			\pi_{\min,1}(w_0^2)f_0 &= -f_0, & \pi_{\min,1}(w_1^2)f_0 &= -f_0, & \pi_{\min,1}(w_2^2)f_0 &= f_0,\\
			\pi_{\min,1}(w_0^2)f_1 &= -f_1, & \pi_{\min,1}(w_1^2)f_1 &= f_1, & \pi_{\min,1}(w_2^2)f_1 &= -f_1,
		\end{align*}
		which implies $\zeta(w_0^2)=-1$ and either $c_0=0$ or $c_1=0$. The claim follows.
		\item In the case $\frakg=\sl(3,\RR)$ the restriction of the lowest $K$-type $W\simeq\CC^2$ to $M\subseteq K$ is irreducible and two-dimensional.\qedhere
	\end{enumerate}
\end{proof}

\begin{remark}
	For the adjoint group $G=G_{2(2)}$, the subgroup $M$ has two connected components (see e.g. \cite[Section 2]{Kab12}) and since $\chi(w_1^2)=\chi(w_2^2)=-1$, it follows that $w_1^2,w_2^2\in M$ are contained in the non-trivial component. Therefore, Corollary~\ref{cor:RelationZetaImageFT} determines the character $\zeta$ completely in this case. We believe that a similar statement is true for the other cases. Note that even for the non-linear group $G=\widetilde{\SL}(3,\RR)$ for which $M$ has $8$ connected components, the elements $w_0^2,w_1^2,w_2^2$ generate the component group $M/M_0$ and hence any representation $\zeta$ of $M$ is uniquely determined by $\zeta(w_0^2),\zeta(w_1^2),\zeta(w_2^2)$.
\end{remark}

Finally, we are able to describe the precise principal series embedding $\pi_\min\hookrightarrow\pi_{\zeta,\nu}$. Recall that, for $u\in I(\zeta,\nu)^{\Omega_\mu(\frakm)}$ (or the corresponding subrepresentations in the cases $\frakg\simeq\sl(n,\RR)$ and $\so(p,q)$) it was shown that
$$ \widehat{u}(\lambda,x,y) = \xi_{-\lambda,0}(x)u_0(\lambda,y)+\xi_{-\lambda,1}(x)u_1(\lambda,y) $$
for some $u_0,u_1\in\calD'(\RR^\times)\otimeshat\calS'(\Lambda)$. Recall further the map
$$ \Phi_\delta:\calD'(\RR^\times)\otimeshat\calS'(\Lambda)\to\calD'(\RR^\times)\otimeshat\calS'(\Lambda), \quad \Phi_\delta u(\lambda,x) = \sgn(\lambda)^\delta|\lambda|^{-s_\min}u(\lambda,\tfrac{x}{\lambda}),\index{1Phidelta@$\Phi_\delta$} $$
then it was shown in Sections~\ref{sec:FTpictureMinRep} that $u\mapsto\Phi_\delta u_\varepsilon$ is $\frakg$-intertwining from $d\pi_{\zeta,\nu}$ to $d\pi_\min$ for any $\delta,\varepsilon\in\ZZ/2\ZZ$. Note that $u_\varepsilon$ could be zero. We determine for which $\delta,\varepsilon\in\ZZ/2\ZZ$ the map $u\mapsto\Phi_\delta u_\varepsilon$ is a $G$-intertwining isomorphism from $\pi_{\zeta,\nu}$ to $\pi_\min$. Note that the case $\frakg=\sl(n,\RR)$ is excluded since here $\widehat{u}(\lambda,x,y)=u_0(\lambda,y)$, so there is no $\varepsilon\in\ZZ/2\ZZ$ to determine.

\begin{corollary}\label{cor:PSEmbedding}
	Let $u\in I(\zeta,\nu)^{\Omega_\mu(\frakm)}$, then
	$$ \widehat{u}(\lambda,x,y) = \xi_{-\lambda,\varepsilon}(x)u_\varepsilon(\lambda,y) $$
	and the map $u\mapsto\Phi_\delta u_\varepsilon$ is a $G$-intertwining isomorphism from $\pi_{\zeta,\nu}$ to $\pi_\min$, where
	\begin{enumerate}[(1)]
		\item In the quaternionic cases $\frakg=\frake_{6(2)},\frake_{7(-5)},\frake_{8(-24)}$ we have
		$$ \delta=\varepsilon=n=-s_\min-1. $$
		\item In the split cases $\frakg=\frake_{6(6)},\frake_{7(7)},\frake_{8(8)}$ we have
		$$ \delta=\varepsilon=0. $$
		\item In the case $\frakg=\frakg_{2(2)}$ we have
		$$ \delta=\varepsilon=1. $$
	\end{enumerate}
\end{corollary}

\begin{proof}
	Let $u\in I(\zeta,\nu)$ such that
	$$ \widehat{u}(\lambda,x,y) = \xi_{-\lambda,0}(x)u_0(\lambda,y) + \xi_{-\lambda,1}(x)u_1(\lambda,y). $$
	For $m\in M$, we have by Proposition~\ref{prop:ActionFTpicture}
	\begin{multline}
		\widehat{\pi}_{\zeta,\nu}(m)\widehat{u}(\lambda,x,y) = \zeta(m)\Bigg(\omega_{\met,-\lambda}(m)\xi_{-\chi(m)\lambda,0}(x)\cdot\omega_{\met,\lambda}(m)u_0(\chi(m)\lambda,y)\\
		+ \omega_{\met,-\lambda}(m)\xi_{-\chi(m)\lambda,1}(x)\cdot\omega_{\met,\lambda}(m)u_1(\chi(m)\lambda,y)\Bigg).\label{eq:RelationZetaImageFT1}
	\end{multline}
	On the other hand, if $d\rho_\min$ integrates to the group representation $\rho_\min$ it follows that
	\begin{equation}
		\widehat{\pi}_{\zeta,\nu}(m)\widehat{u}(\lambda,x,y) = \xi_{-\lambda,0}(x)\cdot\rho_\min(m)u_0(\lambda,y) + \xi_{-\lambda,1}(x)\cdot\rho_\min(m)u_1(\lambda,y).\label{eq:RelationZetaImageFT2}
	\end{equation}
	We compare \eqref{eq:RelationZetaImageFT1} and \eqref{eq:RelationZetaImageFT2} for $m=w_0^2,w_1^2,w_2^2$. Note that for $m=w_0^2,w_1^2,w_2^2$, the action $\omega_{\met,\lambda}(m)$ can be computed using \eqref{eq:DefMetaplecticRep} and \eqref{eq:DefSchroedingerModel} as well as the adjoint action $\Ad(m)$ on $V=\frakg_{-1}$ which is known by Section~\ref{sec:W0} and Lemma~\ref{lem:W1W2}:
	\begin{align*}
		\omega_{\met,\lambda}(w_0^2)u(a,y) &= \pm u(-a,-y),\\
		\omega_{\met,\lambda}(w_1^2)u(a,y) &= \pm u(a,-y),\\
		\omega_{\met,\lambda}(w_2^2)u(a,y) &= \pm u(-a,y).
	\end{align*}
	The sign disambiguity in the formulas is due to the fact that the metaplectic representation $\omega_{\met,\lambda}$ is a projective representation of $\Sp(V,\omega)$ (or, alternatively, a representation of the metaplectic group $\operatorname{Mp}(V,\omega)$, a double cover of $\Sp(V,\omega)$). Moreover, we have
	$$ \chi(w_0^2)=1 \qquad \mbox{and} \qquad \chi(w_1^2)=\chi(w_2^2)=-1. $$
	Finally, $\rho_\min(m)=\Phi_\delta^{-1}\circ\pi_\min(m)\circ\Phi_\delta$ with $\pi_\min(m)$ given by Proposition~\ref{prop:ActionWeylGroupSquares}.
	\begin{enumerate}[(1)]
		\item In the quaternionic cases $\frakg=\frake_{6(2)},\frake_{7(-5)},\frake_{8(-24)}$, comparing \eqref{eq:RelationZetaImageFT1} and \eqref{eq:RelationZetaImageFT2} shows that
		$$ \zeta(w_0^2)=(-1)^{\varepsilon+n}, \qquad \zeta(w_1^2)=(-1)^\delta, \qquad \zeta(w_2^2)=(-1)^{\delta+\varepsilon+n}, $$
		where $n=-s_\min-1$ and
		\begin{align*}
			\widehat{u}(\lambda,x,y) = \xi_{-\lambda,\varepsilon}(x)u_\varepsilon(y).
		\end{align*}
		Comparing with Corollary~\ref{cor:RelationZetaImageFT} shows $\delta=\varepsilon=n$.
		\item In the split cases $\frakg=\frake_{6(6)},\frake_{7(7)},\frake_{8(8)}$, comparing \eqref{eq:RelationZetaImageFT1} and \eqref{eq:RelationZetaImageFT2} shows that
		$$ \zeta(w_0^2)=(-1)^\varepsilon, \qquad \zeta(w_1^2)=(-1)^\delta, \qquad \zeta(w_2^2)=(-1)^{\delta+\varepsilon}. $$
		Comparing with Corollary~\ref{cor:RelationZetaImageFT} shows $\delta=\varepsilon=0$.
		\item In the case $\frakg=\frakg_{2(2)}$, comparing \eqref{eq:RelationZetaImageFT1} and \eqref{eq:RelationZetaImageFT2} shows that
		$$ \zeta(w_0^2)=(-1)^{\varepsilon+1}, \qquad \zeta(w_1^2)=(-1)^\delta, \qquad \zeta(w_2^2)=(-1)^{\delta+\varepsilon+1}. $$
		Comparing with Corollary~\ref{cor:RelationZetaImageFT} shows $\delta=\varepsilon=1$.\qedhere
	\end{enumerate}
\end{proof}

\appendix

\chapter{The $K$-Bessel function}\label{app:KBessel}

We collect some basic information about the $K$-Bessel function $K_\alpha(z)$ and its renormalization $\overline{K}_\alpha(z)$.

\section{The differential equation and asymptotics}\label{sec:BesselDiffEqAsymptotics}

The differential equation
$$ x^2u''(x)+xu'(x)-(x^2+\alpha^2)u(x)=0 $$
has the two linearly independent solutions $I_\alpha(x)$ and $K_\alpha(x)$\index{K1alphax@$K_\alpha(x)$}. While the $I$-Bessel function $I_\alpha(x)$ grows exponentially as $x\to\infty$, the $K$-Bessel function $K_\alpha(x)=K_{-\alpha}(x)$ has the asymptotics
$$ K_\alpha(x) = \sqrt{\frac{\pi}{2x}}e^{-x}\left(1+\calO\left(\frac{1}{x}\right)\right) \qquad \mbox{as }x\to\infty. $$
Near $x=0$ it behaves as follows:
$$ K_\alpha(x) = \begin{cases}\frac{\Gamma(|\alpha|)}{2}(\frac{x}{2})^{-|\alpha|}+o(x^{-|\alpha|})&\mbox{for }\alpha\neq0,\\-\log(\frac{x}{2})+o(\log(\frac{x}{2}))&\mbox{for }\alpha=0.\end{cases} $$
For our purposes it is more convenient to work with the renormalization
$$ \overline{K}_\alpha(x) = x^{-\frac{\alpha}{2}}K_\alpha(\sqrt{x}).\index{K1alphax@$\overline{K}_\alpha(x)$} $$
Then $\overline{K}_\alpha(x)$ solves the differential equation
\begin{equation}
	zu''+(\alpha+1)u'-\frac{1}{4}u=0.\label{eq:BesselDiffEq}
\end{equation}

\section{Identities}

For the derivative of $K_\alpha(x)$, the following two identities hold:
$$ K_\alpha'(x) = \frac{\alpha}{x}K_\alpha(x)-K_{\alpha+1}(x) = -\frac{\alpha}{x}K_\alpha(x)-K_{\alpha-1}(x). $$
They imply in particular the three term recurrence relation
$$ 2\alpha K_\alpha(x) = x(K_{\alpha+1}(x)-K_{\alpha-1}(x)). $$
These can be reformulated in terms of the renormalization $\overline{K}_\alpha$:
\begin{align}
\overline{K}_\alpha'(x) &= -\frac{1}{2}\overline{K}_{\alpha+1}(x),\label{eq:BesselDerivative1}\\
x\overline{K}_\alpha'(x) &= -\frac{1}{2}\overline{K}_{\alpha-1}(x)-\alpha\overline{K}_\alpha(x),\label{eq:BesselDerivative2}
\end{align}
as well as
\begin{equation}
x\overline{K}_{\alpha+1}(x) = 2\alpha\overline{K}_\alpha(x)+\overline{K}_{\alpha-1}(x).\label{eq:BesselRecurrence}
\end{equation}

\section{Half-integer parameters}
	
For $\alpha\in\NN+\frac{1}{2}$, the $K$-Bessel function degenerates to a product of a polynomial and an exponential function times a power function:
$$ K_\alpha(x) = \sqrt{\frac{\pi}{2}}\frac{e^{-x}}{\sqrt{x}}\sum_{j=0}^{|\alpha|-\frac{1}{2}}\frac{(j+|\alpha|-\frac{1}{2})!}{j!(-j+|\alpha|-\frac{1}{2})!}(2x)^{-j}. $$
For the renormalized function $\overline{K}_\alpha(x)$ with $\alpha=-\frac{1}{2}$ this implies
\begin{equation}
\overline{K}_{-\frac{1}{2}}(x) = \sqrt{\frac{\pi}{2}}e^{-\sqrt{x}}\label{eq:KBessel1/2}
\end{equation}
and for $\alpha=\frac{1}{2}+n$
\begin{equation}
\overline{K}_{n+1/2}(x) = \sqrt{\frac{\pi}{2}}x^{-n-\frac{1}{2}}2^{-n}e^{-\sqrt{x}}\sum_{j=0}^n\frac{(2n-j)!}{j!(n-j)!}(2\sqrt{x})^j.\label{eq:KBesselHalfInt}
\end{equation}

\chapter{Examples}

For a few Heisenberg-graded Lie algebras $\frakg$, we provide explicit information about the structure of the grading, the symplectic vector space $V$ and its invariants.

\section{$\frakg=\sl(n,\RR)$}\label{app:SLn}

Let $G=\SL(n,\RR)$ and $\frakg=\sl(n,\RR)$. Put
$$ E = \begin{pmatrix}0&0&1\\&\0_{n-2}&0\\&&0\end{pmatrix}, \qquad F=\begin{pmatrix}0&&\\0&\0_{n-2}&\\1&0&0\end{pmatrix}, \qquad H = \begin{pmatrix}1&&\\&\0_{n-2}&\\&&-1\end{pmatrix}, $$
then $\ad(H)$ has eigenvalues $0,\pm1,\pm2$ on $\frakg$ and, in the above block notation:
$$ \frakg_{-2} = \RR F, \quad \frakg_{-1} = \begin{pmatrix}0&&\\\star&0&\\0&\star&0\end{pmatrix}, \quad \frakg_0 = \begin{pmatrix}\star&&\\&\star&\\&&\star\end{pmatrix}, \quad \frakg_1 = \begin{pmatrix}0&\star&0\\&0&\star\\&&0\end{pmatrix}, \quad \frakg_2 = \RR E. $$
Further,
$$ M = \left\{\begin{pmatrix}a&&\\&g&\\&&b\end{pmatrix}:g\in\GL(n-2,\RR),|a|=|b|,ab\det(g)=1\right\}. $$
We parameterize $\RR^{n-2}\times\RR^{n-2}\simeq\frakg_{-1}$ by
$$ (x,y)\mapsto\begin{pmatrix}0&&\\x&\0_n&\\0&y^\top&0\end{pmatrix}, $$
then $\omega((x,y),(x',y'))=x'^\top y-x^\top y'$ and
$$ \mu(x,y) = \begin{pmatrix}\frac{x^\top y}{2}&&\\&-xy^\top&\\&&\frac{x^\top y}{2}\end{pmatrix}, \qquad \Psi(x,y) = \frac{x^\top y}{2}(x,-y), \qquad Q(x,y) = \frac{(x^\top y)^2}{4}. $$
The choice of Cartan involution $\theta(X)=-X^\top$ gives
$$ J(x,y) = (-y,x) \qquad \mbox{and} \qquad ((x,y)|(x',y')) = \frac{1}{4}(x^\top x'+y^\top y'). $$

\section{$\frakg=\so(p,q)$}\label{app:SOpq}

Let $G=\SO_0(p,q)$, where $\SO(p,q)=\{g\in\SL(p+q,\RR\}:g^\top\1_{p,q}g=\1_{p,q}\}$ with $\1_{p,q}=\diag(\1_p,-\1_q)$. Then $K=\SO(p)\times\SO(q)$. Put
$$ H = \left(\begin{array}{ccc}&&\1_2\\&\0_{p+q-4}&\\\1_2&&\end{array}\right), $$
then $\ad(H)$ has eigenvalues $0,\pm1,\pm2$ on $\frakg=\so(p,q)$ and
\begin{align*}
\frakg_0 &= \left\{\left(\begin{array}{ccccc}0&x&&a&b\\-x&0&&b&d\\&&T&&\\a&b&&0&x\\b&d&&-x&0\end{array}\right):x,a,b,d\in\RR,T\in\so(p-2,q-2)\right\},\\
\frakg_{\pm1} &= \left\{\left(\begin{array}{cccc}&V&\mp W^\top&\\-V^\top&&&\pm V^\top\\\mp W&&&W\\&\pm V&-W^\top&\end{array}\right):\begin{array}{c}V\in M(2\times(p-2),\RR),\\W\in M((q-2)\times2,\RR)\end{array}\right\},\\
\frakg_{\pm2} &= \left\{\left(\begin{array}{ccccc}0&x&&0&\mp x\\-x&0&&\pm x&0\\&&\0_{p+q-4}&&\\0&\pm x&&0&-x\\\mp x&0&&x&0\end{array}\right):x\in\RR\right\}.
\end{align*}
The map
$$ \sl(2,\RR) \to \frakm, \quad \begin{pmatrix}a&b\\c&-a\end{pmatrix}\mapsto\begin{pmatrix}0&\frac{b-c}{2}&&a&\frac{b+c}{2}\\-\frac{b-c}{2}&&&\frac{b+c}{2}&-a\\&&\0_{p+q-4}&&\\a&\frac{b+c}{2}&&0&\frac{b-c}{2}\\\frac{b+c}{2}&-a&&-\frac{b-c}{2}&0\end{pmatrix}, $$
is an isomorphism onto an ideal $\frakm_0\simeq\sl(2,\RR)$ of $\frakm$. Further,
$$ M = \left\{\begin{pmatrix}\frac{g+g^{-\top}}{2}&&\frac{g-g^{-\top}}{2}\\&h&\\\frac{g-g^{-\top}}{2}&&\frac{g+g^{-\top}}{2}\end{pmatrix}:g\in\SL^\pm(2,\RR),h\in\SO(p-2,q-2),\det(g)=\chi(h)\right\}, $$
where $\chi:\SO(p-2,q-2)\to\{\pm1\}$ is the non-trivial character of $\SO(p-2,q-2)$.

We identify $V=\frakg_{-1}\simeq\RR^{2\times(p-2)}\times\RR^{(q-2)\times2}$ by mapping a matrix of the above form to $(V,W)$. We choose
$$ E=\frac{1}{2}\begin{pmatrix}J&&-J\\&\0_{p+q-4}&\\J&&-J\end{pmatrix} \qquad \mbox{and} \qquad F=-\frac{1}{2}\begin{pmatrix}J&&J\\&\0_{p+q-4}&\\-J&&-J\end{pmatrix}, $$
where
\begin{equation}
	J = \begin{pmatrix}0&1\\-1&0\end{pmatrix},\label{eq:DefSympMatrixJ}
\end{equation}
then
$$ \omega((V,W),(V',W')) = 2\big(v_1^\top v_2'-v_1'^\top v_2-w_1^\top w_2'+w_1'^\top w_2\big), $$
where $V=(v_1,v_2)^\top$ and $W=(w_1,w_2)$ with $v_1,v_2\in\RR^{p-2}$, $w_1,w_2\in\RR^{q-2}$. Further,
$$ \mu(V,W) = \begin{pmatrix}&-a&&&-b&c\\a&&&&c&b\\&&2V^\top JV&2V^\top JW^\top&&\\&&-2WJV&-2WJW^\top&&\\-b&c&&&&-a\\c&b&&&a&\end{pmatrix}, $$
with
$$ a = \frac{|v_1|^2+|v_2|^2-|w_1|^2-|w_2|^2}{2}, \quad b = v_1^\top v_2-w_1^\top w_2, \quad c = \frac{|v_1|^2-|v_2|^2-|w_1|^2+|w_2|^2}{2}, $$
and
\begin{align*}
\Psi(V,W) ={}& \big((VV^\top-W^\top W)JV,WJ(W^\top W-VV^\top)\big),\\
={}& \big((w_1^\top w_2-v_1^\top v_2)v_1+(|v_1|^2-|w_1|^2)v_2,(|w_2|^2-|v_2|^2)v_1+(v_1^\top v_2-w_1^\top w_2)v_2,\\
& \quad(w_1^\top w_2-v_1^\top v_2)w_1+(|v_1|^2-|w_1|^2)w_2,(|w_2|^2-|v_2|^2)w_1+(v_1^\top v_2-w_1^\top w_2)w_2\big),\\
Q(V,W) ={}& -(|v_1|^2-|w_1|^2)(|v_2|^2-|w_2|^2) + (v_1^\top v_2-w_1^\top w_2)^2.
\end{align*}

For $p,q>2$ the group $G$ is non-Hermitian, and we choose
$$ O = ((v_0,0),(0,w_0)) $$
for some fixed $v_0\in\RR^{p-2}$, $w_0\in\RR^{q-2}$ with $|v_0|=|w_0|=1$. Then
$$ A = \frac{1}{2}((v_0,-v_0),(-w_0,w_0)) \qquad \mbox{and} \qquad B = \frac{1}{2}((v_0,v_0),(w_0,w_0)), $$
and
\begin{align*}
\calJ &= \{((\lambda v_0+a,\lambda v_0-a),(-\lambda w_0-b,-\lambda w_0+b)):\lambda\in\RR,v_0^\top a=w_0^\top b=0\},\\
\calJ^* &= \{((\lambda v_0+a,-\lambda v_0+a),(\lambda w_0+b,-\lambda w_0+b)):\lambda\in\RR,v_0^\top a=w_0^\top b=0\}.
\end{align*}
Identifying $\calJ$ with $\RR\times v_0^\perp\times w_0^\perp\subseteq\RR\times\RR^{p-2}\times\RR^{q-2}$ by mapping $((\lambda v_0+a,\lambda v_0-a),(-\lambda w_0-b,-\lambda w_0+b))$ to $(\lambda,a,b)$, we find that
$$ n(\lambda,a,b) = 4\lambda(|a|^2-|b|^2). $$

We choose
$$ P=\tfrac{1}{2\sqrt{2}}((v_0,v_0),(-w_0,-w_0)) \qquad \mbox{and} \qquad Q=\tfrac{1}{2\sqrt{2}}((-v_0,v_0),(-w_0,w_0)), $$
then $\omega(P,Q)=1$ and $\calJ=\RR P\oplus\overline{\calJ}$ with
$$ \overline{\calJ} = \{((a,-a),(-b,b)):v_0^\top a=w_0^\top b=0\} \simeq \RR^{p-3}\times\RR^{q-3}. $$
We further choose $\vartheta:\calJ\to\calJ$ to be
$$ \vartheta((a,-a),(-b,b)) = ((a,-a),(b,-b)), $$
then the corresponding map $J:V\to V$ in Proposition~\ref{prop:CartanInvSOpq} is given by
$$ J(V,W) = (JV,WJ) $$
with $J$ on the right hand side as in \eqref{eq:DefSympMatrixJ}.

\section{$\frakg=\frakg_{2(2)}$}\label{app:ExG2}

The structure theory of $\frakg=\frakg_{2(2)}$ is treated in detail in \cite{SS12}. In this case, $V$ can be identified with the space of binary cubics
$$ V = S^3(\RR^2) = \{p=aX^3+3bX^2Y+3cXY^2+dY^3:a,b,c,d\in\RR\} $$
with symplectic form
$$ \omega(p,p') = ad'-da'-3bc'+3cb' $$
and
$$ Q(p) = \frac{1}{4}(a^2d^2-3b^2c^2-6abcd+4b^3d+4ac^3). $$
The action of $\frakm\simeq\sl(2,\RR)$ on $V=S^3(\RR^2)$ is induced by the natural action of $\sl(2,\RR)$ on $\RR^2$. One possible choice of $A$, $B$, $C$ and $D$ is
$$ A = \sqrt{2}X^3, \quad B = \sqrt{2}Y^3, \quad C = -\frac{3}{\sqrt{2}}X^2Y, \quad D = \frac{3}{\sqrt{2}}XY^2. $$

\chapter{A meromorphic family of distributions}\label{app:MeromFamily}

Let $\frakg$ be non-Hermitian and $\frakg\not\simeq\sl(n,\RR),\so(p,q)$. In Theorem~\ref{thm:InvDistributionVector}, it is shown that the space of $\frakm$-invariant distribution vectors in the metaplectic representation $(\omega_{\met,\lambda},L^2(\Lambda))$ is two-dimensional and spanned by the two distributions
$$ \xi_{\lambda,\varepsilon}(a,x) = \sgn(a)^\varepsilon|a|^{s_\min}e^{-i\lambda\frac{n(x)}{a}} \qquad (\varepsilon\in\ZZ/2\ZZ). $$
It is a priori not clear that this formula defines a distribution on $\Lambda$, since both $|a|^{s_\min}$ and $e^{-i\lambda\frac{n(x)}{a}}$ have a singularity at $a=0$ and $|a|^{s_\min}$ is only locally integrable for $s_\min>-1$. In this section, we show that $\xi_{\lambda,\varepsilon}$ is the special value of a meromorphic family of distributions at a regular point.

Fix $\lambda\in\RR^\times$ and $\varepsilon\in\ZZ/2\ZZ$. For $s\in\CC$ put
$$ \psi_{s,\varepsilon}(a,x) = \sgn(a)^\varepsilon|a|^se^{-i\lambda\frac{n(x)}{a}}.\index{1ypsisepsilon@$\psi_{s,\varepsilon}$} $$
Then $\psi_{s,\varepsilon}=\xi_{\lambda,\varepsilon}$ for $s=s_\min=-\frac{1}{6}(\dim\Lambda+2)$. For $\Re(s)>-1$ the function $\psi_{s,\varepsilon}$ is locally integrable and of at most polynomial growth, and hence defines a tempered distribution $\psi_{s,\varepsilon}\in\calS'(\Lambda)$. We show that $\psi_{s,\varepsilon}$ extends meromorphically to $s\in\CC$ and that $s=s_\min$ is not a pole of this meromorphic extension.

\section{Some preliminary formulas}

Let $(e_\alpha)_\alpha$ be a basis of $\calJ=\frakg_{(0,-1)}$ and denote by $(\widehat{e}_\alpha)_\alpha$ the dual basis of $\frakg_{(-1,0)}$ with respect to the symplectic form $\omega$, i.e. $\omega(e_\alpha,\widehat{e}_\beta)=\delta_{\alpha\beta}$.

\begin{lemma}\label{lem:MeromFam1}
For all $x\in\calJ$:
$$ \Psi(\mu(x)B) = 4n(x)^2B. $$
\end{lemma}

\begin{proof}
By the $\frakm$-invariance of $\Psi$ and Lemma~\ref{lem:MuSquared}, we have
\begin{align*}
	\Psi(\mu(x)B) ={}& B_\Psi(\mu(x)B,\mu(x)B,\mu(x)B)\\
	={}& \mu(x)B_\Psi(B,\mu(x)B,\mu(x)B)-2B_\Psi(B,\mu(x)^2B,\mu(x)B)\\
	={}& \mu(x)B_\Psi(B,\mu(x)B,\mu(x)B)+8n(x)B_\Psi(B,x,\mu(x)B)\\
	={}& \frac{1}{2}\mu(x)\Big(\mu(x)B_\Psi(B,B,\mu(x)B)-B_\Psi(B,B,\mu(x)^2B)\Big)\\
	& \qquad\qquad+4n(x)\Big(\mu(x)B_\Psi(B,x,B)-B_\Psi(B,\mu(x)x,B)\Big).
\end{align*}
From the bigrading, it follows that
$$ B_\Psi(B,B,\mu(x)B)=0 \qquad \mbox{and} \qquad B_\Psi(B,B,\mu(x)^2B)=-4n(x)B_\Psi(B,B,x)=0. $$
Further, $\mu(x)x=-3\Psi(x)=-3n(x)A$ and $B_\Psi(A,B,B)=\frac{1}{3}B$ and the claim follows.
\end{proof}

\begin{lemma}\label{lem:MeromFam2}
For all $x\in\calJ$:
$$ \sum_\alpha B_\mu(A,\widehat{e}_\alpha)B_\mu(e_\alpha,B)x = \left(\frac{1}{2}+\frac{\dim\frakg_{(0,-1)}}{6}\right)x. $$
\end{lemma}

\begin{proof}
Let $y\in\frakg_{(-1,0)}$, then by Lemma~\ref{lem:TraceOnG0-1}:
\begin{align*}
	\omega\left(\sum_\alpha B_\mu(A,\widehat{e}_\alpha)B_\mu(e_\alpha,B)x,y\right) &= -\sum_\alpha\omega(B_\mu(B,e_\alpha)x,B_\mu(A,\widehat{e}_\alpha)y)\\
	&= -\sum_\alpha\omega(B_\mu(B,x)e_\alpha,B_\mu(A,y)\widehat{e}_\alpha)\\
	&= \tr(B_\mu(A,y)\circ B_\mu(x,B)|_{\frakg_{(0,-1)}})\\
	&= \left(\frac{1}{2}+\frac{\dim\frakg_{(0,-1)}}{6}\right)\omega(x,y).
\end{align*}
Since $y$ was arbitrary and the symplectic form is non-degenerate on $\frakg_{(0,-1)}\times\frakg_{(-1,0)}$, the desired identity follows.
\end{proof}

\begin{lemma}\label{lem:MeromFam3}
For all $x\in\calJ$:
$$ \sum_\alpha B_\Psi(\mu(x)B,B_\mu(x,e_\alpha)B,\widehat{e}_\alpha) = \left(\frac{1}{2}+\frac{\dim\frakg_{(0,-1)}}{6}\right)n(x)B. $$
\end{lemma}

\begin{proof}
It follows from the bigrading that the left hand side is contained in $\frakg_{(-2,1)}$ and hence a scalar multiple of $B$. Therefore, it suffices to compute the quantity
\begin{multline*}
	\omega\left(A,\sum_\alpha B_\Psi(\mu(x)B,B_\mu(x,e_\alpha)B,\widehat{e}_\alpha)\right) = \sum_\alpha \omega(\mu(x)B,B_\Psi(A,\widehat{e}_\alpha,B_\mu(x,e_\alpha)B))\\
	= -\frac{1}{3}\sum_\alpha\omega(\mu(x)B,B_\mu(A,\widehat{e}_\alpha)B_\mu(x,e_\alpha)B)-\frac{1}{6}\sum_\alpha\omega(\mu(x)B,B_\tau(A,\widehat{e}_\alpha)B_\mu(x,e_\alpha)B).
\end{multline*}
The second sum vanishes since
$$ B_\tau(A,\widehat{e}_\alpha)B_\mu(x,e_\alpha)B = \frac{1}{2}\omega(A,B_\mu(x,e_\alpha)B)\widehat{e}_\alpha+\frac{1}{2}\omega(\widehat{e}_\alpha,B_\mu(x,e_\alpha)B)A $$
and $\omega(A,\frakg_{(-1,0)})=\omega(\frakg_{(-1,0)},\frakg_{(-1,0)})=0$. For the first sum, we have by Lemma~\ref{lem:MeromFam2}:
\begin{align*}
	\sum_\alpha\omega(\mu(x)B,B_\mu(A,\widehat{e}_\alpha)B_\mu(e_\alpha,x)B) &= \sum_\alpha\omega(\mu(x)B,B_\mu(A,\widehat{e}_\alpha)B_\mu(e_\alpha,B)x)\\
	&= \left(\frac{1}{2}+\frac{\dim\frakg_{(0,-1)}}{6}\right)\omega(\mu(x)B,x)\\
	&= -\left(3+\dim\frakg_{(0,-1)}\right)n(x),
\end{align*}
so the claim follows.
\end{proof}

\begin{lemma}\label{lem:MeromFam4}
For all $x\in\calJ$:
$$ \sum_{\alpha,\beta} B_\Psi(\widehat{e}_\alpha,\widehat{e}_\beta,B_\mu(e_\alpha,e_\beta)B) = \frac{\dim\frakg_{(0,-1)}}{6}\left(\frac{1}{2}+\frac{\dim\frakg_{(0,-1)}}{6}\right)B. $$
\end{lemma}

\begin{proof}
From the bigrading it follows that the left hand side is contained in $\frakg_{(-2,1)}$ and hence has to be a scalar multiple of $B$. Therefore, it suffices to compute
\begin{multline*}
	\omega\left(A,\sum_{\alpha,\beta} B_\Psi(\widehat{e}_\alpha,\widehat{e}_\beta,B_\mu(e_\alpha,e_\beta)B)\right) = \sum_{\alpha,\beta} \omega(\widehat{e}_\beta,B_\Psi(A,\widehat{e}_\alpha,B_\mu(e_\alpha,B)e_\beta)\\
	= -\frac{1}{3}\sum_{\alpha,\beta} \omega(\widehat{e}_\beta,B_\mu(A,\widehat{e}_\alpha)B_\mu(e_\alpha,B)e_\beta) - \frac{1}{6}\sum_{\alpha,\beta} \omega(\widehat{e}_\beta,B_\tau(A,\widehat{e}_\alpha)B_\mu(e_\alpha,B)e_\beta).
\end{multline*}
The second sum vanishes as in the proof of Lemma~\ref{lem:MeromFam3}, and the first sum evaluates by Lemma~\ref{lem:MeromFam2} to
\begin{align*}
	-\frac{1}{3}\sum_{\alpha,\beta} \omega(\widehat{e}_\beta,B_\mu(A,\widehat{e}_\alpha)B_\mu(e_\alpha,B)e_\beta) &= -\frac{1}{3}\left(\frac{1}{2}+\frac{\dim\frakg_{(0,-1)}}{6}\right)\sum_\beta\omega(\widehat{e}_\beta,e_\beta)\\
	&= \frac{\dim\frakg_{(0,-1)}}{3}\left(\frac{1}{2}+\frac{\dim\frakg_{(0,-1)}}{6}\right).\qedhere
\end{align*}
\end{proof}

\section{A Bernstein--Sato identity}

We now show a Bernstein--Sato identity which expresses $\psi_{s,\varepsilon}$ in terms of $\psi_{t,\delta}$ for $t\in\{s+1,s+2\}$, $\delta\in\ZZ/2\ZZ$, and hence can be used to meromorphically extend $\psi_{s,\varepsilon}$ to $s\in\CC$.

\begin{lemma}\label{lem:MeromFam5}
We have the following formulas for derivatives of $\psi_{s,\varepsilon}$:
\begin{align*}
	\partial_A\psi_{s,\varepsilon} ={}& s\psi_{s-1,\varepsilon+1}+i\lambda n(x)\psi_{s-2,\varepsilon},\\
	\partial_A^2\psi_{s,\varepsilon} ={}& s(s-1)\psi_{s-2,\varepsilon}+2i(s-1)\lambda n(x)\psi_{s-3,\varepsilon+1}-\lambda^2n(x)^2\psi_{s-4,\varepsilon},\\
	\omega(A,\Psi(\tfrac{\partial}{\partial x}))\psi_{s,\varepsilon} = {}& i\lambda^3n(x)^2\psi_{s-3,\varepsilon+1} - 3\left(\tfrac{1}{2}+\tfrac{\dim\frakg_{(0,-1)}}{6}\right)\lambda^2n(x)\psi_{s-2,\varepsilon}\\
	& \qquad\qquad\qquad\qquad-i\lambda\tfrac{\dim\frakg_{(0,-1)}}{3}\left(\tfrac{1}{2}+\tfrac{\dim\frakg_{(0,-1)}}{6}\right)\psi_{s-1,\varepsilon+1}.
\end{align*}
\end{lemma}

\begin{proof}
The first two identities follow by direct computation. For the third identity, note that
$$ \omega(A,\Psi(\tfrac{\partial}{\partial x}))\psi_{s,\varepsilon} = \sum_{\alpha,\beta,\gamma} \omega(A,B_\Psi(\widehat{e}_\alpha,\widehat{e}_\beta,\widehat{e}_\gamma))\partial_{e_\alpha}\partial_{e_\beta}\partial_{e_\gamma}\psi_{s,\varepsilon}. $$
To compute the derivatives of $\psi_{s,\varepsilon}$, we use formula \eqref{eq:DerivativeOfN} for the derivatives of $n(x)$ and find
\begin{align*}
	\partial_{e_\alpha}\psi_{s,\varepsilon} ={}& \tfrac{i\lambda}{2}\omega(\mu(x)e_\alpha,B)\psi_{s-1,\varepsilon+1},\\
	\partial_{e_\alpha}\partial_{e_\beta}\psi_{s,\varepsilon} ={}& -\tfrac{\lambda^2}{4}\omega(\mu(x)e_\alpha,B)\omega(\mu(x)e_\beta,B)\psi_{s-2,\varepsilon}+i\lambda\omega(B_\mu(x,e_\beta)e_\alpha,B)\psi_{s-1,\varepsilon+1},\\
	\partial_{e_\alpha}\partial_{e_\beta}\partial_{e_\gamma}\psi_{s,\varepsilon} ={}& -\tfrac{i\lambda^3}{8}\omega(\mu(x)e_\alpha,B)\omega(\mu(x)e_\beta,B)\omega(\mu(x)e_\gamma,B)\psi_{s-3,\varepsilon+1}\\
	& -\tfrac{\lambda^2}{2}\omega(B_\mu(x,e_\alpha)e_\beta,B)\omega(\mu(x)e_\gamma,B)\psi_{s-2,\varepsilon}\\
	& -\tfrac{\lambda^2}{2}\omega(B_\mu(x,e_\alpha)e_\gamma,B)\omega(\mu(x)e_\beta,B)\psi_{s-2,\varepsilon}\\
	& -\tfrac{\lambda^2}{2}\omega(B_\mu(x,e_\beta)e_\gamma,B)\omega(\mu(x)e_\alpha,B)\psi_{s-2,\varepsilon}\\
	& +i\lambda\omega(B_\mu(e_\alpha,e_\beta)e_\gamma,B)\psi_{s-1,\varepsilon+1}.
\end{align*}
Then
\begin{align*}
	\omega(A,\Psi(\tfrac{\partial}{\partial x}))\psi_{s,\varepsilon} ={}& \tfrac{i\lambda^3}{8}\omega(A,\Psi(\mu(x)B))\psi_{s-3,\varepsilon+1}\\
	& -\tfrac{3\lambda^2}{2}\sum_{\alpha}\omega(A,B_\Psi(\widehat{e}_\alpha,B_\mu(x,e_\alpha)B,\mu(x)B))\psi_{s-2,\varepsilon}\\
	& -i\lambda\sum_{\alpha,\beta}\omega(A,B_\Psi(\widehat{e}_\alpha,\widehat{e}_\beta,B_\mu(e_\alpha,e_\beta)B))\psi_{s-1,\varepsilon+1},
\end{align*}
and evaluating the three terms with Lemma~\ref{lem:MeromFam1}, Lemma~\ref{lem:MeromFam3} and Lemma~\ref{lem:MeromFam4} proves the third identity.
\end{proof}

Combining the derivatives in the previous lemma immediately shows:

\begin{proposition}[Bernstein--Sato identity]
The following Bernstein--Sato identity holds:
\begin{multline*}
	\omega(A,\Psi(\tfrac{\partial}{\partial x}))\psi_{s+1,\varepsilon+1} + i\lambda\partial_A^2\psi_{s+2,\varepsilon}-\left(3\left(\tfrac{1}{2}+\tfrac{\dim\frakg_{(0,-1)}}{6}\right)+2(s+1)\right)i\lambda\partial_A\psi_{s+1,\varepsilon+1}\\
	= -\left(s+\tfrac{\dim\frakg_{(0,-1)}}{6}+\tfrac{3}{2}\right)\left(s+\tfrac{\dim\frakg_{(0,-1)}}{3}+1\right)i\lambda\psi_{s,\varepsilon}.
\end{multline*}
As a consequence, $\psi_{s,\varepsilon}$ extends to a meromorphic family of distributions which is regular in the right half plane $\{s\in\CC:\Re s>-\min(\frac{3}{2}+\frac{\dim\frakg_{(0,-1)}}{6},1+\frac{\dim\frakg_{(0,-1)}}{3})\}$. In particular, $\psi_{s,\varepsilon}$ is regular at $s=s_\min=-(\frac{1}{2}+\frac{\dim\frakg_{(0,-1)}}{6})$.
\end{proposition}

\chapter{Tables}\label{app:Tables}

The following classification of simple real Lie algebras with Heisenberg parabolic subalgebras is due to Cheng~\cite{Che87}. We also include the maximal compact subalgebra $\frakk$, the Levi factor $\frakm$ of the maximal parabolic subgroup, half the dimension of the symplectic vector space $\frakg_1$ and, whenever it exists, the isomorphism class of the Jordan algebra $\calJ$ studied in Section~\ref{sec:Bigrading}.
\begin{table}[ht]
\centering
\begin{tabular}{llllllll}
	\hline
	Type & $\frakg$ & $\frakk$ & $\frakm$ & $\frac{1}{2}\dim\frakg_1$ & $\calJ$\\
	\hline\hline
	AI & $\sl(n,\RR)$ & $\so(n)$ & $\gl(n-2,\RR)$ & $n-2$ & $-$\\
	AIII & $\su(p,q)$ & $\fraks(\fraku(p)\oplus\fraku(q))$ & $\fraku(p-1,q-1)$ & $p+q-2$ & $-$\\
	\hline
	BDI & $\so(p,q)$ & $\so(p)\oplus\so(q)$ & $\so(p-2,q-2)\oplus\sl(2,\RR)$ & $p+q-4$ & $\RR\oplus\RR^{p-3,q-3}$\\
	\hline
	CI & $\sp(n,\RR)$ & $\fraku(n)$ & $\sp(n-1,\RR)$ & $n-1$ & $-$\\
	\hline
	DIII & $\so^*(2n)$ & $\fraku(n)$ & $\so^*(2n-4)\oplus\su(2)$ & $2n-4$ & $-$\\
	\hline
	EI & $\frake_{6(6)}$ & $\sp(4)$ & $\sl(6,\RR)$ & $10$ & $\Herm(3,\CC_s)$\\
	EII & $\frake_{6(2)}$ & $\su(6)\oplus\su(2)$ & $\su(3,3)$ & $10$ & $\Herm(3,\CC)$\\
	EIII & $\frake_{6(-14)}$ & $\so(10)\oplus\fraku(1)$ & $\su(1,5)$ & $10$ & $-$\\
	\hline
	EV & $\frake_{7(7)}$ & $\su(8)$ & $\so(6,6)$ & $16$ & $\Herm(3,\HH_s)$\\
	EVI & $\frake_{7(-5)}$ & $\so(12)\oplus\su(2)$ & $\so^*(12)$ & $16$ & $\Herm(3,\HH)$\\
	EVII & $\frake_{7(-25)}$ & $\frake_6\oplus\fraku(1)$ & $\so(2,10)$ & $16$ & $-$\\
	\hline
	EVIII & $\frake_{8(8)}$ & $\so(16)$ & $\frake_{7(7)}$ & $28$ & $\Herm(3,\OO_s)$\\
	EIX & $\frake_{8(-24)}$ & $\frake_7\oplus\su(2)$ & $\frake_{7(-25)}$ & $28$ & $\Herm(3,\OO)$\\
	\hline
	FI & $\frakf_{4(4)}$ & $\sp(3)\oplus\su(2)$ & $\sp(3,\RR)$ & $7$ & $\Herm(3,\RR)$\\
	\hline
	G & $\frakg_{2(2)}$ & $\su(2)\oplus\su(2)$ & $\sl(2,\RR)$ & $2$ & $\RR$\\
	\hline
\end{tabular}
\caption{Simple real Lie algebras possessing a Heisenberg parabolic subalgebra}
\label{tab:Classification}
\end{table}

\newpage

We further list the constants $\calC(\frakm')$\index{Cmprime@$\calC(\frakm')$} for each case (see \cite[\S8.10]{BKZ08}). Note that $\calC(\frakm')$ only depends on the complexifications $\frakg_\CC$ of $\frakg$ and $\frakm'_\CC$ of $\frakm'$.

\begin{table}[ht]
\centering
\begin{tabular}{llllllll}
	\hline
	$\frakg_\CC$ & $\frakm_\CC$ & $\calC(\frakm'_\CC)$\\
	\hline\hline
	$\sl(n,\CC)$ & $\sl(n-2,\CC)\oplus\gl(1,\CC)$ & $\calC(\sl(n-2,\CC))=1$\\
	&& $\calC(\gl(1,\CC))=\frac{n}{2}$\\
	$\so(n,\CC)$ & $\so(n-4,\CC)\oplus\sl(2,\CC)$ & $\calC(\so(n-4,\CC))=2$\\
	&& $\calC(\sl(2,\CC))=\frac{n-4}{2}$\\
	$\sp(n,\CC)$ & $\sp(n-1,\CC)$ & $\frac{1}{2}$\\
	$\frake_6(\CC)$ & $\sl(6,\CC)$ & $3$\\
	$\frake_7(\CC)$ & $\so(12,\CC)$ & $4$\\
	$\frake_8(\CC)$ & $\frake_7(\CC)$ & $6$\\
	$\frakf_4(\CC)$ & $\sp(3,\CC)$ & $5/2$\\
	$\frakg_2(\CC)$ & $\sl(2,\CC)$ & $5/3$\\
	\hline
\end{tabular}
\caption{Special values $\calC(\frakm')$}
\label{tab:Cvalues}
\end{table}

\printindex


\providecommand{\bysame}{\leavevmode\hbox to3em{\hrulefill}\thinspace}
\providecommand{\MR}{\relax\ifhmode\unskip\space\fi MR }
\providecommand{\MRhref}[2]{%
	\href{http://www.ams.org/mathscinet-getitem?mr=#1}{#2}
}
\providecommand{\href}[2]{#2}

\end{document}